\documentclass[12pt]{amsart}
\usepackage{amssymb}
\usepackage{fullpage} 
\usepackage[Symbol]{upgreek}
\usepackage{textcomp}
\usepackage[mathscr]{euscript}
\usepackage{pb-diagram}
\usepackage{ wasysym }
\usepackage{graphicx}
\usepackage{enumitem}
\usepackage{comment}
\usepackage{xspace}
\usepackage{xcolor}
\usepackage{appendix}

\DeclareFontFamily{U}{fsy}{}
\DeclareFontShape{U}{fsy}{m}{n}{<->s*[.9]psyr}{}
\DeclareSymbolFont{der@m}{U}{fsy}{m}{n}
\DeclareMathSymbol{\der}{\mathord}{der@m}{182}

\newcommand{\tshift}{T_\mathrm{shift}}
\newcommand{\unisltwo}{\widetilde{\mathrm{Sl}}_2(\R)}
\newcommand{\sltwo}{\mathrm{Sl}_2(\R)}
\newcommand{\psltwo}{\mathrm{PSl}_2(\R)}
\newcommand{\cyc}{\mathrm{Cyc}}

\newcommand{\simplesquare}{\hspace{.1cm}\square_i\hspace{.1cm}}
\newcommand{\bin}{\mathrm{Bin}}
\newcommand{\V}{\mathbb{V}}
\newcommand{\D}{\mathbb{D}}

\newcommand{\pac}{\mathrm{PAC}}
\newcommand{\prc}{\mathrm{PRC}}
\newcommand{\zerozero}{{00}}
\newcommand{\gqot}{G/G^\zerozero}
\newcommand{\dlo}{\mathrm{DLO}}
\newcommand{\alg}{\mathrm{alg}}

\newcommand{\dloeq}{\dlo_{\mathrm{eq}}}
\newcommand{\dprk}{\operatorname{dp}}
\newcommand{\opd}{\operatorname{opd}}
\newcommand{\lex}{<_{\mathrm{Lex}}}
\newcommand{\hyp}{\mathrm{Hyp}}
\newcommand{\relk}{\mathrm{Rel}}
\newcommand{\cplus}{(2^{\aleph_0})^+}

\DeclareSymbolFont{imag@m}{OT1}{cmr}{m}{ui}
\DeclareMathSymbol{\imag}{\mathord}{imag@m}{105}

\newtheorem{theorem}{Theorem}[section]
\newtheorem*{theorem*}{Theorem}
\newtheorem*{wk}{Weak Conjecture}
\newtheorem*{sk}{Strong Conjecture}

\newtheorem{proposition}[theorem]{Proposition}
\newtheorem{lem}[theorem]{Lemma}

\newtheorem{Claim}{Claim}
\newtheorem*{Claim*}{Claim}
\newtheorem{fact}[theorem]{Fact}

\newtheorem{lemma}[theorem]{Lemma}
\newtheorem{corollary}[theorem]{Corollary}
\newtheorem{conj}{Conjecture}

\theoremstyle{definition}

\theoremstyle{remark}

\newcommand{\monster}{\boldsymbol{\Sa M}}
\newcommand{\ionster}{\boldsymbol{\Sa I}}

\newcommand{\nonster}{\boldsymbol{\Sa N}}
\newcommand{\oonster}{\boldsymbol{\Sa O}}

\newcommand{\monsterset}{\boldsymbol{M}}
\newcommand{\nonsterset}{\boldsymbol{N}}
\newcommand{\oonsterset}{\boldsymbol{O}}

\newcommand{\tp}{\operatorname{tp}}
\newcommand{\qftp}{\operatorname{qftp}}
\newcommand{\mfrak}{\mathfrak{m}}
\newcommand{\valp}{\mathrm{Val}_p}

\newcommand{\rfield}{(\R;+,\times)}
\newcommand{\rcf}{\mathrm{RCF}}
\newcommand{\rcvf}{\mathrm{RCVF}}
\newcommand{\doag}{\mathrm{DOAG}}
\newcommand{\pring}{(\Z_p;+,\times}
\newcommand{\pfield}{\Q_p}
\newcommand{\kfield}{K}
\newcommand{\rad}{\operatorname{rad}}

\newcommand{\az}{\alpha \Z}
\newcommand{\rvec}{\R_{\mathrm{Vec}}}

\newcommand{\qvec}{\Q_{\mathrm{Vec}}}

\newcommand{\Th}{\mathrm{Th}}
\ProvideTextCommandDefault{\cprime}{(U+042C)}

\newcommand{\Fraisse}{Fra\"iss\'e\xspace}

\newcommand{\chara}{\operatorname{Char}}

\newcommand{\st}{\operatorname{st}}

\newenvironment{claimproof}[1][\proofname]
               {
                 \proof[#1]
                 
               }
               {
                 \endproof
               }

\newcommand{\cl}{\operatorname{Cl}}

\newcommand{\corank}{\ensuremath{\textup{cork}}}
\newcommand{\coranke}{\ensuremath{\textup{ecork}}}
\newcommand{\rank}{\ensuremath{\textup{rk}}}
\newcommand{\erank}{\ensuremath{\textup{erk}}}

\newcommand{\Sh}[1]{\ensuremath{\mathscr{#1}^{\mathrm{Sh}}}}
\newcommand{\Sq}[1]{\ensuremath{\mathscr{#1}^{\square}}}
\newcommand{\nip}{\mathrm{NIP}}

\newcommand{\Cal}[1]{\ensuremath{\mathcal{#1}}}
\newcommand{\Sa}[1]{\ensuremath{\mathscr{#1}}}

\newcommand{\age}[1]{\ensuremath{\textup{Age{#1}}}}

\newcommand{\air}[1]{\ensuremath{\textup{Air{#1}}}}

\newcommand{\app}{\approx_\uplambda}
\newcommand{\ru}{\mathrm{RU}}
\newcommand{\ruo}{\mathrm{RU}_{\Sa O}}
\newcommand{\rum}{\mathrm{RU}_{\Sa M}}
\newcommand{\mr}{\mathrm{RM}}
\newcommand{\mrm}{\mathrm{RM}_{\Sa M}}
\newcommand{\mro}{\mathrm{RM}_{\Sa O}}

\newcommand{\B}{\mathbb{B}}
\newcommand{\Z}{\mathbb{Z}}
\newcommand{\N}{\mathbb{N}}
\newcommand{\C}{\mathbb{C}}
\newcommand{\Q}{\mathbb{Q}}
\newcommand{\R}{\mathbb{R}}
\newcommand{\F}{\mathbb{F}}
\newcommand{\K}{\mathbb{K}}

\begin{document}
\title[]{Notes on trace equivalence}

\author{Erik Walsberg}
\address{Department of Mathematics, Statistics, and Computer Science\\
Department of Mathematics\\University of California, Irvine, 340 Rowland Hall (Bldg.\# 400),
Irvine, CA 92697-3875}
\email{ewalsber@uci.edu}
\urladdr{https://www.math.uci.edu/\textasciitilde ewalsber/}

\date{\today}

\maketitle

\begin{abstract}
We introduce and study trace equivalence, a weak notion of equivalence for first order theories.
In particular this gives an interesting notion of equivalence  for $\nip$ theories.
\end{abstract}

\section*{Introduction}
Throughout $\Sa M$ and $\Sa O$  are structures, $L,L^*$ are languages, and $T,T^*$ is a complete $L,L^*$-theory, respectively.
Let $\uptau \colon O \to M^m$ be an injection.
We say that $\Sa M$ trace defines $\Sa O$ via $\uptau$ if for every $\Sa O$-definable subset $X$ of $O^n$ there is an $\Sa M$-definable subset $Y$ of $M^{mn}$ such that: 
$$ (\alpha_1,\ldots,\alpha_n) \in X \quad\Longleftrightarrow\quad (\uptau(\alpha_1),\ldots,\uptau(\alpha_n)) \in Y \quad \text{for all  } \alpha_1,\ldots,\alpha_n \in O,$$
and $\Sa M$ \textbf{trace defines} $\Sa O$ if $\Sa M$ trace defines $\Sa O$ via an injection $\uptau \colon O \to M^m$.
We say that $T$ trace defines $T^*$ if every $T^*$-model is trace definable in a $T$-model, $T$ trace defines $\Sa O$ if some $T$-model trace defines $\Sa O$, $T$ and $T^*$ are \textbf{trace equivalent} if $T$ trace defines $T^*$ and vice versa, and $\Sa M$ and $\Sa O$ are trace equivalent if $\Th(\Sa M)$ and $\Th(\Sa O)$ are trace equivalent.
We will see that $T$ trace defines $T^*$ if and only if some $T^*$-model is trace definable in a $T$-model. 

\medskip
These definitions arose out of specific questions on, and examples of, $\nip$ structures.
That is treated below.
Why do I think trace definibility is interesting?
Well, I think that working up to trace equivalence reveals the underlying nippy structure.
More concretely:
\begin{enumerate}[leftmargin=*, itemsep=.1cm]
\item There are many interesting examples of theories $T,T^*$ such that $T$ trace defines but does not interpret $T^*$.
See the second list below.

\item Suppose $T$ does not trace define $T^*$.
Then $T$ does not interpret $T^*$ and, furthermore, any theory trace definable in $T$ does not interpret $T^*$.
So if we can upgrade a non-interpretation result to a non-trace definition result then we can ``spread out" the non-interpretation result.
We will not explicitly mention this here, but in this way we obtain many new non-interpretation results below.

\item Stability and $\nip$-theoretic properties that are preserved under interpretations are usually preserved under trace definitions, see Theorem~\ref{thm:preservation}.
These preservation results are usually more or less obvious from the definitions, so we expect that yet-to-be discovered $\nip$-theoretic properties will also be preserved.
So if $T$ trace defines $T^*$ then we expect satisfaction of any $\nip$-theoretic property, known or unknown, to transfer from $T$ to $T^*$.

\item Some $\nip$-theoretic properties can be characterized in terms of trace definibility, see Theorem~\ref{thm:characterize}.
Indiscernible collapse can be seen as a special case of trace definibility, so properties that can be characterized in terms of indiscernible collapse can be characterized in terms of trace definibility.
A classification of finitely homogeneous structures up to trace equivalence would give a classification of those model-theoretic properties that can be defined in terms of indiscernible collapse.

\item The partial order of trace equivalence classes admits finite joins, this is given by taking disjoint unions.
Under certain situations we can show that a structure is trace equivalent to a disjoint union of simpler structures, this gives a novel way of decomposing structures.

\item In stability, o-minimality, and a few other places, we have ``Zil'ber dichotomies" of the following form: We have a class $\Cal C$ of tame structures and an abstract ``triviality" (``modularity") notion $\mathbf{P}$, and one shows that any $\Sa M\in\Cal C$ satisfies $\mathbf{P}$ if and only if $\Sa M$ does not interpret an infinite group (field).
We are interested in finding such dichotomies in the $\nip$ setting.
This seems to require replacing ``interpret" with ``trace define".
\item There is a long line of work on classification of finitely homogeneous structures.
These classifications become increasing complex as one passes to broader classes of structures.
For example the classification of homogeneous partial orders~\cite{schmerl} is quite simple, while the classification of homogeneous colored partial orders \cite{sousa} is much more complex.
What we know so far suggests that there is a reasonable classification of finitely homogeneous structures up to trace equivalence.
\end{enumerate}
\medskip
Theorem~\ref{thm:preservation} summarizes our results on preservation of $\nip$-theoretic properties.

\begin{theorem}
\label{thm:preservation}
Suppose that $T$ trace defines $T^*$.
If $T$ satisfies one of the following properties then $T^*$ does as well: stability, $\nip$, $k$-independence, total transendence, superstability, strong dependence, finiteness of $U$-rank, finiteness of Morley rank, finiteness of dp-rank, finiteness of op-rank, linear vc-density bounds, the strong Erd\H{o}s-Hajnal property, and near linear Zarankiewicz bounds.
\end{theorem}

We also characterize some classification-theoretic properties.

\begin{theorem}
\label{thm:characterize}
\hspace{.2cm}
\begin{enumerate}[leftmargin=*]
\item $T$ is unstable if and only if $T$ trace defines $(\Q;<)$.
\item $T$ is $\mathrm{IP}$ if and only if $T$ trace defines the Erd\H{o}s-Rado graph.
More generally for all $k\ge 1$, $T$ is $k$-independent if and only if $T$ trace defines the generic countable $(k+1)$-hypergraph.
\item $T$ is not totally transendental if and only if $T$ trace defines an infinite set equipped with a family of unary relations forming a complete binary tree under inclusion.
\end{enumerate}
\end{theorem}

Theorem~\ref{thm:characterize}.3 is due to James Hanson.
Theorem~\ref{thm:characterize} also shows that many classification theoretic properties are not preserved under trace definibility.
For example simplicity is not preserved under trace definibility as an unstable simple structure trace defines $\dlo$.
Most classification-theoretic properties are either on the ``$\nip$-axis" or the ``$\mathrm{NSOP}$-axis", the former are typically preserved under trace definitions and the latter are not.
We will also see that many of the main examples of simple theories trace define \textit{any} theory.
We say that $T$ is \textbf{trace maximal} if $T$ trace defines any theory and a structure is trace maximal if its theory is.
We will show that $T$ is trace maximal if and only if there is $\Sa M\models T$ and infinite $A \subseteq M^m$ such that for every $X\subseteq A^k$ we have $X=A^k\cap Y$ for $\Sa M$-definable $Y\subseteq M^{mk}$.

\begin{theorem}
\label{thm:trace maximal}
The following structures are trace maximal: $(\Z;+,\times)$, any infinite boolean algebra, the expansion of an infinite abelian group by a generic unary relation, any $\pac$ field that is not separably closed, any $\mathrm{PRC}$ field that is not real closed or separably closed.
\end{theorem}

Working up to trace equivalence erases ``$\mathrm{NSOP}$-axis" classification theoretic structure and highlights ``$\nip$-axis" classification theoretic structure.

\newpage
We now give a list of examples.

\begin{enumerate}[leftmargin=*,itemsep=.1cm]
\item Any $\nip$ structure is trace equivalent to its Shelah completion.
A $\nip$ expansion $\Sa O$ of a linear order $(O;\prec)$ is trace equivalent to any expansion of $\Sa O$ by convex subsets of $O$.
\item The following are trace equivalent: $(\Z;+,<)$, $(\R;+,<,\Z)$, $(\R;+,<,\Z,\Q)$, $(\Z^n;+,\prec)$ for any group order $\prec$ on $\Z^n$, $(\Z^n;+,C_\upchi)$ where $C_\upchi$ is the cyclic group order on $\Z^n$ induced by an injective character $\upchi\colon\Z^n\to\R/\Z$, and the disjoint union of $(\Z;+)$ and $(\R;+,<)$.
\item The following structures are trace equivalent: $(\R;+,<)$, $(\R;+,<,\Q)$, $(\R/\Z;+,C)$, and $(\R/\Z;+,C,\Q/\Z)$, here $C$ is the counterclockwise cyclic order on $\R/\Z$.
\item All infinite finitely generated abelian groups are trace equivalent.
\item The following abelian groups are trace equivalent: $\Q$, $\Q/\Z$, and $\Z(p^\infty)$ for any prime $p$.
\item $\Th(\Z;+)$ trace defines any finite rank torsion free abelian group and $\Th(\Z;+,<)$ trace defines any finite rank ordered abelian group.
\item If $(H;+,\prec)$ is an archimedean ordered abelian group (more generally an ordered abelian group which admits only finitely many definable convex subgroups) then $(H;+,\prec)$ is trace equivalent to the disjoint union of $(H;+)$ and $(\R;+,<)$.
\item Let $\mathbf{s}(x)=x+1$.
The following are trace equivalent: $(\Z;<)$, any infinite discrete linear order, $(\R;<,\mathbf{s})$, $(\R;<,\Z)$, $(\R;<,\Z,\mathbf{s})$, and the disjoint union of $(\Z;\mathbf{s})$ and $(\R;<)$.
\item All finite extensions of $\Q_p$ are trace equivalent.
\item The generic countable $k$-hypergraph is trace equivalent to the generic countable $k$-ary relation.
More generally: the generic countable $k$-hypergraph is trace equivalent to any $(k-1)$-independent structure which admits quantifier elimination in a finite relational language of airity $\le k$.
In particular the Erd\H{o}s-Rado graph is trace equivalent to any $\mathrm{IP}$ structure which admits quantifier elimination in a finite binary language.
\item $(\R;<)$ is trace equivalent to, but does not interpret, $(\R;<,\Q)$.
\item A primitive rank $1$ finitely homogeneous $\nip$ structure is trace definable in $(\Q;<)$.
\item Fix positive $\lambda\in\R$.
The expansion of $(\Q;+,<)$ by all sets of the form $$\{ (\beta_1,\ldots,\beta_n)\in\Q^n: (\lambda^{\beta_1},\ldots,\lambda^{\beta_n})\in X\}\quad \text{for semialgebraic } X\subseteq\R^n$$ is trace equivalent to $\rfield$.
This structure does not interpret an infinite field.
\item Let $\B$ be the collection of balls in $\Q_p$ and $\Sa B$ be the structure induced on $\B$ by $\Q_p$.\\
Then $\Sa B$ is trace equivalent to $\Q_p$ but does not interpret an infinite field.
\end{enumerate}

\medskip
Of course these examples are only interesting for us if they cannot be strengthened to interpretations, at least aside from the obvious interpretations.
We gather many non-interpretation results in Appendix A.
Just a few of these are mentioned on the list above.

\medskip
Trace definibility is also considered by Guingona and Parnes~\cite{Guingona-Parnes}, who refer to a trace definition of $\Sa O$ as an ``$\Sa O$-configuration".
They only consider this in the case when $\Sa O$ is the \Fraisse limit of an algebraically trivial \Fraisse class.
Our main innovation is to consider trace definibility for arbitrary structures.
Guingona and Parnes are motivated by work on indiscernible collapse, this also provides motivation for us.
We discuss indiscernible collapse.

\medskip
Suppose that $\Sa I$ is a homogeneous structure in a finite relational language whose age has the Ramsey property.
We let $\Cal C_{\Sa I}$ be the class of theories $T$ such that the monster model of $T$ does not admit an uncollapsed indiscernible picture of $\Sa I$.

\begin{theorem}
\label{thm}
Let $\Sa I, \Sa J$ be homogeneous structures in finite relational languages whose ages have the Ramsey property.
Let $\monster$ be a monster model of $T$.
Then $\monster$ admits an uncollapsed indiscernible picture of $\Sa I$ if and only if $T$ trace defines $\Sa I$.
Furthermore $\Cal C_{\Sa I}=\Cal C_{\Sa J}$ if and only if $\Sa I$ and $\Sa J$ are trace equivalent.
\end{theorem}

This gives us motivation to classify finitely homogeneous structures up to trace equivalence.

\medskip
\textbf{Weak Conjecture:} There are $\aleph_0$ finitely homogeneous structures up to trace equivalence.

\medskip\textbf{Strong Conjecture:} For every $k$ there are only finitely many finitely homogeneous structures in a $k$-ary language up to trace equivalence.

\medskip
Given a finitely homogeneous structure $\Sa O$ we let $\mathcal{C}_{\Sa O}$ be the class of theories which do not trace define $\Sa O$.
The general hope is that each $\Cal C_{\Sa O}$ should be the class of theories satisfying some reasonable $\nip$-theoretic property and that we can find fundamental $\nip$-theoretic properties by classifying finitely homogeneous structures up to trace equivalence.
We prove the binary case of the weak conjecture by combining the recent Onshuus-Simon theorem that there are only $\aleph_0$ homogeneous $\nip$ structures in a binary language up to interdefinibility together with the fact that any homogeneous $\mathrm{IP}$ structure in a finite binary language is trace equivalent to the Erd\H{o}s-Rado graph.
We give more evidence in Section~\ref{section:conj}.

\medskip\noindent
We now describe the original motivation for trace definibility.
We first recall Fact~\ref{fact:ps}, a special case of the Peterzil-Starchenko o-minimal trichotomy~\cite{PS-Tri}.
We let $\rvec$ be the ordered vector space $(\R;+,<,(t \mapsto \lambda t)_{\lambda \in \R})$, recall that $\rvec$ admits quantifier elimination.
A subset of $\R^m$ is \textbf{semilinear} if it is definable in $\rvec$.

\begin{fact}
\label{fact:ps}
The following are equivalent for any o-minimal expansion $\Sa R$ of $(\R;+,<)$.
\begin{enumerate}
    \item $\Sa R$ does not define an isomorphic copy of $\rfield$,
    \item $\Sa R$ does not define an infinite field,
    \item $\Sa R$ is locally modular,
    \item $\Sa R$ is a reduct of $\rvec$.
\end{enumerate}
\end{fact}

\noindent
Fact~\ref{fact:ps} asserts the equivalence of (a) absence of some algebraic structure, (b) an abstract model-theoretic ``linearity" notion, and (c) semilinearity of definable sets.
It is a theorem of Simon~\cite{Simon-dp} that an expansion of $(\R;+,<)$ is o-minimal if and only it has dp-rank one.
It is natural to ask for an analogue of Fact~\ref{fact:ps} for dp-rank one expansions of $(\Q;+,<)$.
In Section~\ref{section:mann} we give an example which shows that we need different notions of definable algebraic structure and model-theoretic linearity.
We prove Theorem~\ref{thm:tri}.

\begin{theorem}
\label{thm:tri}
Suppose that $Q$ is a divisible subgroup of $(\R;+)$ and $\Sa Q$ is a dp-rank one expansion of $(Q;+,<)$.
Then the following are equivalent:
\begin{enumerate}
\item $\Th(\Sa Q)$ does not trace define $\rcf$,
\item $\Th(\Sa Q)$ does not trace define an infinite field,
\item $\Sa Q$ has near linear Zarankiewicz bounds,
\item any $\Sa Q$-definable $X \subseteq Q^n$ is of the form $Y \cap Q^n$ for a semilinear $Y \subseteq \R^n$.
\item $\Sa Q$ is trace equivalent to an ordered vector space.
\end{enumerate}
If $Q = \Q$ then $Th(\Sa Q)$ does not trace define $\rcf$ if and only if $\Sa Q$ is a reduct of the structure induced on $\Q$ by $\rvec$.
\end{theorem}

See Section~\ref{section:almost-lin} for a definition of ``near linear Zarankiewicz bounds".
A theory with this property is $\nip$ and cannot trace define an infinite field.
Ordered vector spaces were recently shown to have near linear Zarankiewicz bounds in \cite{zaran}.
I don't know if this is actually a good model-theoretic linearity notion for general $\nip$ theories, but it's a start.

\medskip
It's natural to ask what happens when we remove the assumption of divisibility from Theorem~\ref{thm:tri}.
See Conjecture~\ref{conj:2} for how to proceed on this.

\medskip
I claim that the right way to establish ``Zil'ber dichotomies" between combinatorial simplicity and interpretable algebraic structure in $\nip$ is to substitute ``trace definable" for ``interpretable".
For this claim to be reasonable I must show that the existing dichotomies are unaffected by this substitution.
Theorem~\ref{thm:rigid} shows that the claim is partially reasonable.

\begin{theorem}
\label{thm:rigid}
\hspace{.1cm}
\begin{enumerate}[leftmargin=*]
\item If $\Sa M$ is o-minimal then $\Sa M$ trace defines an infinite group iff $\Sa M$ defines an infinite group.
\item If $\Sa M$ is an o-minimal expansion of an ordered abelian group then $\Sa M$ trace defines an infinite field if and only if $\Sa M$ defines an infinite field.
\item If $\Sa M$ is $\aleph_0$-stable and $\aleph_0$-categorical then $\Sa M$ trace defines an infinite group if and only if $\Sa M$ interprets an infinite group.
\end{enumerate}
\end{theorem}

Rather amusingly (3) is proven as a corollary to (1) and Lachlan's theorem that a trivial $\aleph_0$-stable $\aleph_0$-categorical structure is interpretable in a dense linear order.
I don't know of another case in which a result on stable structures is proven via o-minimality, possibly because results on stability are usually proven by people who know stability theory.


\medskip
\begin{center}
\textbf{The contents by section.}
\end{center}

\vspace{.2cm}\noindent
\textbf{\ref{section:conventions}}\quad Background information and terminology (not all of which is standard).

\vspace{.1cm}\noindent\textbf{\ref{section:few general}}\quad The minimal amount of general results needed to discuss some examples.

\vspace{.1cm}\noindent
\textbf{\ref{section:few}}\quad Some interesting examples.

\vspace{.1cm}\noindent
\textbf{\ref{section: basic trace}}\quad 
General theory, indiscernible collapse, preservation of classification-theoretic properties.

\vspace{.1cm}\noindent
\textbf{\ref{section:examples}}\quad More examples.

\vspace{.1cm}\noindent
\textbf{\ref{section:dicho}}\quad Dichotomies, proofs of Theorems~\ref{thm:tri} and \ref{thm:rigid}.
Near linear Zarankiewicz bounds.

\vspace{.1cm}\noindent
\textbf{\ref{section:fields}}\quad Some results on trace definitions between fields.

\vspace{.1cm}\noindent
\textbf{\ref{section:open core}}\quad Some comments on expansions of $(\R;+,<)$ trace definable in o-minimal structures.

\vspace{.1cm}\noindent
\textbf{\ref{section:conj}}\quad    Trace definibility between finitely homogeneous structures.

\subsection*{Acknowledgements}
Six years ago or so John Goodrick told me that he would like to have an analogue of the Peterzil-Starchenko trichotomy for dp-rank one ordered structures.
Theorem~\ref{thm:tri} came from this.
I was also influenced by conversations with Artem Chernikov and Pierre Simon.
Artem told me how to use linear Zarankiewicz bounds to show that the Shelah expansion of an ordered vector space cannot interpret an infinite field and Pierre told me that there should be more structure in $\nip$, and more properties definable via indiscernible collapse, then was thought.

\section{Conventions and background}
\label{section:conventions}

\subsection{Notation}
Throughout $n,m,k$ are natural numbers.
We use $\monster,\nonster,\oonster,\ldots$ to denote monster models and $\monsterset,\nonsterset,\oonsterset,\ldots$ to denote their domains.
Given an $L$-structure $\Sa M$ and $L^* \subseteq L$ we let $\Sa M \! \upharpoonright \! L^*$ be the $L^*$-reduct of $\Sa M$.
Given a structure $\Sa M$, $a \in M^n$, and $A \subseteq M$ we let $\tp_{\Sa M}(a|A)$ be the type of $a$ over $A$ and let $\tp_{\Sa M}(a) = \tp_{\Sa M}(a|\emptyset)$.
If $\Sa M$ and $\Sa M^*$ are two structures with domain $M$ then we say that $\Sa M$ and $\Sa M^*$ are \textbf{interdefinable} if $\Sa M$ is a reduct of $\Sa M^*$ and vice versa.
We will also say that $\Sa M$ and $\Sa N$ are \textbf{isointerdefinable}\footnote{The only name for this that I have seen in the literature is ``interdefinable", but this conflicts with the usual meaning of ``interdefinable". Unable to coin a good name, I use a suggestive one.} if there is a structure $\Sa N^*$ on $N$ such that $\Sa N$ is interdefinable with $\Sa N^*$ and $\Sa N^*$ is isomorphic to $\Sa M$.

\medskip
We use some standard acronyms for theories.
An \textbf{oag} is a (totally) ordered abelian group, $\rcf$ is the theory of real closed fields, $\mathrm{RCVF}$ is the theory of a real closed field expanded by a non-trivial convex valuation, $\dlo$ is the theory of dense linear orders without endpoints, and $\doag$ is the theory of divisible ordered abelian groups.

\medskip\noindent
Let $\Sa M$ be a structure and $A \subseteq M^m$.
The \textbf{structure induced} on $A$ by $\Sa M$ is the structure $\Sa A$ with an $n$-ary predicate $P_X$ defining $X \cap A^n$ for every $\Sa M$-definable $X \subseteq M^{mn}$.
Note that the induced structure admits quantifier elimination if and only if every $\Sa A$-definable subset of each $A^n$ is of the form $X \cap A^n$ for some $\Sa M$-definable $X \subseteq M^{mn}$.

\medskip\noindent
Suppose that $\Sa M$ expands a linear order $(M;<)$.
Then $\Sa M$ is \textbf{weakly o-minimal} if every definable subset of $M$ is a finite union of convex sets and $\Th(\Sa M)$ is weakly o-minimal if the same holds in every elementary extension of $\Sa M$.
Note that $\Th(\Sa M)$ is weakly o-minimal iff for every formula $\varphi(x,y)$ with $|y| = 1$ there is $n$ such that $\{ a \in M : \Sa M \models \varphi(b,a) \}$ is a union of $\le n$ convex sets for all $b \in M^{|x|}$.
Fact~\ref{fact:wom-induced} follows from this by o-minimal cell decomposition.

\begin{fact}
\label{fact:wom-induced}
Suppose that $\Sa M$ is o-minimal, $A$ is a subset of $M$, and the structure $\Sa A$ induced on $A$ by $\Sa M$ admits quantifier elimination.
Then $\Th(\Sa A)$ is weakly o-minimal.
\end{fact}

We also use the description of one-types in weakly o-minimal theories.

\begin{fact}
\label{fact:wom types}
Suppose that $\Sa M$ is weakly o-minimal.
Given a non-realized one-type $p(x)$ over $M$ we let $C_p : = \{ \alpha \in M : p \models \alpha < x \}$.
Then $p$ is definable if and only if $C_p$ is definable.
Then $p$ is determined by $C_p$ and $p$ is definable iff $C_p$ is definable.
For every downwards closed $C \subseteq M$ without a supremum there is a unique non-realized one-type $p$ such that $C = C_p$.
\end{fact}

Fact~\ref{fact:wom types} is an easy exercise.

\subsection{Homogeneous structures}
\label{sectin:homo}
A \textbf{homogeneous} structure is a countable relational structure such that every finite partial automorphism extends to a total automorphism.
A \textbf{finitely homogeneous} structure is a homogeneous structure in a finite language.
Let $\age(\Sa M)$ be the age of an $L$-structure $\Sa M$ for relational $L$, i.e. the class of finite $L$-structures that embed into $\Sa M$.
A finitely homogeneous structure $\Sa M$ is \textbf{Ramsey} if $\age(\Sa M)$ has the Ramsey property.

\medskip
Fact~\ref{fact:homo} is standard, see \cite[Thm~2.1.3, Prop~3.1.6]{macpherson-survey}.

\begin{fact}
\label{fact:homo}
Suppose $\Sa M$ is relational.
\begin{enumerate}[leftmargin=*]
\item $\Sa M$ is homogeneous if and only if $\age(\Sa M)$ is a \Fraisse class with \Fraisse limit $\Sa M$.
\item If $\Sa M$ is finitely homogeneous then $\Sa M$ is $\aleph_0$-categorical and admits quantifier elimination.
\item If $\Sa M$ is $\aleph_0$-categorical and admits quantifer elimination, then $\Sa M$ is homogeneous.
\end{enumerate}
\end{fact}

\noindent
Suppose $k \ge 2$.
We say that a structure or theory is \textbf{$k$-ary} if every formula is equivalent to a boolean combination of formulas of airity $\le k$ and is \textbf{finitely $k$-ary} if there are formulas $\varphi_1,\ldots,\varphi_n$ of airity $\le k$ such that any $n$-ary formula $\phi(x_1,\ldots,x_n)$ is a boolean combination of formulas of the form $\varphi_j(x_{\imag_1},\ldots,x_{\imag_\ell})$ for $\imag_1,\ldots,\imag_\ell \in \{1,\ldots,n\}$.
A finitely $k$-ary theory is $\aleph_0$-categorical and a finitely homogeneous structure is finitely $k$-ary for some $k$.
We define $\air(T)$ to be the minimal $k$ such that $T$ is $k$-ary if there is such $k$ and $\air(T) = \infty$ otherwise.

\medskip
Fact~\ref{fact:airity} is an exercise.

\begin{fact}
\label{fact:airity}
The following are equivalent:
\begin{enumerate}
\item $T$ is $k$-ary.
\item If $\Sa M \models T$, $A\subseteq M$, $(\alpha_1,\ldots,\alpha_n)$ and $(\beta_1,\ldots,\beta_n)$ are in $M^n$, and 
\[
\tp(\alpha_{i_1},\ldots,\alpha_{i_k}|A) = \tp(\beta_{i_1},\ldots,\beta_{i_k}|A) \quad \text{for all} \quad 1 \le i_1 < \ldots <  i_k \le n,
\]
then $\tp(\alpha|A) = \tp(\beta|A)$.
\end{enumerate}
Furthermore if $\Sa M$ is $\aleph_0$-categorical then $\Sa M$ is $k$-ary if and only if $\Sa M$ is finitely $k$-ary.
\end{fact}

Note that the theory $T$ of an infinite structure always satisfies $\air(T) \ge 2$.
This is because we can take $\alpha,\beta$ to be distinct realizations of the same one-type.

\medskip\noindent
A \textbf{$k$-hypergraph} $(V;E)$ is a set $V$ equipped with a symmetric $k$-ary relation $E$ such that $E(\alpha_1,\ldots,\alpha_k)$ implies that the $\alpha_i \ne \alpha_j$ when $i \ne j$.
Finite $k$-hypergraphs form a \Fraisse class, we refer to the \Fraisse limit of this class as the \textbf{generic countable $k$-hypergraph}.
The generic countable $2$-hypergraph is the \textbf{Erd\H{o}s-Rado graph}.
The class of structures $(V;E)$ where $V$ is finite and $E$ is an arbitrary $k$-ary relation is a \Fraisse class, we refer to the \Fraisse limit as the \textbf{generic countable $k$-ary relation}.

\medskip\noindent
A \textbf{bipartite graph} $(V,W;E)$ consists of sets $V,W$ and $E \subseteq V \times W$.
Let $(V,W;E)$ be a bipartite graph.
The \textbf{generic countable bipartite graph} is the Fra\"iss\'e limit of the class of finite bipartite graphs.

\subsection{Shelah completeness}
The Shelah completion is usually referred to as the ``Shelah expansion".
I use ``completion" because it is more suggestive and because we already use ``expansion" enough.
Let $\Sa M \prec \Sa N$ be $|M|^+$-saturated.
A subset $X$ of $M^n$ is \textbf{externally definable} if $X = M^n \cap Y$ for some $\Sa N$-definable subset $Y$ of $N^n$.
An application of saturation shows that the collection of externally definable sets does not depend on choice of $\Sa N$.

\begin{fact}
\label{fact:convex}
Suppose that $X$ is an $\Sa M$-definable set and $<$ is an $\Sa M$-definable linear order on $X$.
Then any $<$-convex subset of $X$ is externally definable.
\end{fact}

\noindent
Fact~\ref{fact:convex} is well-known and easy.
Lemma~\ref{lem:lambda} is a saturation exercise.

\begin{lemma}
\label{lem:lambda}
Suppose that $\uplambda$ is a cardinal, $\Sa M$ is $\uplambda$-saturated, $X \subseteq M^m$ is externally definable, and $A \subseteq M^m$ satisfies $|A| < \uplambda$.
Then there is a definable $Y \subseteq M^m$ such that $X \cap A = Y \cap A$.
\end{lemma}

\noindent
Lemma~\ref{lem:chain} is an easy generalization of Fact~\ref{fact:convex}.

\begin{lemma}
\label{lem:chain}
Suppose that $(X_a : a \in M^n)$ is an $\Sa M$-definable family of subsets of $M^m$.
If $A \subseteq M^n$ is such that $(X_a : a \in A )$ is a chain under inclusion then $\bigcup_{a \in A} X_a$ and $\bigcap_{a \in A} X_a$ are both externally definable.
\end{lemma}



\noindent
The \textbf{Shelah completion} $\Sh M$ of $\Sa M$ is the expansion of $\Sa M$ by all externally definable sets, equivalently the structure induced on $M$ by $\Sa N$.
Fact~\ref{fact:shelah} is due to Shelah~\cite{Shelah-external}.

\begin{fact}
\label{fact:shelah}
Suppose $\Sa M$ is $\nip$.
Then the structure induced on $M$ by $\Sa N$ admits quantifier elimination.
Equivalently: every $\Sh M$-definable set is externally definable.
\end{fact}






\subsection{Dp-rank}
We first recall the definition of the dp-rank $\dprk_{\Sa M} X$ of an $\Sa M$-definable set $X \subseteq M^m$.
It suffices to define the dp-rank of a definable set in the monster model.
Let $X$ be a definable set and $\uplambda$ be a cardinal.
An $(\monster,X,\uplambda)$-array is a sequence $(\varphi_\alpha(x_\alpha , y) : \alpha < \uplambda )$ of parameter free formulas and an array $(a_{\alpha,i} \in \monsterset^{|x_\alpha|} : \alpha < \uplambda , i < \upomega )$ such that for any function $f \colon \uplambda \to \upomega$ there is a $b \in X$ such that
$$ \monster \models \varphi_\alpha(a_{\alpha,k},b) \quad \text{if and only if} \quad f(\alpha) = k \quad \text{for all   } \alpha<\uplambda,i<\upomega. $$
Then $\dprk_{\monster} X \ge \uplambda$ if there is an $(\monster, X, \uplambda)$-array.
We declare $\dprk_{\monster} X = \infty$ if $\dprk_{\monster} X \ge \uplambda$ for all cardinals $\uplambda$.
We let $\dprk_{\monster} X = \max\{ \uplambda : \dprk_{\monster}  X \ge \uplambda \}$ if this maximum exists and otherwise
$$ \dprk_{\monster} X = \sup\{ \uplambda : \dprk_{\monster} X \ge \uplambda \} - 1. $$
We also define $\dprk \Sa M = \dprk_{\Sa M} M$.
Of course this raises the question of what exactly $\uplambda - 1$ is for an infinite cardinal $\uplambda$.
But that doesn't matter at all.

\medskip\noindent
The first three claims of Fact~\ref{fact:dp-rank} are immediate consequences of the definition of dp-rank.
The fourth is due to Kaplan, Onshuus, and Usvyatsov~\cite{dp-rank-additive}.

\begin{fact}
\label{fact:dp-rank}
Suppose $X,Y$ are $\Sa M$-definable sets.
Then
\begin{enumerate}
\item $\dprk \Sa M < \infty$ if and only if $\Sa M$ is $\nip$,
\item $\dprk_{\Sa M} X = 0$ if and only if $X$ is finite,
\item If $f : X \to Y$ is a definable surjection then $\dprk_{\Sa M} Y \leq \dprk_{\Sa M} X$,
\item $\dprk_{\Sa M} X \times Y \le \dprk_{\Sa M} X + \dprk_{\Sa M} Y$.
\end{enumerate}
\end{fact}



\noindent
Fact~\ref{fact:pw-weak} is proven in \cite{SW-dp}.

\begin{fact}
\label{fact:pw-weak}
Suppose that $Q$ is a divisible subgroup of $(\R,+)$ and $\Sa Q$ expands $(Q,+,<)$.
Then the following are equivalent:
\begin{enumerate}
\item $\Sa Q$ has dp-rank one,
\item $\Sa Q$ is weakly o-minimal,
\item $\Th(\Sa Q)$ is weakly o-minimal.
\item there is an o-minimal expansion $\Sq Q$ of $(\R;+,<)$ such that the structure induced on $Q$ by $\Sq Q$ eliminates quantifiers and is interdefinable with $\Sh Q$.
\end{enumerate}
\end{fact}

\noindent
The implication $(4) \Rightarrow (3)$ is not stated in \cite{SW-dp} but follows from Fact~\ref{fact:wom-induced}.

\subsection{Abelian groups}
\label{section:abelian}
Fact~\ref{fact:direct sum} follows by Feferman-Vaught.

\begin{fact}
\label{fact:direct sum}
Suppose that $A\equiv A'$ and $B\equiv B'$ are abelian groups.
Then $A\oplus B \equiv A'\oplus B'$.
\end{fact}

Let $A$ be an abelian group, written additively.
Given  $k \in \N$, we write $k|\alpha$ when $\alpha = k \beta$ for some $\beta \in A$.
Here we consider each $k|$ to be a unary relation symbol.
We let $L_{\mathrm{div}}$ be the expansion of the language of abelian groups by $(k| : k \in \N)$.
Fact~\ref{fact:abelian qe} is due to Szmielew, see~\cite[Theorem~A.2.2]{Hodges}.

\begin{fact}
\label{fact:abelian qe}
Any abelian group admits quantifier elimination in $L_\mathrm{div}$.
\end{fact}

Fact~\ref{fact:qe-for-divisible} is a special case of Fact~\ref{fact:abelian qe}.

\begin{fact}
\label{fact:qe-for-divisible}
A divisible abelian group eliminates quantifiers in the language of abelian groups.
\end{fact}

The \textbf{rank} $\rank(A)$ of $A$ is the $\Q$-vector space dimension of $A\otimes\Q$.
Let $p$ range over primes and $\F_p$ be the field with $p$ elements.
Let $A[p]$ be the $p$-torsion subgroup of $A$, i.e $\{\alpha\in A : p\alpha=0\}$.
The \textbf{$p$-rank} $\rank_p(A)$ is the $\F_p$-vector space dimension of $A[p]$ and the \textbf{$p$-corank} $\corank_p(A)$ of $A$ is the $\F_p$-vector space dimension of $A/pA$.
Hence $|A[p]|=p^{\rank_p(A)}$and $|A/pa|=p^{\corank_p(A)}$.

\begin{fact}
\label{fact:additive}
Suppose that $(A_i : i \in I)$ is a family of abelian groups.
Then for all primes $p$:
\[
\rank_p\left(\bigoplus_{i\in I} A_i\right) = \sum_{i\in I} \rank_p(A_i)\quad\text{and}\quad\corank_p\left(\bigoplus_{i\in I}A_i\right)=\sum_{i\in I} \corank_p(A_i).
\]
\end{fact}

\begin{proof}
For the first claim note that $\left(\bigoplus_{i\in I} A_i\right)[p] = \bigoplus_{i\in I} A_i[p]$.
For the second claim note
\[
\bigoplus_{i\in I} A_i \bigg/ p\bigoplus_{i\in I} A_i = \bigoplus_{i \in I} A_i/p A_i.
\]
\end{proof}

\medskip
We let $\Z(p^\infty)$ be the Pr\"ufner $p$-group, i.e. the subgroup $\{m/p^k +\Z :k\in\N,m\in\Z\}$ of $\Q/\Z$.

\begin{fact}
\label{fact:divisible classification}
Any divisible abelian group $A$ is isomorphic to $\Q^{\rank(A)}\oplus\bigoplus_{p}\Z(p^\infty)^{\rank_p(A)}$.
In particular $\Q/\Z$ is isomorphic to $\bigoplus_p \Z(p^\infty)$.
\end{fact}

The first claim of Fact~\ref{fact:divisible classification} is \cite[Theorem~3.1]{fuchs}.
The second follows by computing ranks.

\medskip
The \textbf{elementary rank} $\erank(A)$ of $A$ is $\min\{\rank_p(A),\aleph_0\}$.
Likewise define the elementary $p$-rank $\erank_p(A)$ and elementary $p$-corank $\coranke_p(A)$ of $A$.
Note that these only depend on $\Th(A)$.
Fact~\ref{fact:divisible elemen} follows from Fact~\ref{fact:divisible classification} and a back-and-forth argument.
It is also a special case of Szmielew's classification of abelian groups up to elementary equivalence.

\begin{fact}
\label{fact:divisible elemen}
If $A$ is a divisible abelian group then $A\equiv\Q^{\erank(A)}\oplus\bigoplus_{p}\Z(p^\infty)^{\erank_p(A)}$.
Two divisible abelian groups are elementarily equivalent if and only if they have the same elementary ranks.
\end{fact}

We now discuss torsion free abelian groups.
Let $\Z_{(p)}$ be the localization of $\Z$ at $p$, i.e. $\Q\cap\Z_p$.

\begin{fact}
\label{fact:szmielew}
If $A$ is a torsion free abelian group then $A$ is elementarily equivalent to  $\bigoplus_{p} \Z^{\coranke_p(A)}_{(p)}$.
In particular $\Z$ is elementarily equivalent to $\bigoplus_p \Z_{(p)}$.
Two torsion free abelian groups are elementarily equivalent if and only if they have the same elementary $p$-coranks.
\end{fact}

The first claim of Fact~\ref{fact:szmielew} is a special case of Szmielew's classification of abelian groups up to elementary equivalence~\cite[Lemma~A.2.3]{Hodges}, see also Zakon~\cite{zakon} for an easy direct treatment of the torsion free case.
The second claim follows easily from the first claim by computing ranks and the third follows from the first.

\begin{fact}
\label{fact:finite rank}
Suppose that $A$ is a finite rank torsion free abelian group.
Then the $p$-corank of $A$ is less then or equal to the rank of $A$ for all primes $p$.
\end{fact}

I am sure Fact~\ref{fact:finite rank} is known, but I couldn't find it so I will give a proof.

\begin{lemma}
\label{lem:finite rank}
Suppose $A'\to A\to A'' \to 0$ is an exact sequence of abelian groups.
Then for all primes $p$ we have $\corank_p(A)\le \corank_p(A')+\corank_p(A'')$.
\end{lemma}

\begin{proof}
Tensoring is right exact and $B \otimes \Z/p\Z = B/pB$ for any abelian group $B$.
Therefore tensoring by $\Z/p\Z$ gives an exact sequence $A'/pA' \to A/pA \to A''/pA''\to 0$ of $\F_p$-vector spaces.
So $\corank_p(A)=\dim(A/pA)\le\dim(A'/pA')+\dim(A''/pA'')=\corank_p(A')+\corank_p(A'')$.
\end{proof}

We now prove Fact~\ref{fact:finite rank}.

\begin{proof}
We apply induction on $\rank(A)$.
The rank one case follows by \cite[Example~9.10]{fuchs}.
Suppose $\rank(A) > 1$.
Fix non-zero $\beta\in A$ and let $A'$ be the set of $\beta^*\in A$ such that $\beta^*=q\beta$ for some $q\in\Q$.
Then $A'$ is a rank one subgroup of $A$ and it is easy to see that $A'':=A/A'$ is torsion free.
We have $\rank(A'')=\rank(A)-1$ so by applying induction to $A'$, $A''$ and applying Lemma~\ref{lem:finite rank} to the  exact sequence $0\to A'\to A\to A'' \to 0$ we get the following.
\[
\corank_p(A)\le\corank_p(A')+\corank_p(A'')\le 1 + (\rank(A) - 1) = \rank(A).
\]
\end{proof}

\subsection{Ordered abelian groups}
\label{section:l order}
We will use some things from the theory of (totally) ordered abelian groups.
All oag's have at least two elements.
Note that if $(H_i;+,\prec_i)$ is an ordered abelian group for $i \in \{1,\ldots,n\}$ then the lexicographic order on $H_1 \times\cdots\times \times H_n$ is a group order.
The resulting oag is referred to as the lexicographic product of the $(H_i;+,\prec_i)$.

\medskip
An oag $(H;+,\prec)$ is \textbf{archimedean} if it satisfies one of the following equivalent conditions:
\begin{enumerate}
\item For any positive $\alpha,\beta \in H$ there is $n$ such that $n\alpha > \beta$.
\item $(H;+,\prec)$ does not have a non-trivial convex subgroup.
\item There is a (unique up to rescaling) embedding $(H;+,\prec) \to (\R;+,<)$.
\end{enumerate}
(1) $\Longleftrightarrow$ (3) follows by Hahn embedding and (1) $\Longleftrightarrow$ (2) is an easy exercise.
An ordered abelian group $(H;+,\prec)$ is \textbf{regular} if it satisfies one of the following equivalent conditions:
\begin{enumerate}
\item $(H;+,\prec)$ is elementarily equivalent to an archimedean ordered abelian group.
\item $(H;+,\prec)$ does not have a non-trivial definable convex subgroup.
\item If $n \ge 1$ and $I \subseteq H$ is an interval containing at least $n$ elements then $I \cap nH \ne \emptyset$.
\item Either $(H;+,\prec)$ is a model of Presburger arithmetic or $(H;\prec)$ is dense and $nH$ is dense in $H$ for every $n \ge 1$.
\end{enumerate}
(3)$\Longleftrightarrow$(4) follows by considering the usual axiomization of Presburger arithmetic.
It is easy to see that (1) implies (4), Robinson and Zakon showed that (3) implies (1) \cite{robinson-zakon}.
What is above shows that (1) implies (2).
Belegradek showed that (2) implies (3)~\cite[Theorem~3.15]{poly-regular}.

\medskip
We let $L_\mathrm{ordiv}$ be the expansion of the language $L_\mathrm{div}$ defined in the previous section by a binary relation $<$.
Consider ordered abelian groups to be $L_\mathrm{div}$-structures in the natural way.
Fact~\ref{fact:qe for oag} below is due to Weispfenning~\cite{Weispfenning}.
It is also essentially the simplest case of the description of definable sets in arbitrary ordered abelian groups, see \cite{coag,expanded,Sch.habil}.

\begin{fact}
\label{fact:qe for oag}
Any regular ordered abelian group admits quantifier elimination $L_\mathrm{ordiv}$.
\end{fact}

Note that convex subgroups are linearly ordered under inclusion.
Given a convex subgroup $J$ of $H$, we consider $H/J$ to be an ordered abelian group by declaring $a + J < b + J$ when every element of $a + J$ is strictly less than every element of $b + J$.
Fact~\ref{fact:oag} is an exercise.

\begin{fact}
\label{fact:oag}
Suppose $J$ is a convex subgroup of $H$.
Then $I \mapsto I/J$ gives a one-to-one correspondence between convex subgroups $I$ of $H$ containing $J$ and convex subgroups of $H/J$.
\end{fact}

Fact~\ref{fact:rank} follows by Fact~\ref{fact:oag} an easy induction on $k$, see \cite[pg~101]{trans}.

\begin{fact}
\label{fact:rank}
If $(H;+)$ is rank $k$ then $(H;+,\prec)$ has at most $k - 1$ non-trivial convex subgroups.
\end{fact}

The \textbf{regular rank} of $(H;+,\prec)$ is $n$ if there are exactly $n$ non-trivial definable convex subgroups, and $\infty$ if there are infinitely many.
A lexicographic product of $k$ regular abelian groups has regular rank at most $k$. 
Fact~\ref{fact:regular rank} is due to Belegradek~\cite[Thm~3.10]{poly-regular}.

\begin{fact}
\label{fact:regular rank}
Suppose that $(H;+,\prec)$ is an $\aleph_1$-saturated ordered abelian group of regular rank $n$.
Then there are non-trivial definable convex subgroups $\{0\}=J_0 \subseteq J_1 \subseteq \ldots \subseteq J_n=H$ of $H$ such that $J_i/J_{i - 1}$ is a regular ordered abelian group for each $i \in \{1,\ldots,n\}$ and $(H;+,\prec)$ is isomorphic to the lexicographic product $J_n/J_{n-1} \times \cdots\times J_2/J_2 \times J_1/J_0$.
\end{fact}

Fact~\ref{fact:regular rank} shows that $(H;+,\prec)$ is of regular rank at most $n$ if and only if $(H;+,\prec)$ is elementarily equivalent to a lexicographically ordered subgroup of $(\R^n;+)$~\cite[Theorem~3.13]{poly-regular}.

\subsection{Cyclically ordered abelian groups}
\label{section:cyclic order}
We now discuss the more obscure but quite natural class of cyclically ordered abelian groups.
A \textbf{cyclic order} on a set $G$ is a ternary relation $\cyc$ such that for all $a,b,c \in G$, the following holds:
\begin{enumerate}
\item If $\cyc(a,b,c)$, then $\cyc(b,c,a)$.
\item If $\cyc(a,b,c)$, then $\cyc(c,b,a)$,
\item If $\cyc(a,b,c)$ and $\cyc(a,c,d)$ then $\cyc(a,b,d)$.
\item If $a,b,c$ are distinct, then either $\cyc(a,b,c)$ or $\cyc(c,b,a)$.
\end{enumerate}

A \textbf{cyclic group order} on an abelian group $(G;+)$ is a cyclic order on $G$ which is preserved under $+$.
In this case, we call  $(G; +, \cyc)$ a \textbf{cyclically ordered abelian group}. 

\medskip \noindent Let $(G; +)$ be an abelian group. Suppose $(H; +,<)$ is a linearly ordered abelian group, $u$ is a positive element of $H$ such that $u\N$ is cofinal in $(H; <)$, and $\uppi\colon H \to G$ induces an isomorphism $(H \slash u\Z; +)\to(G; +)$.
Define the relation $\cyc$ on $G$ by declaring:  
$$ \cyc(\uppi(a), \uppi(b), \uppi(c)) \ \text{ if }\  a<b<c \text{ or } b< c< a \text{ or } c< a < b \quad \text{for all}\quad 0 \le a,b,c < 1.$$
We can easily check that $\cyc$ is an cyclic group order on  $(G;+)$. 
We call $(H; u, +,<)$ as above a {\bf universal cover} of $(G;+,\cyc)$.
We let $(\R/\Z;+,C)$ be the cyclically ordered group with universal cover $(\R;1,+,<)$, so $C$ is the usual counterclockwise cyclic order on $\R/\Z$.
A cyclically ordered abelian group is \textbf{archimedean} if it is isomorphic to a substructure of $(\R/\Z;+,C)$, equivalently if its universal cover is archimedean.
See \cite{Sw} for an intrinsic definition of the class of archimedean cyclically ordered abelian groups.

\medskip
Fact~\ref{fact:cover} is not hard.
See \cite{TW-cyclic} for a proof.

\begin{fact}
\label{fact:cover}
Let $(G; +, \cyc)$ be a cyclically ordered abelian group. Then $(G; +, \cyc)$  has a universal cover  $(H; u, +,<)$   which is unique up to unique isomorphism.  Moreover, $(G; +, \cyc)$ is isomorphic to $([0, u); +^*,\cyc^*)$ where $+^*$ and $\cyc^*$ are definable in $(H; u, +,<)$.
In particular $(G;+,\cyc)$ is interpretable in $(H;+,<)$.
\end{fact}

We describe how to construct $+^*$ and $\cyc^*$.
We let $\cyc^*(a,a',a'')$ when either $a < a' < a''$ or $a'<a''<a$ or $a''<a<a'$.
We let $a +^* a' = a+a'$ when $a+a'<u$ and $a+^*a'=a+a'-1$ otherwise.
Note that all of this is definable in $(H;+,<)$.




\newpage

\section{A few general results}
\label{section:few general}
My first goal is to give a few nice examples.
This will require some general results, which are given here.
All proofs could be left as exercises to the reader.
Here $\uptau$  is an injection $O \to M^m$.
We also let $\uptau$ be the function $O^n \to M^{mn}$ given by $(a_1,\ldots,a_n) \mapsto (\uptau(a_1),\ldots,\uptau(a_n))$.

\medskip
Proposition~\ref{prop:trace-basic} is immediate from the definitions.

\begin{proposition}
\label{prop:trace-basic}
Let $\Sa M, \Sa O$ and $\Sa P$ be structures.
\begin{enumerate} 
\item If $\Sa M$ trace defines $\Sa O$ and $\Sa O$ trace defines $\Sa P$ then $\Sa M$ trace defines $\Sa P$.
\item If $T$ trace defines $T^*$ and $T^*$ trace defines $T^{**}$ then $T$ trace defines $T^{**}$.
\item If $\Sa O \prec \Sa M$ then $\Sa O$ is trace definable in $\Sa M$.
\end{enumerate}
\end{proposition}

\noindent
Thus trace definibility is a quasi-order and trace equivalence is an equivalence relation.

\begin{proposition}
\label{prop:trace-interpret}
If $\Sa M$ interprets $\Sa O$ then $\Sa M$ trace defines $\Sa O$.
\end{proposition}

\noindent
In the proof below $\uppi$ will denote a certain map $X \to O$ and we will also use $\uppi$ to denote the map $X^n \to O^n$ given by
$ \uppi(a_1,\ldots,a_n) = (\uppi(a_1),\ldots,\uppi(a_n))$.

\begin{proof}
Suppose $\Sa O$ is interpretable in $\Sa M$.
Let $X \subseteq M^m$ be  $\Sa M$-definable, $E$ be an $\Sa M$-definable equivalence relation on $X$, and $\uppi \colon X \to O$ be a surjection such that
\begin{enumerate} 
\item for all $a,b \in E$ we have $E(a,b) \Longleftrightarrow \uppi(a) = \uppi(b)$, and
\item if $X \subseteq O^n$ is $\Sa O$-definable  then $\uppi^{-1}(X)$ is $\Sa M$-definable.
\end{enumerate}
Let $\uptau \colon O \to X$ be a section of $\uppi$ and $A = \uptau(O)$.
If $Y \subseteq O^n$ is $\Sa O$-definable then $\uppi^{-1}(X)$ is $\Sa M$-definable and $X = \{ a \in O^n : \uptau(a) \in \uppi^{-1}(Y) \}$.
Therefore $\Sa M$ trace defines $\Sa O$ via $\uptau$.
\end{proof}

We now prove some results that will be used to get trace definitions.

\begin{proposition}
\label{prop:qe-trace}
Suppose that $\Sa O$ is an $L$-structure which admits quantifier elimination and $\uptau\colon\Sa O\to\Sa M$ is an embedding of $L$-structures.
Then $\Sa M$ trace defines $\Sa O$ via $\uptau$.
\end{proposition}

We will often apply Proposition~\ref{prop:qe-trace} in the case when $\uptau$ is the identity.


\begin{proof}
Let $X \subseteq O^n$ be $\Sa O$-definable and $\varphi(x)$ be a formula which defines $X$.
Declare $Y$ to be $\{ a \in M^n : \Sa M \models \varphi(a) \}$.
We may suppose that $\varphi$ is quantifier free.
Then $\uptau^{-1}(Y) = X$.
\end{proof}

\begin{proposition}
\label{prop:eq neg 1}
Suppose that $L^*$ is an expansion of $L$ by relations and $\Sa O$ is an $L^*$-structure with quantifier elimination.
Suppose that $\Sa P$ is an $\Sa M$-definable $L$-structure and $\uptau \colon O \to P$ gives an $L$-embedding $\Sa O \!\upharpoonright \! L \to \Sa P$.
Suppose that for every $k$-ary relation $R \in L^* \setminus L$ there is an $\Sa M$-definable $Y \subseteq P^k$ such that
\[
\Sa O \models R(a) \quad \Longleftrightarrow \quad a \in Y \quad \text{for all} \quad a \in O^k.
\]
Then $\Sa M$ trace defines $\Sa O$ via $\uptau$.
\end{proposition}

\begin{proof}
We expand $\Sa P$ to an $L^*$-structure.
For each $k$-ary $R \in L^* \setminus L$ fix an $\Sa M$-definable subset $Y_R$ of $P^{k}$ such that $\uptau^{-1}(Y_R)$ agrees with $\{ a \in O^k : \Sa O \models R(a) \}$.
Let $\Sa P^*$ be the $L^*$-structure expanding $\Sa P$ such that for every $k$-ary $R \in L^*\setminus L$ and $a \in P^{k}$ we have $\Sa P^* \models R(a) \Longleftrightarrow a \in Y_R$.
Then $\Sa M$ defines $\Sa P^*$ and $\Sa P^*$ trace defines $\Sa O$ via $\uptau$ by Proposition~\ref{prop:qe-trace}.
\end{proof}

\begin{proposition}
\label{prop:qe0}
Suppose that $L^*$ is an expansion of $L$ by relations, $L^{**}$ is an arbitrary expansion of $L$, $\Sa O$ is an $L^*$-structure with quantifier elimination, and $\Sa M$ is an $L^{**}$-structure such that $\Sa M \! \upharpoonright \! L$ extends $\Sa O \! \upharpoonright \! L$.
Suppose that for every $k$-ary relation $R \in L^* \setminus L$ there is an $\Sa M$-definable $Y \subseteq M^k$ such that
\[
\Sa O \models R(a) \quad \Longleftrightarrow \quad a \in Y \quad \text{for all} \quad a \in O^k.
\]
Then $\Sa M$ trace defines $\Sa O$ via the identity $O \to M$.
\end{proposition}

\begin{proof}
Apply Proposition~\ref{prop:eq neg 1} with $\Sa P = \Sa M\!\upharpoonright\!L$.
\end{proof}

\begin{proposition}
\label{prop:qe}
Suppose that $L^*$ is relational, $\Sa O$ is an $L^*$-structure with quantifier elimination, and $\uptau \colon O \to M^m$ is an injection.
Suppose that for every $k$-ary relation $R \in L^*$ there is an $\Sa M$-definable $Y \subseteq M^{mk}$ such that
$$ \Sa O \models R(a) \quad \Longleftrightarrow\quad \uptau(a)   \in Y \quad\text{for all}\quad a \in O^k. $$
\end{proposition}

\begin{proof}
Apply Proposition~\ref{prop:eq neg 1} with $L$ the empty language.
\end{proof}

\noindent
Corollary~\ref{cor:qe} is a reformulation of Proposition~\ref{prop:qe}.

\begin{corollary}
\label{cor:qe}
Suppose that $L^*$ is relational and $\Sa O$ is an $L^*$-structure with quantifier elimination.
Then the following are equivalent:
\begin{enumerate}
[leftmargin=*]
\item $\Sa M$ trace defines $\Sa O$, and
\item $\Sa O$ is isomorphic to an $L^*$-structure $\Sa P$ such that $P \subseteq M^m$ and for every $k$-ary $R \in L$ there is an $\Sa M$-definable $Y \subseteq M^{mk}$ such that for any $a \in P^k$ we have $\Sa P \models R(a) \Longleftrightarrow a \in Y$.
\end{enumerate}
\end{corollary}

Recall our standing assumption that $T,T^*$ is a complete $L,L^*$-theory, respectively.

\begin{proposition}
\label{prop:trace-theories}
Suppose that $T$ trace defines $T^*$.
Let $\uplambda \ge |L^*|$ be an infinite cardinal, $\Sa O \models T^*$, $|O| < \uplambda$, and $\Sa M \models T$ be $\uplambda$-saturated.
Then $\Sa M$ trace defines $\Sa O$.
\end{proposition}

Thus if $T$ trace defines $T^*$ then every $T^*$-model is trace definable in a $T$-model.
We make constant use of Proposition~\ref{prop:trace-theories} below, often without reference.

\begin{proof}
Suppose that $\Sa P \models T^*$ is trace definable in $\Sa N \models T$.
We may suppose that $P \subseteq N^m$ and $\Sa N$ trace defines $\Sa P$ via the identity $P \to N^m$.
By Morleyization we may suppose that $L^*$ is relational and $T^*$ admits quantifier elimination.
For each $n$-ary relation $R \in L^*$ we fix an $L(N)$-formula $\delta_R(x)$ such that for all $a \in P^n$ we have $\Sa N \models \delta_R(a) \Longleftrightarrow \Sa P \models R(a)$.
Let $A$ be the set of coefficients of the $\delta_R$ for $R \in L$.
Assume that $L$ and $L^*$ are disjoint.
Let $L_\cup$ be $L(A) \cup L^*  \cup \{P\}$ where $P$ is a new $m$-ary predicate.
Given an $L_\cup$-structure $\Sa S$ we define the associated $L^*$-structure $\Sa S_{L^*}$ by letting $\Sa S_{L^*}$ have domain $\{ a \in S^m : \Sa S \models P(a) \}$, and letting the interpretation of each $n$-ary $R \in L^*$ be 
$$\{ (b_1,\ldots,b_n) \in S^{mn} : \Sa S \models P(b_1) \land \cdots \land P(b_n) \land \delta_R(b_1,\ldots,b_n) \}.$$
Let $T_\cup$ be the $L_\cup$-theory where an $L_\cup$-structure $\Sa S$ satisfies $T_\cup$ if and only if: 
\begin{enumerate}
\item $\Sa S\!\upharpoonright\! L(A)$ satisfies the $L(A)$-theory of $\Sa N$, and
\item $\Sa S_{L^*} \models T^*$.
\end{enumerate}
It is clear that $T_\cup$ is consistent.
Proposition~\ref{prop:qe} shows that if $\Sa S \models T_\cup$ then $\Sa S \!\upharpoonright\!L$ trace defines $\Sa S_{L^*}$.
Let $\Sa S \models T_\cup$ be $\uplambda$-saturated.
Then $\Sa S_L$ is also $\uplambda$-saturated, hence there is an elementary embedding $\Sa O \to \Sa S_L$.
We suppose that $\Sa O$ is an elementary substructure of $\Sa S_L$.
Let $\Sa T$ be an elementary substructure of $\Sa S$ which contains $O \cup A$ and satisfies $|T| \le \uplambda$.
Then $\Sa T\!\upharpoonright\!L$ trace defines $\Sa T_L$, and hence trace defines $\Sa O$.
By saturation there is an elementary embedding $\Sa T_L \to \Sa M$, so we suppose $\Sa T_L$ is an elementary submodel of $\Sa M$.
Then $\Sa M$ trace defines $\Sa O$.
\end{proof}

We need one more tool for our first examples, disjoint unions.
The collection of trace equivalence classes is a partial order under trace definibility.
Lemma~\ref{lem:disjoint union} shows that this partial order admits joins, the join of the trace equivalence classes of $\Sa M_1,\ldots,\Sa M_n$ is the class of the disjoint union $\Sa M_1 \sqcup\cdots \sqcup \Sa M_n$.

\medskip
We first discuss the disjoint union $\bigsqcup_{i = 1}^n \Sa M_i$ of structures $\Sa M_1, \ldots, \Sa M_n$. After possibly Morleyzing we suppose that each $L_i$ is relational and each $\Sa M_i$ admits quantifier elimination\footnote{Technically we only define the disjoint union up to definable equivalence.}.
We may suppose that the $L_i$ are pairwise disjoint.
Let $L_\sqcup$ be  $L_1 \cup\cdots\cup L_n$ together with new unary relations $P_1,\ldots,P_n$.
Let $\Sa M_1 \sqcup\cdots \sqcup \Sa M_n$ be the $L_\sqcup$-structure with domain $M_1 \sqcup\cdots \sqcup M_n$, where each $P_i$ defines $M_i$, and for all $k$-ary $R \in L_i$ and $a_1,\ldots,a_k \in M_1 \sqcup\cdots \sqcup  M_n$ we have
\begin{align*}
\Sa M_1 \sqcup\cdots \sqcup \Sa M_n \models R(a_1,\ldots,a_k) \quad \Longleftrightarrow \quad \text{either} \quad & a_1,\ldots,a_k \in M_j \text{  for some  } j \ne i, \\ \text{or}\quad & a_1,\ldots,a_k \in M_i \text{  and  } \Sa M_i \models R(a_1,\ldots,a_k).
\end{align*}
We also consider each $\Sa M_i$ to be an $L_\sqcup$-structure by considering $\Sa M_i$ to be a substructure of $\Sa M_1 \sqcup\cdots \sqcup \Sa M_n$ in the natural way.
It is then easy to see that the structure induced on $M_1 \times\cdots\times M_n$ by $\Sa M_1 \sqcup\cdots \sqcup \Sa M_n$ is interdefinable with $\Sa M_1 \times\cdots\times \Sa M_n$, where the product is the usual product of $L_\sqcup$-structures.
By Feferman-Vaught any $\Sa M_1 \sqcup\cdots\sqcup \Sa M_n$-definable subset of $M_1 \times\cdots\times M_n$ is a finite union of sets of the form $X_1 \times\cdots\times X_n$ where each $X_i \subseteq M_i$ is $\Sa M_i$-definable, see \cite[Corollary~9.6.4]{Hodges}.
Hence $\Sa M_1 \sqcup\cdots\sqcup \Sa M_n$ admits quantifier elimination.

\begin{lemma}
\label{lem:disjoint union}
If $\Sa M$ trace defines each $\Sa O_1,\ldots,\Sa O_n$ then $\Sa M$ trace defines $\Sa O_1 \sqcup\cdots \sqcup \Sa O_n$.
\end{lemma}

\begin{proof}
Suppose each $\Sa O_i$ is an $L_i$-structure.
As above we may suppose that each $L_i$ is relational and each $\Sa O_i$ admits quantifier elimination.
We may suppose that we have $O_i \subseteq M^{m_i}$ for each $i$ and that each $\Sa O_i$ is trace definable via the identity.
Let $\uptau \colon O_1 \sqcup\cdots\sqcup O_n \to M^{m_1} \sqcup\cdots \sqcup M^{m_n}$ be the natural injection.
For each $k$-ary $R \in L_i$ we fix $\Sa M$-definable $X_R \subseteq M^{km_i}$ such that $X_R \cap O^k_i = \{a \in O^k_i : \Sa O_i \models R(a)\}$ and let $Y_R$ be the set of $a_1,\ldots,a_k \in (M^{m_1} \sqcup\cdots \sqcup M^{m_n})^k$ such that either $a_1,\ldots,a_k \in M^{m_j}$ for some $j \ne i$ or $a_1,\ldots,a_k \in M^{m_i}$ and $(a_1,\ldots,a_k) \in X_R$.
Let $P$ be the image of $\uptau$.
Then $\uptau$ gives an isomorphism $\Sa O_1 \sqcup\cdots\sqcup \Sa O_n \to (P ; (Y_R)_{R \in L_\sqcup})$
Note that each $X_R$ is $\Sa M$-definable.
Quantifier elimination for $\Sa O_1 \sqcup\cdots\sqcup\Sa O_n$ and Corollary~\ref{cor:qe} together show that $\Sa M$ trace defines $\Sa O_1 \sqcup\cdots\sqcup \Sa O_n$ via $\uptau$.
\end{proof}

Lemma~\ref{lem:disjoint union 1} is immediate from Lemma~\ref{lem:disjoint union}.

\begin{lemma}
\label{lem:disjoint union 1}
Suppose that $\Sa M_i$ is trace equivalent to $\Sa M^*_i$ for each $i \in \{1,\ldots,n\}$.
Then $\Sa M_1 \sqcup\cdots \sqcup \Sa M_n$ is trace equivalent to $\Sa M^*_1 \sqcup\cdots \sqcup \Sa M^*_n$
\end{lemma}

We finally characterize stability in terms of trace definibility.

\begin{lemma}
\label{lem:stable-0}
$T$ is unstable if and only if $T$ trace defines $\dlo$.
\end{lemma}

It follows that any infinite linear order trace defines $\dlo$.
By Fact~\ref{fact:zp} a discrete linear order cannot interpret $\dlo$.
 
\begin{proof}
The right to left implication follows from the first claim of Proposition~\ref{prop:stable-0}.
By Proposition~\ref{prop:trace-theories} it is enough to suppose that $\Sa M$ is unstable and $\aleph_1$-saturated and show that $\Sa M$ trace defines $(\Q;<)$.
There is a sequence $(\alpha_q : q \in \Q)$ of elements of $M^m$ and a formula $\phi(x,y)$ such that for all $p,q \in \Q$ we have $\Sa M \models \phi(\alpha_p,\alpha_q)$ if and only if $p < q$.
Let $\uptau \colon \Q \to M^m$ be given by declaring $\uptau(q) = \alpha_q$ for all $q \in \Q$.
As $(\Q;<)$ admits quantifier elimination an application of Proposition~\ref{prop:qe} shows that $\Sa M$ trace defines $(\Q;<)$ via $\uptau$.
\end{proof}

\subsection{The multisorted case}
One can also define trace definibility between multisorted structures in a natural way, and basically all of our results adapt to this setting.
One thing gained by using multisorted structures is that one can take disjoint unions of arbitrary collections of structures and prove a generalization of Lemma~\ref{lem:disjoint union 1}.
If we allow multisorted structures then the partial order of trace equivalence classes admits arbitrary joins.
This is nice but at present I don't think it's worth the pain of writing down.

\section{A few easy examples}
\label{section:few}
We now give some attractive examples.
Many of these will be generalized in later sections.

\medskip
A convex valuation on an ordered field is a non-trivial valuation with convex valuation ring.
$\mathrm{RCVF}$ is the theory of a real closed field equipped with a convex valuation.
More precisely $\rcvf$ is the theory of $(K,\prec)$ where $K$ is a real closed ordered field, $\prec$ is a quasi-order on $K\setminus\{0\}$, and there is a convex valuation $v$ on $K$ such that such that $\beta\prec \beta^* \Longleftrightarrow v(\beta^*)> v(\beta)$ for all non-zero $\beta,\beta^*\in K$.
$\rcvf$ is complete and admits quantifier elimination by work of Cherlin and Dickmann, see \cite{cd} or \cite[Thm~3.6.6]{trans}.
The archimedean valuation on an ordered field is the valuation whose valuation ring is the convex hull of $\Z$ in $K$, in this case we have $\beta\prec\beta^*$ if and only if $n|\beta|<|\beta^*|$ for all $n$.
Note that $\rcvf$ is the theory of the archimedean valuation on a non-archimedean real closed field.

\begin{proposition}
\label{prop:rcvf}
$\rcvf$ is trace equivalent to $\rcf$.
\end{proposition}

This is generalized in Section~\ref{section:completions}.
$\rcvf$ is not interpretable in an o-minimal structure: o-minimality implies rosiness, rosiness is preserved under interpretations, $\rcvf$ is not rosy.

\begin{proof}
We need to show that $\rcf$ trace defines $\rcvf$.
Suppose that $K$ is a non-archimedean real closed field and $\prec$ is induced by the archimedean valuation on $K$.
We let $F$ be a $|K|^+$-saturated elementary extension of $K$.
We show that $F$ trace defines $(K,\prec)$ via the identity.
We suppose that $X\subseteq K^m$ is $(K,\prec)$-definable and produce $F$-definable $Y\subseteq F^m$ such that $X=K^m \cap X$.
By quantifier elimination for $\rcvf$ we may suppose that $X$ is either definable in the language of ordered fields or $X = \{\alpha\in K^m : f(\alpha)\prec g(\alpha)\}$ for some $f,g\in K[x_1,\ldots,x_m]$.
In the first case we proceed as in the proof of Proposition~\ref{prop:qe-trace}.
Suppose $X = \{\alpha\in K^m : f(\alpha)\prec g(\alpha)\}$.
Then for any $\alpha\in K^n$ we have $\alpha\in X$ if and only if $n|f(\alpha)| < |g(\alpha)|$ for all $n$.
Fix $\lambda\in F$ such that $\lambda > \N$ and $\lambda<\beta$ for all $\beta\in K$ satisfying $\beta > \N$.
Let $Y = \{\alpha\in F^m : \lambda |g(\alpha)| <  |f(\alpha)|\}$.
It is now easy to see that $X = Y\cap K^m$.
\end{proof}

\begin{proposition}
\label{prop:fg-group}
All infinite finitely generated abelian groups are trace equivalent.
\end{proposition}

This example is generalized in Proposition~\ref{prop:torsion free in Z}.
By Corollary~\ref{cor:group in a group} below $(\Z^m;+)$ interprets $(\Z^n;+)$ if and only if $m$ divides $n$.
Any expansion of $(\Z^n;+)$ trace definable in $\Th(\Z;+)$ is a reduct of the structure induced on $\Z^n$ by $(\Z;+)$, in particular $\Th(\Z;+)$ does not trace define a proper expansion of $(\Z;+)$.
This follows from Proposition~\ref{prop:U rank} below and the Palac\'{\i}n-Sklinos theorem that any expansion of $(\Z^n;+)$ of finite $U$-rank is such a reduct~\cite{palacin-sklinos}.

\begin{proof}
Any finitely generated abelian group is definable in $(\Z;+)$.
It is enough to show that an infinite finitely generated abelian group $A$ trace defines $(\Z;+)$.
We have $A = (\Z;+)\oplus A'$ for a finitely generated abelian group $A'$, so we consider $(\Z;+)$ to be a subgroup of $A$ in the natural way.
Note that if $\alpha\in\Z$ and $k\in\N$ then $k$ divides $\alpha$ in $(\Z;+)$ if and only if $k$ divides $\alpha$ in $A$.
By Fact~\ref{fact:abelian qe} and Proposition~\ref{prop:qe-trace} $A$ trace defines $(\Z;+)$ via the identity.
\end{proof}

\begin{proposition}
\label{prop:p adic fields}
Fix a prime $p$.
All finite extensions of $\Q_p$ are trace equivalent.
\end{proposition}

Halevi, Hasson, and Peterzil~\cite{halevi-hasson-peterzil} show that an infinite field interpretable in $\Q_p$ is isomorphic to a finite extension of $\Q_p$.
Their work goes through for finite extensions of $\Q_p$.
Thus two finite extensions of $\Q_p$ are mutually interpretable if and only if they are isomorphic.
For example we see that $\Q_p(\sqrt{q})$ trace defines but does not interpret $\Q_p$ when $q \ne p$ is a prime.







\begin{proof}
Recall that $\Q_p$ defines any finite extension of $\Q_p$.
It is therefore enough to fix a finite extension $\K$ of $\Q_p$ and show that $\K$ trace defines $\Q_p$ via the identity $\pfield \to \K$.
We consider $\pfield$ as a structure in the Macintyre language.
That is, we let $L'$ be the expansion of the language of rings by unary relations $(P_n : n \ge 2)$ and consider $\Q_p$ to be an $L'$-structure by declaring $P_n$ to be the set of non-zero
$n$th powers.
Recall that $\pfield$ admits quantifier elimination in $L'$~\cite{macintyre-p-adic}.
We apply Proposition~\ref{prop:qe0} with $L = L^{**}$ the language of fields, $L^* = L'$, $\Sa O = \pfield$, and $\Sa N = \K$.
It is enough to fix $n \ge 2$ and produce $\K$-definable $X \subseteq \K$ such that $\alpha \in P_n \Longleftrightarrow \alpha \in X$ for all $\alpha \in \pfield$.
By Fact~\ref{fact:p adic field} there is $m$ such that every $\alpha\in\Q_p$ which is an $m$th power in $\K$ is an $n$th power in $\Q_p$.
Let $Y = \{ \alpha^m : \alpha \in \K^\times\}$.
Then $Y \cap \Q^\times_p$ is a subgroup of $\Q^\times_p$ contained in $P_n$, this subgroup is finite index as $Y$ is a finite index subgroup of $\K^\times$.
Hence there are $\beta_1,\ldots,\beta_k \in P_n$ such that
\begin{align*}
    P_n &= \beta_1(Y \cap \Q^\times_p) \cup \cdots \cup \beta_k(Y \cap \Q^\times_p) \\ &= (\beta_1 Y \cup \cdots \cup \beta_k Y) \cap \Q^\times_p.
\end{align*}
Take $X = \beta_1 Y \cup \cdots \cup \beta_k Y$.
\end{proof}

\begin{proposition}
\label{prop:dense pair}
Let $C$ be the ternary relation on $\R/\Z$ where
\[
C(\alpha+\Z,\alpha'+\Z,\alpha''+\Z)\quad\Longleftrightarrow\quad (\alpha<\alpha'<\alpha'')\vee(\alpha'<\alpha''<\alpha)\vee(\alpha''<\alpha<\alpha')
\]
for all $0\le\alpha,\alpha',\alpha''<1$.
The following structures are trace equivalent:
\begin{enumerate}
\item $(\R;+,<)$,
\item $(\R;+,<,\Q)$,
\item $(\R/\Z;+,C)$,
\item $(\R/\Z;+,C,\Q/\Z)$.
\end{enumerate}
\end{proposition}

By Corollary~\ref{cor:rama-2} $(\R;+,<,\Q)$ is not interpretable in an o-minimal expansion of an oag, a similar argument shows that $(\R/\Z;+,C,\Q/\Z)$ is not interpretable in an o-minimal expansion of an oag.
By Proposition~\ref{prop:cyclic order 1} $(\R/\Z;+,C)$ does not interpret $(\R;+,<)$.
See Section~\ref{section:open core} for more on sharpness of Proposition~\ref{prop:dense pair}.

\begin{proof}
We first show that $(\R;+,<)$ trace defines $(\R;+,<,\Q)$.
By \cite{DMS1} $(\R;+,<,\Q)$ admits quantifier elimination after we add a unary function for scalar multiplication $\R \to \R$ by each rational.
Note that  $(\R/\Q;+)$ is a continuum size $\Q$-vector space, hence there is an isomorphism $\upchi \colon (\R/\Q;+) \to (\R;+)$.
Let $\uptau \colon \R \to \R \times \R$ be given by $\uptau(\alpha) = (\upchi(\alpha + \Q),\alpha)$.
We show that $(\R;+,<)$ trace defines $(\R;+,<,\Q)$ via $\uptau$.
Note that $\uptau$ is a $\Q$-vector space embedding $(\R;+) \to (\R^2;+)$.
We apply Proposition~\ref{prop:eq neg 1} with $L$ the language of $\Q$-vector spaces, $L^*$ the expansion of $L$ by $<$ and a unary relation, $\Sa O = (\R;+,<,\Q)$, and $\Sa P$ the $\Q$-vector space $\R^2$.
It suffices to produce $(\R;+,<)$-definable $X \subseteq \R^2$ and $Y \subseteq \R^4$ such that 
\begin{align*}
\uptau(\alpha) \in X \quad \Longleftrightarrow \quad \alpha \in \Q \quad &\text{for all  } \alpha \in \R\\
(\uptau(\alpha),\uptau(\beta)) \in Y \quad \Longleftrightarrow \quad \alpha < \beta \quad &\text{for all  } \alpha,\beta \in \R
\end{align*}
Let $X$ be the set of $(s,t)\in\R^2$ such that $s = 0$ and $Y=\{((s,t),(s',t')) \in \R^4 : t < t' \}$.

\medskip
Note $(\R;1,+,<)$ is the universal cover of $(\R/\Z;+,C)$, hence $(\R;+,<)$ interprets $(\R/\Z;+,C)$.
A similar construction shows that $(\R;+,<,\Q)$ interprets $(\R/\Z;+,C,\Q/\Z)$.
It is therefore enough to show that $\Th(\R/\Z;+,C)$ trace defines $\doag$.
Let $(G;+,C)$ be an $\aleph_1$-saturated elementary extension of $(\R/\Z;+,C)$.
Let $J$ be the set of $\alpha \in G$ such that $C(\beta,\alpha,\beta^*)$ holds for all $\beta,\beta^* \in \R/\Z$ satisfying $C(\beta,0,\beta^*)$, i.e. $J$ is the subgroup of infinitesimals.
Let $\triangleleft$ be binary relation on $J$ where we have $\beta \triangleleft \beta^*$ when either $C(\beta,0,\beta^*)$, $C(0,\beta,\beta^*)$, or $C(\beta,\beta^*,0)$.
It is easy to see that $(J;+,\triangleleft)\models\doag$.
Apply Proposition~\ref{prop:qe0} and qe for $\doag$.
\end{proof}

Given $\alpha\in\R\setminus\Q$ we let $\upchi_\alpha:\Z\to\R/\Z$ be the character given by $\upchi_\alpha(m) = m\alpha+\Z$ and let $C_\alpha$ be the pullback of the usual counterclockwise cyclic order $C$ on $\R/\Z$ by $\upchi_\alpha$.

\begin{proposition}
\label{prop:final Z}
The following structures are trace equivalent:
\begin{enumerate}
\item\label{item:1} $(\Z;+,<)$,
\item\label{item:2} $(\R;+,<,\Z)$,
\item\label{item:3} $(\R;+,<,\Z,\Q)$,
\item\label{item:4} $(\Z;+)\sqcup(\R;+,<)$, 
\item\label{item:6} $(\Z;+,C_\alpha)$ for any $\alpha\in\R\setminus\Q$,
\item\label{item:5} and $(\Z^n;+,\prec)$ for any archimedean group order $\prec$,
\end{enumerate}
\end{proposition}

This example is generalized in Section~\ref{section:presburger}, e.g. we remove the archimedean assumption.

\medskip
By Corollary~\ref{cor:group in a group} $(\Z;+)\sqcup(\R;+,<)$ does not interpret a non-divisible ordered abelian group.
Zapryagaev and Pakhomov have shown that any linear order interpretable in $(\Z;+,<)$ is scattered of finite Hausdorff rank, see Fact~\ref{fact:Zapryagaev and Pakhomov}.
Hence $(\Z;+,<)$ does not interpret $(\R;<)$.
By Proposition~\ref{prop:general cyclic} $(\Z^n;+,\prec)$ does not interpret $(\Z;<)$ when $\prec$ is archimedean and $n \ge 2$, by Corollary~\ref{cor:cyclic order} $(\Z;+,C_\alpha)$ doesn't interpret $(\Z;<)$, and by Fact~\ref{fact:zp} $(\Z;+,<)$ does not interpret $(\Z;+,C_\alpha)$.
I would also conjecture that if $\alpha,\beta\in\R\setminus\Q$ are not rational multiples of each other then $(\Z;+,C_\alpha)$ does not interpret $(\Z;+,C_\beta)$ and vice versa, but I do not have a proof of this.
Something similar should hold for archimedean group orders on $(\Z^n;+)$.

\medskip
Presburger arithmetic does not trace define a proper expansion of $(\Z;+,<)$.
This follows by Proposition~\ref{prop:dp-rank} below, strong dependence of Presburger arithmetic, and  Dolich-Goodrick's theorem that there are no non-trivial strongly dependent expansions of $(\Z;+,<)$~\cite{DG}.

\medskip
Prop~\ref{prop:final Z} raises the question of whether $\doag$ can trace define $(\Z;+)$.
If it does then $(\Z;+,<)$ and $(\R;+,<)$ are trace equivalent.
I would conjecture more generally that an o-minimal structure cannot trace define $(\Z;+)$, but I don't have a clue as to how to get it.

\begin{proof}
We first show that Presburger arthimetic trace defines $(\Z;+)\sqcup(\R;+,<)$.
By Lemma~\ref{lem:disjoint union} it is enough to show that Presburger arithmetic trace defines $\doag$.
Let $(Z;+,<)$ be an $\aleph_1$-saturated model of Presburger arithmetic and fix positive $\beta\in Z$ such that $\beta$ is divisible by every $k\in\N,k\ge 1$.
Let $\uptau\colon\Q\to Z$ be given by $\uptau(q)=q\beta$.
Then $\uptau$ gives an embedding $(\Q;+,<)\to(Z;+,<)$.
Apply Proposition~\ref{prop:qe-trace} and quantifier elimination for $\doag$.

\medskip
We now suppose that $\prec$ is an archimedean group order on $(\Z^n;+)$.
We show that $(\Z^n;+,\prec)$ is trace definable in $(\Z;+)\sqcup(\R;+,<)$.
In particular this shows that $(\Z;+)\sqcup(\R;+,<)$ trace defines $(\Z;+,<)$.
Let $\upchi$ be the unique-up-to-rescaling embedding $(\Z^n;+,\prec)\to(\R;+,<)$.
We let $\Sa P$ be the structure $(\Z^n\times\R;+,\triangleleft,(P_k)_{k\in\N})$ where $+$ is the usual addition, $\triangleleft$ is the binary relation given by $(m,t)\triangleleft(m^*,t^*)\Longleftrightarrow t < t^*$, and each $P_k$ is a unary relation with $P_k(m_1,\ldots,m_n,t)\Longleftrightarrow k|m_1,\ldots,k|m_n$.
Note that $\Sa P$ is definable in $(\Z;+)\sqcup(\R;+,<)$.
Let $\uptau\colon\Z^n\to\Z^n\times\R$ be given by $\uptau(\alpha)=(\alpha,\upchi(\alpha))$.
Identify $\Z^n$ with its image under $\uptau$ and hence identify $(\Z^n;+,\prec)$ with a substructure of $\Sa P$.
Fact~\ref{fact:qe for oag} and Proposition~\ref{prop:qe-trace} together show that $(\Z;+)\sqcup(\R;+,<)$ trace defines $(\Z^n;+,\prec)$ via $\uptau$.

\medskip
We show that $(\Z;+,C_\alpha)$ is interpretable in $(\Z^2;+,\prec)$ for an archimedean group order $\prec$.
Let $H=\{m\alpha+m^*:m,m^*\in\Z\}$.
Then $(H;1,+,<)$ is the universal cover of $(\Z;+,C_\alpha)$ as seen in Section~\ref{section:cyclic order}, so $(H;+,<)$ interprets $(\Z;+,C_\alpha)$.
We let $\upchi$ be the map $\Z^2\to H$ be given by $\upchi(m,m^*)=m\alpha+m^*$ and let $\prec$ be the pullback of $<$ by $\upchi$.
Then $\upchi$ gives an isomorphism $(\Z^2;+,\prec)\to(H;+,<)$.
Hence $(\Z;+,C_\alpha)$ is interpretable in $(\Z^2;+,\prec)$.

\medskip
We show that $(\Z;+)\sqcup(\R;+,<)$ trace defines $(\R;+,<,\Z,\Q)$.
By Proposition~\ref{prop:dense pair} and Lemma~\ref{lem:disjoint union} it's enough to show $(\Z;+,<)\sqcup(\R;+,<,\Q)$ interprets $(\R;+,<,\Z,\Q)$.
Let $\Sa O$ be the structure $(\Z\times[0,1);\Tilde{+},\prec,\Z\times\{0\}, \Z\times[\Q\cap[0,1)])$ where $(m,t)\Tilde{+}(m^*,t^*) = (m+m^*,t+t^*)$ when $t + t^*<1$ and $(m,t)\Tilde{+}(m^*,t^*) = (m+m^*+1,t+t^*-1)$ otherwise, and $\prec$ is the lexicographic order on $\Z\times[0,1)$.
Note that $\Sa O$ is definable in $(\Z;+,<)\sqcup(\R;+,<,\Q)$ and the map $\Z\times[0,1)\to\R$, $(m,t)\mapsto m+t$ gives an isomorphism $\Sa O\to(\R;+,<,\Z,\Q)$.

\medskip
Fix $\alpha\in\R\setminus\Q$.
We show that $\Th(\Z;+,C_\alpha)$ trace defines $(\Z;+)\sqcup(\R;+,<)$.
By Lemma~\ref{lem:disjoint union} and Proposition~\ref{prop:trace-theories} it is enough to show that $\Th(\Z;+,C_\alpha)$ trace defines $\doag$.
Let $(Z;+,C)$ be an $\aleph_1$-saturated elementary extension of $(\Z;+,C_\alpha)$.
An application of Kronecker density shows that for any non-zero $a\in\Z$ and $k\ge 1$ there is $a'\in\Z$ such that $k$ divides $a'$ and $C_\alpha(0,a,a')$.
Hence there is $\beta\in Z$ such that every $k\ge 1$ divides $\beta$ and $C(0,\beta,a)$ holds for every $a\in\Z$.
Let $\uptau\colon\Q\to Z$ be given by $\uptau(q)=q\beta$.
Let $\triangleleft$ be the binary relation on $Z$ defined by declaring $\alpha\triangleleft\alpha^*$ when either $C(0,\alpha,\alpha^*)$, $C(\alpha,0,\alpha^*)$, or $C(\alpha,\alpha^*,0)$, so $\triangleleft$ is $(Z;+,C)$-definable.
It is easy to see that $\uptau$ gives an embedding $(\Q;+,<)\to(Z;+,\triangleleft)$.
Apply Proposition~\ref{prop:qe-trace} and quantifier elimination for $\doag$.
\end{proof}

\begin{proposition}
\label{prop:Q/Z}
$(\Q;+)$ and $(\Q/\Z;+)$ are trace equivalent.
\end{proposition}

This example is generalized in Proposition~\ref{prop:ab}, Lemma~\ref{lem:cover}, and Proposition~\ref{prop:divisible abelian}.
By Fact~\ref{fact:group in a group} neither of these structures can interpret the other.

\begin{proof}
As $(\Q/\Z;+)$ is an abelian group of unbounded exponent $(\Q/\Z;+)$ is elementarily equivalent to $(\Q/\Z;+)\oplus(\Q;+)$ \cite[Lemma~A.2.4]{Hodges}.
Let $\uptau\colon (\Q;+)\to(\Q/\Z;+)\oplus(\Q;+)$ be the natural embedding.
By Proposition~\ref{prop:qe-trace} and quantifier elimination for divisible abelian groups $(\Q/\Z;+)\oplus(\Q;+)$ trace defines $(\Q;+)$ via $\uptau$.

\medskip
We show that $(\Q;+)$ trace defines $(\Q/\Z;+)$.
Let $J = [0,1)\cap \Q$ and let $\uptau\colon\Q/\Z\to J$ be the bijection defined by letting $\uptau(\gamma+\Z)$ be the unique element of $[\gamma+\Z]\cap[0,1)$ for all $\gamma\in\Q$.
Hence $\uptau(\gamma+\Z)=\gamma$ for all $\gamma\in J$.
Note that $\uptau$ is a section of the quotient map $\uppi\colon\Q\to\Q/\Z$, we will use this.
Suppose that $X$ is a $(\Q/\Z;+)$-definable subset of $(\Q/\Z)^n$.
We construct $(\Q;+)$-definable $Y\subseteq\Q^n$ such that $X=\uptau^{-1}(Y)$, equivalently $X=\uppi(Y\cap J^n)$.
By quantifier elimination for divisible abelian groups we suppose that $X = \{\alpha\in(\Q/\Z^n) : T(\alpha)+\beta=0\}$ for a term $T(x_1,\ldots,x_n)=m_1x_1+\cdots+m_nx_n$, $m_1,\ldots,m_n\in\Z$ and $\beta\in(\Q/\Z)^n$.
For any $\alpha\in\Q/\Z$ we have $\alpha\in X$ if and only if $T(\uptau(\alpha))+\uptau(\beta)\in\Z$.
By construction $|\uptau(\alpha)|<1$ for all $\alpha\in\Q/\Z$.
Let $m=\max\{|m_1|,\ldots,|m_n|\}$.
For any $\alpha=(\alpha_1,\ldots,\alpha_n)\in(\Q/\Z)^n$ we have
\begin{align*}
|T(\uptau(\alpha))+\uptau(\beta)| &= |m_1\uptau(\alpha_1)+\cdots+m_n\uptau(\alpha_n)+\uptau(\beta)|\\&\le|m_1||\uptau(\alpha_1)|+\cdots+|m_n||\uptau(\alpha_n)|+|\uptau(\beta)|<mn+1.
\end{align*}
Let $I$ be the set of $a\in\Z$ such that $|a|\le mn$.
Then for any $\alpha\in(\Q/\Z)^n$ we have $\alpha\in X$ if and only if $T(\uptau(\alpha))+\uptau(\beta)\in I$.
Let $Y$ be the set of $\alpha\in\Q^n$ such that $T(\alpha)+\uptau(\beta)\in I$.
Note that $Y$ is $(\Q;+)$-definable as $I$ is finite and we have $X=\uptau^{-1}(Y)$. 
\end{proof}

We now give some ``relational" examples.

\begin{proposition}
\label{prop:random}
Fix $k \ge 2$.
The following structures are trace equivalent:
\begin{enumerate}
\item the generic countable $k$-hypergraph,
\item the generic countable $k$-ary relation,
\item the generic countable ordered $k$-hypergraph,
\item the generic countable ordered $k$-ary relation.
\end{enumerate}
\end{proposition}

This is generalized in Lemma~\ref{lem:free homo} and Corollary~\ref{cor:henson}.

\medskip
By Theorem~\ref{thm:west} the generic countable $k$-hypergraph does not interpret the generic countable $k$-ary relation.
(Theorem~\ref{thm:west} is essentially due to Harry West.)
Hence the Erd\H{o}s-Rado graph trace defines but doesn't interpret the generic countable directed graph.
The generic countable $k$-ary relation cannot interpret an infinite linear order as it is simple.

\begin{proof}
Let $(W;R)$ be the generic countable $k$-ary relation.
Define a $k$-ary relation $E$ on $W$ by declaring $E(a_1,\ldots,a_k)$ when $a_i \ne a_j$ for all $i \ne j$ and $R(a_{\upsigma(1)},\ldots,a_{\upsigma(k)})$ holds for some permutation $\upsigma$ of $\{1,\ldots,k\}$.
It is easy to see that $E$ satisfies the extension axioms for the generic countable $k$-hypergraph, so $(W;E)$ is a copy of the generic countable $k$-hypergraph.
The same construction shows that the generic countable ordered $k$-ary relation defines the generic countable $k$-hypergraph.

\medskip\noindent
Let $(V;E)$ be the generic countable $k$-hypergraph, we show that $(V;E)$ trace defines $(W;R)$.
We will produce a $(V;E)$-definable $k$-ary relation $R_E$ with domain $V^k$ such that the generic countable $k$-ary relation embedds into $(V^k;R_E)$.
This suffices by Proposition~\ref{prop:qe} and quantifier elimination for $\relk_k$.

\medskip\noindent
We first describe $R_E$, which is defined more generally for any $k$-hypergraph.
Let $(X;F)$ be a $k$-hypergraph.
Given elements $\alpha^1,\ldots,\alpha^k$ of $X^k$ with $\alpha^i = (\alpha^i_1,\ldots,\alpha^i_k)$ we declare $$R_{F}(\alpha^1,\ldots,\alpha^k)\quad \Longleftrightarrow\quad F(\alpha^1_1,\alpha^2_2,\ldots,\alpha^k_k).$$

\begin{Claim*}
Let $X$ be a set and $S$ be a $k$-ary relation on $X$.
Then there is a $k$-hypergraph $(Y;F)$ such that $(X;S)$ embedds into $(Y^k;R_F)$.
If $Y$ is infinite then we may take $|Y| = |X|$.
\end{Claim*}

\noindent
We first explain why the claim suffices.
Suppose that the claim holds.
By the claim we may suppose that $(W;R)$ is a substructure of $(Y^k;R_{F})$ for some countable $k$-hypergraph $(Y;F)$.
Every countable $k$-hypergraph embedds into $(V;E)$, so we may suppose that $(Y;F)$ is a substructure of $(V;E)$.
Finally, it is easy to see that $(Y^k; R_F)$ is a substructure of $(V^k; R_E)$.

\medskip\noindent
We now prove the claim.
Fix a $k$-ary relation $S$ on a set $X$.
Let $Y = X \times \{1,\ldots,k\}$.
We define a $k$-hypergraph $F$ on $Y$.
Let $(\alpha_1,\imag_1),\ldots,(\alpha_k,\imag_k) \in Y$.
We declare $F((\alpha_1,\imag_1),\ldots,(\alpha_k,\imag_k))$ when $\{\imag_1,\ldots,\imag_k\} = \{1,\ldots,k\}$ and $R(\alpha_{\upsigma(1)},\ldots,\alpha_{\upsigma(k)})$ where $\upsigma$ is the unique permutation of $\{1,\ldots,k\}$ with $\imag_{\upsigma(1)} = 1,\ldots,\imag_{\upsigma(k)} = k$.
We now give an embedding $\mathbf{e} \colon (X;R) \to (Y^k;R_F)$.
Let $\mathbf{e} \colon X \to Y^k$ be given by $\mathbf{e}(\alpha) = ((\alpha,1),\ldots,(\alpha,k))$.
Suppose that $\alpha_1,\ldots,\alpha_k \in X$.
Then
\begin{align*}
R(\alpha_1,\ldots,\alpha_k) &\Longleftrightarrow  F((\alpha_1,1),(\alpha_2,2),\ldots,(\alpha_k,k)) \\
&\Longleftrightarrow R_{F}(\mathbf{e}(\alpha_1),\ldots,\mathbf{e}(\alpha_k))
\end{align*}
Hence $\mathbf{e}$ is an embedding.

\medskip
We finally show that the $(W;R)$ trace defines the generic countable ordered $k$-ary relation.
Fix an order $\prec$ on $W$ such that $(W;R,\prec)$ is the generic countable ordered $k$-ary relation.
Fix $\beta\in W^{k - 2}$ and let $R^*$ be the binary relation on $W$ where $R^*(\alpha,\alpha')\Longleftrightarrow R(\alpha,\alpha',\beta)$.
It is easy to see that $(W;R^*)$ is a copy of the generic countable binary relation.
Hence there is an embedding $\upiota\colon (W;\prec)\to(W;R^*)$.
Let $\uptau\colon W\to W\times W$ be given by $\uptau(\alpha)= (\upiota(\alpha),\alpha)$.
An easy application of Proposition~\ref{prop:qe-trace} shows that $(W;R)$ trace defines $(W;R,\prec)$ via $\uptau$.


\end{proof}

We now consider discrete linear orders.
Let $\mathbf{s}\colon\R\to\R$ be given by $\mathbf{s}(x) = x + 1$ and let $\mathbf{s}^{(n)}$ be the $n$-fold compositional iterate of $\mathbf{s}$.
We declare $\tshift$ to be the theory of $(M;<,f)$ where $(M;<)\models\dlo$ and $f\colon M\to M$ is  continuous  with $\lim_{t\to\infty}f(t)=-\infty$, $\lim_{t\to\infty}f(t)=\infty$, and  $f(t)>t$ for all $t\in M$.
A standard back-and-forth argument shows that $\tshift$ is complete, hence $\tshift$ is $\Th(\R;<,\mathbf{s})$.
Once can also use a back-and-forth argument to show that all models of $\tshift$ expanding $(\R;<)$ are isomorphic~\cite[Theorem~12]{stewart-baldwin}.
It follows that if $(\R;<,f)\models\tshift$ then $(\R;<,f)$ is isomorphic to $(\R;<,\mathbf{s})$.

\begin{proposition}
\label{prop:example}
The following structures are trace equivalent:
\begin{enumerate}
\item $(\Z;<)$,
\item $(\R;<,\mathbf{s},\Z)$,
\item $(\Z;\mathbf{s})\sqcup(\R;<)$,
\item any infinite discrete linear order,
\item $(\R;<,\mathbf{s})$, or more generally any model of $\tshift$.
\end{enumerate}
\end{proposition}


O-minimal expansions of oag's cannot interpret $(\Z;<)$ as they eliminate imaginaries and $\exists^\infty$.
So $(\R;+,<)$ does not interpret $(\Z;<)$, hence $(\R;<,\mathbf{s})$ does not interpret $(\Z;<)$.
By Fact~\ref{fact:zp} $(\Z;+,<)$ does not interpret $\dlo$.
If $(D;\prec)$ is an infinite discrete linear order with a minimum and a maximum then $(D;\prec)$ is pseudofinite and so cannot interpret $(\Z;<)$.

\medskip 
Let $L_\mathrm{dis}$ be the language containing a binary relation $<$ and a binary relation $D_n$ for each $n$.
We consider $(\Z;<)$ to be an $L_\mathrm{dis}$-structure by letting $D_n(m,m')\Longleftrightarrow \mathbf{s}^{(n)}(m) = m'$ for all $n$.
It is well-known that $(\Z;<)$ admits quantifier elimination in $L_\mathrm{dis}$ \cite[7.2]{Poizat}.

\begin{proof}
It is easy to see that if $\Sa D$ is an $\aleph_1$-saturated infinite discrete linear order then there is an $L_\mathrm{dis}$-embedding $(\Z;<)\to\Sa D$, so by Proposition~\ref{prop:qe-trace} $\Sa D$ trace defines $(\Z;<)$.

\medskip
We show that $(\R;<,\mathbf{s})$ trace defines $(\Z;<)$.
We consider $(\R;<,\mathbf{s})$ to be an $L_\mathrm{dis}$-structure by declaring $D_n(t,t^*)$ if and only if $t^* = \mathbf{s}^{(n)}(t)$.
Then $(\Z;<,\mathbf{s})$ is an $L_\mathrm{dis}$-substructure of $(\R;<,\mathbf{s})$.
By Proposition~\ref{prop:qe} $(\R;<,\mathbf{s})$ trace defines $(\Z,<)$ via the identity.

\medskip
We show that $(\R;<,\mathbf{s},\Z)$ is trace definable in $\Th(\Z;<)$.
Let $I = \{ t \in \R : 0 \le t < 1  \}$.
By Lemma~\ref{lem:stable-0} and Proposition~\ref{prop:trace-theories} $\Th(\Z;<)$ trace defines $(I;<)$.
By Lemma~\ref{lem:disjoint union} it is enough to show that $(\R;<,\mathbf{s},\Z)$ is interpretable in the disjoint union of $(\Z;<)$ and $(I;<)$.
Identify $\Z \times I$ with $\R$ via the bijection given by $(m,\alpha) \mapsto m + \alpha$.
This identifies $<$ with the lexicographic order on $\Z \times I$ and identifies $\mathbf{s}$ with the map $(m,\alpha) \mapsto (\mathbf{s}(m),\alpha)$, note that both are definable in $(\Z,<) \sqcup (I;<)$.

\medskip
We have seen that $\Th(\Z;<)$ trace defines $(\R;<)$, so by Lemma~\ref{lem:disjoint union} $\Th(\Z;<)$ trace defines $(\Z;\mathbf{s})\sqcup (\R;<)$.
Let $\uptau\colon\Z\to\Z\times\R$ be given by $\uptau(m)=(m,m)$.
We show that $(\Z;\mathbf{s})\sqcup(\R;<)$ trace defines $(\Z;<)$ via $\uptau$.
By Proposition~\ref{prop:qe} it is enough to produce $(\Z;\mathbf{s})\sqcup(\R;<)$-definable $X\subseteq (\Z\times\R)^2$ such that $(\uptau(m),\uptau(m^*))\in X\Longleftrightarrow m < m^*$ for all $m,m^*\in\Z$ and a sequence $(Y_n:n\in\N,n\ge 1)$ of $(\Z;\mathbf{s})\sqcup(\R;<)$-definable subsets of $(\Z\times\R)^2$ such that for all $m,m^*\in\Z$ we have $(\uptau(m),\uptau(m^*))\in Y_n \Longleftrightarrow D_n(m,m^*)$.
Let $X$ be the set of $((m,t),(m^*,t^*))$ such that $t<t^*$ and $Y_n$ be the set of $((m,t),(m^*,t^*))$ such that $m^* = \mathbf{s}^{(n)}(m)$.
\end{proof}


By Lemma~\ref{lem:stable-0} $T$ trace defines a dense linear order if and only if $T$ is unstable.
What happens if we replace ``dense" with ``discrete"?
Well, I don't actually know whether or not $\dlo$ trace defines $(\Z;<)$, so it's possible that infinite discrete linear orders are trace equivalent to dense linear orders.
I suspect this is not the case.
By Proposition~\ref{prop:example} $T$ trace defines a discrete linear order if and only if $T$ is unstable and trace defines $(\Z;\mathbf{s})$, so we can just consider $(\Z;\mathbf{s})$.
Of course, $\Th(\Z;\mathbf{s})$ is the theory of a cycle-free bijection, so an unstable theory which defines a cycle-free bijection trace defines $(\Z;<)$.

\medskip
Given a function $f\colon X\to X$ we let $f^{(n)}$ be the $n$-fold compositional iterate of $f$ and define the \textbf{orbit} of $\alpha\in X$ under $f$ to be $\{f^{(n)}(\alpha):n\in\N\}$.

\begin{lemma}
\label{lem:shift}
Suppose that $\Th(\Sa M)$ does not trace define $(\Z;\mathbf{s})$, $X$ is an $\Sa M$-definable set, and $f\colon X\to X$ is definable.
Then there is $n$ such that the orbit of any $\alpha\in X$ has cardinality at most $n$.
Hence any definable injection $X\to X$ is surjective
\end{lemma}

An unstable structure which defines a group with unbounded exponent trace defines $(\Z;<)$.

\begin{proof}
If $f\colon X\to X$ is injective and not surjective then the orbit of any $\alpha\in X\setminus f(X)$ is infinite, so the second claim follows from the first.
It is easy to see that $(\N;\mathbf{s})$ interprets $(\Z;\mathbf{s})$, hence it is enough so show that $\Th(\Sa M)$ trace defines $(\N;\mathbf{s})$. 
After passing to an elementary extension we fix $\alpha\in X$ with infinite orbit.
Let $\uptau\colon\N\to X$ be  $\uptau(n)=f^{(n)}(\alpha)$.
Then $\uptau$ gives an injective homomorphism $(\N;\mathbf{s})\to (X;f)$.
Apply quantifier elimination for $(\N;\mathbf{s},0)$ \cite[Section~7.1]{Poizat} and Proposition~\ref{prop:qe-trace}.
\end{proof}

Hence $(\Z;\mathbf{s})$ and $(\N;\mathbf{s})$ are trace equivalent.
By Prop~\ref{prop:euler} $(\Z;\mathbf{s})$ doesn't interpret $(\N;\mathbf{s})$.

\section{General Theory}
\label{section: basic trace}
\noindent
This section contains the rest of the general results.


\begin{lemma}
\label{lem:she--1}
Suppose that $A \subseteq M^m$, $\Sa A$ is the structure induced on $A$ by $\Sa M$, and $\Sa A$ admits quantifier elimination.
Then $\Sa M$ trace defines $\Sa A$.
\end{lemma}

\begin{proof}
Immediate from the definitions.
One can also note that $\Sa A$ is a substructure of the structure induced on $M^m$ by $\Sa M$ and apply Proposition~\ref{prop:qe-trace} 
\end{proof}

\begin{proposition}
\label{prop:zero-def}
Suppose $\uptau \colon O \to M^m$ is an injection and for every subset $X$ of $O^n$ which is definable in $\Sa O$ without parameters there is $\Sa M$-definable $Y \subseteq
M^{mn}$ such that for all $a \in O^n$ we have $ a \in X \Longleftrightarrow \uptau(a) \in Y$.
Then $\Sa M$ trace defines $\Sa O$ via $\uptau$.
\end{proposition}

\begin{proof}
Let $L^*$ be the language with a $k$-ary relation $R_X$ for each parameter-free definable $X \subseteq O^k$ and $\Sa O^*$ be the $L^*$-structure on $O$ where $\Sa O^* \models R_X(a) \Longleftrightarrow a \in X$ for all $R \in L^*$, $a \in O^k$.
Then $\Sa O^*$ has quantifier elimination.
Apply Proposition~\ref{prop:qe} to $\Sa M$ and $\Sa O^*$.
\end{proof}

\begin{proposition}
\label{prop:omega-cat-0}
Suppose that $\Sa O$ is $\aleph_0$-categorical and $\uptau \colon O \to M^m$ is an injection.
Then $\Sa M$ trace defines $\Sa O$ via $\uptau$ if and only if for any $p \in S_k(\Sa O)$ there is an $\Sa M$-definable $X \subseteq M^{mk}$ such that if $a \in O^n$ then $\tp_{\Sa O}(a) = p \Longleftrightarrow \uptau(a) \in X$.
\end{proposition}

\begin{proof}
By Ryll-Nardzewski a subset of $O^k$ is definable without parameters iff it is a finite union of sets of the form $\{ a \in O^n : \tp_{\Sa O}(a) = p \}$ for $p \in S_k(\Sa O)$.
Apply Proposition~\ref{prop:zero-def}.
\end{proof}

\begin{proposition}
\label{prop:ec}
Suppose that $\Sa O$ is a substructure of $\Sa M$, $\Sa O$ is existentially closed in $\Sa M$, and $\Sa O$ is model complete.
Then $\Sa M$ trace defines $\Sa O$ via the identity.
\end{proposition}

\begin{proof}
Follow the proof of Proposition~\ref{prop:qe-trace} replacing ``quantifier free" with ``existential".
\end{proof}



\noindent
Corollary~\ref{cor:trace-theories} follows from Proposition~\ref{prop:trace-theories} and $\aleph_1$-saturation of nonprinciple ultrapowers.

\begin{corollary}
\label{cor:trace-theories}
Suppose $\Sa M \models T$ and $\Sa O \models T^*$ are countable.
The following are equivalent:
\begin{enumerate}
\item $T$ trace defines $T^*$,
\item every nonprinciple ultrapower of $\Sa M$ trace defines $\Sa O$,
\item some nonprinciple ultrapower of $\Sa M$ trace defines $\Sa O$.
\end{enumerate}
\end{corollary}

Corollary~\ref{cor:ls} follows from the proof of Proposition~\ref{prop:trace-theories} and downward L\"owenheim-Skolem.

\begin{corollary}
\label{cor:ls}
Suppose that $T$ trace defines $\Sa O$.
Then there is $\Sa M\models T$ such that $\Sa O$ is trace definable in $\Sa M$ and $|M|=\max\{|T|,|O|\}$.
\end{corollary}

Proposition~\ref{prop:omega cat} follows from Corollary~\ref{cor:ls}.

\begin{proposition}
\label{prop:omega cat}
Suppose that $\uplambda$ is an infinite cardinal and $\Sa M$ is a structure in a language of cardinality $\le\uplambda$.
If $\Sa M$ is $\uplambda$-categorical and $|O|\le\uplambda$ then $\Th(\Sa M)$ trace defines $\Sa O$ if and only if $\Sa M$ trace defines $\Sa O$.
In particular two $\uplambda$-categorical structures in languages of cardinality $\le\uplambda$ are trace equivalent if and only if each trace defines the other.
\end{proposition}

We mainly apply Proposition~\ref{prop:omega cat} in the case $\uplambda=\aleph_0$.


\medskip
We now discuss a way of building up a structure from simpler structures.

\begin{proposition}
\label{prop:fusion}
Suppose that $L_1,\ldots,L_n$ are pairwise disjoint relational languages and let $L = L_1 \cup \cdots \cup L_n$.
Suppose that $\Sa O$ is an $L$-structure with quantifier elimination, $\Sa O_i = \Sa O \! \upharpoonright \! L_i$ for each $i \in \{1,\ldots,n\}$, and each 
$\Sa O_i$ is a substructure of an $L_i$-structure $\Sa P_i$.
If each $\Sa P_i$ is trace definable in $\Sa M$ then $\Sa O$ is trace definable in $\Sa M$.
In particular if each $\Sa O_i$ is trace definable in $\Sa M$ then $\Sa O$ is trace definable in $\Sa M$.
\end{proposition}

\begin{proof}
The second claim follows from the first.
Suppose that each $\Sa P_i$ is trace definable in $\Sa M$.
By Lemma~\ref{lem:disjoint union} $\Sa M$ trace defines $\Sa P_1 \sqcup\cdots \sqcup \Sa P_n$.
Hence it is enough to show that $\Sa P_1 \sqcup\cdots \sqcup \Sa P_n$ trace defines $\Sa O$.
Let $\uptau \colon O  \to P_1 \times\cdots\times P_n$ be given by $\uptau(a) = (a, \ldots,a)$.
We show that $\Sa P_1 \sqcup\cdots \sqcup \Sa P_n$ trace defines $\Sa O$ via $\uptau$.
For each $i \in \{1,\ldots,n\}$ let $\uppi_i : P_1 \times\cdots\times P_n \to P_i$ be the projection.
For each $i \in \{1,\ldots,n\}$ and $k$-ary $R \in L_i$ fix an $\Sa M$-definable $X_R \subseteq P^k_i$ such that $X_R \cap O^k = \{ a \in O^k : \Sa O_i \models R(a)\}$ and $Q_R$ be the subset of $(P_1 \times\cdots\times P_n)^k$ such that $(a_1,\ldots,a_k) \in Q_R \Longleftrightarrow (\uppi_i(a_1),\ldots,\uppi_i(a_k)) \in X_R$.
We show that $\uptau$ gives an isomorphism $\Sa O \to (\uptau(O); (Q_R)_{ R \in L })$.
For any $k$-ary $R \in L_i$ and $a_1,\ldots,a_k \in O$ we have
\begin{align*}
\Sa O \models R(a_1,\ldots,a_k) &\Longleftrightarrow \Sa O_i \models R(a_1,\ldots,a_k) \\ &\Longleftrightarrow \Sa P_i \models R(a_1,\ldots,a_k) \\ &\Longleftrightarrow (a_1,\ldots,a_k) \in X_R \\ &\Longleftrightarrow ( \uppi_i(\uptau(a_1)),\ldots,\uppi_i(\uptau(a_k))) \in X_R \\ &\Longleftrightarrow (\uptau(a_1),\ldots,\uptau(a_k)) \in Q_R.
\end{align*}
By Corollary~\ref{cor:qe} $\Sa M$ trace defines $\Sa O$ via $\uptau$.
\end{proof}

\begin{proposition}
\label{prop:new order}
Suppose that $\Sa M$ is relational, $\prec$ is a linear order on $M$, and $(\Sa M,\prec)$ admits quantifier elimination.
Then $(\Sa M,\prec)$ is trace equivalent to $\Sa M\sqcup(\Q;<)$.
If $\Sa M$ is unstable then $(\Sa M,\prec)$ is trace equivalent to $\Sa M$.
\end{proposition}

\begin{proof}
Lemmas~\ref{lem:stable-0} and \ref{lem:disjoint union} together show that if $\Sa M$ is unstable then $\Sa M$ is trace equivalent to $\Sa M\sqcup(\Q;<)$.
Hence it is enough to prove the first claim.
Again, Lemmas~\ref{lem:stable-0} and \ref{lem:disjoint union} show that $\Th(\Sa M;\prec)$ trace defines $\Sa M\sqcup(\Q;<)$.
It is enough to show that $\Th(\Sa M\sqcup(\Q;<))$ trace defines $(\Sa M,\prec)$.
Suppose that $(M;\prec)$ is a substructure of $(I;\triangleleft)\models\dlo$.
We apply Proposition~\ref{prop:fusion} with $\Sa O=(\Sa M,\prec)$, $n=2$, $\Sa O_1=\Sa M$, $\Sa O_2=(M;\prec)$, $\Sa P_1=\Sa M$, and $\Sa P_2=(I;\triangleleft)$.
This shows that $\Sa M\sqcup(I;\triangleleft)$ trace defines $(\Sa M,\prec)$.
\end{proof}


We now give some equivalent definitions in the relational case.

\begin{lemma}
\label{lem:substructure}
Suppose that $L$ is relational and $\Sa O$ is an $L$-structure with quantifier elimination.
Then the following are equivalent:
\begin{enumerate}
\item $\Th(\Sa M)$ trace defines $\Sa O$,
\item There is $m$ and a collection $(Y_R \subseteq M^{mk} : R \in L, R \text{ is $k$-ary})$ of $\Sa M$-definable sets such that for every $\Sa P \in \age(\Sa O)$ there is an injection $\uptau \colon P \to M^m$ satisfying
\[ 
\Sa P \models R(a) \quad \Longleftrightarrow \quad \uptau(a) \in Y_R \quad \text{for all} \quad \text{$k$-ary } R \in L, a \in P^k.
\] 
\end{enumerate}
\end{lemma}

It is instructive to think about the case $\Sa O = (\Q;<)$ and recall some definitions of stability.

\begin{proof}
The definitions directly show that (1) implies (2).
The other implication follows by applying compactness and Proposition~\ref{prop:qe}.
\end{proof}

\begin{lemma}
\label{lem:substructure1}
Suppose that $L$ is relational, $\Sa O$ is an $L$-structure with quantifier elimination, and $\Sa O^*$ is an $L$-structure such that $\age(\Sa O) = \age(\Sa O^*)$.
Suppose that $\Sa M$ is $\max\{|L|,|O|\}^+$-saturated.
The following are equivalent:
\begin{enumerate}
\item $\Sa M$ trace defines $\Sa O$.
\item There is an injection $\uptau^* \colon O^* \to M^m$ such that for every $k$-ary $R \in L$ there is an $\Sa M$-definable $Y \subseteq M^{mk}$ such that for all $a \in (O^*)^k$ we have
$$ \Sa O^* \models R(a) \quad \Longleftrightarrow\quad \uptau^*(a)   \in Y. $$
\end{enumerate}
\end{lemma}

Again, think about the case $\Sa O = (\Q;<)$ and $\Sa O^* = (\N;<)$, recall some definitions of stability.

\begin{proof}
Apply Lemma~\ref{lem:substructure} and saturation.
\end{proof}

\subsection{Trace maximality}
We say that $T$ is \textbf{trace maximal} if $T$ trace defines every theory and that $\Sa M$ is trace maximal if $\Th(\Sa M)$ is trace maximal.

\begin{lemma}
\label{lem:max}
Suppose that $\Sa M$ is $\aleph_1$-saturated.
Then the following are equivalent:
\begin{enumerate}
\item $\Sa M$ is trace maximal,
\item There is an infinite $A \subseteq M^m$ such that for any $X \subseteq A^k$ there is an  $\Sa M$-definable $Y \subseteq M^{mk}$ such that $X = A^k \cap Y$,
\item there is an infinite $A \subseteq M^m$, a sequence $(\varphi_k(x_k,y_k) : k < \upomega)$ of formulas with $|x_k| = mk$, and elements $( \alpha_{k,X} \in M^{|y_k|} : k < \upomega, X \subseteq A^k)$ such that for every $\beta \in A^k$ we have $\Sa M \models \varphi_k(\beta,\alpha_{k,X})\Longleftrightarrow \beta \in X$,
\item there is an infinite $A \subseteq M^m$, a sequence $(\varphi_k(x_k,y_k) : k < \upomega )$ with $|x_k| = mk$, and $(\alpha_{k,X,Y} \in M^{|y_k|} : k < \upomega, X,Y \subseteq A^k, |X \cup Y| < \aleph_0, X \cap Y = \emptyset)$ such that for any $k < \upomega$, disjoint finite $X,Y \subseteq A^k$, and $\beta \in A^k$ we have 
\begin{align*}
\beta \in X &\quad\Longrightarrow\quad \Sa M \models \hspace{8pt} \varphi_k(\beta,\alpha_{k,X,Y})  \\
\beta \in Y &\quad\Longrightarrow\quad \Sa M \models \neg \varphi_k(\beta,\alpha_{k,X,Y}). 
\end{align*}

\end{enumerate}
\end{lemma}

\noindent
It's reasonable to view trace maximality as ``$\infty$-independence".
There are structures which are $k$-independent for all $k$ but not trace maximal, see the remarks at the end of this section.

\begin{proof}
We first show that $(1)$ implies $(2)$.
Suppose that $T$ is trace maximal.
Let $A$ be a countable set and $\Sa A$ be a structure on $A$ which defines every subset of every $A^k$.
Then $\Th(\Sa A)$ is trace definable in $T$, by Proposition~\ref{prop:trace-theories} $\Sa M$ trace defines $\Sa A$.
It is easy to see that $(3)$ and $(4)$ are equivalent.

\medskip\noindent
We show that $(3)$ implies $(1)$.
It is easy to see that for every infinite cardinal $\uplambda$ there is $\Sa M \prec \Sa N$ and $A \subseteq N^m$ such that $|A| = \uplambda$ and every subset of every $A^k$ is of the form $Y \cap A^k$ for some $\Sa N$-definable $Y \subseteq N^{mk}$.
This yields trace maximality.
We finish the proof by showing that $(2)$ implies $(4)$.
Trace maximality follows directly.
Suppose that $A \subseteq M^m$ satisfies the condition of $(2)$.
The collection of disjoint pairs of nonempty subsets of $A$ has cardinality $|A|$.
Hence we can fix a subset $Z$ of $A^{k+1}$ such that for any disjoint finite $X,Y \subseteq A^k$ there is $b_{X,Y} \in A$ such that for all $a \in A^k$ we have
\begin{align*}
a \in X &\quad\Longrightarrow\quad
(a,b_{X,Y}) \in Z \\
a \in Y &\quad\Longrightarrow\quad (a,b_{X,Y}) \notin Z    
\end{align*}
Let $\varphi(x,y)$ be a formula, possibly with parameters, such that for any $(a,b) \in A^{k+1}$ we have $\Sa M \models \varphi(a,b) \Longleftrightarrow (a,b) \in Z$.
$(4)$ follows.
\end{proof}

Proposition~\ref{prop:cat max} follows from Proposition~\ref{prop:omega cat}.

\begin{proposition}
\label{prop:cat max}
If $\Sa M$ is countable, $\aleph_0$-categorical, and trace maximal, then any countable structure in countable language is trace definable in $\Sa M$.
\end{proposition}

Lemma~\ref{lem:max-1} shows that trace maximality is equivalent to a symmetrized version of trace maximality, this will be useful below.

\begin{lemma}
\label{lem:max-1}
Suppose that $\Sa M$ is an $\aleph_1$-saturated $L$-structure.
The following are equivalent:
\begin{enumerate}
\item $\Sa M$ is trace maximal, 
\item there is an infinite $A \subseteq M^m$ such that for every $k$-hypergraph $E$ on $A$ there is an $\Sa M$-definable $X \subseteq M^{mk}$ with $E(a_1,\ldots, a_k) \Longleftrightarrow (a_1,\ldots,a_k) \in X$ for all $a_1,\ldots,a_k \in A$,
\item there is a sequence $(a_i : i < \upomega )$ of elements of some $M^m$ such that for any $k$-hypergraph $E$ on $\upomega$ there is an $\Sa M$-definable $Y \subseteq M^{mk}$ with $E(i_1,\ldots,i_k) \Longleftrightarrow (a_{i_1},\ldots,a_{i_k}) \in Y$ for all $i_1,\ldots,i_k < \upomega$.
\end{enumerate}
\end{lemma}

\begin{proof}
It is clear that $(2)$ and $(3)$ are equivalent.
Lemma~\ref{lem:max} shows that $(1)$ implies $(2)$.
We show that $(2)$ implies $(1)$.
Suppose $(2)$.
Let $E$ be a graph on $A$ and $(a_i : i < \upomega)$, $(b_j : j < \upomega)$ be sequences of distinct elements of $A$ such that for all $i,j$ we have $E(a_i,b_j) \Longleftrightarrow i < j$.
Let $\delta(x,y)$ be an $L(M)$-formula such that for all $a,a^* \in A$ we have $E(a,a^*) \Longleftrightarrow \Sa M \models \delta(a,a^*)$.
For each $i$ let $c_i = (a_i,b_i)$ and let $\phi(x_1,y_1,x_2,y_2)$ be $\delta(x_1,y_2)$.
Then for any $i,j$ we have
\begin{align*}
\Sa M \models \phi(c_i,c_j) &\Longleftrightarrow \Sa M \models \phi(a_i,b_i,a_j,b_j)\\
&\Longleftrightarrow \Sa M \models \delta(a_i,b_j) \\
&\Longleftrightarrow i < j.
\end{align*}
We now show that for any $X \subseteq \upomega^k$ there is an $\Sa M$-definable $Y \subseteq M^{2k}$ which satisfies $(i_1,\ldots,i_k) \in X \Longleftrightarrow (c_{i_1},\ldots,c_{i_k}) \in Y$ for all $(i_1,\ldots,i_k) \in \upomega^k$.
Trace maximality of $\Sa M$ follows by Lemma~\ref{lem:max}.
We apply induction on $k \ge 1$.

\medskip\noindent
Suppose $k = 1$ and $X \subseteq \upomega$.
Let $F$ be a graph on $A$ and $d \in A$ be such that $E(a_i,d) \Longleftrightarrow i \in X$ for all $i < \upomega$.
Let $\theta(x,y)$ be an $L(M)$-formula such that $F(a,a^*) \Longleftrightarrow \Sa M \models \theta(a,a^*)$ for all $a,a^* \in A$.
Let $Y$ be the set of $(a,b) \in M^2$ such that $\Sa M \models \theta(a,d)$.
Then for any $i < \upomega$ we have $i \in X \Longleftrightarrow c_i \in Y$.

\medskip\noindent
We now suppose that $k \ge 2$ and $X \subseteq \upomega^k$.
Abusing notation we let $\upomega$ denote the structure $(\upomega; <)$.
We let $\qftp_\upomega(\imag_1,\ldots,\imag_k)$ be the quantifer free type (equivalently: order type) of $(\imag_1,\ldots,\imag_k) \in \upomega^k$ and let $S_k(\upomega)$ be the set of quantifier free $k$-types.
Note that $S_k(\upomega)$ is finite.
For each $p \in S_k(\upomega)$ we fix an $\Sa M$-definable $Y_p \subseteq M^{2k}$ such that $$\qftp_{\upomega}(\imag_1,\ldots,\imag_k) = p\quad \Longleftrightarrow\quad (c_{\imag_1},\ldots,c_{\imag_k}) \in Y_p\quad \text{for all}\quad  (\imag_1,\ldots,\imag_k) \in \upomega^k.$$
We show that for every $p \in S_k(\upomega)$ there is an $\Sa M$-definable subset $X_p$ of $M^{2k}$ such that $(c_{\imag_1},\ldots,c_{\imag_k}) \in X \Longleftrightarrow (c_{\imag_1},\ldots,c_{\imag_k}) \in X_p$ for any $(\imag_1,\ldots,\imag_k) \in \upomega^k$ with $\qftp_\upomega(\imag_1,\ldots,\imag_k) = p$.
For any $(\imag_1,\ldots,\imag_k) \in \upomega^k$ we have 
$$(c_{\imag_1},\ldots,c_{\imag_k}) \in X \quad\Longleftrightarrow\quad (c_{\imag_1},\ldots,c_{\imag_k}) \in \! \bigcup_{p \in S_k(\upomega)} (Y_p \cap X_p).$$

\medskip\noindent
We fix $p(x_1,\ldots,x_k) \in S_k(\upomega)$ and produce $X_p$.
We first treat the case when $p \models (x_i = x_j)$ for some $i \ne j$.
To simplify notation we suppose that $p \models (x_1 = x_2)$.
Let $X^*$ be the set of $(\imag_1,\ldots,\imag_{k - 1}) \in \upomega^{k - 1}$ such that $(\imag_1,\imag_1,\imag_2,\ldots,\imag_{k - 1}) \in X$.
By induction there is an $\Sa M$-definable $Y^* \subseteq M^{2(k - 1)}$ such that we have $(\imag_1,\ldots,\imag_{k - 1}) \in X^* \Longleftrightarrow (c_{\imag_1},\ldots,c_{\imag_{k- 1}}) \in Y^*$ for all $(\imag_1,\ldots,\imag_{k - 1}) \in \upomega^{k - 1}$ .
Let $X_p$ be the set of $(d_1,\ldots,d_k) \in M^{2k}$ such that $d_1 = d_2$ and $(d_2,\ldots,d_k) \in Y^*$, note that $X_p$ is definable in $\Sa M$.

\medskip\noindent
Now suppose that  $p \models (x_i \ne x_j)$ when $i \ne j$.
Then if $(\imag_1,\ldots,\imag_k) \in \upomega^k$ and $|\{ \imag_1,\ldots,\imag_k \}| = k$ then there is a unique permutation $\sigma_p$ of $\{1,\ldots,k\}$ with $\qftp_{\upomega}(\imag_{\sigma_p(1)},\ldots,\imag_{\sigma_p(k)}) = p$.
Let $H^*$ be the $k$-hypergraph on $\{c_\imag : \imag < \upomega \}$ where $H^*(c_{\imag_1},\ldots,c_{\imag_k})$ if and only if $|\{ \imag_1,\ldots,\imag_k \}| = k$ and $(\imag_{\sigma_p(1)},\ldots,\imag_{\sigma_p(k)}) \in X$.
Let $H$ be the $k$-hypergraph on $\{a_\imag : \imag < \upomega\}$ where we have $H(a_{\imag_1},\ldots,a_{\imag_k}) \Longleftrightarrow H^*(c_{\imag_1},\ldots,c_{\imag_k})$ for all $(\imag_1,\ldots,\imag_k) \in \upomega^k$.
By assumption there is an $L(M)$-formula $\varphi(x_1,\ldots,x_k)$ such that $H(a_{\imag_1},\ldots,a_{\imag_k}) \Longleftrightarrow \Sa M \models \varphi(a_{\imag_1},\ldots,a_{\imag_k})$ for any $(\imag_1,\ldots,\imag_k)$ in $\upomega^k$.
Finally, let $X_p$ be the set of $((d_1,e_1),\ldots,(d_k,e_k)) \in M^{2k}$ with $\Sa M \models \varphi(d_1,\ldots,d_k)$.
\end{proof}

There are structures which are $k$-independent for all $k$ but not trace maximal.
If we allow ourselves to work with multisorted structures we can let $\Sa H_k$ be the generic countable $k$-hypergraph for $k\ge 2$ and take $\Sa H_\upomega=\bigsqcup_{k\ge 2} \Sa H_k$.
Then $\Sa H_\upomega$ is $k$-independent for all $k$ but not trace maximal.
There are more-or-less obvious ways of making a one-sorted structure along the same lines.
Note also that a theory $T$ trace defines $\Sa H_\upomega$ if and only if $T$ is $k$-independent for all $k$.
(It is not hard to see that there is no one-sorted structure with this property.)



\subsection{Indiscernible collapse}
\label{section:collapse}
\noindent
We discuss trace definibility and indiscernible collapse.
We first recall some background definitions on indiscernible collapse.
Some of our conventions are a bit different then other authors, so some care is warranted.
We suppose that $\monster$ is a monster model of an $L$-theory $T$ and $A$ is a small set of parameters from $\monsterset$.
Let $L_{\mathrm{indis}}$ be a finite relational language, $\Sa I$ be a homogeneous $L_{\mathrm{indis}}$-structure, and suppose that $\Sa I$ is Ramsey.
A \textbf{picture} of $\Sa I$ in $\monster$ is an injection $\upgamma \colon I \to \monsterset^n$ for some $n$.
Given an injection $\upgamma \colon I \to \monsterset^n$ and $a = (a_1,\ldots,a_k) \in I^k$ we let $\upgamma(a) = (\upgamma(a_1),\ldots,\upgamma(a_k))$.
A picture $\upgamma$ of $\Sa I$ in $\monsterset$ is \textit{$A$-indiscernible} if
$$ \tp_{\Sa I}(a) = \tp_{\Sa I}(b) \quad \Longrightarrow \quad \tp_{\monster}(\upgamma(a)|A) = \tp_{\monster}(\upgamma(b)|A)$$
for any $k$ and $a,b \in I^k$.
An $A$-indiscernible picture $\upgamma$ of $\Sa I$ in $\monster$ is \textit{uncollapsed} if
$$ \tp_{\Sa I}(a) = \tp_{\Sa I}(b) \quad \Longleftrightarrow \quad \tp_{\monster}(\upgamma(a)|A) = \tp_{\monster}(\upgamma(b)|A)$$
for any $k \in \N$ and $a,b \in I^k$.
Recall that $\Sa I$ has quantifier elimination so we could use $\qftp_{\Sa I}(a)$ in place of $\tp_{\Sa I}(a)$.
Let $\upgamma, \upgamma^*$ be pictures of $I$ in $\monster$ and suppose that $\upgamma^*$ is indiscernible.
Then $\upgamma^*$ is \textit{based on} $\upgamma$ if for any finite $L_0 \subseteq L$ and $a \in I^k$ there is $b \in I^k$ such that $\qftp_{\Sa I}(a) = \qftp_{\Sa I}(b)$ and $\tp_{\monster}(\upgamma^*(a)|A)\!\upharpoonright\! L_0 = \tp_{\monster}(\upgamma(b)|A)\!\upharpoonright\! L_0$.
It is shown in \cite[Theorem~2.13]{gcs} that if $\upgamma$ is a picture of of $\Sa I$ in $\monster$ then there is an $A$-indiscernible picture $\upgamma^*$ of $\Sa I$ in $\monster$ such that $\upgamma^*$ is based on $\upgamma$.
(This requires that $\Sa I$ is Ramsey.)

\begin{proposition}
\label{prop:picture}
Suppose as above that $\Sa I$ is a Ramsey finitely homogeneous structure.
Then the following are equivalent:
\begin{enumerate}
\item $\monster$ trace defines $\Sa I$ via an injection $I \to \monsterset^m$,
\item there is a small set $A$ of parameters and a non-collapsed $A$-indiscerniable picture $I \to \monsterset^m$ of $\Sa I$ in $\monster$.
\end{enumerate}
Therefore the following are equivalent:
\begin{enumerate}
\setcounter{enumi}{3}
\item $\monster$ trace defines $\Sa I$,
\item any $\aleph_1$-saturated elementary substructure of $\monster$ trace defines $\Sa I$,
\item there is a small set $A$ of parameters and a non-collapsed $A$-indiscernible picture of $\Sa I$ in $\monster$.
\end{enumerate}
\end{proposition}

\noindent
As $\Sa I$ is $\aleph_0$-categorical we can and will apply Proposition~\ref{prop:qe}.

\begin{proof}
Note that the equivalence of $(4) - (6)$ follows from the equivalence of $(1) - (2)$ and Proposition~\ref{prop:trace-theories}.
Suppose that $\monster$ trace defines $\Sa I$ via $\uptau \colon I \to \monsterset^m$.
Then for every $k \in \N$ and $p \in S_k(\Sa I)$ there is an $\monster$-definable $Y_p \subseteq \monsterset^{mk}$ such that $\qftp_{\Sa I}(a) = p \Longleftrightarrow \uptau(a) \in Y_p$ for any $a \in I^k$.
We reduce to the case when the $Y_p$ are pairwise disjoint.
Note that the sets $Y_p \cap \uptau(I)^k$ are pairwise disjoint.
For each $p \in S_k(\Sa I)$ we let $Y^*_p = Y_p \setminus \bigcup_{q \in S_k(\Sa I) \setminus \{p\}} Y_q$.
Then the $Y^*_p$ are pairwise disjoint and $Y^*_p \cap \uptau(I)^k = Y_p \cap \uptau(I)^k$ for all $p \in S_k(\Sa I)$.
Therefore after replacing each $Y_p$ with $Y^*_p$ we may suppose that the $Y_p$ are pairwise disjoint.

\medskip\noindent
Let $A$ be a small set of parameters such that each $Y_p$ is $A$-definable and $\upgamma$ be an $A$-indiscernible picture of $\Sa I$ in $\monster$ which is based off of $\uptau$.
We show that $\upgamma$ is uncollapsed.
Suppose that $a_0,a_1 \in I^k$ and $p_0 := \qftp_{\Sa I}(a_0) \ne \qftp
_{\Sa I}(a_1) =: p_1$.
Let $L_0 \subseteq L$ be finite such that $Y_p$ is $L_0(A)$-definable for all $p \in S_k(\Sa I)$.
As $\upgamma$ is based on $\uptau$ there are $b_0,b_1 \in I^k$ such that for each $i \in \{0,1\}$ we have
\begin{enumerate}
\item $\qftp_{\Sa I}(b_i) = p_i$, and
\item $\tp_{\monster}(\upgamma(b_i)|A)\! \upharpoonright\! L_0 = \tp_{\monster}(\uptau(a_i)|A) \!\upharpoonright\! L_0$.
\end{enumerate}
We have $\uptau(a_0) \in Y_{p_0}$ and $\uptau(a_1) \in Y_{p_1}$.
Therefore $\upgamma(b_0) \in Y_{p_0}$ and $\upgamma(b_1) \in Y_{p_1}$.
As $Y_{p_0} \cap Y_{p_1} = \emptyset$ we have  $\tp_{\monster}(\upgamma(b_0)|A) \ne \tp_{\monster}(\upgamma(b_1)|A)$.

\medskip\noindent
Now suppose that $\upgamma \colon I \to \monsterset^m$ is a non-collapsed $A$-indiscernible picture of $\Sa I$ in $\monster$.
We show that $\monster$ trace defines $\Sa I$ via $\upgamma$.
We fix $k$ and let $S_k(\Sa I) = \{p_1,\ldots,p_n\}$ where the $p_i$ are distinct.
We produce $A$-definable $Y_1,\ldots,Y_n$ such that for any $a \in I^k$ we have $\qftp_{\Sa I}(\upgamma(a)) = p_i$ if and only if $\uptau(a) \in Y_i$.
Fix $a_1,\ldots,a_n \in I^k$ such that $\qftp_{\Sa I}(a_i) = p_i$ for each $i$.
As $\upgamma$ is non-collapsed $\tp_{\monster}(\upgamma(a_i)|A) \ne \tp_{\monster}(\upgamma(a_j)|A)$ when $i \ne j$.
For each $i,j$ there is an $A$-definable $Y_{ij} \subseteq \monsterset^{mk}$ such that $\upgamma(a_i) \in Y_{ij}$ and $\upgamma(a_j) \notin Y_{ij}$.
Let $Y_i = \bigcap_{j = 1}^{n} Y_{ij}$.
We have $\upgamma(a_i) \in Y_i$ and $\upgamma(a_j) \notin Y_i$ for all $i \ne j$.
Fix $b \in I^k$.
We show that $\tp_{\Sa I}(b) = p_i \Longleftrightarrow \upgamma(b) \in Y_i$.
Suppose that $\tp_{\Sa I}(b) = p_i$.
As $\upgamma$ is indiscernible we have $\tp_{\monster}(\upgamma(b)|A) = \tp_{\monster}(\upgamma(a_i)|A)$, so $b \in Y_i$.
Now suppose that $\upgamma(b) \in Y_i$.
Then $\upgamma(b) \notin Y_j$ when $j \ne i$, so $\tp_{\Sa I}(b) \ne p_j$ for all $j \ne i$.
Hence $\tp_{\Sa I}(b) = p_i$.
\end{proof}

\noindent
We continue to suppose that $\Sa I$ is finitely homogeneous and Ramsey.
We let $\Cal C_{\Sa I}$ be the class of complete first order theories $T$ such that any monster model of $T$ does not admit an uncollapsed indiscernible picture of $\Sa I$.
Corollary~\ref{cor:picture} is immediate from Proposition~\ref{prop:picture}.

\begin{corollary}
\label{cor:picture}
Let $\Sa I$ be a finitely homogeneous structure and suppose that $\Sa I$ is Ramsey.
If $T \in \Cal C_{\Sa I}$ and $T$ trace defines $T^*$ then $T^* \in \Cal C_{\Sa I}$.
\end{corollary}

\begin{proposition}
\label{prop:ramsey}
Suppose that $\Sa I$ and $\Sa J$ are finitely homogeneous structures and both $\Sa I$ and $\Sa J$ are Ramsey.
Then $\Cal C_{\Sa J} \subseteq \Cal C_{\Sa I}$ if and only if $\Th(\Sa I)$ trace defines $\Th(\Sa J)$.
Hence $\Cal C_{\Sa I} = \Cal C_{\Sa J}$ if and only if $\Sa I$ and $\Sa J$ are trace equivalent.
\end{proposition}

\begin{proof}
It is enough to prove the first claim.
Suppose that $\Th(\Sa I)$ trace defines $\Th(\Sa J)$.
By Proposition~\ref{prop:trace-theories} any $\aleph_1$-saturated elementary extension of $\Sa I$ trace defines $\Sa J$.
Suppose that $T \notin \Cal C_{\Sa I}$ and $\monster \models T$.
By Proposition~\ref{prop:picture} $\monster$ admits an uncollapsed indiscernible picture of $\Sa I$, so $\monster$ trace defines $\Sa I$, so by Proposition~\ref{prop:trace-theories} $\monster$ trace defines an $\aleph_1$-saturated elementary extension of $\Sa I$, so $\monster$ trace defines $\Sa J$, so $\monster$ admits an uncollapsed indiscernible picture of $\Sa J$.
Then $T \notin \Cal C_{\Sa J}$.
Now suppose that $\Cal C_{\Sa J} \subseteq \Cal C_{\Sa I}$.
We have $\Th(\Sa I) \notin \Cal C_{\Sa I}$, hence $\Th(\Sa I) \notin \Cal C_{\Sa J}$, so $\ionster$ admits an uncollapsed indiscernible picture of $\Sa J$, so $\Th(\Sa I)$ trace defines $\Th(\Sa J)$. 
\end{proof}

\noindent
Of course, many of the main examples of finitely homogeneous structures are not Ramsey.
Many finitely homogeneous structures become Ramsey after they are expanded by a suitable linear order, we show that this does not change trace equivalence in the unstable case.



\medskip
Corollary~\ref{cor:add order} follows from Propositions~\ref{prop:new order} and \ref{prop:ramsey}.

\begin{corollary}
\label{cor:add order}
Suppose that $\Sa O$ is finitely homogeneous and unstable and let $\prec$ be a linear order on $O$ such that $(\Sa O,\prec)$ is finitely homogenous and Ramsey.
Then $\Cal C_{(\Sa O,\prec)}$ is the class of theories that do not trace define $\Sa O$.
\end{corollary}

We describe a class of finitely homogeneous structures that admit Ramsey expansions by a dense linear order.
Let $\Cal E$ be a \Fraisse class of structures in a finite relational language.
Then $\Cal E$ admits \textbf{free amalgamation} if for any embeddings $e_i\colon \Sa O \to \Sa O_i$, $i \in \{1,2\}$ between elements of $\Cal E$, there is $\Sa P \in \Cal E$ and embeddings $f_i\colon \Sa O_i \to \Sa P$, $i \in \{1,2\}$ such that the associated square commutes and $\Sa P \models \neg R(a_1,\ldots,a_k)$ for any $k$-ary relation $R$ whenever $\{a_1,\ldots,a_k\}$ intersects $f_i(O_i) \setminus f_i(e_i(O))$ for $i \in \{1,2\}$.
We say that a finitely homogeneous structure $\Sa O$ is \textbf{free homogeneous} if $\age(\Sa O)$ admits free amalgamation.
The generic countable $k$-hypergraph is free homogeneous.
(Consider disjoint unions.)

\medskip\noindent
We may suppose that $\triangleleft$ is not in the language of $\Sa E$.
We let $(\Cal E,\triangleleft)$ be the class of structures in $\Cal E$ expanded by an arbitrary linear order $\triangleleft$.
See \cite[6.5.3]{macpherson-survey} for Fact~\ref{fact:lo-exp}.

\begin{fact}
\label{fact:lo-exp}
If $\Cal E$ has free amalgamation then $(\Cal E,\triangleleft)$ is a \Fraisse class with the Ramsey property.
\end{fact}

Given a free homogeneous $\Sa O$ we let $(\Sa O,\triangleleft)$ be the \Fraisse limit of $(\mathrm{Age}(\Sa O),\triangleleft)$.
This is the expansion of $\Sa O$ by a generic linear order.
By Fact~\ref{fact:lo-exp} $(\Sa O,\triangleleft)$ is Ramsey.
Lemma~\ref{lem:free homo} is a special case of Proposition~\ref{prop:new order}.

\begin{lemma}
\label{lem:free homo}
An unstable free homogeneous structure $\Sa O$ is trace equivalent to $(\Sa O,\triangleleft)$.
\end{lemma}

Proposition~\ref{prop:ramse1} is a special case of Corollary~\ref{cor:add order}.

\begin{proposition}
\label{prop:ramse1}
Suppose that $\Sa O$ is an unstable free homogeneous structure.
Then $\Cal C_{(\Sa O,\triangleleft) }$ is the class of theories which do not trace define $\Sa O$.
\end{proposition}



We will use indiscernible collapse below in two ways.
First, we will use it to connect trace definibility to classification-theoretic notions defined in terms of indiscernible collapse.
Secondly we will apply Proposition~\ref{prop:airity}.

\begin{proposition}
\label{prop:airity}
Suppose that $\Sa I$ is either a finitely homogeneous Ramsey structure or an unstable free homogeneous structure and $T$ trace defines $\Sa I$.
Then $\air(\Sa I) \le \air(T)$.
\end{proposition}

Proposition~\ref{prop:airity} fails in general.
We describe a ternary structure interpretable in the trivial theory of a set with equality, see \cite[Example 3.3.2]{macpherson-survey} or \cite[Example 1.2]{thomas-hyper}.
Let $X$ be an infinite set, $O$ be $\{ \{a,a'\} : a,a' \in X, a \ne a' \}$, and $E$ be the binary relation on $O$ where we have $E(\beta,\beta') \Longleftrightarrow |\beta\cap\beta'|=1$.
Note that if $\beta_1,\beta_2,\beta_3 \in O$ and $E(\beta_i,\beta_j)$ for distinct $i,j \in \{1,2,3\}$ then $|\beta_1 \cap \beta_2 \cap \beta_3|$ is either $0$ or $1$.
Let $R$ be the ternary relation on $O$ where
\[
R(\beta_1,\beta_2,\beta_3) \quad \Longleftrightarrow \quad |\beta_1 \cap \beta_2 \cap \beta_3| = 1 \hspace{.2cm}\land\hspace{.2cm} \bigwedge_{i \ne j} E(\beta_i,\beta_i).
\]
Let $\Sa O = (O;E,R)$.
Then $\Sa O$ is homogeneous and interpretable in $X$ (considered as a set with equality).
Any permutation of $X$ induces an automorphism on $\Sa O$ in an obvious way and it is easy to see that this induces a transitive action on $\{ (\beta,\beta^*) \in O^2 : E(\beta,\beta^*) \}$.
Hence $\tp_{\Sa O}(\beta,\beta^*) = \tp_{\Sa O}(\gamma,\gamma^*)$ when $E(\beta,\beta^*),E(\gamma,\gamma^*)$.
Fix distinct $a,b,b',b'',x,y,z,w$ from $X$.
Let
\[
\beta_1 = \{a,b\}, \beta_2 = \{a,b'\}, \beta_3 = \{a,b''\} \quad\text{and}\quad \gamma_1 = \{x,y\}, \gamma_2 = \{y,z\}, \gamma_3 = \{z,w\}.
\]
Then $E(\beta_i,\beta_j),E(\gamma_i,\gamma_j)$ hence $\tp_{\Sa O}(\beta_i,\beta_j) = \tp_{\Sa O}(\gamma_i,\gamma_j)$ when $i,j \in \{1,2,3\}$ are distinct.
But we have $R(\beta_1,\beta_2,\beta_3)$, $\neg R(\gamma_1,\gamma_2,\gamma_3)$ so $\tp_{\Sa O}(\beta_1,\beta_2,\beta_3) \ne \tp_{\Sa O}(\gamma_1,\gamma_2,\gamma_3)$.
Thus $\Sa O$ is ternary.

\medskip
Given $\alpha = (\alpha_1,\ldots,\alpha_n)$, $1 \le i_1 < \ldots < i_k \le n$, $I = \{i_1,\ldots,i_k\}$ we let $\alpha_I = (\alpha_{i_1},\ldots,\alpha_{i_k})$.

\begin{proof}[Proof of Proposition~\ref{prop:airity}]
We first address the case when $\Sa I$ is finitely homogeneous Ramsey.
Applying Proposition~\ref{prop:picture} we let $\monster \models T$, $A$ be a small set of parameters from $\monster$, and $\upgamma \colon I \to \monsterset^m$ be an uncollapsed $A$-indiscernible picture of $\Sa I$ in $\monster$.
The case $\air(T) = \infty$ is trivial, so we suppose $\air(T) = k $.
We apply Fact~\ref{fact:airity}.
We fix $\alpha = (\alpha_1,\ldots,\alpha_n)$ and $\beta = (\beta_1,\ldots,\beta_n)$ in $I^n$.
Suppose that $\tp_{\Sa I}( \alpha_I) = \tp_{\Sa I}(\beta_I)$ for all $I \subseteq \{1,\ldots,n\}$, $|I| = k$.
By $A$-indiscernibility we have $\tp_{\monster}(\upgamma(\alpha_I)|A) = \tp_{\monster}(\upgamma(\beta_I)|A)$ for all $I \subseteq \{1,\ldots,n\}$, $|I| = k$.
As $\air(T) = k$ we have $\tp_{\monster}(\upgamma(\alpha)|A) = \tp_{\monster}(\upgamma(\beta)|A)$.
As $\upgamma$ is uncollapsed, $\tp_{\Sa I}(\alpha) = \tp_{\Sa I}(\beta)$.

\medskip
We now suppose that $\Sa I$ is free homogeneous and unstable.
Let $(\Sa I,\triangleleft)$ be as defined above.
By Lemma~\ref{lem:free homo} $\Sa I$ is trace equivalent to $(\Sa I,\triangleleft)$.
An easy application of homogenity of $(\Sa I,\triangleleft)$ shows that $\air(\Sa I) = \air(\Sa I,\triangleleft)$.
Apply the Ramsey case to $(\Sa I,\triangleleft)$.
\end{proof}

\subsection{Stability}
\label{section:stability}
We show that many stability theoretic properties are preserved under trace definibility.
Recall that by Lemma~\ref{lem:stable-0} $T$ is unstable if and only if $T$ trace defines $\dlo$.

\begin{proposition}
\label{prop:stable-0}
If $T$ is stable and $T^*$ is trace definable in $T$ then $T^*$ is stable.
Furthermore $T$ is stable if and only if $T$ trace defines $\mathrm{DLO}$.
\end{proposition}

\begin{proof}
The first claim of Prop~\ref{prop:stable-0} is clear.
The second claim follows from Lemma~\ref{lem:stable-0}.
\end{proof}

\noindent
Given a structure $\Sa M$, a subset $A$ of $M$, and a cardinal $\uplambda$, we let $S_n(\Sa M,A)$ be the set of $n$-types over $A$ and $S_n(\Sa M,\uplambda)$ be the supremum of $\{ |S_n(\Sa M,A)| : A \subseteq M, |A| \le \uplambda \}$.
Given a theory $T$ we let $S(T,\uplambda)$ be the supremum of $\{ S_n(\Sa M,\uplambda) : n \in \N, \Sa M \models T \}$.

\begin{proposition}
\label{prop:lambda-0}
Suppose that $\uplambda \ge |L^*|$ is an infinite cardinal and $T^*$ is trace definable in $T$.
Then $S(T^*,\uplambda) \le S(T,\uplambda)$.
Hence if $T$ is $\uplambda$-stable then $T^*$ is $\uplambda$-stable.
\end{proposition}

\noindent
Before proving Proposition~\ref{prop:lambda-0} we give an immediate corollary.

\begin{corollary}
\label{cor:superstable}
Suppose that $T^*$ is trace definable in $T$.
If $T$ is superstable then $T^*$ is superstable.
If $L^*$ is countable and $T$ is $\aleph_0$-stable then $T^*$ is $\aleph_0$-stable.
\end{corollary}


\begin{proof}[Proof of Proposition~\ref{prop:lambda-0}]
By Proposition~\ref{prop:trace-theories} it is enough to suppose that $\Sa O$ is $\uplambda^+$-saturated, $\Sa N$ is $\uplambda^+$-saturated, and $\Sa N$ trace defines $\Sa O$ via an injection $O \to N^{m}$, and show that for any $n$, $S_n(\Sa O,\uplambda) \le S_{nm}(\Sa N,\uplambda)$.
Fix $n$ and $A \subseteq O$ such that $|A| \le \uplambda$.
Let $\Cal B$ be the collection of all subsets of $O^n$ that are $\Sa O$-definable with parameters from $A$.
As $|L^*|,|A| \le \uplambda$ we have $|\Cal B| \le \uplambda$.
For each $X \in \Cal B$ we let $Y_X$ be an $\Sa N$-definable subset of $N^{nm}$ such that $Y_X \cap O^n = X$, and $\Cal C = \{ Y_X : X \in \Cal B\}$.
We have $|\Cal C| \le \uplambda$.
Fix $B \subseteq N$ such that $|B| \le \uplambda$ and every $Y_X$ is definable with parameters from $B$.
It suffices to show that $|S_n(\Sa O,A)| \le |S_{nm}(\Sa N,B)|$.
We produce an injection $\upiota \colon S_n(\Sa O,A) \to S_{nm}(\Sa N,B)$.
Applying $\uplambda^+$-saturation we fix for each $p \in S_n(\Sa O,A)$ an $a_p \in O^n$ such that $\tp_{\Sa O}(a_p|A) = p$.
We declare $\upiota(p) = \tp_{\Sa N}(a_p|B)$.
We show that $\upiota$ is injective.
Fix distinct $p,q \in S_n(\Sa O,A)$.
Then there is $X \in \Cal B$ such that $p$ concentrates on $X$ and $q$ does not, hence $a_p \in X$ and $a_q \notin X$.
Then $a_p \in Y_X$ and $a_q \notin Y_X$, so $\tp_{\Sa N}(a_p|B) \ne \tp_{\Sa N}(a_q|B)$.
\end{proof}

\begin{corollary}
\label{cor:tot-trans}
Suppose that $T$ is totally transendental and $T^*$ is trace definable in $T$.
Then $T^*$ is totally transendental.
\end{corollary}

\begin{proof}
We apply the fact that an $L^*$-structure $\Sa O$ is totally transcendental if and only if $\Sa O\!\upharpoonright\! L^{**}$ is $\aleph_0$-stable for every countable $L^{**} \subseteq L^*$.
Suppose that $\Sa M \models T$ trace defines $\Sa O \models T^*$.
We may suppose that $O \subseteq M^m$ and $\Sa M$ trace defines $\Sa O$ via the identity $O \to M^m$.
Fix countable $L^* \subseteq L^*$.
For each $n$-ary parameter free $L^*$-formula $\phi(x)$ fix an $mn$-ary $L(M)$-formula $\varphi_{\phi}(x)$ such that for every $a \in O^n$ we have $\Sa O \models \phi(a) \Longleftrightarrow \Sa M \models \varphi_{\phi}(a)$.
Let $L^{**}$ be a countable sublanguage of $L$ such that each $\varphi_\phi$ is an $L^{**}(M)$-formula.
Proposition~\ref{prop:zero-def} shows that $\Sa M\!\upharpoonright\! L^{**}$ trace defines $\Sa O \!\upharpoonright\! L^*$.
By Corollary~\ref{cor:superstable} $\Sa O \!\upharpoonright\! L^*$ is $\aleph_0$-stable.
\end{proof}

We give a characterization of totally transendental theories due to Hanson.
We use Fact~\ref{fact:unary}.


\begin{fact}
\label{fact:unary}
Any structure in a unary relational language has quantifier elimination.
\end{fact}

Fact~\ref{fact:unary} is presumably classical.
We leave it as an exercise to the reader.

\medskip
Let $\Sa U$ be the structure with domain $\{0,1\}^{\upomega}$ and unary relations $(U_\upsigma : \upsigma \in \{0,1\}^{< \upomega})$ where 
\[
 U_\upsigma(\alpha) \quad\Longleftrightarrow\quad \upsigma \text{  is an initial segment of  } \alpha \text{  for all  } \upsigma \in \{0,1\}^{<\upomega}, \alpha \in \{0,1\}^\upomega.
 \]
Proposition~\ref{prop:hanson} is due to Hanson~\cite{james-hanson}.

\begin{proposition}
\label{prop:hanson}
The~theory~$T$~is~not~totally~transendental~if~and~only~if~$T$~trace~defines~$\Sa U$.
\end{proposition}

\begin{proof}
The right to left implication follows from Corollary~\ref{cor:tot-trans} as $\Sa U$ is not totally transendental.
Suppose that $\Sa M$ is $\aleph_1$-saturated and not totally transendental.
Then there is a family $(X_\upsigma : \upsigma \in \{0,1\}^{< \upomega})$ of nonempty definable subsets of $M$ such that $X_\upsigma \subseteq X_{\upeta}$ when $\upeta$ extends $\upsigma$.
By saturation there is an injection $\uptau \colon \{0,1\}^\upomega \to M$ such that we have
\[
\uptau(\alpha) \in X_\upsigma \quad \Longleftrightarrow \quad \alpha \text{  extends  } \upsigma \quad\text{for all  } \alpha \in \{0,1\}^\upomega, \upsigma \in \{0,1\}^{< \upomega}  .
\]
Fact~\ref{fact:unary} and Proposition~\ref{prop:qe} together show that $\Sa M$ trace defines $\Sa U$ via $\uptau$.
\end{proof}

\begin{proposition}
\label{prop:morley rank}
If $T$ has finite Morley rank and $T^*$ is trace definable in $T$ then $T^*$ has finite Morley rank.
\end{proposition}

Given a structure $\Sa M$ and definable $X \subseteq M^m$ we let $\mr_{\Sa M}(X)$ be the Morley rank of $X$.
Proposition~\ref{prop:morley rank} follows from Proposition~\ref{prop:morely rank fine}.

\begin{proposition}
\label{prop:morely rank fine}
Suppose that $O \subseteq M^m$ and $\Sa M$ trace defines $\Sa O$ via the identity $O \to M^m$.
Suppose that $X$ is an $\Sa O$-definable subset of $O^k$ and $Y$ is an $\Sa M$-definable subset of $M^{mk}$ such that $X = Y \cap O^k$.
Then $\mr_{\Sa O}(X) \le \mr_{\Sa M}(Y)$. 
\end{proposition}

Proposition~\ref{prop:morely rank fine} can be used to give another proof that total transdentality is preserved.

\begin{proof}
To simplify notation we only treat the case when $m = k = 1$, the general case follows in the same way.
Let $\uplambda$ be an ordinal and suppose that $\mro(X) \ge \uplambda$.
We show that $\mrm(Y) \ge \uplambda$.
We apply induction on $\uplambda$.
If $\mro(X) \ge 0$ then $X$ is nonempty, hence $Y$ is nonempty, hence $\mrm(Y) \ge 0$.
The case when $\uplambda$ is a limit ordinal is clear, we suppose that $\uplambda$ is a successor ordinal.
Fix $n$.
It is enough to produce pairwise disjoint $\Sa M$-definable $Y_1,\ldots,Y_n \subseteq Y$ such that $\mro(Y_i) \ge \uplambda - 1$ for each $i$.
As $\mro(X) > \uplambda - 1$ there are pairwise disjoint $\Sa O$-definable $X_1,\ldots,X_n \subseteq X$ such that $\mro(X_i) \ge \uplambda - 1$ for each $i$.
For each $i$ fix $\Sa M$-definable $Y_i \subseteq M$ such that $Y_i \cap O = X_i$.
After replacing each $Y_i$ with $Y \cap Y_i$ we suppose that each $Y_i$ is contained in $Y$.
After replacing each $Y_i$ with $Y_i \setminus \bigcup_{j \ne i} Y_j$ we suppose that the $Y_i$ are pairwise disjoint.
By induction we have $\mrm(Y_i) \ge \mro(X_i) \ge \uplambda - 1$.
\end{proof}

We now prove the analogous result for $U$-rank.
We let $\ru$ be the $U$-rank.
Given an ordinal $\uplambda$ and $n$ we let $n\cdot\uplambda$ be the ordinal produced by multiplying every coefficient in the Cantor normal form of $\uplambda$ by $n$.
If $\uplambda<\upomega$ then this is the usual product.

\begin{proposition}
\label{prop:U rank}
If $\Sa O$ is trace definable in $\Sa M$ via an injection $O \to M^m$ then we have $\ru(\Sa O)\le m\cdot\ru(\Sa M)$.
In particular if $T$ has finite $U$-rank and $T^*$ is trace definable in $T$ then $T^*$ has finite $U$-rank.
\end{proposition}

Proposition~\ref{prop:U rank} follows from Proposition~\ref{prop:U rank fine} and the Lascar inequalities \cite[19.2]{Poizat}.

\begin{proposition}
\label{prop:U rank fine}
Suppose that $O \subseteq M^m$ and $\Sa M$ trace defines $\Sa O$ via the identity $O\to M^m$.
Suppose that $X$ is an $\Sa O$-definable subset of $O^k$ and $Y$ is an $\Sa M$-definable subset of $M^{mk}$ such that $X = Y\cap O^k$.
Then $\ruo(X)\le\rum(Y)$.
\end{proposition}

We first make some remarks on $U$-rank that will be used below to prove Proposition~\ref{prop:U rank fine}.\\
I am sure all of this is trivial to people who actually know stability.

\begin{fact}
\label{fact:U rank}
Suppose that $T$ is superstable and $\Sa M\models T$ is $\kappa$-saturated for a cardinal $\kappa$.
For each ordinal $\uplambda$ we let $S_\uplambda$ be the set of global one-types in $\Sa M$ with $U$-rank $\ge \uplambda$.
\begin{enumerate}
\item If $\ru(\Sa M)\le\uplambda$ then $|S_\uplambda|\le 2^{|T|}$.
\item If $\ru(\Sa M)>\uplambda$ then $|S_\uplambda|\ge\kappa$.
\end{enumerate}
\end{fact}

Here a \textit{type} is a one-type.

\begin{proof}
Suppose $\ru(\Sa M)=\uplambda$.
If $p_0$ is the restriction of $p\in S_\uplambda$ to the empty set then we have $\ru(p)\le\ru(p_0)\le\ru(\Sa M)$, hence $\ru(p_0)=\uplambda$, hence $p$ is a non-forking extension of $p_0$.
Thus any $p\in S_\uplambda$ is a non-forking extension of the restriction of $p$ to $\emptyset$.
There are $\le 2^{|T|}$-types over $\emptyset$ and every type over $\emptyset$ has at most $2^{|T|}$ nonforking extensions~\cite[Proposition~2.20(iv)]{pillay-book}.
If $\ru(\Sa M)<\uplambda$ then $S_\uplambda$ is empty.
So we see that if $\ru(\Sa M)\le\uplambda$ then $|S_\uplambda|\le 2^{|T|}$.
Now suppose that $\ru(\Sa M)>\uplambda$.
Then there is a type $p$ over $\emptyset$ such that $\ru(p)>\uplambda$.
By saturation and the definition of $U$-rank there is a set $A$ of parameters from $\Sa M$ such that $p$ has $\ge\kappa$ extensions $q$ over $A$ of $U$-rank $\ge\uplambda$, and each extension $q$ has a global non-forking extension $p'$, and $p'$ is necessarily in $S_\uplambda$.
Hence $|S_\uplambda|\ge\kappa$.
\end{proof}

\medskip
Let $\mathrm{Def}(\Sa M)$ be the collection of definable subsets of $M$.
Given a ordinal $\uplambda$ we let $\app$ be the equivalence relation on $\mathrm{Def}(\Sa M)$ given by $X\app X'\Longleftrightarrow \rum(X\triangle X')<\uplambda$.

\begin{lemma}
\label{lem:app}
Suppose that $T$ is superstable, $\uplambda$ is an ordinal.
The following are equivalent:
\begin{enumerate}
\item $\ru(T)\le\uplambda$.
\item $|\mathrm{Def}(\Sa M)/\!\app\!|\le 2^{2^{|T|}}$ for all $\Sa M\models T$.
\end{enumerate}
\end{lemma}

\begin{proof}
Suppose $\Sa M\models T$ and let $S_\uplambda$ be as above. 
If $X,X'\subseteq M$ are definable then $X\app X'$ if and only if $X$ and $X'$ contain the same types of rank $\ge\uplambda$.
Hence $|\mathrm{Def}(\Sa M)/\!\app\!|\le 2^{|S_\uplambda|}$.
If $p\in S_\uplambda$ and $X\app X'$ are definable then $p$ is not in $X\triangle X'$, so $p$ is in $X$ if and only if $p$ is in $X'$.
Hence we also have $|S_\uplambda|\le2^{|\mathrm{Def}(\Sa M)/\!\app\!|}$.
Apply Fact~\ref{fact:U rank}.
\end{proof}

We now prove Proposition~\ref{prop:U rank fine}.

\begin{proof}
To simplify notation we only treat the case when $X=O$ and $Y=M$, the general case follows in the same way by replacing $\Sa O$, $\Sa M$ with the structure induced on $X$, $Y$ by $\Sa O$, $\Sa M$, respectively.
Suppose as above that $\Sa M$, $\Sa O$ is a $T$, $T^*$-model, respectively.
After possibly passing to elementary extensions and adding constant symbols we may suppose that $|T|=|T^*|$.
If $\Sa M$ is not superstable then $\ru(\Sa M)=\infty$ and the inequality trivially holds.
We suppose that $\Sa M$ is superstable.
Then $\Sa O$ is superstable by Corollary~\ref{cor:superstable}.

\medskip
Let $\uplambda$ be an ordinal.
We apply induction on $\uplambda$ to show that if $\ru(\Sa M)\le\uplambda$ then $\ru(\Sa O)\le\uplambda$.
Suppose $\uplambda=0$.
Then $Y$ is finite, hence $X$ is finite, hence $\ruo(X)=0$.

\medskip
Suppose $\uplambda > 0$.
By Lemma~\ref{lem:app} and Prop~\ref{prop:trace-theories} it is enough to show that $|\mathrm{Def}(\Sa O)/\!\app\!|\le 2^{2^{|T^*|}}$.
By Lemma~\ref{lem:app} it is enough to show that $|\mathrm{Def}(\Sa O)/\!\app\!|\le|\mathrm{Def}(\Sa M)/\!\app\!|$.
It is enough to fix $\Sa O$-definable $X,X' \subseteq O$ and $\Sa M$-definable $Y,Y'\subseteq M$, suppose that $X=Y\cap O$, $X'=Y'\cap O$ and $Y\app Y'$, and show that $X \app X'$.
We have $\ru_{\Sa M}(Y\triangle Y')<\uplambda$ as $Y\app Y'$.
Note that $X\triangle X' \subseteq Y\triangle Y'$, so $\ru_{\Sa O}(X\triangle X')<\uplambda$ by induction, hence $X\app X'$.
\end{proof}

 \subsection{NIP}
\label{section:trace nip}
We now discuss $\nip$-theoretic properties away from stability.

\begin{proposition}
\label{prop:trace}
If $T$ is $\nip$ and $T^*$ is trace definable in $T$ then $T^*$ is $\nip$.
\end{proposition}

Proposition~\ref{prop:trace} is easy and left to the reader.

\begin{proposition}
\label{prop:trace-0}
The following are equivalent:
\begin{enumerate}
\setcounter{enumi}{3}
\item $T$ is $\mathrm{IP}$,
\item $T$ trace defines the Erd\H{o}s-Rado graph,
\item $T$ trace defines the generic countable bipartite graph.
\end{enumerate}
\end{proposition}

One can also prove the equivalence of (4) and (5) using Proposition~\ref{prop:picture} and the characterization of $\nip$ in terms of uncollapsed indiscernibles, see \cite{scow} and Proposition~\ref{prop:k-dep-0} below.

\begin{proof}
By Proposition~\ref{prop:trace-theories} it is enough to suppose that $\Sa M$ is $\aleph_1$-saturated and show that the following are equivalent:
\begin{enumerate}
\item $\Sa M$ is $\mathrm{IP}$,
\item $\Sa M$ trace defines the Erd\H{o}s-Rado graph,
\item $\Sa M$ trace defines the generic countable bipartite graph.
\end{enumerate}
Proposition~\ref{prop:trace} shows that $(2)$ and $(3)$ both imply $(1)$.
We show that $(1)$ implies both $(2)$ and $(3)$.
Suppose that $\Sa M$ is $\mathrm{IP}$.
We first show that $\Sa M$ trace defines the generic countable bipartite graph.
Fix a formula $\varphi(x,y)$, a sequence $(a_i : i \in \N)$ of elements of $M^{|x|}$, and a family $(b_I : I \subseteq \N)$ of elements of $M^{|y|}$ such that $\Sa M \models \varphi(a_i,b_I) \Longleftrightarrow i \in I$ for any $I \subseteq \N$ and $i \in \N$.
Let $V = \{ a_i : i \in \N \}$, $W = \{ b_I : I \subseteq \N \}$, and $E = \{ (a,b) \in V \times W : \Sa M \models \varphi(a,b) \}$.
Any countable bipartite graph embeds into $(V,W;E)$, so in particular the generic countable bipartite graph embeds into $(V,W;E)$.
Apply Proposition~\ref{prop:qe} and quantifier elimination for the generic countable bipartite graph.
We now show that $\Sa M$ trace defines the Erd\H{o}s-Rado graph.
Laskowski and Shelah~\cite[Lemma~2.2]{Laskowski-shelah-karp} show that the Erd\H{o}s-Rado graph embeds into a definable graph.
Apply Proposition~\ref{prop:qe}.
\end{proof}

Strong dependence and dp-finiteness are also preserved.

\begin{proposition}
\label{thm:dp-rank}
Suppose that $T^*$ is trace definable in $T$.
If $T$ has finite dp-rank then $T^*$ has finite dp-rank, and if $T$ has infinite dp-rank then $\dprk T^* \le \dprk T$.
In particular if $T$ is strongly dependent then $T^*$ is strongly dependent.
\end{proposition}

The first two claims of Proposition~\ref{thm:dp-rank} are immediate from Proposition~\ref{prop:dp-rank}.
The last claim follows from the previous as $T$ is strongly dependent $\Longleftrightarrow$ $\dprk T \le \aleph_0-1$.

\begin{proposition}
\label{prop:dp-rank}
Suppose that $\Sa M$ trace defines $\Sa O$ via an injection $O \to M^m$.
If $\Sa M$ has finite dp-rank then $\dprk \Sa O \le m \dprk \Sa M$ and if $\Sa M$ has infinite dp-rank then $\dprk \Sa O \le \dprk \Sa M$.
\end{proposition}

Proposition~\ref{prop:dp-rank} follows from Proposition~\ref{prop:dp-rank-0} and subadditivity of dp-rank (Fact~\ref{fact:dp-rank}).

\begin{proposition}
\label{prop:dp-rank-0}
Suppose that $O \subseteq M^m$ and $\Sa M$ trace defines $\Sa O$ via the identity $O \to M^m$.
Suppose that $X$ is an $\Sa O$-definable subset of $O^k$ and $Y$ is an $\Sa M$-definable subset of $M^{mk}$ such that $X = Y \cap O^k$.
Then $\dprk_{\Sa O}(X) \le \dprk_{\Sa M}(Y)$. 
\end{proposition}


\begin{proof}[Proof of Proposition~\ref{prop:dp-rank-0}]
We only treat the case when $m = k = 1$, the general case follows in the same way.
By Proposition~\ref{prop:trace-theories} we may suppose that $\Sa M \models T$ and $\Sa O \models T^*$ are both $\aleph_1$-saturated, $O \subseteq M^m$, and $\Sa M$ trace defines $\Sa O$ via the identity $O \to M^m$.
Suppose that $\uplambda$ is a cardinal and $\dprk_{\Sa O} X \ge \uplambda$.
Fix an $(\Sa O,X,\uplambda)$-array consisting of parameter free $L^*$-formulas $(\varphi_\alpha(x_\alpha , y) : \alpha < \uplambda )$ and tuples $(a_{\alpha,i} \in O^{|x_\alpha|} : \alpha < \uplambda , i < \upomega )$ .
For each $\alpha < \uplambda$ we fix an $L(M)$-formula $\theta_\alpha(z_\alpha,w)$, $|z_\alpha| = |x_\alpha|$, $|w| = 1$, such that for any $a \in O^{|x_\alpha|}$, $b \in O$ we have $\Sa O \models \varphi_\alpha(a,b) \Longleftrightarrow \Sa M \models \theta_\alpha(a,b)$.
It is now easy to see that $(\theta_\alpha(z_\alpha,w) : \alpha < \uplambda)$ and $( a_{\alpha,i} : \alpha < \uplambda, i < \upomega )$ forms an $(\Sa M,Y,\uplambda)$ array.
Thus $\dprk_{\Sa M} Y \ge \uplambda$.
Hence $\dprk_{\Sa M} Y \ge \dprk \Sa O$.
\end{proof}

We now prove Proposition~\ref{prop:dp-rank}.

\begin{proof}
By Proposition~\ref{prop:dp-rank-0} $\dprk \Sa O \le \dprk_{\Sa M} M^m$.
Suppose $\dprk \Sa M < \aleph_0$.
By Fact~\ref{fact:dp-rank} $\dprk_{\Sa M} M^m$ is at most $m \dprk_{\Sa M} M$.
Suppose $\dprk \Sa M\ge \aleph_0$.
By Fact~\ref{fact:dp-rank} $\dprk_{\Sa M} M^m = \dprk_{\Sa M} M$.
So $\dprk \Sa O \le \dprk \Sa M$.
\end{proof}

As a corollary we can show that there are class-many trace equivalence classes.
For each cardinal $\uplambda$ let $L_\uplambda$ consist of binary relations $(E_\kappa : \kappa < \uplambda)$, $T_\uplambda$ be the theory of $\uplambda$ equivalence relations, and $T^*_\uplambda$ be the model companion of $T_\uplambda$.
Then $T_\uplambda$ has dp-rank $\uplambda$ (each equivalence relation gives one row).
It is easy to see that $T_\upeta$ is the $L_\upeta$-reduct of $T_\uplambda$ for all $\upeta < \uplambda$.
Proposition~\ref{thm:dp-rank} shows that if $\uplambda, \upeta$ are infinite cardinals then $T_\uplambda$ trace defines $T_\upeta$ iff $\upeta \le \uplambda$.

\medskip
We now show that finiteness of op-dimension is preserved.
Op-dimension was introduced by Guingona and Hill~\cite{gh-op}.
Let $\opd_{\Sa M}(X)$ be the op-dimension of an $\Sa M$-definable set $X$ and $\opd(T)$ be the op-dimension of $T$.
Recall that $\opd(T) = \opd(M)$ for any $\Sa M \models T$.
A structure with finite op-dimension has $\nip$, see \cite[Section 3.1]{gcs}.
We will also apply subadditivity of op-dimension, see \cite[Theorem~2.2]{gh-op}.
A $k$-order is a structure $(P;<_1,\ldots,<_k)$ where each $<_k$ is a linear order on $P$.
Finite $k$-orders form a \Fraisse class, we denote the  \Fraisse limit by $\Sa P_k$.
Fact~\ref{fact:op-collapse} is \cite[Theorem~3.4]{gcs}.

\begin{fact}
\label{fact:op-collapse}
Suppose that $\monster$ is $\nip$ and let $X$ be an $\monster$-definable set.
Then the following are equivalent for any $k \ge 1$:
\begin{enumerate}[leftmargin=*]
\item $\opd_{\monster}(X) \ge k$
\item there is a small set $A$ of parameters and an uncollapsed $A$-indiscernible picture $P \to X$ of $\Sa P_k$ in $\monster$.
\end{enumerate}
\end{fact}

\begin{proposition}
\label{prop:op}
If $\Sa M$ trace defines $\Sa O$ via an injection $O \to M^m$ then $\opd(\Sa O) \le m \opd(\Sa M)$.
If $T$ has finite op-dimension and $T^*$ is trace definable in $T$ then $T^*$ has finite op-dimension.
\end{proposition}

One could also prove an analogue of Propositions~\ref{prop:morely rank fine} and \ref{prop:dp-rank-0} here.
This would require making Section~\ref{section:collapse} a bit more techinical, so I didn't do it.


\begin{proof}
We prove the first claim, the second claim follows.
By Proposition~\ref{prop:trace-theories} we may suppose that $\Sa M$ and $\Sa O$ are both $\aleph_1$-saturated, $O \subseteq M^m$, and $\Sa M$ trace defines $\Sa O$ via the identity $O \to M^m$.
If $\opd(\Sa M) = \infty$ then the inequality trivially holds, so we suppose $\Sa M$ has finite op-dimension.
Hence $\Sa M$ is $\nip$.
By Proposition~\ref{prop:trace} $\Sa O$ is $\nip$.
Hence by Fact~\ref{fact:op-collapse} there is an uncollapsed indiscernible picture $P \to \oonsterset$ of $\Sa P_k$ in $\oonster$.
By Proposition~\ref{prop:trace-theories} $\Sa O$ trace defines $\Sa P_k$ via an injection $\uptau \colon P \to O$.
Hence $\Sa M$ trace defines $\Sa P_k$ via $\uptau \colon P \to M^m$.
By Proposition~\ref{prop:picture} there is an uncollapsed indiscernible picture $P \to \monsterset^m$ of $\Sa P_k$ in $\monster$. 
Thus $\opd_{\monster}(\monsterset^m) \ge k$.
By subadditivity of op-dimension $m\opd(\Sa M) \ge k$.
\end{proof}

We now make some comments concerning VC-density.
We follow \cite[3.2]{toomanyI}.
Given a theory $T$ and $n$ we let $\mathrm{vc}(T,n)\in\R\cup\{\infty\}$ be the supremum of all positive $r\in\R$ such that for every $\Sa M\models T$ and formula $\varphi(x,y)$ with $|y|=n$ there is $\lambda\in\R$ such that if $X\subseteq M^{|x|}$ is finite then there are at most $\lambda|X|^r$ subsets of $X$ of the form $\{ a \in X : \Sa M\models\varphi(a,b)\}$, $b$ ranging over $M^n$.
For many $\nip$ theories of interest $\mathrm{vc}(T,n)$ is bounded above by a linear function of $n$.

\begin{proposition}
\label{prop:vc}
Suppose that $\Sa M\models T$, $\Sa O\models T^*$, and $\Sa M$ trace defines $\Sa O$ via $O \to M^m$.
Then $\mathrm{vc}(T^*,n)\le\mathrm{vc}(T,mn)$ for all $n$.
If $T$ trace defines $T^*$ and $\mathrm{vc}(T,n)$ is bounded above by a linear function of $n$ then $\mathrm{vc}(T^*,n)$ is bounded above by a linear function of $n$.
\end{proposition}

Proposition~\ref{prop:vc} should be basically obvious by now.

\subsection{Higher airity dependence}
\label{section:k-dep}
\noindent
Suppose $k \ge 2$.
We refer to \cite{cpt} for the definition of $k$-dependence.
Proposition~\ref{prop:k-dep} is clear from the usual definition of $k$-dependence.

\begin{proposition}
\label{prop:k-dep}
If $T$ is $k$-dependent and $T^*$ is trace definable in $T$ then $T^*$ is $k$-dependent.
\end{proposition}

Proposition~\ref{prop:k-dep-0} follows from Proposition~\ref{prop:random}, the Ramsey property for ordered hypergraphs, Proposition~\ref{prop:ramsey}, and the fact that $\monster$ is $k$-dependent if and only if  $\monster$ admits an uncollapsed generic countable ordered $(k+1)$-hypergraph indiscernible \cite[Theorem~5.4]{cpt}.

\begin{proposition}
\label{prop:k-dep-0}
The theory $T$ is $k$-dependent if and only if $T$ does not trace define the generic countable $(k+1)$-hypergraph.
\end{proposition}

We discuss two further interesting $\nip$-theoretic properties in Sections~\ref{section:almost-lin} and \ref{section:eh}.


\section{Examples}
\label{section:examples}

\subsection{Shelah Completions}
We first discuss Shelah completions and some associated examples.
\label{section:completions}
\noindent
One has the feeling that a $\nip$ structure and its Shelah completion should be equivalent in some sense.
Trace equivalence seems to be the right sense.

\begin{lemma}
\label{lem:she-last}
Let $\uplambda$ be a cardinal.
Suppose that $\Sa M$ is $\nip$ and $\uplambda$-saturated and $\Sa O$ is an elementary submodel of $\Sa M$ of cardinality $< \uplambda$.
Then $\Sa M$ trace defines $\Sh O$.
\end{lemma}

\begin{proof}
Suppose that $X \subseteq O^n$ is $\Sh O$-definable.
By Fact~\ref{fact:shelah} $X$ is externally definable in $\Sa O$.
By Lemma~\ref{lem:lambda} there is an $\Sa M$-definable $Y \subseteq M^n$ such that $X = Y \cap O^n$.
\end{proof}

\noindent
It is clear that the Shelah completion of $\Sa M$ trace defines $\Sa M$, the other direction of Proposition~\ref{prop:she-0} follows from Lemma~\ref{lem:she-last}.

\begin{proposition}
\label{prop:she-0}
Every $\nip$ structure is trace equivalent to its Shelah completion.
\end{proposition}



\noindent
Corollary~\ref{cor:poizat-b} follows from Proposition~\ref{prop:she-0} and Fact~\ref{fact:convex}.

\begin{corollary}
\label{cor:poizat-b}
Suppose that $\Sa M$ is a $\nip$ expansion of a linear order $(M;<)$ and $\Cal C$ is a collection of convex subsets of $M$.
Then $(\Sa M,\Cal C)$ is trace equivalent to $\Sa M$.
\end{corollary}

\noindent
Fact~\ref{fact:rigid} shows that if $\Sa M$ is an o-minimal expansion of a field and $C$ is a convex subset of $M$ which is not an interval then $\Sa M$ does not interpret $(\Sa M;C)$.

\begin{corollary}
\label{cor:hensel}
Suppose that $K$ is a $\nip$ field and $v$ is a Henselian valuation on $K$ such that the residue field of $v$ is not separably closed.
Then $K$ and $(K,v)$ are trace equivalent.
\end{corollary}

\noindent
Corollary~\ref{cor:hensel} follows from Proposition~\ref{prop:she-0} and the fact that if $K$ and $v$ satisfy the conditions above then $v$ is externally definable in $K$, see \cite{jahnke-when}.

\medskip\noindent
We now assume that the reader is somewhat familiar with the theory of definable groups in $\nip$ structures.
Suppose that $\monster$ is $\nip$ and $G$ is a definable group.
Let $\uppi$ be the quotient map $G \to \gqot$.
The structure induced on $\gqot$ by $\monster$ is the expansion of the pure group $\gqot$ by all sets $\{ (\uppi(\alpha_1),\ldots,\uppi(\alpha_n)) : \alpha_1,\ldots,\alpha_n \in X \}$ for definable $X \subseteq G^n$.

\begin{proposition}
\label{prop:gqot}
Suppose that $\monster$ is $\nip$, $G$ is a definably amenable definable group, and $\gqot$ is a  Lie group.
Then the structure induced on $\gqot$ by $\monster$ is trace definable in $\monster$.
If $\monster$ is o-minimal and $G$ is a definably compact definable group then the structure induced on $\gqot$ by $\monster$ is trace definable in $\monster$.
\end{proposition}

In particular if $\monster$ is o-minimal and $G$ is a definably compact definable group then the induced structure on $\gqot$ is trace definable in $\monster$, so $\monster$ trace defines $\gqot$ as a group.

\begin{proof}
By Proposition~\ref{prop:trace-theories} and Proposition~\ref{prop:she-0} it is enough to show that $\monster^{\mathrm{Sh}}$ interprets the induced structure on $\gqot$.
This follows from work of Pillay and Hrushovski who showed that $G^\zerozero$ is externally definable in $\monster$.
They proved this in the case when $\monster$ is o-minimal and $G$ is definably compact \cite[Lemma~8.2]{HP-invariant} and pointed out that the proof goes through when $\monster$ is $\nip$, $G$ is definably amenable, and $\gqot$ is a Lie group~\cite[Remark 8.3]{HP-invariant}.
\end{proof}

\subsection{Abelian groups}
\label{section:trace torsion free}
We now discuss trace definibility between abelian groups.
\textbf{In this section we change some notational conventions!}
All of our structures here are abelian groups so in this section we write $\Z$ where in other sections we would write $(\Z;+)$ and so on.
Yes, this does conflict with the notation elsewhere.
(I tried to write it the other way but it looked too bad.)
Throughout $p$ ranges over primes.


\medskip
An embedding $f\colon A \to B$ of abelian groups is \textbf{pure} if it is an $L_\mathrm{div}$-embedding, i.e. for all $k$ and $\alpha \in A$, $k$ divides $f(\alpha)$ in $B$ if and only if $k$ divides $\alpha$ in $A$.
(See Section~\ref{section:abelian} for $L_\mathrm{
div}$.)

\begin{lemma}
\label{lem:pure}
Suppose that $A$ and $A^*$ are abelian groups and $\uptau \colon A \to A^*$ is a pure embedding.
Then $A^*$ trace defines $A$ via $\uptau$.
If $A$ is a direct summand of $A^*$ then $A^*$ trace defines $A$.
\end{lemma}

\begin{proof}
The first claim follows from Fact~\ref{fact:abelian qe} and Proposition~\ref{prop:qe-trace}.
The second follows from the first claim and the fact that the natural embedding $A \to A\oplus B$ is pure.
\end{proof}

If $A'$ is a pure subgroup of $A$ and $A'$ is $\aleph_1$-saturated then $A'$ is direct summand of $A$, hence $A$ is isomorphic to $A\oplus (A/A')$, see for example \cite[10.7.1, 10.7.3]{Hodges} or \cite[Corollary~3.3.38]{trans}.
It follows by Fact~\ref{fact:direct sum} that if $A'$ is a pure subgroup of $A$ then $A\equiv A'\oplus (A/A')$.

\begin{lemma}
\label{lem:oplus}
Suppose that $A,A_1,\ldots,A_n$ are abelian groups.
Then $A_1\oplus\cdots\oplus A_n$ is trace equivalent to $A_1 \sqcup\cdots\sqcup A_n$.
Furthermore the following structures are trace equivalent:
\begin{enumerate}[leftmargin=*]
\item $A$,
\item $A^n$ for any $n$, and
\item the disjoint union of the torsion subgroup $\mathrm{Tor}(A)$ of $A$ and the torsion free group $A/\mathrm{Tor}(A)$.
\end{enumerate}
\end{lemma}

Thus if we want to understand abelian groups trace definable in a theory $T$ we may consider the torsion and torsion-free cases separately.

\begin{proof}
It is easy to see that $A_1\sqcup\cdots\sqcup A_n$ defines $A_1\oplus\cdots\oplus A_n$.
By Lemma~\ref{lem:pure} $A_1\oplus\cdots\oplus A_n$ trace defines each $A_i$.
Hence $A_1\oplus\cdots\oplus A_n$ trace defines $A_1\sqcup\cdots\sqcup A_n$ by Lemma~\ref{lem:disjoint union}.
Furthermore $A^n$ is trace equivalent to the disjoint union of $n$ copies of $A$ and this disjoint union is mutually interpretable with $A$.
Finally, $A\equiv\mathrm{Tor}(A)\oplus(A/\mathrm{Tor}(A))$ by the remark before Lemma~\ref{lem:oplus}.
(Note that $\mathrm{Tor}(A)$ is a pure subgroup of $A$).
\end{proof}

\begin{proposition}
\label{prop:ab}
Any group $G$ of infinite exponent trace defines $\Th(\Q)$.
\end{proposition}

\noindent
So $\Th(\Z)$ trace defines $\Th(\Q)$.
Presburger arithmetic does not interpret $\Th(\Q)$ by Fact~\ref{fact:pres group}.

\begin{proof}
By Proposition~\ref{prop:qe-trace} it is enough to construct an elementary extension $G^*$ of $G$ and an embedding $(\Q;+)\to G^*$.
After possibly passing to an elementary extension fix a non-torsion $\beta\in G$, let $Z$ be the subgroup generated by $\beta$, and $(G^*,Z^*)$ be an $\aleph_1$-saturated elementary extension of $(G,Z)$.
By the proof of Proposition~\ref{prop:Q/Z} there is an embedding $(\Q;+)\to Z^*$.
\end{proof}

We consider abelian groups that are trace definable in $\Th(\Q)$.
By Corollary~\ref{cor:superstable} any structure trace definable in $\Th(\Q)$ is $\aleph_0$-stable.
Macintyre showed that an abelian group $A$ is $\aleph_0$-stable iff $\mathrm{Tor}(A)$ has finite exponent and $A/\mathrm{Tor}(A)$ is divisible \cite[Thm~A.2.11]{Hodges}.
I don't know if $\Th(\Q)$ can trace define an infinite abelian group of finite exponent or vice versa.

\begin{proposition}
\label{prop:divisible abelian}
Suppose $A$ is a non-trivial divisible abelian group and suppose that there is $m\in\N$ such that $\rank_p(A)\le m$ for all primes $p$.
Then $A$ is trace equivalent to $\Q$.
In particular $\Z(p^\infty)$ is trace equivalent to $\Q$ for any prime $p$.
\end{proposition}

By Fact~\ref{fact:group in a group} any group interpretable in a torsion free divisible abelian group has a torsion-free subgroup of finite index, hence a torsion free divisible abelian group cannot interpret a divisible abelian group with torsion.
By Proposition~\ref{prop:morley rank} any structure trace definable in $\Th(\Q)$ has finite Morley rank, and by work of Macintyre a divisible abelian group $A$ has finite Morley rank if and only if $\rank_p(A)<\aleph_0$ for all $p$, see \cite[Theorem~6.7]{borovik-nesin}.

\begin{proof}
By Proposition~\ref{prop:ab} $\Th(A)$ trace defines $\Th(\Q)$.
We show that $A$ is trace definable in $\Th(\Q)$.
By Proposition~\ref{prop:Q/Z} and Lemma~\ref{lem:oplus} $\Q$ trace defines $(\Q/\Z)^m$.
Recall that by Fact~\ref{fact:divisible classification} $\Q/\Z$ is isomorphic to $\bigoplus_p \Z(p^\infty)$ and $A$ is isomorphic to $\Q^{\rank(A)}\oplus\bigoplus_p \Z(p^\infty)^{\rank_p(A)}$.
By Lemma~\ref{lem:oplus} and Proposition~\ref{prop:ab} we may suppose that $\rank(A) = 0$.
Therefore $A$ is a summand of $(\Q/\Z)^m$ and hence $A$ is trace definable in $(\Q/\Z)^m$.
\end{proof}


We now consider abelian groups that are trace definable in $\Th(\Z)$, starting with the torsion free case.
Recall that $\corank_p(A)$ is the dimension of the $\F_p$-vector space $A/pA$.

\begin{proposition}
\label{prop:torsion free in Z}
Suppose that $A$ is a torsion free abelian group and suppose that there is $m\in\N$ such that $\corank_p(A)\le m$ for all $p$.
Then $A$ is trace definable in $\Th(\Z)$.
If $\corank_p(A)\ge 1$ for all $p$ then $A$ is trace equivalent to $\Z$.
\end{proposition}

The first claim is close to sharp.
Recall that $\Z$ is superstable, so by Corollary~\ref{cor:superstable} any structure trace definable in $\Th(\Z)$ is superstable.
By a result of Rogers \cite[Thm~A.2.13]{Hodges} a torsion free abelian group $A$ is superstable if and only if $\corank_p(A)<\aleph_0$ for all $p$.

\begin{proof}
By Fact~\ref{fact:szmielew} we may suppose that $A$ is of the form $\bigoplus_{p} \Z^{\mu_p}_{(p)}$ for natural numbers  $\mu_p \le m$.
Then $A$ is a direct summand of $\bigoplus_p \Z^m_{(p)}$, so $A$ is trace definable in $\bigoplus_p \Z^m_{(p)}$.
By Fact~\ref{fact:szmielew} we have $\Z^m\equiv\bigoplus_p\Z^m_{(p)}$, so $A$ is trace definable in $\Th(\Z)$.
If $\mu_p\ge 1$ for all primes $p$ then $\bigoplus_p \Z_{(p)}$ is a direct summand of, and is hence trace definable in, $A$.
\end{proof}

A similar argument shows that any torsion free abelian group is trace definable in $\Th(\Z^\upomega)$.

\medskip
Proposition~\ref{prop:finite rank} follows from Proposition~\ref{prop:torsion free in Z} and Fact~\ref{fact:finite rank}.

\begin{proposition}
\label{prop:finite rank}
Any finite rank torsion free abelian group is trace definable in $\Th(\Z)$.
\end{proposition}


We now discuss the torsion case.
I don't know if $\Th(\Z)$ can trace define an infinite vector space over a finite field.

\begin{lemma}
\label{lem:cover}
Let $\uppi$ be the quotient map $\Q^n\to(\Q/\Z)^n$.
Suppose that $A$ is a subgroup of $(\Q/\Z)^n$ and that $H$ is the subgroup of $\Q^n$ given by $H=\uppi^{-1}(A)$.
Then $H$ trace defines $A$.
\end{lemma}

However, $A$ may not trace define $H$.
Suppose that $A=\Z(p^\infty)$, where we take $\Z(p^\infty)$ to be the subgroup of $\Q/\Z$ consisting of elements of the form $(k/p^n)+\Z$ where $0\le k\le p^n$.
Then $H=\{k/p^n : n\in\N, k \in\Z\}$.
Then $H$ is torsion free and not divisible by any prime other than $p$, so $H$ is not trace definable in an $\aleph_0$-stable structure such as $\Z(p^\infty)$.

\begin{proof} 
Note that $H$ contains $\Z^n = \uppi^{-1}(0)$.
Note that $\uppi(\gamma)=\gamma+\Z^n$ for all $\gamma\in H$.
Let $J = [0,1)^n\cap H$ and $\uptau\colon A\to J$ be the bijection given by declaring $\uptau(\gamma+\Z^n)$ to be the unique element of $[\gamma+\Z^n]\cap J$, so $\uptau(\gamma+\Z^n)=\gamma$ for all $\gamma\in J$.
Then $\uptau$ is a section of $\uppi$, we will use this.
We show that $H$ trace defines $A$ via $\uptau$.
Let $T$ be the term given by $T(x_1,\ldots,x_k)=m_1x_1+\cdots+m_kx_k$ and let
\[
X = \{ \alpha\in A^n : j \text{  divides  }T(\alpha)+\beta \} \quad\text{for some}\quad j\in\N, \beta\in A^k.
\]
By Fact~\ref{fact:abelian qe} it is enough to produce $H$-definable $Y\subseteq H^n$ such that $X=\uptau^{-1}(Y)$.
Note that $j$ divides $T(\alpha)+\beta$ if and only if $T(\uptau(\alpha))+\uptau(\beta)$ is in $jH+\Z^n$.
We first suppose that $j \ge 1$.
Then we have \[jH+\Z^n=jH +(j\Z^n+\{0,\ldots,j-1\}^n)=jH+\{0,1,\ldots,j-1\}^n.\]
The first equality holds by the remainder theorem and the second equality holds as $j\Z^n\subseteq jH$.
Hence $jH+\Z^n$ is $H$-definable.
Let $Y=\{a\in H^n:T(a)+\uptau(\beta)\in jH+\Z^n\}$, then $X=\uptau^{-1}(Y)$.

\medskip
We now treat the case when $j = 0$, so $X$ is the set of $\alpha\in A^n$ such that $T(\alpha)+\beta=0$.
So for any $\alpha\in A^n$ we have $\alpha\in X$ if and only if $T(\uptau(\alpha))+\uptau(\beta)\in\Z^n$.
Proceed as in the second part of the proof of Proposition~\ref{prop:Q/Z} with $\Q$, $\Z$ replaced by $\Q^n$, $\Z^n$, respectively, and the absolute value on $\Q$ replaced with the $\ell_\infty$-norm on $\Q^n$.
\end{proof}

\begin{proposition}
\label{prop:finite total rank}
Any subgroup $A$ of $(\Q/\Z)^m$ is trace definable in $\Th(\Z)$.
\end{proposition}

\begin{proof}
Let $H$ be the pre-image of $A$ under the quotient map $\Q^m\to(\Q/\Z)^m$, so $H$ is a subgroup of $\Q^m$.
Hence $H$ is a finite rank torsion free group, so $H$ is trace definable in $\Th(\Z)$ by Proposition~\ref{prop:finite rank}.
Apply Lemma~\ref{lem:cover}.
\end{proof}

\begin{corollary}
\label{cor:torus}
Any finite rank subgroup $A$ of $(\R/\Z)^m$ is trace definable in $\Th(\Z)$.
\end{corollary}

\begin{proof}
By Lemma~\ref{lem:oplus} it is enough to show that $\mathrm{Tor}(A)$ and $A/\mathrm{Tor}(A)$ are both trace definable in $\Th(\Z)$.
Note that $\mathrm{Tor}(A)$ is a subgroup of $(\Q/\Z)^m$, so we can apply Proposition~\ref{prop:finite total rank}.
Furthermore $A/\mathrm{Tor}(A)$ is finite rank torsion free, apply Proposition~\ref{prop:torsion free in Z}.
\end{proof}

We now give a rephrasing of Proposition~\ref{prop:finite total rank}.
Let $\Z(p^n)=\Z/p^n\Z$ for all $n$.
Recall that each $\Z(p^n)$ is subgroup of $\Z(p^\infty)$ and that by Fact~\ref{fact:divisible classification} $\Q/\Z$ is isomorphic to $\bigoplus_p \Z(p^\infty)$.
Hence the subgroups of $(\Q/\Z)^m$ are exactly the groups described in Corollary~\ref{cor:prufner}.

\begin{corollary}
\label{cor:prufner}
For each prime $p$ fix $\kappa_p \in \N\cup\{\infty\}$ and $\mu_p\in\N$.
Suppose there is $m$ such that $\mu_p\le m$ for all $p$.
Then $\bigoplus_p \Z(p^{\kappa_p})^{\mu_p}$ is trace definable in $\Th(\Z)$.
\end{corollary}

In particular $\bigoplus_p \Z(p)$ is trace definable in $\Th(\Z)$.
Let $A=\bigoplus_p \Z(p^{\kappa_p})^{\mu_p}$ where $\kappa_p \in\N\cup\{\infty\}$ and $\mu_p\in\N\cup\{\aleph_0\}$ for each prime $p$.
If $\mu_p=\aleph_0$ for some $p$ then $A$ trace defines an infinite vector space over a finite field, so in this case I do not know if $A$ is trace definable in $\Th(\Z)$.
If $\mu_p=\aleph_0$ for infinitely many $p$ then $A$ has infinite dp-rank, see \cite{halevi-palacin} or \cite[Lemma~5.2]{toomanyII}, hence $A$ is not trace definable in $\Th(\Z)$ by Proposition~\ref{thm:dp-rank}.

\medskip
The arguments above show that if $A$ is a non-divisible torsion free abelian group then $\Th(A)$ trace defines $\Z_{(p)}$ for some $p$.
For this reason I would like to know if o-minimal structures can trace define $\Z_{(p)}$.
If they can't then one could prove many corollaries.

\subsection{Some ordered examples}
\label{section:ordered examples}
We first prove an analogue of Proposition~\ref{prop:ab}.

\begin{proposition}
\label{prop:re-oag}
The theory of any ordered abelian group or infinite cyclically ordered abelian group trace defines $\doag$.
\end{proposition}

By Fact~\ref{fact:zp} there are ordered abelian groups that do not interpret $\mathrm{DOAG}$.

\begin{proof}
The first claim follows by the proof that Presburger arithmetic trace defines $\doag$, see Proposition~\ref{prop:final Z}.
Suppose that $(H;+,C)$ is a cyclically ordered abelian group and let $(G;+,C)$ be an $\aleph_1$-saturated elementary extension of $(H;+,C)$.
By Fact~\ref{fact:cover} $(H;+,C)$ is interpretable in an ordered abelian group, hence $(H;+,C)$ is $\nip$.
Hence by Proposition~\ref{prop:she-0} it is enough to show that $(G;+,C)^\mathrm{Sh}$ defines an ordered abelian group.
This follows by the proof of Proposition~\ref{prop:dense pair} as the set $J$ defined in that proof is externally definable.
\end{proof}

\begin{proposition}
\label{prop:ordered-fields}
If $\Sa R$ is an ordered field then $\Th(\Sa R)$ trace defines $\rcf$.
\end{proposition}

There are ordered fields $\Sa S$ such that $\Th(\Sa S)$ does not interpret $\rcf$.
Let $K$ be a bounded pseudo real closed field that is not $\mathrm{PAC}$ and is not real closed.
As $K$ is not $\mathrm{PAC}$ $K$ admits a field order $<$.
Furthermore $<$ is $K$-definable, see \cite{M-prc}.
As $K$ is not real closed every finite extension of $K$ is not real closed by the Artin-Schreier theorem.
Montenegro~\cite{samaria-imaginary} has shown that $K$ admits elimination of imaginaries.
Finally, $K$ is \'ez by \cite{ez-fields}, hence any infinite field definable in $K$ is isomorphic to a finite extension of $K$, see Fact~\ref{fact:ez}.

\begin{proof}
Suppose that $\Sa R$ is an $(2^{\aleph_0})^+$-saturated ordered field.
We show that $\Sa R$ trace defines $\rfield$.
Let $V$ be the convex hull of $\Z$ in $R$ and $\mfrak$ be the set of $\alpha \in V$ such that $|\alpha| < 1/n$ for all $n \ge 1$.
Let $\st \colon V \to \R$ be given by $\st(\alpha) = \sup\{ q \in \Q : q <\alpha \}$.
Then $\st$ is surjective by saturation, so we identify $V/\mfrak$ with $\R$.
Then $V$ is a valuation ring with maximal ideal $\mfrak$ and $\st$ is the residue map.
Let $\uptau \colon \R \to V$ be a section of $\st$.
We show that $\Sa R$ trace defines $\rfield$ via $\uptau$.
Suppose that $X$ is an $\rfield$-definable subset of $\R^n$.
By Proposition~\ref{prop:zero-def} we may suppose that $X$ is definable without parameters.
By quantifier elimination for real closed fields we may suppose that $X = \{ \alpha \in \R^n : f(\alpha) \ge 0\}$ for some $f \in \Z[x_1,\ldots,x_n]$.
Hence for any $\alpha \in \R^n$ we have $\alpha \in X$ if and only if $f(\alpha) = f(\st(\uptau(\alpha)) \ge 0$.
We have $f(\st(\beta)) = \st(f(\beta))$ for all $\beta \in V^n$, so
\begin{align*}
\alpha \in X &\Longleftrightarrow \st(f(\uptau(\alpha))) \ge 0 \\
&\Longleftrightarrow f(\uptau(\alpha)) \ge c \text{  for some  } c \in \mfrak.
\end{align*}
By saturation the downwards cofinality of $\mfrak$ is at least $\cplus$, so there is $c \in \mfrak$ such that for any $\alpha \in \R^n$ we have $\st(f(\uptau(\alpha))) \ge 0$ if and only if $f(\uptau(\alpha)) \ge c$.
Let $Y$ be the set of $\alpha \in R^n$ such that $f(\alpha) \ge c$.
Then $Y$ is $\Sa R$-definable and $X = \uptau^{-1}(Y)$.
\end{proof}

\subsection{Presburger arithmetic}
\label{section:presburger}
We give some examples of structures that are trace definable in $\Th(\Z;+,<)$.
See Section~\ref{section:l order} for background on oag's.
We first prove a more general result.

\begin{theorem}
\label{thm:regular rank}
Suppose that $(H;+,\prec)$ is an ordered abelian group which admits only finitely many definable convex subgroups.
Then $(H;+,\prec)$ is trace equivalent to $(H;+)\sqcup(\R;+,<)$.
In particular if $(H;+,\prec)$ is archimedean then $(H;+,\prec)$ is trace equivalent to $(H;+)\sqcup(\R;+,<)$.
\end{theorem}

Again we do not know if $(\R;+,<)$ can trace define a non-divisible torsion free abelian group.

\begin{proof}
It is clear that $(H;+,\prec)$ defines $(H;+)$ and $\Th(H;+,\prec)$ trace defines $(\R;+,<)$ by Proposition~\ref{prop:re-oag}.
By Lemma~\ref{lem:disjoint union} $\Th(H;+,\prec)$ trace defines $(H;+) \sqcup (\R;+,<)$.

\medskip
We show that $\Th((H;+)\sqcup(\R;+,<))$ trace defines $(H;+,\prec)$.
We first treat the case when $(H;+,\prec)$ does not admit any non-trivial definable convex subgroups (i.e. is regular).
We may suppose that $(H;+,\prec)$ is archimedean and hence let $\upchi$ be the unique-up-to-rescaling embedding $(H;+,\prec)\to(\R;+,<)$.
The proof of Proposition~\ref{prop:final Z} shows that $(H;+)\sqcup(\R;+,<)$ trace defines $(H;+,\prec)$ via the map $\uptau\colon H\to H\times\R$ given by $\uptau(\alpha)=(\alpha,\upchi(\alpha))$.

\medskip
We now address the general case.
We may suppose that $(H;+,\prec)$ is $\aleph_1$-saturated.
Let $J_0 \subseteq J_1 \subseteq \ldots \subseteq J_n$ be as provided by Fact~\ref{fact:regular rank}.
Declare $H_i = J_i/J_{i - 1}$ and let $\prec_i$ be the order on $H_i$ for all $i \in \{1,\ldots,n\}$.
Then each $(H_i;+,\prec_i)$ is regular and $(H;+,\prec)$ is isomorphic to the lexicographic product $(H_1;+,\prec_1)\times\cdots\times(H_n;+,\prec_n)$.
In particular $(H;+)$ is isomorphic to $(H_1;+)\oplus\cdots\oplus(H_n;+)$, so $(H;+)$ is trace equivalent to $(H_1;+)\sqcup\cdots\sqcup(H_n;+)$ by Lemma~\ref{lem:oplus}.
Furthermore $(H;+)\sqcup(\R;+,<)$ and $(H_1;+)\sqcup\cdots\sqcup(H_n;+)\sqcup(\R;+,<)$ are trace equivalent by Lemma~\ref{lem:disjoint union 1}.
Each $(H_i;+,\prec_i)$ is trace equivalent to $(H_i;+)\sqcup(\R;+,<)$ by the regular case.
By Lemma~\ref{lem:disjoint union} $(H_1;+)\sqcup\cdots\sqcup(H_n;+)\sqcup(\R;+,<)$ is trace equivalent to $(H_1;+,\prec_1)\sqcup\cdots\sqcup(H_n;+,\prec_n)$.
Finally, it is easy to see that the lexicographic product $(H_1;+,\prec_1)\times\cdots\times(H_n;+,\prec_n)$ is definable in $(H_1;+,\prec_1)\sqcup\cdots\sqcup(H_n;+,\prec_n)$.
\end{proof}

\begin{proposition}
\label{prop:presbur1.5}
Suppose that $(H;+,\prec)$ is an ordered abelian group and $(H;+)$ is finite rank.
Then $(H;+,\prec)$ is trace definable in Presburger arithmetic.
\end{proposition}

\begin{proof}
By Fact~\ref{fact:oag} and Theorem~\ref{thm:regular rank} $(H;+,\prec)$ is trace equivalent to $(H;+)\sqcup(\R;+,<)$.
By Proposition~\ref{prop:torsion free in Z} $\Th(\Z;+)$ trace defines $(H;+)$, so $\Th((\Z;+)\sqcup(\R;+,<))$ trace defines $(H;+)\sqcup(\R;+,<)$.
Finally, Presburger arithmetic trace defines $(\Z;+)\sqcup(\R;+,<)$ by Prop~\ref{prop:final Z}.
\end{proof}

\begin{proposition}
\label{prop:cyclic oag}
Suppose that $(G;+)$ is a finite rank ordered abelian group and $C$ is a cyclic group order on $G$.
Then $(G;+,C)$ is trace definable in Presburger arithmetic.
If $G$ is infinite and finitely generated then $(G;+,C)$ is trace equivalent to $(\Z;+,<)$.
\end{proposition}

By Fact~\ref{fact:zp} $(\Z;+,<)$ cannot interpret $(G;+,C)$ if $C$ is dense. Corollary~\ref{cor:cyclic order} shows that if $(G;+)$ is a finite rank group and $C$ is an archimedean cyclic group order on $G$ then $(G;+,C)$ does not interpret $(\Z;+,<)$.

\begin{proof}
We first show that $(G;+,C)$ is trace definable in Presburger arithmetic.
We let $(H;u,+,\prec)$ be the universal cover of $(G;+,C)$, so $(G;+,C)$ is interpretable in $(H;+,\prec)$.
It is enough to show that $(H;+,\prec)$ is trace definable in Presburger arthimetic.
We have an exact sequence $0 \to \Z \to H \to G \to 0$, hence $(H;+)$ is finite rank.
Apply Proposition~\ref{prop:presbur1.5}.

\medskip
We now suppose that $G$ is infinite and finitely generated and show that $\Th(G;+,C)$ trace defines $(\Z;+,<)$.
By Theorem~\ref{thm:regular rank} and Lemma~\ref{lem:disjoint union} it is enough to show that $\Th(G;+,C)$ trace defines both $(\Z;+)$ and $(\R;+,<)$.
Note that $(G;+)$ trace defines $(\Z;+)$ by Proposition~\ref{prop:fg-group} and 
$\Th(G;+,C)$ trace defines $(\R;+,<)$ by Proposition~\ref{prop:re-oag}.
\end{proof}

\subsection{$p$-adic examples}

We first give an example of a natural structure interpretable in $\pfield$ which is trace equivalent to $\pfield$ but does not interpret an infinite field.
For our purposes a \textbf{$p$-adically closed field}\footnote{Some authors use ``$p$-adically closed field" for fields that are elementarily equivalent to finite extensions of $\Q_p$. Our work generalizes to this setting.} is a field elementarily equivalent to $\Q_p$.

\medskip
We discuss the induced structure on the set of balls in $\Q_p$.
We let $\valp \colon \Q^\times_p \to \Z$ be the $p$-adic valuation on $\Q_p$.
Given $a \in \Q_p$ and $r \in \Z$ we let $B(a,r)$ be the ball with center $a$ and radius $r$, i.e. the set of $b \in \Q_p$ such that $\valp(a - b) > r$.
Let $\B$ be the set of balls in $\Z_p$.
We consider $\B$ to be a $\pfield$-definable set of imaginaries.
Let $\approx$ be the equivalence relation on $\Z_p \times \N_{\ge 1}$ where $(a,r) \approx (a^*,r^*)$ if and only if $(a,r) = B(a,r^*)$.
Identify $\B$ with $(\Z \times \N_{\ge 1})/\!\approx$ and let $\Sa B$ be the structure induced on $\B$ by $\pfield$.

\begin{proposition}
\label{prop:padic}
$\Sa B$ is trace equivalent to $\pfield$.
\end{proposition}

\begin{proof}
We need to show that $\Th(\Sa B)$ trace defines $\Q_p$.
By Proposition~\ref{prop:she-0} it suffices to produce $\Sa B\prec \Sa D$ such that $\Sh D$ interprets $\pfield$.
Let $\kfield$ be an $\aleph_1$-saturated elementary extension of $\pfield$ of cardinality $2^{\aleph_0}$.
Let $\Sa D'$ be the structure induced on $\B_K$ by $\kfield$.
We let $L$ be the language of $\Sa B$, recall that $L$ contains an $n$-ary relation $R_X$ for each $\pfield$-definable $X \subseteq \B^n$.
We naturally consider $L$ to be a sublanguage of the language of $\Sa D'$.
We define $\Sa D = \Sa D'\!\upharpoonright\! L$.
Note that $\Sa D$ is an $\aleph_1$-saturated elementary extension of $\Sa B$.

\medskip\noindent
We first realize $\Q_p$ as an $\kfield^{\mathrm{Sh}}$-definable set of imaginaries.
We let $\valp \colon K^\times \to \Gamma$ be the $p$-adic valuation on $K$.
Note that $\Z$ is the minimal non-trivial convex subgroup of $\Gamma$.
By Fact~\ref{fact:convex} $\Z$ is externally definable.
Let $v \colon K^\times \to \Gamma/\Z$ be the composition of $\valp$ with the quotient $\Gamma \to \Gamma/\Z$.
We equip $\Gamma/\Z$ with a group order by declaring $a + \Z \leq b + \Z$ when $a \leq b$, so $v$ is an externally definable valuation on $K$.
Let $W$ be the valuation ring of $v$ and $\mfrak_W$ be the  maximal ideal of $W$.
Then $W$ is the set of $\alpha \in K$ such that $\valp(\alpha) \geq m$ for some $m \in \Z$ and $\mfrak_W$ is the set of $\alpha \in K$ such that $\valp(\alpha) > m$ for all $m \in \Z$.
It is easy to see that for every $\alpha \in W$ there is a unique $\alpha^* \in \Q_p$ such that $\alpha - \alpha^* \in \mfrak_W$.
We identify $W/\mfrak_W$ with $\Q_p$ so that the residue map $\st \colon W \to \Q_p$ is the usual standard part map.
Then $\Q_p$ is a $K^{\mathrm{Sh}}$-definable set of imaginaries.

\medskip\noindent
The claim is a saturation exercise.

\begin{Claim*}
\label{lem:st}
Suppose that $X$ is a closed definable subset of $\Z^n_p$ and $X^*$ is the subset of $W^n$ defined by any formula defining $X$.
Then $\st(X^*)$ agrees with $X$.
\end{Claim*}

\noindent
Given $B \in \B_K$ such that $B = B(a,t)$ we let $\rad(B) = t$.
As $\rad \colon \B \to \Gamma_\geq$ is surjective and $\kfield$-definable we consider $\Gamma_\geq$ to be an imaginary sort of $\Sa B$ and $\rad$ to be a $\Sa B$-definable function.
Fix $\gamma \in \Gamma$ such that $\gamma > \N$.
Let $E=\{B \in \B_K:\rad(B) = \gamma\}$, so $E$ is $\Sa D$-definable.

\medskip\noindent
We finally show that $\Sh D$ interprets $\pring)$.
Note that for any $\alpha \in V$ and $\beta \in B(\alpha,\gamma)$ we have $\st(\alpha) = \st(\beta)$.
We define a surjection $f \colon E \to \Z_p$ by declaring $f( B(\alpha,\gamma) ) = \st(\alpha)$ for all $\alpha \in V$.
Let $\approx$ be the equivalence relation on $E$ where $B_0 \approx B_1$ if and only if $f(B_0) = f(B_1)$.
Note that for any $B_0,B_1 \in \B_K$ we have $B_0 \approx B_1$ if and only if 
$$ \{ B^* \in \B_K : \rad(B^*) \in \N , B_0 \subseteq B^* \} = \{ B^* \in \B : \rad(B^*) \in \N , B_1 \subseteq B^* \}. $$
Hence $\approx$ is $\Sh D$-definable.
Let $f \colon E^n \to \Z^n_p$ be given by 
$$ f(B_1,\ldots,B_n) = (f(B_1),\ldots,f(B_n)) \quad \text{for all} \quad B_1,\ldots,B_n \in E. $$
Suppose that $X$ is a $\pfield$-definable subset of $\Z^n_p$.
We show that $f^{-1}(X)$ is $\Sh D$-definable.
As $X$ is $\pfield$-definable it is a boolean combination of closed $\pfield$-definable subsets of $\Z^n_p$, so we may suppose that $X$ is closed.
Let $X^*$ be the subset of $V^n$ defined by any formula defining $X$.
Let $Y_0$ be the set of $B \in E$ such that $B \cap X^* \neq \emptyset$ and $Y$ be the set of $B \in E$ such that $B \approx B_0$ for some $B_0 \in Y_0$.
Observe that $Y_0$ is $\Sa D$-definable and $Y$ is $\Sh D$-definable.
Note that $Y_0$ is the set of balls of the form $B(\alpha,\gamma)$ for $\alpha \in X^*$.

\medskip\noindent
We show that $Y = f^{-1}(X)$.
Suppose $B(\alpha,\gamma) \in f^{-1}(X)$.
Then $\st(\alpha) \in X$.
We have $B(\st(\alpha),\gamma) \in Y_0$ and $B(\st(\alpha),\gamma) \approx B(\alpha,\gamma)$, so $B(\alpha,\gamma) \in Y$.
Now suppose that $B(\alpha,\gamma) \in Y$ and fix $B(\beta,\gamma) \in Y_0$ such that $B \approx B_0$.
We may suppose that $b \in X^*$.
As $X$ is closed an application of the claim shows $\st(\beta) \in X$.
Then $\st(\alpha) = \st(\beta) \in X$, so $B(\alpha,\gamma) \in f^{-1}(X)$.
\end{proof}

\begin{proposition}
\label{prop:p-adic-0}
$\Sa B$ does not interpret an infinite field.
\end{proposition}

We apply Fact~\ref{fact:obv}, due to Halevi, Hasson, and Peterzil~\cite{halevi-hasson-peterzil}.

\begin{fact}
\label{fact:obv}
Suppose that $\kfield$ is a $p$-adically closed field.
Then any infinite field interpretable in $\kfield$ is definably isomorphic to a finite extension of $\kfield$.
\end{fact}

Similar results have been proven for algebraically and real closed valued fields~\cite{metastable,hasson-peterzil}.

\medskip\noindent
We first prove a lemma.
Given a $p$-adically closed field $K$ we let $\B_K$ be the set of balls in $K$ and $\Sa B_K$ be the structure induced on $\B_K$ by $K$.

\begin{lem}
\label{lem:map}
If $\kfield$ is $p$-adically closed then a definable function $\B^m_K \to K^n$ has finite image.
\end{lem}

\begin{proof}
If $f \colon \B^m_K \to K^n$ has infinite image then there is a coordinate projection $e \colon K^m \to K$ such that $e \circ f$ has infinite image.
We suppose $m = 1$.
Recall that if $X$ is a definable subset of $K$ then $X$ is either finite or has interior.
It therefore suffices to show that the image of any definable function $\B^m_K \to K$ has empty interior.
It is enough to show that the image of any function $\B^m \to \Q_p$ has empty interior.
This holds as $\B$ is countable and every nonempty open subset of $\Q_p$ is uncountable.
\end{proof}

We now prove Proposition~\ref{prop:p-adic-0}.

\begin{proof}
Suppose that $\kfield$ is $p$-adically closed and $\Sa B_K$ interprets an infinite field $L$.
By Fact~\ref{fact:obv} we may suppose that $L$ has underlying set $K^n$, so for some $m$ there is a $\kfield$-definable surjection $\B^m_K \to K^n$, this contradicts Lemma~\ref{lem:map}.
\end{proof}

We now give an example of a structure that is trace equivalent to $\Q_p$.
I believe, but cannot prove, that this structure cannot interpret an infinite field.
Fact~\ref{fact:mari}  is due to Mariaule~\cite{Ma-adic}.

\begin{fact}
\label{fact:mari}
Suppose that $A$ is a dense finitely generated subgroup of $\Z^\times_p$ and $\Sa A$ is the structure induced on $A$ by $\Q_p$.
Then any $\Sa A$-definable subset of $A^k$ is a finite union of sets of the  form $X \cap Y$ where $X\subseteq A^k$ is $(A;\times)$-definable and $Y\subseteq\Z^k_p$ is $\Q_p$-definable.
\end{fact}

\begin{proposition}
\label{prop:mari}
Let $A$ and $\Sa A$ be as in Fact~\ref{fact:mari}.
Then $\Sa A$ is trace equivalent to $\Q_p$.
\end{proposition}

We write ``semialgebraic" for ``$\Q_p$-definable".

\begin{proof}
A straightforward variation of the proof  of Lemma~\ref{lem::} together with an application of the fact that every $\Q_p$-definable set is a boolean combination of closed $\Q_p$-definable sets shows that $\Th(\Sa A)$ trace defines $\Q_p$.
We show that $\Q_p$ trace defines $\Sa A$.
Note that $\Q_p$ interprets $(\Z;+)$, hence $\Q_p$ trace defines $(\Z;+)\sqcup\Q_p$ by Lemma~\ref{lem:disjoint union}.
By Proposition~\ref{prop:fg-group} $(\Z;+)\sqcup\Q_p$ is trace equivalent to $(A;\times)\sqcup\Q_p$, so it is enough to show that $(A;\times)\sqcup\Q_p$ trace defines $\Sa A$.
Let $\uptau$ be the injection $A \to A \times \Q_p$ given by $\uptau(\alpha) = (\alpha,\alpha)$.
We show that $(A;\times) \sqcup \Q_p$ trace defines $\Sa A$ via $\uptau$.
We let $L'$ be the expansion of $L_\mathrm{div}$ by a $k$-ary relation symbol $R_X$ defining $X \cap A^k$ for every semialgebraic $X\subseteq\Q^k_p$.
By Fact~\ref{fact:gh} $\Sa A$ admits quantifier elimination in $L'$.
We apply Proposition~\ref{prop:eq neg 1} with $L$ the language of abelian groups, $L^{*}=L'$, $\Sa O=\Sa A$, $\Sa M = (A;\times)\sqcup\Q_p$, and $\Sa P$ the abelian group $(A;\times) \oplus \Z^\times_p$.
It is enough to fix $k$ and semialgebraic $X \subseteq \Q^n_p$ and produce $(A;+) \sqcup \Q_p$-definable $Y \subseteq A \times \Q_p$, $Z \subseteq (A \times \Q_p)^k$ such that for all $a \in A, b \in A^k$:
\begin{align*}
\uptau(a) \in Y \quad &\Longleftrightarrow \quad k|a \\
\uptau(b) \in Z \quad &\Longleftrightarrow \quad b \in X.
\end{align*}
Let $Y$ be the set of $(a,b) \in A \times \R$ such that $k$ divides $a$ and $Z = A^k \times X$.
\end{proof}

We finally give a $p$-adic analogue of the results of Section~\ref{section:presburger}.
I don't know if the structures below are trace definable in Presburger arithmetic, I suspect they are not.
Given a prime $p$ let $\prec_p$ be the partial order on $\Z$ where $\alpha \prec_p \beta$ when the $p$-adic valuation of $\alpha$ is strictly less than the $p$-adic valuation of $\beta$.
Fact~\ref{fact:alouf-elbee} is due to Alouf and d'Elb\'{e}e~\cite{AldE}.

\begin{fact}
\label{fact:alouf-elbee}
Let $P$ be a set of primes and $L_\mathrm{div-val}$ be the expansion of $L_\mathrm{div}$ by $(\prec_p : p \in P)$.
Then $(\Z;+,(\prec_p)_{p \in P})$ admits quantifier elimination in $L_\mathrm{div-val}$ and is $\nip$.
\end{fact}

We prove Proposition~\ref{prop:p adic integer}.

\begin{proposition}
\label{prop:p adic integer}
Let $p_1,\ldots,p_n$ be primes and let $\prec_i$ be $\prec_{p_i}$ for all $i \in \{1,\ldots,n\}$.
Then $(\Z;+,\prec_1,\ldots,\prec_n)$ is trace equivalent to $(\Z;+) \sqcup (\Z_{p_1};+,\prec_1)\sqcup\cdots\sqcup(\Z_{p_n};+\prec_n)$.
\end{proposition}

In particular $(\Z;+,\prec_p)$ is trace equivalent to $(\Z;+)\sqcup(\Z_p;+,\prec_p)$.

\begin{proof}
We  show that $(\Z;+,\prec_1,\ldots,\prec_n)$ trace defines $(\Z;+) \sqcup (\Z_{p_1};+,\prec_1)\sqcup\cdots\sqcup(\Z_{p_n};+,\prec_n)$.
By Lemma~\ref{lem:disjoint union} it is enough to fix a prime $p$ and show that $\Th(\Z;+,\prec_p)$ trace defines $(\Z_p;+,\prec_p)$.
Let $\Sa Z$ be an $\aleph_1$-saturated elementary extension of $(\Z;+,\prec_p)$.
By Fact~\ref{fact:alouf-elbee} and Proposition~\ref{prop:she-0} it is enough to show that $\Sh Z$ interprets $(\Z_p;+,\prec_p)$.
Let $\st \colon Z \to \Z_p$ be the usual $p$-adic standard part map.
Note that $\st$ is surjective by $\aleph_1$-saturation and definibility of the $p$-adic topology on $Z$.
Let $\mfrak$ be the set of $\alpha \in Z$ such that $m \prec \alpha$ for all $m \in \Z$.
Then $\mfrak$ is the kernal of $\st$, so $\st$ induces an isomorphism $Z/\mfrak \to \Z_p$.
We identify $Z/\mfrak$ with $\Z_p$.
Note that $\mfrak$ is $\Sh Z$-definable, so we regard $\Z_p$ as an imaginary sort of $\Sh Z$.
It is easy to that $\Sh Z$ defines the addition and partial order on $\Z_p$.

\medskip
We show that $(\Z;+) \sqcup (\Z_{p_1};+,\prec_1)\sqcup\cdots\sqcup(\Z_{p_n};+,\prec_n)$ trace defines $(\Z;+,\prec_1,\ldots,\prec_n)$.
Let 
\[
\uptau\colon\Z \to \Z \times \Z_{p_1}\times\cdots\times\Z_{p_n} \quad\text{be given by}\quad \uptau(m) = (m,m,\ldots,m) \quad \text{for all  } m\in\Z.
\]
An adaptation of the proof of Proposition~\ref{prop:final Z} and an application of Fact~\ref{fact:alouf-elbee} shows that $(\Z;+) \sqcup (\Z_{p_1};+,\prec_1)\sqcup\cdots\sqcup(\Z_{p_n};+,\prec_n)$ trace defines $(\Z;+,\prec_1,\ldots,\prec_p)$ via $\uptau$.
\end{proof}

Proposition~\ref{prop:p adic integer} shows that the field $\Q_p$ trace defines $(\Z;+,\prec_p)$.
This follows from Lemma~\ref{lem:disjoint union} as $\Q_p$ defines $(\Z_p;+,\prec_p)$ and interprets $(\Z;+)$ (as the value group).
I don't think that $\Q_p$ can interpret $(\Z;+,\prec_p)$, but I don't think that this follows from known results\footnote{We would need to know more about interpretable groups in $\Q_p$.}.

\subsection{The integers and the reals}
\label{section:Z and R}
In this section we give three examples of structures that are trace equivalent to $(\Z;+) \sqcup \rfield$.
I do not know if $\rcf$ trace defines $\Th(\Z;+)$.

\medskip
Fact~\ref{fact:gh} is due to van den Dries and G\"unaydin, \cite[Theorem~7.2]{vddG}.

\begin{fact}
\label{fact:gh}
Suppose that $A$ is a dense finite rank subgroup of $(\R_{>0};\times)$ and let $\Sa A$ be the structure induced on $A$ by $\rfield$.
Then every $\Sa A$-definable subset of $A^k$ is a finite union of sets of the form $X \cap Y$ where $X \subseteq A^k$ is $(A;\times)$-definable and $Y \subseteq \R^k$ is semialgebraic.
\end{fact}

\begin{proposition}
\label{prop:mordell lang}
Suppose that $A$ is a dense finite rank subgroup of $(\R_{>0};\times)$ and $\Sa A$ is the structure induced on $A$ by $\rfield$.
Then $\Sa A$ is trace equivalent to $(A;\times) \sqcup \rfield$.
Hence $\Sa A$ is trace definable in $(\Z;+) \sqcup \rfield$ and if $A$ is finitely generated then $\Sa A$ is trace equivalent to $(\Z;+)\sqcup\rfield$.
\end{proposition}

By Proposition~\ref{prop:two sorts} and Corollary~\ref{cor:rama-2} any expansion of an archimedean oag interpretable in $(\Z;+)\sqcup\rfield$ is o-minimal, hence $(\Z;+)\sqcup\rfield$ does not interpret $\Sa A$.

\begin{proof}
The other claims follows from the first by Proposition~\ref{prop:fg-group}, Proposition~\ref{prop:finite rank}, and Lemma~\ref{lem:disjoint union 1}.
We prove the first claim.
Clearly $\Sa A$ defines $(A;\times)$.
By \cite{GH-Dependent} $\Sa A$ is $\nip$, so $\Th(\Sa A)$ trace defines $\rfield$ by Lemma~\ref{lem::}.
By Lemma~\ref{lem:disjoint union} $\Th(\Sa A)$ trace defines $(A;\times) \sqcup \rfield$.
Finally, the proof of Proposition~\ref{prop:mari} shows that $\Sa A$ is trace definable in $(A;\times) \sqcup \rfield$.

\end{proof}

Fact~\ref{fact:bndd Z} is proven in \cite{big-nip}.

\begin{fact}
\label{fact:bndd Z}
Suppose that $\Cal B$ is a collection of bounded subsets of Euclidean space such that $(\R;+,<,\Cal B)$ is o-minimal.
Then a subset $X$ of $\R^n$ is definable in $(\R;+,<,\Cal B,\Z)$ if and only if it is a finite union of sets of the form $\bigcup_{a \in X} a + Y$ for $(\Z;+,<)$-definable $X \subseteq \R^n$ and $(\R;+,<,\Cal B)$-definable $Y \subseteq [0,1)^n$.
\end{fact}

If $\Sa R$ is a strongly dependent expansion of $(\R;+,<)$ by closed sets then $\Sa R$ is either o-minimal or of the form $(\R;+,<,\Cal B,\alpha\Z)$ for $\alpha \in \R$ and a collection $\Cal B$ of bounded subsets of Euclidean space such that $(\R;+,<,\Cal B)$ is o-minimal~\cite{big-nip}.

\begin{proposition}
\label{prop:expansion by Z}
Suppose that $\Cal B$ is a collection of bounded subsets of Euclidean space with $(\R;+,<,\Cal B)$ o-minimal.
Then $(\R;+,<,\Cal B,\Z)$ is trace equivalent to $(\Z;+) \sqcup (\R;+,<,\Cal B)$.
\end{proposition}

By the remarks after Proposition~\ref{prop:mordell lang} $(\Z;+)\sqcup(\R;+,<,\Cal B)$ does not interpret $(\R;+,<,\Cal B,\Z)$.

\begin{proof}
By Lemma~\ref{lem:disjoint union 1} $(\R;+,<,\Cal B,\Z)$ trace defines $(\Z;+) \sqcup (\R;+,<,\Cal B)$.
Follow the proof of Proposition~\ref{prop:final Z} and apply Fact~\ref{fact:bndd Z} to show that $(\R;+,<,\Cal B,\Z)$ is definable in the disjoint union $(\R;+,<,\Cal B) \sqcup (\Z;+,<)$.
By Proposition~\ref{prop:final Z} and Lemma~\ref{lem:disjoint union 1} $(\Z;+,<) \sqcup (\R;+,<,\Cal B)$ is trace equivalent to $(\Z;+) \sqcup (\R;+,<,\Cal B)$.
\end{proof}

Let $\Sa R_\mathrm{bnd}$ be the expansion of $(\R;+,<)$ by all bounded semialgebraic sets.
If $X$ is a bounded subset of $\R^n$ which is semialgebraic but not semilinear then $(\R;+,<,X)$ is interdefinable with $\Sa R_\mathrm{bnd}$ by a theorem of of Marker, Pillay, and Peterzil~\cite{MPP}.
By definition $\Sa R_\mathrm{bnd}$ is a reduct  of $\rfield$ and it is easy to see that $\Sa R_\mathrm{bnd}$ defines an isomorphic copy of $\rfield$.
Hence $\Sa R_\mathrm{bnd}$ is trace equivalent to $\rfield$.
Proposition~\ref{prop:expansion by Z} shows that $(\Sa R_\mathrm{bnd},\Z)$ is trace equivalent to $(\Z;+) \sqcup \Sa R_\mathrm{bnd}$, hence $(\Sa R_\mathrm{bnd},\Z)$ is trace equivalent to $(\Z;+)\sqcup\rfield$.

\medskip
We mention another example of Proposition~\ref{prop:example}.
It is easy to see that $(\R;+,<,\sin)$ is interdefinable with $(\R;+,<,\sin\!|_{[0,2\pi]},\pi\Z)$ and $(\R;+,<,\sin\!|_{[0,2\pi]})$ is o-minimal as $\sin\!|_{[0,2\pi]}$ is definable in $\R_\mathrm{an}$.
Hence $(\R;+,<,\sin)$ is trace equivalent to $(\Z;+) \sqcup (\R;+,<,\sin\!|_{[0,2\pi]})$.

\medskip
We now let $\unisltwo$ be the universal cover of $\sltwo$, considered as a group.

\begin{proposition}
\label{prop:sl2}
$\unisltwo$ is mutually interpretable with $(\Z;+) \sqcup \rfield$.
\end{proposition}

\begin{proof}
Let $K$ be the center of $\unisltwo$.
Then $K$ is isomorphic to $(\Z;+)$ and $\unisltwo/K$ is $\psltwo$.
By \cite{o-minimal-simple} $\psltwo$ is bi-interpretable with $\rfield$.
Note that $\unisltwo$ defines $K$ and hence interprets $\psltwo$.
Hence $\unisltwo$ interprets $(\Z;+)\sqcup\rfield$.
By \cite[Example 1.1]{two-disjoint-sorts} $(\Z;+)\sqcup\rfield$ interprets $\unisltwo$. 
\end{proof}

\subsection{Examples of trace maximal structures}
We first discuss some standard examples of structures that do not satisfy any positive classification-theoretic properties.
Standard coding arguments show that $(\Z;+,\times)$ satisfies Lemma~\ref{lem:max}.4 with $A = \Z$.
Thus $(\Z;+,\times)$ is trace maximal.
A common example of a tame structure that does not satisfy any positive classification theoretic property is an infinite boolean algebra.
We show infinite boolean algebras are trace maximal, this gives an example of an $\aleph_0$-categorical trace maximal structure.

\begin{proposition}
\label{prop:bool}
Every infinite boolean algebra is trace maximal.
\end{proposition}


\noindent
We say that a subset $A$ of a boolean algebra is \textbf{independent} if the subalgebra generated by $A$ is free, equivalently for any $\alpha_1,\ldots,\alpha_k,\beta_1,\ldots,\beta_\ell \in A$ with $\alpha_i \ne \beta_j$ for all $i,j$ we have 
\[(\alpha_1 \land \cdots \land \alpha_k) \land ( \neg \beta_1 \land \cdots \land \neg \beta_\ell) \ne 0\]

\begin{lemma}
\label{lem:boolean}
Suppose that $\Sa B$ is a boolean algebra, $A$ is an independent subset of $B$, and $(\beta_i : i \le k)$, $(\alpha^i_j : i \in \{1,\ldots,n\}, 1 \le j \le k)$ are elements of $A$ such that for all $i \in \{1,\ldots,n\}$:
\begin{enumerate}
\item $|\{ \beta_1,\ldots,\beta_k\}| = k = |\{\alpha^i_1,\ldots,\alpha^i_k\}|$, and
\item $\{ \beta_1,\ldots,\beta_k \} \ne \{\alpha^i_1,\ldots,\alpha^i_k\}$. 
\end{enumerate}
Then $(\beta_1 \land \cdots \land \beta_k) \nleqslant \bigvee_{i = 1}^{n} (\alpha^i_1 \land \cdots \land \alpha^i_k)$.
\end{lemma}

\begin{proof}
Note that for each $i \in \{1,\ldots,n\}$ there is $j(i)$ such that $\alpha^i_{j(i)} \notin \{\beta_1,\ldots,\beta_k \}$.
By independence we have
$$ \gamma := (\beta_1\land \cdots \land \beta_k) \land (\neg \alpha^1_{j(1)} \land \cdots \land \neg \alpha^n_{j(n)}) \ne 0. $$
Then $\gamma \le (\beta_1\land \cdots \land \beta_k)$ and $\gamma \land \bigvee_{i = 1}^{n} (\alpha^i_1 \land \cdots \land \alpha^i_k) = 0$.
\end{proof}

\noindent
We now prove Proposition~\ref{prop:bool}.

\begin{proof}
We first make some basic remarks on boolean algebras which can probably be found in any standard reference\footnote{Or easily proven via Stone duality.}.
Any finite boolean algebra is isomorphic to the boolean algebra of subsets of $\{1,\ldots,n\}$ for some $n$.
It follows that if $\Sa B,\Sa B^{*}$ are finite boolean algebras with $|B| \le |B^{*}|$ then there is an embedding $\Sa B \to \Sa B^{*}$.
The free boolean algebra on $n$ generators has cardinality $2^{2^n}$ and clearly has an independent subset of cardinality $n$.
It follows that any finite boolean algebra with $\ge 2^{2^n}$ elements contains an independent subset of cardinality $n$.
Recall also that a finitely generated boolean algebra is finite.
It follows that any infinite boolean algebra contains an independent subset of cardinality $n$ for all $n$.

\medskip\noindent
Let $\Sa B = (B;\land, \vee, < ,0,1)$ be an infinite boolean algebra.
We may suppose $\Sa B$ is $\aleph_1$-saturated.
Applying the previous paragraph and saturation we fix a countably infinite independent subset $A$ of $B$.
By Lemma~\ref{lem:max-1} it suffices to suppose $E$ is a $k$-hypergraph on $A$ and produce  definable $Y \subseteq B^k$ such that $ E(a_1,\ldots,a_k) \Longleftrightarrow (a_1,\ldots,a_k) \in Y$ for all $a_1,\ldots,a_k \in A$.
Let $f \colon A^k \to B$ be given by $f(a_1,\ldots,a_k) = a_1 \land \ldots \land a_k$.
Let $D$ be the set of $(a_1,\ldots,a_k) \in A^k$ such that $|\{a_1,\ldots,a_k\}| = k$.
Let $(a^i : i < \upomega)$ be an enumeration of $E$.
Lemma~\ref{lem:boolean} shows that if $b \in D$ and $\neg E(b)$ then $f(b) \nleqslant \bigvee_{i = 1}^{n} f(a^i)$ for all $n$.
Hence for any $n$ there is a $c \in B$ such that $f(a^i) \le c$ for all $i \leqslant n$ and $f(b) \nleqslant c$ for all $b \in D$ such that $\neg E(b)$.
By saturation there is $c \in B$ such that $f(a^i) \leqslant c$ for all $i < \upomega$ and $f(b) \nleqslant c$ for all $b \in D$ with $\neg E(b)$.
Therefore we let $Y$ be the set of $a \in B^k$ such that $a \in D$ and $f(a) \leqslant c$.
\end{proof}

Corollary~\ref{cor:bool} follows from Proposition~\ref{prop:cat max}.

\begin{corollary}
\label{cor:bool}
The countable atomless boolean algebra trace defines any countable structure in a countable language.
\end{corollary}

\subsubsection{Trace maximal simple theories}

We first describe two simple simple trace maximal theories.
First, suppose $L$ is a relational language and $L$ contains an $n$-ary relation for all $n \ge 2$.
Let $\O_L$ be the empty $L$-theory and $\O^*_L$ be the model companion of $\O_L$, which exists by a theorem of Winkler~\cite{Winkler}.
By \cite{KRExp} $\O^*_L$ is simple.
We show that $\O^*_L$ is trace maximal.

\medskip
Suppose that $\Sa M \models \O^*_L$ is $\aleph_1$-saturated and let $A \subseteq M$ be countably infinite.
It is easy to see that for each $X \subseteq A^k$ and $(k + 1)$-ary $R \in L$ there is $c \in M$ such that for all $a_1,\ldots,a_k \in A$ we have $(a_1,\ldots,a_k) \in X \Longleftrightarrow \Sa M \models R_{k + 1}(a_1,\ldots,a_k,c)$.
Apply Lemma~\ref{lem:max}.

\medskip\noindent
We now describe a symmetric analogue of $\O^*_L$.
For each $k \ge 2$ let $E_k$ be a $k$-ary relation and let $L = \{ E_k : k < \upomega \}$.
Let $T$ be the $L$-theory such that $(M;(E_k)_{k \ge 2})$ if and only if each $E_k$ is a $k$-hypergraph on $M$.
Then $T$ has a model companion $T^*$ and $T^*$ is simple.
Suppose that $\Sa M \models T$ is $\aleph_1$-saturated and $A \subseteq M$ is countably infinite.
We show that $\Sa M$ is trace maximal.
It is easy to see that for any $k$-hypergraph $E$ on $M$ there is $c \in M$ such that we have $E(a_1,\ldots,a_k) \Longleftrightarrow \Sa M \models E_{k + 1}(a_1,\ldots,a_k,c)$ for all $a_1,\ldots,a_k \in A$.
Apply Lemma~\ref{lem:max-1}.

\medskip\noindent
We now discuss expansions by generic predicates.
Suppose that $P$ is a unary relation not in $L$ and let $T_P$ be $T$ considered as an $L \cup \{P\}$ theory.
After possibly Morleyizing we suppose that $T$ is model complete.
If $T$ eliminates $\exists^\infty$ then $T_P$ has a model companion $T^*_P$ and $T^*_P$ is simple when $T$ is stable~\cite{Cha-Pi,DMS-generic}.

\begin{proposition}
\label{prop:generic-predicate}
Suppose that $A$ is an infinite abelian group, $T$ expands $\Th(A)$, and $T$ eliminates $\exists^\infty$.
Then $T^*_P$ is trace maximal.
\end{proposition}

In particular the expansion of an algebraically closed field or $(\Z;+)$ by a generic unary predicate is simple and trace maximal.
The latter structure is natural.
Suppose that $\upmu$ is the Bernoulli coin-tossing probability measure on the power set $\mathcal{P}(\Z)$.
One can show that 
\[
\upmu \left(\{ X \subseteq \Z : (\Z;+,X) \models \Th(\Z;+)^*_P\} \right) = 1.
\]
Informally, the expansion of $(\Z;+)$ by a random subset of $\Z$ is trace maximal.

\begin{proof}
We may suppose that $\Sa A\models T$ is $\aleph_1$-saturated.
Let $X$ be an infinite $\Z$-linearly independent subset of $V$.
So if $\alpha_1,\ldots,\alpha_k \in A$ and $\beta_1,\ldots,\beta_k \in A$ are distinct then we have
\[
\alpha_1 + \cdots + \alpha_k = \beta_1 + \cdots + \beta_k \quad\Longleftrightarrow\quad \{\alpha_1,\ldots,\alpha_k\} = \{\beta_1,\ldots,\beta_k\}.
\]
Let $X[k] = \{ \alpha_1 + \cdots + \alpha_k : \alpha_1,\ldots,\alpha_k \in X, \alpha_i\ne\alpha_j \text{ when } i \ne j\}$ for each $k \ge 1$.
Let $\Cal H_k$ be the collection of $k$-hypergraphs on $X$ for each $k \ge 1$.
For each $E \in \Cal H_k$ on $X$ fix $\beta_E \in A$ such that $(X[k] + \beta_E : k \ge 1, E \in \Cal H_k)$ is a collection of pairwise disjoint sets.
Let $P$ be a unary predicate on $A$ such that
\[
E(\alpha_1,\ldots,\alpha_k) \quad \Longleftrightarrow \quad P(\alpha_1 + \cdots + \alpha_k + \beta_E)
\] for all $k \ge 1$, $E \in \Cal H_k$, and distinct $\alpha_1,\ldots,\alpha_k \in X$.
Then $(\Sa A,P)$ extends to a $T^*_P$-model, so we may suppose that $(\Sa V,P) \models T^*_P$.
Lemma~\ref{lem:max-1} shows that $(\Sa V,P)$ is trace maximal.
\end{proof}

In the next section we show that another important example of a simple structure, psuedofinite fields, are trace maximal.
As a corollary any completion of $\mathrm{ACFA}$ is trace maximal.

\subsubsection{PAC and PRC fields}
\label{section:pac}
I do not know of an $\mathrm{IP}$ field which is not trace maximal.
Chernikov and Hempel conjectured that a field which is $k$-independent for some $k \ge 2$ is $k$-independent for all $k$~\cite[Conjecture 1]{hempel-chernkov}.
In this section we show that many of the main examples of model-theoretically tame $\mathrm{IP}$ fields are trace maximal.

\medskip
Let $K$ be a field.
Then $K$ is $\mathrm{PAC}$ if $K$ is existentially closed in any purely transendental field extension of $K$ and $K$ is $\prc$ if $K$ is existentially closed in any formally real purely transendental field extension of $K$.

\begin{proposition}
\label{prop:pseudofinite}
Let $K$ be a field and suppose that one of the following holds:
\begin{enumerate}
\item $K$ is $\mathrm{PAC}$ and not separably closed, or
\item $K$ is pseudo real closed and not real closed.
\end{enumerate}
Then $K$ is trace maximal.
\end{proposition}

\noindent
Psuedofinite fields and infinite algebraic extensions of finite fields are $\mathrm{PAC}$.
It follows that pseudofinite fields and  infinite proper subfields of $\mathbb{F}^\mathrm{alg}_p$ are trace maximal.
These fields are bounded $\mathrm{PAC}$ and hence simple, see \cite{Chatzidakis}.

\medskip\noindent
We make use of Duret's proof~\cite{duret} that a non-separably closed $\mathrm{PAC}$ field is $\mathrm{IP}$.
Given a field $K$ of characteristic $p \ne 0$ we let $\wp \colon K \to K$ be the Artin-Schreier map $\wp(x) = x^p - x$.
Fact~\ref{fact:macintyre} was essentially proven by Macintryre~\cite{Macintyre-omegastable}.

\begin{fact}
\label{fact:macintyre}
Suppose that $K$ is not separably closed.
Then there is a finite extension $F$ of $K$ such that one of the following holds:
\begin{enumerate}
\item the Artin-Schreier map $\wp \colon F \to F$ is not surjective, or
\item there is a prime $p \ne \chara(K)$ such that $F$ contains a primitive $p$th root of unity and the $p$th power map $F^\times \to F^\times$ is not surjective.
\end{enumerate}
\end{fact}

\noindent
Fact~\ref{fact:duret} is~\cite[Lemma~6.2]{duret}.

\begin{fact}
\label{fact:duret}
Suppose that $K$ is a $\mathrm{PAC}$ field, $p \ne \mathrm{Char}(K)$ is a prime, $K$ contains a primitive $p$th root of unity, and the $p$th power map $K^\times \to K^\times$ is not surjective.
Let $A,B$ be disjoint finite subsets of $K$.
Then there is $c \in K$ such that $c + a$ is a $p$th power for every $a \in A$ and $c + b$ is not a $p$th power for every $b \in B$.
\end{fact}

\noindent
Fact~\ref{fact:duret-1} is~\cite[Lemma~6.2]{duret}.

\begin{fact}
\label{fact:duret-1}
Suppose that $K$ is a $\mathrm{PAC}$ field of characteristic $p \ne 0$.
Suppose that the Artin-Schreier map $\wp \colon K \to K$ is not surjective and $A,B$ are disjoint finite subsets of $K$ such that $A \cup B$ is linearly independent over $\F_p$.
Then there is $c \in K$ such that $ac$ is in the image of $\wp \colon K \to K$ for any $a \in A$ and $bc$ is not in the image of $\wp \colon K \to K$ for any $b \in B$.
\end{fact}

\noindent
We now prove Proposition~\ref{prop:pseudofinite}.
Let $\imag^2 = -1$.

\begin{proof}
Suppose that $K$ is $\mathrm{PRC}$ and not real closed.
It is immediate from the definitions that $K(\imag)$ is $\mathrm{PAC}$.
By the Artin-Schreier theorem $K(\imag)$ is not algebraically closed, hence $K(\imag)$ is not separably closed as $\chara(K) = 0$.
Hence the $\mathrm{PRC}$ follows from the $\mathrm{PAC}$ case.

\medskip
We suppose that $K$ is $\mathrm{PAC}$ and not separably closed.
We may also suppose that $K$ is $\aleph_1$-saturated.
By Fact~\ref{fact:macintyre} there is a finite field extension $F$ of $K$ such that:
\begin{enumerate}
\item $\chara(K) \ne 0$ and the Artin-Schreier map $\wp \colon F \to F$ is not surjective, or
\item there is a prime $q \ne \chara(K)$ such that $F$ contains a primitive $q$th root of unity and the $q$th power map $F^\times \to F^\times$ is not surjective.
\end{enumerate}
A finite extension of a $\mathrm{PAC}$ field is $\mathrm{PAC}$, so $F$ is $\mathrm{PAC}$.
If $F$ is separably closed then $K$ is either real closed or separably closed.
A real closed field is not $\mathrm{PAC}$, so $F$ is not separably closed.
Recall that $F$ is interpretable in $K$ so $F$ is $\aleph_1$-saturated.
It is enough to show that $F$ is trace maximal.
Suppose $(2)$ and fix a relevant prime $q$.
Let $A$ be an infinite subset of $K$ and $t$ be an element of $K$ which is not in the algebraic closure of $\Q(A)$.
Let $f \colon A^k \to K$ be given by $f(\alpha_0,\ldots,\alpha_{k - 1}) = t^k + \alpha_{k - 1} t^{k - 1} + \cdots + \alpha_1 t + \alpha_0$.
Let $X$ be a subset of $A^k$.
By injectivity of $f$, Fact~\ref{fact:duret}, and saturation there is $\gamma \in K$ such that $\alpha \in X \Longleftrightarrow F \models \exists x(x^q = f(\alpha) + \gamma)$ for all $\alpha \in A^k$.
By Lemma~\ref{lem:max} $F$ is trace maximal.

\medskip\noindent
We now suppose $(1)$.
Let $A$ be an infinite subset of $F$ which is algebraically independent over $\F_p$.
By Lemma~\ref{lem:max-1} it is enough to suppose that $E$ is a $k$-hypergraph on $A$ and produce definable $Y \subseteq F^k$ such that $E(a_1,\ldots,a_k) \Longleftrightarrow (a_1,\ldots,a_k) \in Y$ for all $a_1,\ldots,a_k \in A$.
Let $D$ be the set of $(a_1,\ldots,a_k) \in A^k$ such that $|\{ a_1,\ldots,a_k \}| = k$ and $f \colon D \to A$ be given by $f(a_1,\ldots,a_k) = a_1 a_2 \ldots a_k$.
An application of algebraic independence shows that $f(a) = f(b) \Longleftrightarrow \{a_1,\ldots,a_k\} = \{b_1,\ldots,b_k\}$ for all $a = (a_1,\ldots,a_k), b = (b_1,\ldots,b_k) \in D$, and that $\{ f(a) : a \in D \}$ is linearly independent over $\F_p$.
Applying Fact~\ref{fact:duret-1} and saturation we obtain $c \in F$ such that $E(a) \Longleftrightarrow F \models \exists x( x^p - x = f(a)c )$ for all $a \in D$.
\end{proof}

\begin{corollary}
\label{cor:acfa}
Suppose $(K,\updelta)\models\mathrm{ACFA}$.
Then $(K,\updelta)$ is trace maximal.
\end{corollary}

\begin{proof}
By \cite[Thm~7]{macintyre-generic} $\{\alpha \in K : \updelta(\alpha
) = \alpha\}$ is a pseudofinite subfield of $K$.
Apply Prop~\ref{prop:pseudofinite}.
\end{proof}

A field extension $L/K$ is \textbf{regular} if every element of $L$ is transcendental over $K$.
Duret showed that if a field $L$ has a $\pac$ subfield such that $L/K$ is regular then $L$ is $\mathrm{IP}$~\cite{duret}.

\begin{corollary}
\label{cor:duret ext}
Suppose that $K$ is a non-separably closed $\pac$ field and $L$ is a regular field extension of $K$.
Then $L$ is trace maximal.
If $L$ is positive characteristic and the algebraic closure of the prime subfield is infinite but not algebraically closed then $L$ is trace maximal.
\end{corollary}

This would be much more interesting if one could show that a non-trace maximal field is Artin-Schreier closed, as it would then follow that a non-trace maximal field of positive characteristic contains the algebraic closure of its prime subfield.
It would follow that any infinite field of characteristic $p$ trace defines the algebraic closure of the field with $p$ elements.

\begin{proof}
The second claim follows from the first claim and the fact that an infinite algebraic extension of a finite field is $\pac$.
As $K$ is $\pac$ $K$ is existentially closed in $L$.
The formulas witnessing trace maximality in $K$ produced in the proof of Proposition~\ref{prop:pseudofinite} are existential, so they also witness trace maximality in $L$.
\end{proof}

As mentioned above, it is tempting to conjecture that every $\mathrm{IP}$ field is trace maximal.
This would imply that if $\mathbf{P}$ is a non-trivial property of structures and $\mathbf{P}$ is preserved under trace definibility then any field with $\mathbf{P}$ is $\nip$.
There are expansions of fields which are $\mathrm{IP}$ and not trace maximal.
For example one can let $K$ be an uncountable algebraically closed field, let $A$ be an infinite algebraically independent subset of $K$, and let $E$ be a copy of the Erd\H{o}s-Rado graph with vertex set $A$.
It should be easy to show that $(K;A,E)$ is $\mathrm{IP}$ and $2$-dependent\footnote{I remember seeing this example somewhere, but I forgot where and I couldn't find it again.}.
I think that the conjecture on pure fields is probably false, but I think that it might hold (and might even be provable) for large fields.
(We can prove the stable fields conjecture for large fields \cite{1stpaper}, so this may not be a pipe dream.)

\section{Dichotomies, proof of Theorems~\ref{thm:tri} and \ref{thm:rigid}}
\label{section:dicho}

We first discuss a combinatorial property that rules out trace definibility of infinite fields.
We use it to show that ordered vector spaces do not trace define infinite fields.

\subsection{Near linear Zarankiewicz bounds}
\label{section:almost-lin}
Let $(V,W;E)$ be a bipartite graph.
We say $(V,W;E)$ is \textit{not} $K_{kk}$-free if there are $V^* \subseteq V$, $W^* \subseteq W$ with $|V^*| = k = |W^*|$ and $V^* \times W^* \subseteq E$.
We say that $(V,W;E)$ has \textbf{near linear Zarankiewicz bounds} if for any $k$ and $\varepsilon \in \R_{>0}$ there is $\lambda \in \R_{>0}$ such that if $V^* \subseteq V$, $W^* \subseteq W$ are finite and $(V^*,W^*;E \cap [V^*\times W^*])$ is $K_{kk}$-free then $|E \cap (V^* \times W^*)| \le \lambda |V^*\cup W^*|^{1 + \varepsilon}$.
Then $\Sa M$ has near linear Zarankiewicz bounds if every $\Sa M$-definable bipartite graph has near linear Zarankiewicz bounds and $T$ has near linear Zarankiewicz bounds if every (equivalently: some) $\Sa M \models T$ has near linear Zarankiewicz bounds.
The name\footnote{Please come up with a better one.} comes from Zarankiewicz's problem, see \cite{zaran}.

\medskip\noindent
Lemma~\ref{lem:blank-2} is clear from the definition.

\begin{lemma}
\label{lem:blank-2}
Suppose that $T$ has near linear Zarankiewicz bounds and $T^*$ is trace definable in $T$.
Then $T^*$ has near linear Zarankiewicz bounds.
\end{lemma}

\noindent
The first claim of Fact~\ref{fact:blank-0} shows that this notion is non-trivial.

\begin{fact}
\label{fact:blank-0}
The generic countable bipartite graph does not have near linear Zarankiewicz bounds.
A theory with near linear Zarankiewicz bounds is $\nip$.
\end{fact}

\noindent
The first claim follows from a probabilistic construction given in \cite{erdos-zaran} as every finite bipartite graph is a substructure of the generic countable bipartite graph.
The second claim follows from the first claim and Proposition~\ref{prop:trace-0}.

\begin{proposition}
\label{prop:field-blank}
An infinite field does not have near linear Zarankiewicz bounds.
\end{proposition}

\noindent
Proposition~\ref{prop:field-blank} follows from Lemma~\ref{lem:fp} and some standard incidence geometry.

\begin{lemma}
\label{lem:fp}
Suppose that $K$ is an infinite $\nip$ field of characteristic $p > 0$.
Then $K$ trace defines the algebraic closure of the field with $p$ elements.
\end{lemma}

Let $\F^\alg_p$ be the algebraic closure of the field with $p$ elements.

\begin{proof}
As $K$ is $\nip$ it follows from Kaplan, Scanlon, and Wagner~\cite[Corollary~4.5]{Kaplan2011} that there is a field embedding $\uptau \colon \F^\alg_p \to K$.
Proposition~\ref{prop:qe-trace} and quantifier elimination for algebraically closed fields shows that $K$ trace defines $\F^\alg_p$ via $\uptau$.
\end{proof}

\noindent
We now prove Proposition~\ref{prop:field-blank}.

\begin{proof}
We let $K$ be an infinite field.
Let $V = K^2$, $W$ be the set of lines in $K^2$, and $E$ be the set of $(p,\ell) \in V \times W$ such that $p$ is on $\ell$.
Note that $(V,W;E)$ is $K_{22}$-free.

\medskip\noindent
We first treat the case when $K$ is of characteristic zero.
We just need to recall the usual witness for sharpness in the lower bounds of Szemeredi-Trotter.
We fix $n \ge 1$, declare $V^* = \{1,\ldots,n\} \times \{1,\ldots,2n^2\}$ and let $W^*$ be the set of lines with slope in $\{1,\ldots,n\}$ and $y$-intercept in $\{1,\ldots,n^2\}$.
Then $|V^* + W^*| = 3n^3$ and $|E \cap (V^* \times W^*)| = n^4$.
Hence $(V,W;E)$ does not have near linear Zarankiewicz bounds.

\medskip\noindent
Now suppose that $K$ has characteristic $p > 0$.
By Fact~\ref{fact:blank-0} we may suppose that $K$ is $\nip$.
By Lemmas~\ref{lem:blank-2} and \ref{lem:fp} we may suppose that $K$ is the algebraic closure of the field with $p$ elements.
Fix $n \ge 1$ and let $L$ be the subfield of $K$ with $q = p^n$ elements.
Let $V^* = L^2$ and $W^*$ be the set of non-vertical lines between elements of $L^2$.
Let $|V^*| = q^2 = |W^*|$ and as every $\ell \in W^*$ contains $q$ points in $L^2$ we have $|E \cap (V^* \times W^*)| = q^3$.
Hence $(V,W;E)$ does not have near linear Zarankiewicz bounds.
\end{proof}

\noindent
You might wonder if we really need Lemma~\ref{lem:fp} to prove Proposition~\ref{prop:field-blank}.
One might suspect that the point-line bipartite graph $(V,W;E)$ discussed above does not have near linear Zarankiewicz bounds over an arbitrary infinite field $K$.
This appears to be open\footnote{More precisely I asked this on math overflow and did not get an answer.}.

\medskip
An \textbf{ordered vector space} is an ordered vector space over an ordered division ring.
The real ordered vector space $\rvec$ is shown to have near linear Zarankiewicz bounds in \cite[Theorem~C]{zaran}.
Fact~\ref{fact:mct} is not explicitly stated in \cite{zaran}, but follows directly from the proof of Theorem C in that paper and quantifier elimination for ordered vector spaces.

\begin{fact}
\label{fact:mct}
An ordered vector space has near linear Zarankiewicz bounds.
\end{fact}

\noindent
Theorem~\ref{thm:vector-field} follows from Fact~\ref{fact:mct}, Proposition~\ref{prop:field-blank}, and Lemma~\ref{lem:blank-2}.

\begin{theorem}
\label{thm:vector-field}
An ordered vector space cannot trace define an infinite field.
\end{theorem}

We prove Proposition~\ref{prop:press}.

\begin{proposition}
\label{prop:press}
$(\Z;+,<)$ has near linear Zarankiewicz bounds.
\end{proposition}

\noindent
Fact~\ref{fact:mct} shows $(\R;+,<)$ has near linear Zarankiewicz bounds.
Prop~\ref{prop:press} follows by Prop~\ref{prop:pres}.

\begin{proposition}
\label{prop:pres}
Any bipartite graph which is definable in $(\Z;+,<)$ embeds into a bipartite graph which is definable in $(\R;+,<)$.
\end{proposition}

Proposition~\ref{prop:pres} shows that we need something more than the combinatorics of definable bipartite graphs to show $(\R;+,<)$ cannot trace define $(\Z;+)$ (unless it does).

\begin{proof}
It is enough to consider bipartite graphs of the form $(\Z^m,\Z^n;E)$ where $E$ is a $(\Z;+,<)$-definable subset of $\Z^{m + n}$.
We construct an $(\R;+,<)$-definable bipartite graph $(V_0,V_1;F)$ and an embedding $\mathbf{e} \colon (\Z^m,\Z^n;E) \to (V_0,V_1;F)$.
The usual quantifier elimination for Presburger arithmetic shows that there is $\ell \in \N$ such that $E \subseteq \Z^{m + n}$ is a finite union of finite intersections of sets of the following two forms:
\begin{enumerate}[leftmargin=*]
\item $X \cap \Z^{m + n}$ for $(\R;+,<)$-definable $X \subseteq \R^{m + n}$, and
\item $\{ (a_1,\ldots,a_{m + n}) \in \Z^{m + n} : a_i \equiv r \pmod{\ell} \}$ for $r \in \{0,\ldots,\ell - 1\}$, and $i \in \{1,\ldots, m + n\}$.
\end{enumerate}
Let $A$ be $\{0,\ldots,\ell - 1\}^{m + n}$.
For each $a = (a_1,\ldots,a_{m + n}) \in A$ we let $C_a$ be the set of $(b_1,\ldots,b_{m + n}) \in \Z^{m + n}$ such that $b_i \equiv a_i \pmod{\ell}$ for all $i$.
It easily follows that for each $a \in A$ there is an $(\R;+,<)$-definable $E_a \subseteq \R^{m + n}$ such that $E \cap C_a = E_a \cap C_a$.

\medskip\noindent
We let $A_0$ be $\{0,\ldots,\ell - 1 \}^{m}$, $A_1$ be $\{0,\ldots,\ell - 1\}^n$, and $a^{\frown}_0a_1 \in A$ be the concatenation of $a_0 \in A_0$ and $a_1 \in A_1$.
Let $V_0 = \R^m \times A_0$, $V_1 = \R^n \times A_1$, and declare 
\[
F((b,a), (b^*,a^*)) \quad\Longleftrightarrow\quad (b,b^*) \in E_{a^{\frown}a^*} \quad \text{for all} \quad (b,a) \in V_0, (b^*,a^*) \in V_1.
\]
Then $(V_0,V_1;F)$ is $(\R;+,<)$-definable.
We let $\mathbf{e} \colon \Z^m \to V_0$ be given by $e(b_1,\ldots,b_m) = (b_1,\ldots,b_m,r_1,\ldots,r_m)$ where each $r_i$ is the remainder mod $\ell$ of $b_i$.
We define $\mathbf{e} \colon \Z^n \to V_2$ in the same way.
It is easy to see that $\mathbf{e}$ is an embedding.
\end{proof}

Theorem~\ref{thm:press 2} follows from Proposition~\ref{prop:press} and Proposition~\ref{prop:field-blank}. 

\begin{theorem}
\label{thm:press 2}
$\Th(\Z;+,<)$ does not trace define an infinite field.
\end{theorem}

Corollary~\ref{cor:presbur} now follows by Section~\ref{section:presburger} and Proposition~\ref{prop:final Z}.

\begin{corollary}
\label{cor:presbur}
$(\R;+,<,\Z,\Q)$, finite rank oag's, and finite rank cyclically ordered abelian groups, all have near linear Zarankiewicz bounds and hence cannot trace define infinite fields. 
\end{corollary}

I think we need a different property for the general case.

\begin{conj}
\label{conj:oag}
An ordered abelian group cannot trace define an infinite field.
\end{conj}

It should certainly be possible to show that many more examples of oag's have near linear Zarankiewicz bounds.
I think we would need something different to handle the lexicographically ordered group of infinite integer sequences with pointwise addition.

\subsection{Proof of Theorem~\ref{thm:rigid}}
\label{section:rigid}
Let $\mathrm{Add}$ be a ternay relation symbol.
Given an ordered group $(R;+,<)$ and an interval $I\subseteq R$ we let $\mathrm{Add}_I$ be the ternary relation on $I$ where $\mathrm{Add}_I(a,a',b) \Longleftrightarrow b = a + a'$ for all $a,a',b \in I$.
A \textbf{DOAG-interval} is a structure $(I;\mathrm{Add},\prec)$ such that there is $(R;+,<) \models \doag$, an interval $J \subseteq R$ containing $0$, and an isomorphism $(I;\mathrm{Add}, \prec) \to (J;\mathrm{Add}_J,<)$.
Fact~\ref{fact:lovey} a special case of o-minimal trichotomy~\cite{PS-Tri, loveys}.

\begin{fact}
\label{fact:lovey}
Suppose that $\Sa R$ is o-minimal.
Then exactly one of the following holds:
\begin{enumerate}[leftmargin=*]
\item There is an interval $I \subseteq R$ and definable $X \subseteq I^3$ such that $(I;X,<)$ is a $\doag$-interval.
\item $\Sa R$ is trivial.
\end{enumerate}
Suppose in addition that $\Sa R$ expands an ordered group $(R;+,<)$.
Then exactly one of the following holds:
\begin{enumerate}[leftmargin=*]
\item There is an interval $I \subseteq R$ and definable $\oplus,\otimes \colon I^2 \to I$ such that $(I;\oplus,\otimes) \models \rcf$.
\item There is an ordered division ring $\mathbb{D}$ and an ordered $\D$-vector space $\V$ expanding $(R;+,<)$ such that $\Sa R$ is a reduct of $\V$.
\end{enumerate}
If $(R;+,<) = (\R;+,<)$ then we may take $\V$ to be $\rvec$.
\end{fact}

Recall that $\rvec$ is $(\R;+,<,(t \mapsto \lambda t)_{\lambda \in \R})$.

\medskip
We let $\D$ be an ordered division ring and $\Sa V$ be an ordered $\D$-vector space.
Given a division subring $\D^*$ of $\D$ we let $\Sa V^*$ be the ordered $\D^*$-vector space reduct of $\Sa V$.

\begin{lemma}
\label{lem:ovs}
Let $\D$ and $\Sa V$ be as above.
Suppose that $\Sa M$ is a reduct of $\Sa V$ expanding $(V;+,<)$.
Then there is a division subring $\D^*$ of $\D$ such that $\Sa M$ is trace equivalent to $\Sa V^*$.
So any reduct of $\rvec$ expanding $(\R;+,<)$ is trace equivalent to $(\R;+,<,(t\mapsto\lambda t)_{\lambda\in F})$ for a subfield $F\subseteq\R$.
\end{lemma}

Let $\Sa O$ be the expansion of $(\R;+,<)$ by all functions $[0,1]\to\R$, $t\mapsto\lambda t$ for $\lambda\in\R$.
The proof of Lemma~\ref{lem:ovs} shows that $\Sa O$ is trace equivalent to $\rvec$.
However $\Sa O$ is a proper reduct of $\rvec$ as $\R\to\R$, $t\mapsto\lambda t$ is only $\Sa O$-definable when $\lambda\in\Q$~\cite{MPP}.
It should be possible to show that $\Sa O$ does not interpret $\rvec$, but I don't have a proof of this.

\begin{proof}
Let $\Sa M'$ be the expansion of $(V;+,<)$ by all  $f\colon [0,\gamma)\to V$ such that:
\begin{enumerate}
\item $\gamma\in V\cup\{\infty\}$,
\item there is $\lambda\in\D$ such that $f(\alpha)=\lambda\alpha$ for all $0\le\alpha<\gamma$,
\item and $f$ is definable in $\Sa M$.
\end{enumerate}
Then $\Sa M'$ is a reduct of $\Sa M$.
We show that $\Sa M$ and $\Sa M'$ are interdefinable.
Following the semilinear cell decomposition \cite[1.7.8]{lou-book} one can show that any $\Sa M$-definable subset of $V^n$ is a finite union of $\Sa M$-definable semilinear cells.
An easy induction on $n$ shows that any $\Sa M$-definable semilinear cell $X \subseteq V^n$ is definable in $\Sa M'$.
So we suppose $\Sa M=\Sa M'$.

\medskip
Let $\D^*$ be the set of $\lambda\in\D$ such that $[0,\gamma)\to V$, $\alpha\mapsto\lambda\alpha$ is $\Sa M$-definable for some positive $\gamma\in V$.
Note that $\D^*$ is a division subring of $\D$.
We show that $\Sa M$ and $\Sa V^*$ are trace equivalent.
By construction $\Sa M$ is a reduct of $\Sa V^*$.
We show that $\Th(\Sa M)$ trace defines $\Sa V^*$.
Suppose that $\Sa V\prec\Sa W$ is $|\D|^+$-saturated and let $\Sa N$ be the reduct of $\Sa W$ such that $\Sa M\prec\Sa N$.
Then $\Sa N$ is a $|\D|^+$-saturated elementary extension of $\Sa M$.
It is enough to show that $\Sa N$ trace defines some ordered $\D^*$-vector space as the theory of ordered $\D^*$-vector spaces is complete.
Let $I$ be the set of $\alpha\in W$ such that $|\alpha|<|\beta|$ for all non-zero $\beta\in V$.
Note that $I$ is a convex subgroup of $(W;+,<)$ and $I$ is non-trivial by saturation.
If $\lambda\in\D^*\setminus\{0\}$, $\alpha\in I$, and $\beta\in V\setminus\{0\}$, then $|\alpha|<|\beta/\lambda|$ as $\beta/\lambda\in V\setminus\{0\}$, so multiplying through by $|\lambda|$ yields $|\lambda\alpha|<|\beta|$.
Hence $I$ is closed under multiplication by any $\lambda\in\D^*$, so $(I;+,<)$ has a natural $\D^*$-vector space expansion $\Sa I$.
By definition of $\D^*$ for every $\lambda\in\D^*$ there is an $\Sa N$-definable $f\colon W\to W$ such that $f(\alpha)=\lambda\alpha$ for all $\alpha\in I$.
By convexity $I$ is $\Sh N$ definable, hence $\Sa I$ is definable in $\Sh N$.
An application of Proposition~\ref{prop:she-0} shows that $\Th(\Sa N)$ trace defines $\Sa I$.
\end{proof}

Now everything is in place to prove Theorem~\ref{thm:rigid}.3.

\begin{proposition}
\label{prop:o-min-rigid}
Suppose that $\Sa R$ is an o-minimal expansion of an oag $(R;+,<)$.
Then the following are equivalent:
\begin{enumerate}[leftmargin=*]
\item $\Sa R$ does not define a real closed field.
\item $\Th(\Sa R)$ does not trace define an infinite field.
\item $\Sa R$ has near linear Zarankiewicz bounds.
\item $\Sa R$ is trace equivalent to  an ordered vector space over an ordered division ring.\\
(If $(R;+,<)$ is $(\R;+,<)$ then the ordered division ring is an ordered subfield of $\R$.)
\item There is an ordered division ring $\D$ and an ordered $\D$-vector space $\V$ expanding $(R;+,<)$ so that $\Sa R$ is a reduct of $\V$.
(If $(R;+,<)=(\R;+,<)$  we  have $\V=\rvec$.)
\end{enumerate}
\end{proposition}

\begin{proof}
Apply Fact~\ref{fact:lovey}, Theorem~\ref{thm:vector-field}, and Lemma~\ref{lem:ovs}.
\end{proof}

We now discuss Theorem~\ref{thm:rigid}.1.
A geometric (in particular o-minimal or strongly minimal) theory $T$ is \textbf{trivial} if whenever $\Sa M \models T$ and $\beta \in M$ is in the algebraic closure of $A \subseteq M$ then $\beta$ is in the algebraic closure of some $\alpha \in A$.
A structure is trivial when its theory is.

\begin{theorem}
\label{thm:trivial}
Suppose that $\Sa M$ is o-minimal.
Then the following are equivalent:
\begin{enumerate}
\item $\Sa M$ defines an infinite group.
\item $\Th(\Sa M)$ trace defines an infinite group.
\item $\Th(\Sa M)$ trace defines $(\R;+,<)$.
\item $\Sa M$ is non-trivial.
\end{enumerate}
In particular $\dlo$ does not trace define an infinite group.
\end{theorem}

We give some examples of trivial o-minimal structures.
Given an arbitrary o-minimal structure $\Sa M$ let $\Sa M_{\mathrm{mon}}$ be the expansion of $(M;<)$ by all increasing $\Sa M$-definable functions $I \to M$, for $I$ an interval.
By \cite[Proposition~4.1]{Simon-dp} $\Sa M_\mathrm{mon}$ is binary and by \cite{mekler-rubin-steinhorn} a binary o-minimal structure is trivial.
Hence $\Sa M_{\mathrm{mon}}$ is trivial o-minimal.
For example the expansion of $(\R;<)$ by all increasing semialgebraic functions $I \to \R$, $I \subseteq \R$ an interval, is trivial. 

\begin{lemma}
\label{lem:doag interval}
A $\doag$-interval is trace equivalent to $\doag$ and defines an infinite group.
\end{lemma}

Proposition~\ref{prop:semilinear group} shows that if $(R;+,<) \models \doag$ and $I \subseteq R$ is a bounded interval then the structure induced on $I$ does not interpret an ordered abelian group.

\begin{proof}
Suppose that $\Sa D$ is a $\doag$-interval.
Then $\Sa D$ is definable in a divisible ordered abelian group.
The proof of Lemma~\ref{lem:ovs} shows that $\Th(\Sa D)$ trace defines $\doag$.
The construction given after Fact~\ref{fact:cover} shows that $\Sa D$ defines an infinite group.
\end{proof}

We now gather some tools for Theorem~\ref{thm:trivial}.
 
\medskip 
Suppose that $X \subseteq M^m$ and $f \colon X \to Y$.
Then $f$ is \textbf{essentially one variable} if there is $i \in \{1,\ldots,m\}$ and a function $g \colon M \to Y$ such that
\[
f(a_1,\ldots,a_m) = g(a_i) \quad \text{for all} \quad (a_1,\ldots,a_m) \in X.
\]
Fact~\ref{fact:trivial} follows from the definition of triviality, the fact that algebraic and definable closure agree in an o-minimal structure, and a routine compactness argument.

\begin{fact}
\label{fact:trivial}
Suppose that $\Sa M$ is a trivial o-minimal structure and $X \subseteq M^m$, $f \colon X \to M^n$ are definable.
Then there is a partition $X_1,\ldots,X_k$ of $X$ into definable sets such that the restriction of $f$ to each $X_i$ is essentially one-variable.
\end{fact}

We will need a certain cell decomposition for trivial o-minimal structures.
I don't claim originality, but I do claim that it is routine.
Suppose that $\Sa M$ is o-minimal.
We first give some conventions that will be useful when working with cells.
We let $M_{\pm \infty}$ be $M \cup \{\pm\infty\}$.
We consider definable functions $f \colon X \to M_{\pm\infty}$.
We assume that if $f(a) = \infty$ for some $a \in X$ then $f(a) = \infty$ for all $a \in X$, likewise for $-\infty$.
We let $\square$ range over $\{<,=\}$.
We consider terms of the form $f \hspace{.1cm}\square \hspace{.1cm} g$ for definable continuous functions $f,g\colon X \to M_{\pm\infty}$.
We suppose that if $\square$ is $=$ then $f = g$ and $f$ is not constant $\pm\infty$.
We also consider $M^0$ to be a singleton.

\medskip
We recall the definition of a cell~\cite[Chapter 3]{lou-book}, which we give in a non-inductive way because we will want to keep track of all definable functions involved.
A subset $X$ of $M^m$ is a cell if there are definable sets $X_0 = M^0, X_1 \subseteq M, \ldots, X_{i} \subseteq M^{i},\ldots, X_m = X$, definable continuous functions $f_i,g_i \colon X_i \to M_{\pm\infty}$ for $i \in \{1,\ldots,m \}$, and $\square_1,\ldots,\square_m\in \{<,=\}$ such that if $i \ge 1$ then $X_i$ is the set of $(\beta_1,\ldots,\beta_i) \in M^i$ such that
\[
(\beta_1,\ldots,\beta_{i - 1}) \in X_{i - 1} \quad \text{and} \quad f_i(\beta_1,\ldots,\beta_{i - 1}) \simplesquare \beta_i \simplesquare g_i(\beta_1,\ldots,\beta_{i - 1}).
\]
We refer to $f_1,g_1,\ldots,f_m,g_m$ as the \textbf{defining functions} of $X$.
A \textbf{simple cell} is a cell whose defining functions are essentially one variable.

\begin{proposition}
\label{prop:cell decomp}
Suppose that $\Sa M$ is a trivial o-minimal structure and $X$ is a definable subset of $M^m$.
Then $X = X_1 \cup\cdots\cup X_n$ for pairwise disjoint simple cells $X_1,\ldots,X_n$.
\end{proposition}

\begin{proof}
We apply induction on $m$.
The case $m = 1$ is easy.
Suppose $m \ge 2$.
By cell decomposition we may suppose that $X$ is a cell.
Let $(f_i,g_i \colon X_{i - 1} \to M_{\pm\infty} : i \in \{1,\ldots,m\})$ be the defining functions of $X$.
By Fact~\ref{fact:trivial} there is a partition $Y_1,\ldots,Y_n$ of $X_{m - 1}$ into definable sets such that the restriction of $f_m,g_m$ to each $Y_i$ is essentially one-variable.
By induction we can partition each $Y_i$ into finitely many simple cells, so we may suppose that each $Y_i$ is a simple cell.
Finally $[Y_i \times M] \cap X$ is a simple cell for each $i \in \{1,\ldots,n\}$.
\end{proof}

\begin{proof}[Proof of Theorem~\ref{thm:trivial}]
Suppose (4).
By Fact~\ref{fact:lovey} $\Sa M$ defines a $\doag$-interval.
By Lemma~\ref{lem:doag interval} $\Th(\Sa M)$ trace defines $\doag$ and $\Sa M$ defines an infinite group.
Hence (4) implies (1) and (3).
It is clear that (1) implies (2) and (3) implies (2).

\medskip
We finish by showing that (2) implies (4).
Towards a contradiction suppose that $\Sa M$ is a trivial o-minimal structure and $\Sa M$ trace define an infinite group $\Sa O$.
We may suppose that $O \subseteq M^m$ and that $\Sa O$ is trace definable via the identity $O \to M^m$.
Fix $k > 2m$.
Let $X \subseteq M^{mk} \times M^m$ be an $\Sa M$-definable set such that \[
(\alpha_1,\ldots,\alpha_k,\beta) \in X \quad \Longleftrightarrow \quad \beta = \alpha_1 \alpha_2 \cdots \alpha_k \quad \text{for all} \quad \alpha_1,\ldots,\alpha_k,\beta \in O.
\]
Given $\alpha = (\alpha_1,\ldots,\alpha_k) \in M^{mk}$ we let $X_\alpha$ be the set of $\beta \in M^m$ with $(\alpha,\beta) \in X$, so $\alpha_1\alpha_2\cdots\alpha_k$ is the unique element of $O$ contained in $X_\alpha$.
By Proposition~\ref{prop:cell decomp} we have $X = X_1 \cup \cdots \cup X_n$ for simple cells $X_1,\ldots,X_n$.
Applying Ramsey's theorem\footnote{It'd be more conventional to use an indiscernible sequence here, but we only need to apply Ramsey once.} we fix a sequence $(\upzeta_i : i \in \{2,3,\ldots,2k+1\})$ of distinct elements of $O$ and $j \in \{1,\ldots,n\}$ such that $(\upzeta_{i_1}, \upzeta_{i_2}, \ldots ,\upzeta_{i_k},\upzeta_{i_1} \upzeta_{i_2}\cdots \upzeta_{i_k}) \in X_j$ for all $2\le i_1 < i_2 < \ldots < i_k \le 2k + 1$.
Let $Y := X_j$.
Let $f_1,g_1,\ldots,f_{km + m},g_{km+ m}$ be the defining functions of $Y$, so each $f_i,g_i$ is essentially one-variable.
If $a \in M^{km}$ then $Y_a$ is the set of $(\beta_1,\ldots,\beta_m) \in M^m$ such that
\begin{align*}
f_{km + 1}(a) \hspace{.1cm}\square_1\hspace{.1cm} &\beta_1 \hspace{.1cm} \square_1\hspace{.1cm} g_{km + 1}(a)\\
f_{km + 2}(a,\beta_1) \hspace{.1cm}\square_2\hspace{.1cm} &\beta_2 \hspace{.1cm} \square_2\hspace{.1cm} g_{km + 2}(a,\beta_1) \\
f_{km + 3}(a,\beta_1,\beta_2) \hspace{.1cm}\square_3\hspace{.1cm} &\beta_3 \hspace{.1cm} \square_3\hspace{.1cm} g_{km + 3}(a,\beta_1,\beta_2) \\ &\vdots \\f_{km + m}(a,\beta_1,\ldots,\beta_{m - 1}) \hspace{.1cm}\square_m\hspace{.1cm} &\beta_m \hspace{.1cm} \square_m\hspace{.1cm} g_{km + m}(a,\beta_1,\ldots,\beta_{m - 1}). 
\end{align*}
Here the satisfaction of the first $i$th terms ensures that $f_{km + i + 1}$ is defined on $(a,\beta_1,\ldots,\beta_i)$.
We consider each one of the defining functions to be a function in the variables 
\[
x_1,\ldots,x_m,\ldots,x_{2m},\ldots,x_{km},x_{km + 1},\ldots,x_{km + m - 1}.
\]
Note that each defining function depends on only one of these variables.
As $k > 2m$ there is $j \in \{0,\ldots,k - 1\}$ such that $f_{km + 1},g_{km + 1},f_{km + 2},g_{km + 2},\ldots,f_{km + m},g_{km + m}$ do not depend on any of the variables $x_{jm + 1},x_{jm + 2},\ldots,x_{jm + m}$.
Hence if $\alpha,\beta \in M^{mk}$ differ only in these coordinates then $Y_\alpha = Y_{\beta}$.
Let 
\[\upalpha = (\upzeta_2,\upzeta_4,\ldots,\upzeta_{2j - 2}, \upzeta_{2j}, \upzeta_{2j+2} \ldots, \upzeta_{2k})\quad \text{and} \quad\upalpha' = (\upzeta_2,\upzeta_4,\ldots,\upzeta_{2j-2},\upzeta_{2j  + 1}, \upzeta_{2j + 2},\ldots, \upzeta_{2k}).
\]
The product of the entries of $\upalpha$ does not equal the product of the entries of $\upalpha'$ as $\upzeta_{2j} \ne \upzeta_{2j + 1}$.
However $Y_\upalpha = Y_{\upalpha'}$, hence the product of the entries of $\upalpha'$ is in $Y_{\upalpha}$.
This is a contradiction as the product of the entries of $\upalpha$ is the unique element of $O$ contained in $Y_{\upalpha}$.
\end{proof}

We say that an $\aleph_0$-categorical, $\aleph_0$-stable structure is trivial if the induced structure each strongly minimal definable set of imaginaries is trivial.
We now prove Theorem~\ref{thm:rigid}.2.

\begin{proposition}
\label{prop:o-stable-rigid}
Suppose $\Sa M$ is $\aleph_0$-categorical and $\aleph_0$-stable.
The following are equivalent:
\begin{enumerate}
\item $\Sa M$ interprets an infinite group.
\item $\Sa M$ trace defines an infinite group.
\item $\Sa M$ is non-trivial.
\end{enumerate}
\end{proposition}

\begin{proof}
By Proposition~\ref{prop:trace-interpret} (1) implies (2).
If $\Sa M$ is trivial then $\Sa M$ is interpretable in $\dlo$ by a theorem of Lachlan~\cite{lachlan-order} and we apply Theorem~\ref{thm:trivial}.
Thus (2) implies (3).
Well-known results show that (3) implies (1).
Recall that $\Sa M$ is coordinatized by finitely many strongly minimal sets interpretable
in $\Sa M$ and  the combinatorial geometry associated to each strongly minimal set is either
trivial or an affine or projective geometry over a finite field~\cite{CHL}.
If the geometry on some
strongly minimal interpretable set is affine or projective over a finite field then $\Sa M$ interprets
an infinite group, see e.g. \cite[Chapter 5]{pillay-book}.
\end{proof}

By Theorem~\ref{thm:trivial} $\dlo$ does not trace define an infinite group.
We give a shorter proof.
A \textbf{colored linear order} is an expansion of a linear order  by a family of unary relations.

\begin{proposition}
\label{prop:lo-group}
A colored linear order cannot trace define an infinite group.
\end{proposition}

\noindent
Poizat showed that a colored linear order cannot interpret an infinite group~\cite{poizat-propos}.
We follow the proof of this result as presented by Hodges~\cite[A.6]{Hodges}.
Fact~\ref{fact:lo} is \cite[Lemma~A.6.8]{Hodges}.

\begin{fact}
\label{fact:lo}
Suppose that $\Sa O$ is a colored linear order.
Suppose that $a = (a_1,\ldots,a_m) \in O^m$ and $b = (b_1,\ldots, b_n) \in O^n$.
Then there is $I \subseteq \{1,\ldots,m\}$ such that $|I| \le 2n$ and if $c = (c_1,\ldots,c_m)$ is in $O^m$, $\tp_{\Sa O}(c) = \tp_{\Sa O}(a)$, and $\tp_{\Sa O}(b,(a_i)_{i \in I}) = \tp_{\Sa O}(b,(c_i)_{i \in I})$, then $\tp_{\Sa O}(b,a) = \tp_{\Sa O}(b,c)$.
\end{fact}

The main idea in the proof of Fact~\ref{fact:lo} is to select for each $b_i$ the minimal element of $\{a_1,\ldots,a_m\} \cap (\infty,b_i]$ and the maximal element of $\{a_1,\ldots,a_m\} \cap [b_i,\infty)$.
This gives a set $A$ of $\le 2n$ elements of $\{a_1,\ldots,a_m\}$ and one can show that $\tp_{\Sa O}(ab)$ is determined by $\tp_{\Sa O}(a)$ and $\tp_{\Sa O}(Ab)$.
For $\dlo$ this follows immediately by considering automorphisms of $(\Q;<)$.
In general one needs the model theory of linear orders developed by Rubin~\cite{rubin}.

\medskip
We now prove Proposition~\ref{prop:lo-group}.

\begin{proof}
Suppose that $\Sa O$ is a colored linear order and let $L$ be the language of $\Sa O$.
Suppose that $\Sa O$ trace defines an infinite group $G$.
We suppose that $G$ is a subset $O^m$ and that $G$ is trace definable via the identity $G \to O^m$.
By Proposition~\ref{prop:trace-theories} we may suppose that $\Sa O$ and $G$ are both highly saturated.
For each $k$ there is an $L(O)$-formula $\varphi(x_1,\ldots,x_k,y)$ such that for any $\alpha_1,\ldots,\alpha_k,\beta \in G$ we have $\Sa O \models \varphi(\alpha_1,\ldots,\alpha_k,\beta) \Longleftrightarrow \alpha_1 \alpha_2 \cdots \alpha_k = \beta$.
After adding new constant symbols to $L$ we may suppose that each $\varphi_k$ is parameter-free.

\medskip\noindent
Fix a sequence $(\upzeta_q : q \in \Q)$ of distinct elements of $G$ which is indiscernible in $\Sa O$.
Let $\upzeta_1 \upzeta_2  \cdots \upzeta_k = \upbeta_k$ for all $k \ge 1$.
By Fact~\ref{fact:lo} there is $k$ and $i \le k$ such that $\tp_{\Sa O}(\upzeta_1,\ldots,\upzeta_k,\upbeta_k)$ is determined by $\tp_{\Sa O}(\upzeta_1,\ldots,\upzeta_{i - 1},\upzeta_{i + 1},\ldots,\upzeta_k,\upbeta_k)$ and $\tp_{\Sa O}(\upzeta_1,\ldots,\upzeta_k)$.
By indiscernibility:
$$ \tp_{\Sa O}(\upzeta_1,\ldots,\upzeta_{i - 1},\upzeta_{i + \frac{1}{2}},\upzeta_{i + 1},\ldots,\upzeta_k,\upbeta_k) =  \tp_{\Sa O}(\upzeta_1,\ldots,\upzeta_{i - 1},\upzeta_i,\upzeta_{i + 1},\ldots,\upzeta_k,\upbeta_k).$$
Then $\Sa O \models \varphi_k(\upzeta_1,\ldots,\upzeta_{i - 1},\upzeta_{i + \frac{1}{2}},\upzeta_{i + 1},\ldots,\upzeta_k,\upbeta_k)$, so 
$$\upzeta_1 \cdots \upzeta_{i - 1} \upzeta_{i + \frac{1}{2}} \upzeta_{i + 1} \cdots \upzeta_k = \upbeta_k.$$
But this implies $\upzeta_i = \upzeta_{i + \frac{1}{2}}$, contradiction.
\end{proof}

It must be noted that the proofs of Theorem~\ref{thm:trivial} and Proposition~\ref{prop:lo-group} are \textit{bad}, or at least not yet good.
A good proof would be to define a reasonable $\nip$-theoretic notion which is preserved under trace definibility, not satisfied by any infinite group, and is satisfied by a trivial o-minimal structure or colored linear order.
Maybe a notion of nippy triviality.
(Of course this raises the question of exactly which $\nip$ structures should be ``trivial".
Should trees be ``trivial"??)
Maybe the right notion is a notion of ``almost a linear order".

\medskip
Linear orders are $\nip$, and they remain $\nip$ after expansion by unary relations.
Maybe this is actually the property that we should look at here.
A structure $\Sa M$ is \textbf{monadically $\nip$} if any expansion of $\Sa M$ by unary relations is $\nip$.
Trees (considered as semilinear orders) are monadically $\nip$ by \cite[Proposition~4.6]{Simon-dp}.
Monadic $\nip$ and monadic stability seem to be important model-theoretic notions, see \cite{brlas, nese}.
It is easy to see that if $\Sa M$ is monadically $\nip$ then $\Sa M$ cannot define an infinite group via an injection taking values in $M$.
I conjecture that a monadically $\nip$ structure cannot trace define an infinite group.
By \cite{hempel-chernkov} a monadically $\nip$ geometric structure has trivial algebraic closure, so there is a connection to triviality.

\subsection{O-minimal induced structures}
\label{section:induced structure}
We consider structures $\Sa A$ of the following form: $\Sa R$ is o-minimal, $A \subseteq R$ is dense, $\Sa A$ is the structure induced on $A$ by $\Sa R$, and $\Sa A$ admits quantifier elimination.
Note that by Fact~\ref{fact:wom-induced} $\Th(\Sa A)$ is weakly o-minimal hence $\nip$.


\begin{lemma}
\label{lem:}
Suppose that $\Sa R$ is an o-minimal expansion of an ordered group $(R;+,<)$, $A$ is a dense subset of $R$, $\Sa A$ is the structure induced on $A$ by $\Sa R$, and $\Sa A$ admits quantifier elimination.
Then $\Sa A$ and $\Sa R$ are trace equivalent.
\end{lemma}

If $A$ is co-dense then $\Sa A$ is not interpretable in an o-minimal expansion of an ordered group by Corollary~\ref{cor:rama-1}.
The examples in Section~\ref{section:mann} show that $\Th(\Sa A)$ may not interpret $\Sa R$.
Let
\[
\|a\| = \max\{|a_1|,\ldots,|a_n|\} \quad \text{for all} \quad a = (a_1,\ldots,a_n) \in R^n
\]
and let $\cl(X)$ be the closure of $X \subseteq R^m$.
We use the trivial identity $\|(a,b)\| = \max\{ \|a\|, \|b\| \}$.
We also use some basic facts on one-types in weakly o-minimal structures, see Fact~\ref{fact:wom types}.

\begin{proof}
By Lemma~\ref{lem:she--1} $\Sa R$ trace defines $\Sa A$.
We show that $\Th(\Sa A)$ trace defines $\Sa R$.
Let $\Sa B$ be an elementary extension of $\Sa A$ such that:
\begin{enumerate}
\item $\tp_{\Sa B}(\alpha|A)$ is definable for all $\alpha \in B^m$, and
\item every definable type over $A$ is realized in $\Sa B$.
\end{enumerate}
Such $\Sa B$ exists by Fact~\ref{fact:artem} and weak o-minimality of $\Sa A$.
By Proposition~\ref{prop:she-0} it is enough to show that $\Sh B$ interprets $\Sa R$.
We first realize $R$ as a $\Sh B$-definable set of imaginaries.
Let $D$ be the set of $(c,a,b) \in A^3$ such that $c > 0$ and $| a - b | < c$.
Note that $D$ is $\Sa A$-definable.
Let $D^*$ be the subset of $B^3$ defined by the same formula as $D$.
Note that $(D^*_c)_{c > 0}$ is a chain under inclusion.
Let $E = \bigcap_{c \in A, c > 0} D^*_c$ and $V = \bigcup_{c \in A, c > 0} \{ a \in B : (0,a,c) \in D^* \}$.
By Lemma~\ref{lem:chain} $V$ and $E$ are both externally definable in $\Sa B$.
Note that $V$ is the set of $b \in B$ such that $|b| < r$ for some positive $r \in R$ and $E$ is the set of $(a,b) \in B^2$ such that $| a - b | < r$ for all positive $r \in R$.
So $E$ is an equivalence relation on $V$.
We construct a canonical bijection $V/E \to R$.

\medskip
For our purposes a \textbf{cut} in $A$ is a nonempty downwards closed bounded above set $C \subseteq A$ such that either $C$ does not have a supremum or $C$ contains its supremum.
(So if $\alpha \in A$ then $(-\infty,\alpha]$ is a cut but $(-\infty,\alpha)$ is not.)
Quantifier elimination for $\Sa A$, o-minimality of $\Sa R$, and density of $A$ in $R$ together imply that the $\Sa A$-definable cuts in $A$ are exactly those sets of the form $\{ \alpha \in A : \alpha \le r \}$ for unique $r \in R$.
So we identify $R$ with the set of $\Sa A$-definable cuts in $A$.
Given $\beta \in V$ we let $C_\beta := \{ \alpha \in A : \alpha < \beta \}$ if this set does not have a supremum and otherwise let $C_\beta = \{ \alpha \in A : \alpha < \beta \} \cup \{\gamma\}$ if $\gamma$ is the supremum.
Each $C_\beta$ is a cut.
By (1) each $C_\beta$ is $\Sa A$-definable and by (2) every definable cut in $A$ is of the form $C_\beta$ for some $\beta \in V$.
Hence we identify $R$ with $(C_\beta : \beta \in V)$.
Note that for all $\beta,\beta^* \in V$ we have $(\beta,\beta^*) \in E$ if and only if $C_\beta = C_{\beta^*}$, so we may identify $V/E$ with $R$ and consider $R$ to be a $\Sh B$-definable set of imaginaries.
Let $\st \colon V \to R$ be the quotient map.
Then $\st$ is monotone, so $\Sh B$ defines the order on $R$.
Hence the closure of a $\Sh B$-definable subset of $R^n$ is $\Sh B$-definable.

\medskip
We show that $\Sa R$ is a reduct of the structure induced on $R$ by $\Sh B$.
We first suppose that $X \subseteq A^m$ is $\Sa A$-definable and show that $\cl(X)$ is $\Sh B$-definable.
Let $X^*$ be the subset of $B^m$ defined by the same formula as $X$.
It is easy to see that $X \subseteq \st(X^* \cap V^m) \subseteq \cl(X)$, hence $\cl(X) = \cl(\st(X^* \cap V^m))$ is $\Sh B$-definable.

\medskip\noindent
Suppose $Y$ is a nonempty $\Sa R$-definable subset of $R^m$.
We show that $Y$ is $\Sh B$-definable.
By o-minimal cell decomposition $Y$ is a boolean combination of closed $\Sa R$-definable subsets of $R^m$, so we may suppose that $Y$ is closed.
Let $W$ be the set of $(\varepsilon,c) \in R_{>0} \times R^m$ for which there is $c^* \in Y$ satisfying $\| c - c^* \| < \varepsilon$.
Then $W \cap (A \times A^m)$ is $\Sa A$-definable and $Z := \cl(W \cap (A \times A^m))$ is $\Sh B$-definable.
Let $Y^*$ be $\bigcap_{t \in R,t > 0} Z_t$.
Then $Y^*$ is $\Sh B$-definable.
We show that $Y = Y^*$.
Each $Z_t$ is closed, so $Y^*$ is closed.
Suppose $\varepsilon \in A_{> 0}$.
Then $W_\varepsilon$ is open, so $(W \cap (A \times A^n))_\varepsilon$ is dense in $W_\varepsilon$, so $Z_\varepsilon$ contains $W_\varepsilon$, hence $Z_\varepsilon$ contains $Y$.
Thus $Y \subseteq Y^*$.
We now prove the other inclusion.
Suppose that $p^* \in Y^*$.
We show that $p^* \in Y$.
As $Y$ is closed it suffices to fix $\varepsilon \in R_{>0}$ and find $p \in Y$ such that $\| p - p^* \| < \varepsilon$.
We may suppose that $\varepsilon \in A$.
We have $(\varepsilon,p^*) \in Z$, so there is $(\delta,q) \in W \cap (A \times A^n)$ such that $\| (\varepsilon,p^*) - (\delta,q)\| < \varepsilon$.
By definition of $W$ we obtain $p \in Y$ such that $\| p - q \| < \delta$.
We have $| \varepsilon - \delta | < \varepsilon$, so $\delta < 2\varepsilon$, hence $\|p - q\| < 2 \varepsilon$.
We also have $\| p^* - q \| < \varepsilon$, so $\| p^* - p \| < 3\varepsilon$.
\end{proof}




Corollary~\ref{cor:import} now follows from Proposition~\ref{prop:o-min-rigid} and Lemma~\ref{lem:}.

\begin{corollary}
\label{cor:import}
Suppose that $\Sa R$ is an o-minimal expansion of an ordered group $(R;+,<)$, $A$ is a dense subset of $R$, and the structure $\Sa A$ induced on $A$ by $\Sa R$ eliminates quantifiers.
The following are equivalent:
\begin{enumerate}[leftmargin=*]
\item $\Th(\Sa A)$ does not trace define $\rcf$.
\item $\Th(\Sa A)$ does not trace define an infinite field.
\item $\Sa A$ has near linear Zarankiewicz bounds.
\item $\Sa A$ is trace equivalent to  an ordered vector space over an ordered division ring $\D$.\\
(If $(R;+,<)$ is $(\R;+,<)$ then $\D$ is a subfield of $\R$.)
\item there is an ordered division ring $\D$ and an ordered $\D$-vector space $\V$ expanding $(R;+,<)$ such that every $\Sa A$-definable $X \subseteq A^n$ is of the form $Y \cap A^n$ for $\V$-definable $Y \subseteq R^n$.
\\(If $(R;+,<)$ is $(\R;+,<)$ then we  take $\V=\rvec$.)
\end{enumerate}
\end{corollary}

Corollary~\ref{cor:import3} follows from Theorem~\ref{thm:trivial} and Lemma~\ref{lem:}.

\begin{corollary}
\label{cor:import3}
Suppose that $\Sa R$ is o-minimal, $A \subseteq R$ is dense, and the structure $\Sa A$ induced on $A$ by $\Sa R$ eliminates quantifiers.
Then the following are equivalent:
\begin{enumerate}
\item $\Th(\Sa A)$ trace defines an infinite group.
\item $\Th(\Sa A)$ trace defines $\doag$.
\item $\Sa R$ is non-trivial.
\end{enumerate}
\end{corollary}

In Lemma~\ref{lem:} we assume that $\Sa A$ eliminates imaginaries.
What did we use this for?
First to show that $\Sa R$ trace defines $\Sa A$.
Secondly, to construct $\Sa B$, and thirdly to show that $\Sa A$ is $\nip$ which allows an application of $\Sh B$.
Suppose that $\Sa R$ expands $(\R;+,<)$.
Then every cut in $A$ is $\Sa A$-definable, so we can take $\Sa B$ to be an $\aleph_1$-saturated elementary extension of $\Sa A$ and show that $\Sh B$ interprets $\Sa R$ in the same way.
Lemma~\ref{lem::} follows.

\begin{lemma}
\label{lem::}
Suppose that $\Sa R$ is an o-minimal expansion of $(\R;+,<)$, $A$ is a dense subset of $\R$, $\Sa A$ is the structure induced on $A$ by $\Sa R$, and $\Sa A$ is $\nip$.
Then $\Th(\Sa A)$ trace defines $\Sa R$.
\end{lemma}

In Prop~\ref{prop:mordell lang} we give an example of a dense subset $A$ of $\R$ such that the structure induced on $A$ by $\rfield$ is trace equivalent to $(\Z;+)\sqcup\rfield$.
Hence in this case $\rcf$ trace defines $\Sa A$ if and only if $\rcf$ trace defines $(\Z;+)$.
I don't know if $\rcf$ trace defines $(\Z;+)$.

\subsection{Two examples}
\label{section:two examples}

These examples show that Corollaries~\ref{cor:import} and \ref{cor:import3} are sharp.

\subsubsection{A Mann pair}
\label{section:mann}
This example shows that Corollary~\ref{cor:import} is sharp.
Fix a real number $\lambda > 0$, let $\lambda^\Q = \{\lambda^q : q \in \Q\}$, and $\Sa Q_\lambda$ be the expansion of $(\Q;+,<)$ by all sets of the form $\{ (q_1,\ldots,q_m) \in \Q^m : (\lambda^{q_1},\ldots,\lambda^{q_m}) \in X\}$ for semialgebraic $X \subseteq \R^m$.
Note that $q \mapsto \lambda^q$ gives an isomorphism between $\Sa Q_\lambda$ and the structure induced on $\lambda^\Q$ by $\rfield$.

\begin{fact}
\label{fact:mann-vdd}
Every $(\R;+,\times,\lambda^\Q)$-definable subset of $(\lambda^\Q)^n$ is of the form $(\lambda^\Q)^n \cap X$ for semialgebraic $X \subseteq \R^n$.
So the structure induced on $\lambda^\Q$ by $\rfield$ admits quantifier elimination.
\end{fact}

\noindent Fact~\ref{fact:mann-vdd} is a special case of a result of van den Dries and G\"{u}naydin~\cite[Theorem~7.2]{vddG}.

\begin{corollary}
\label{cor:example}
$\Sa Q_\lambda$ is trace equivalent to $\rfield$.
\end{corollary}

\noindent
Corollary~\ref{cor:example} follows from Fact~\ref{fact:mann-vdd} and Lemma~\ref{lem:}.

\begin{proposition}
\label{prop:mann}
$\Sa Q_\lambda$ does not interpret an infinite field.
\end{proposition}

\noindent
It follows from Fact~\ref{fact:mann-vdd} that algebraic closure in $\Sa Q_\lambda$ agrees with algebraic closure in $(\Q;+)$, from this one can deduce that $\Sa Q_\lambda$ does not define an infinite field.
However, Berenstein and Vassiliev have essentially already proven Proposition~\ref{prop:mann}, so we just apply their results.

\begin{proof}
By Eleftheriou~\cite{E-small} $\Sa Q_\lambda
$ eliminates imaginaries.
Hence it suffices to show that $\Sa Q_\lambda$ does not define an infinite field.
It is a special case of a theorem of Berenstein and Vassiliev~\cite[Proposition~3.16]{BV-one-based} that $\Sa Q_\lambda$ is weakly one-based and a weakly one-based theory cannot define an infinite field by~\cite[Proposition~2.11]{BV-one-based}.
\end{proof}

\subsubsection{The induced structure on an independent set}
\label{section:o-minimal}
\noindent
We describe a weakly o-minimal expansion of $(\Q;<)$ that is trace equivalent to $(\R;+,\times)$ but does not interpret an infinite group.
This example shows that Corollary~\ref{cor:import} is sharp.
Given an o-minimal structure $\Sa M$ we say that $A \subseteq M$ is \textbf{independent} if it is independent with respect to algebraic closure in $\Sa M$.

\medskip
Fact~\ref{fact:independent} is due to Dolich, Miller, and Steinhorn~\cite{DMS-Indepedent}.

\begin{fact}
\label{fact:independent}
Suppose that $\Sa M$ is o-minimal and $H$ is a dense independent subset of $M$.
Any subset of $H^n$ definable in $(\Sa M, H)$ is of the form $X \cap H^n$ for some $\Sa M$-definable $X \subseteq M^n$.
\end{fact}

\noindent
Corollary~\ref{cor:independent-1} follows from Fact~\ref{fact:independent} and Lemma~\ref{lem:}.

\begin{corollary}
\label{cor:independent-1}
Suppose $\Sa R$ is an o-minimal expansion of an oag, $H$ is a dense independent subset of $R$, and $\Sa H$ is the structure induced on $H$ by $\Sa R$.
Then $\Sa H$ is trace equivalent to $\Sa R$.
\end{corollary}

\noindent
We prove Proposition~\ref{prop:no-interpret}. 

\begin{proposition}
\label{prop:no-interpret}
Suppose that $\Sa R$ is an o-minimal expansion of an oag and $H\subseteq R$ is dense independent.
The structure $\Sa H$ induced on $H$ by $\Sa R$ cannot interpret an infinite group.
\end{proposition}

\begin{proof}
A theorem of Eleftheriou~\cite[Theorem~C]{Elef-small-sets} shows that $\Sa H$ eliminates imaginaries, so it is enough to show that $\Sa H$ does not define an infinite group.
This follows by another result of Berenstein and Vassiliev~\cite[Corollary~6.3]{BV-independent}
\end{proof}

Thus if $H\subseteq \R$ is  dense algebraically independent and $\Sa H$ is the structure induced on $H$ by $\rfield$ then $\Sa H$ is trace equivalent to $\rfield$ but doesn't interpret an infinite group.

\subsection{Proof of Theorem~\ref{thm:tri}}
\label{section:dense-pairs}
Suppose that $Q$ is a divisible subgroup of $(\R;+)$ and $\Sa Q$ is a dp-rank one expansion of $(Q;+,<)$.
We first describe the o-minimal completion $\Sq Q$ of $\Sa Q$.
This completion is closely associated to the o-minimal completion of a non-valuational weakly o-minimal structure constructed by Wencel~\cite{Wencel-1}, see also Keren~\cite{Keren-thesis}.
Indeed $\Sq Q$ is interdefinable with the Wencel completion of $\Sh Q$.
Let $\Sa N$ be an $\aleph_1$-saturated elementary extension of $\Sa Q$.
Let 
\begin{align*}
V &= \{ a \in N : |a| < t, \text{ for some } t \in Q, t > 0\}\\
\mfrak &= \{ a \in N : |a| < t, \text{ for all } t \in Q, t > 0 \}.
\end{align*}
Then $V$ and $\mfrak$ are convex subgroups of $(N,+,<)$ and we may identify $V/\mfrak$ with $\R$ and the quotient map $V \to \R$ with the usual standard part.
Note that $V$ and $\mfrak$ are externally definable, so we consider $\R$ to be an $\Sh N$-definable set of imaginaries.
We define $\Sq Q$ to be the structure induced on $\R$ by $\Sh N$.

\begin{lemma}
\label{lem:sq-o-min}
$\Sq Q$ is o-minimal.
\end{lemma}

\begin{proof}
By Fact~\ref{fact:pw-weak} $\Th(\Sa Q)$ is weakly o-minimal, so $\Sa N$ is weakly o-minimal.
By Fact~\ref{fact:shelah} $\Sh N$ is weakly o-minimal.
The quotient map $O \to \R$ is monotone, and an image of a convex set under a monotone map is convex, so $\Sq Q$ is weakly o-minimal.
Note that every convex subset of $\R$ is an interval, so $\Sq Q$ is o-minimal.
\end{proof}

\noindent
Fact~\ref{fact:completion-induced} is a special case of the results of \cite{big-nip}.

\begin{fact}
\label{fact:completion-induced}
The structure induced on $Q$ by $\Sq Q$ admits quantifier elimination and is interdefinable with $\Sh Q$.
\end{fact}

\noindent
Corollary~\ref{cor:completion-induced} follows from Fact~\ref{fact:completion-induced}, Proposition~\ref{prop:she-0}, and Lemma~\ref{lem:}.

\begin{corollary}
\label{cor:completion-induced}
$\Sa Q$ and $\Sq Q$ are trace equivalent.
\end{corollary}

\noindent
Laskowski and Steinhorn~\cite{LasStein} show that if $\Sa Q$ is o-minimal then $\Sa Q$ is an elementary submodel of a unique o-minimal expansion of $(\R;+,<)$.
In fact, if $\Sa Q$ is o-minimal then $\Sq Q$ is interdefinable with this elementary extension, see \cite{big-nip}.
If $\Sa Q$ is not o-minimal then Corollary~\ref{cor:rama-2} shows that $\Sa Q$ is not interpretable in an o-minimal expansion of an ordered group, so in particular $\Th(\Sa Q)$ is not interpretable in $\Th(\Sq Q)$.

\medskip
We now prove Theorem~\ref{thm:tri}.
We first restate the theorem in a slightly different form.
A \textbf{semilinear} set is an $\rvec$-definable set.

\begin{theorem}
\label{thm:1.3}
Suppose that $Q$ is a divisible subgroup of $(\R;+)$ and $\Sa Q$ is a dp-rank one expansion of $(Q;+,<)$.
Then the following are equivalent:
\begin{enumerate}
\item $\Th(\Sa Q)$ does not trace define $\rcf$.
\item $\Th(\Sa Q)$ does not trace define an infinite field.
\item $\Sa Q$ has near linear Zarankiewicz bounds.
\item $\Sa Q$ is trace equivalent to $(\R;+,<,(t\mapsto\lambda t)_{\lambda\in F})$ for a subfield $F\subseteq\R$.
\item any $\Sa Q$-definable $X \subseteq Q^n$ is of the form $Y \cap Q^n$ for semilinear $Y \subseteq \R^n$.
\item $\Sq Q$ is a reduct of $\rvec$.
\end{enumerate}
\end{theorem}

\begin{proof}
Note that we may replace $\Sa Q$ with $\Sh Q$.
Apply Corollary~\ref{cor:import} and Fact~\ref{fact:completion-induced}.
\end{proof}

Let $\qvec$ be the structure induced on $\Q$ by $\rvec$. 

\begin{corollary}
\label{cor:on-Q}
Suppose that $\Sa Q$ is a dp-rank one expansion of $(\Q;+,<)$.
Then $\mathrm{Th}(\Sa Q)$ does not trace define $\rcf$ if and only if $\Sa Q$ is a reduct of $\qvec$.
\end{corollary}

\begin{proof}
By \cite{GoHi-Pairs} $(\rvec,\Q)$ is $\nip$ and all $(\rvec,\Q)$-definable sets are of the form $X \cap \Q^n$ for semilinear $X \subseteq \R^n$.
So $\qvec$ admits quantifier elimination.
Apply Theorem~\ref{thm:1.3}.
\end{proof}



\subsection{An extension of Theorem~\ref{thm:tri}}
We prove an extension of Theorem~\ref{thm:tri} using somewhat different tools.
In this section we suppose that $\Sa R$ expands a divisible ordered abelian group $(R;+,<)$.
As above, a \textbf{cut} in $R$ is a nonempty downwards closed bounded above $C \subseteq R$ such that either $C$ does not have a supremum or $C$ contains its supremum.
A cut $C$ in $R$ is \textbf{valuational} if there is positive $\alpha \in R$ such that $C + \alpha = C$.
We say that $\Sa R$ is \textbf{non-valuational} if $\Sa R$ satisfies the following equivalent conditions:
\begin{enumerate}
\item There are no non-trivial definable convex subgroups of $(R;+,<)$.
\item Every non-trivial definable cut in $R$ is non-valuational.
\end{enumerate}
It is easy to see that (1) and (2) are equivalent.
Note that $(R;+,<)$ is regular if and only if $(R;+,<)$ is non-valuational and $(R;+,<)$ is archimedean if and only if any expansion of $(R;+,<)$ is non-valuational.
Fact~\ref{fact:dp-rank-one} is proven in \cite{SW-dp}.

\begin{fact}
\label{fact:dp-rank-one}
If $\Sa R$ is non-valuational then $\Sa R$ is dp-rank one iff $\Sa R$ is weakly o-minimal.
\end{fact}

Fact~\ref{fact:dp-rank-one} fails without divisibility.
Suppose that $\Sa R$ is weakly o-minimal non-valuational.
We define the Wencel completion $\overline{\Sa R}$ of $\Sa R$.
We let $\overline{R}$ be the set of  definable cuts in $R$.
The natural addition on cuts and the inclusion ordering makes $\overline{R}$ an ordered abelian group.
We identify each $\alpha \in R$ with $(-\infty,\alpha]$ and hence consider $R$ to be an ordered subgroup of $\overline{R}$.

\medskip\noindent
The Wencel completion of $\overline{\Sa R}$ is the expansion of $(\overline{R};+,<)$ by all sets of the form $\cl(X)$ for $X$ an $\Sa R$-definable subset of $R^n$.
Fact~\ref{fact:wencel} is proven in \cite{BHP-pairs}.

\begin{fact}
\label{fact:wencel}
Suppose that $\Sa R$ is weakly o-minimal non-valuational.
Then $\overline{\Sa R}$ is o-minimal, the structure $\Sa R^*$ induced on $R$ by $\overline{\Sa R}$ eliminates quantifiers, and $\Sa R$ is interdefinable with $\Sa R^*$.
\end{fact}

Theorem~\ref{thm:wencel} follows from Fact~\ref{fact:wencel} and Lemma~\ref{lem:}.

\begin{theorem}
\label{thm:wencel}
If $\Sa R$ is weakly o-minimal non-valuational then $\Sa R$ and $\overline{\Sa R}$ are trace equivalent.
\end{theorem}

Theorem~\ref{thm:wencel tri} is now a special case of Corollary~\ref{cor:import}.

\begin{theorem}
\label{thm:wencel tri}
Suppose $\Sa R$ is dp-rank $1$ and non-valuational.
The following are equivalent:
\begin{enumerate}
\item $\Th(\Sa R)$ does not trace define $\rcf$,
\item $\Th(\Sa R)$ does not trace define an infinite field,
\item $\Sa R$ has near linear Zarankiewicz bounds,
\item $\Sa R$ is trace equivalent to an ordered vector space over an ordered division ring,
\item there is an ordered division ring $\D$ and an ordered $\D$-vector space $\V$ so that $(R;+,<)$ is an ordered subgroup of $\V$ and every $\Sa R$-definable $X \subseteq R^n$ is of the form $Y \cap R^n$ for $\V$-definable $Y \subseteq V^n$.
\end{enumerate}
\end{theorem}

The expansion $(\overline{\Sa R}, R)$ of $\overline{\Sa R}$ by a predicate defining $R$ is studied in \cite{BHP-pairs}.
This structure is $\nip$ and the structure induced on $R$ by $(\overline{\Sa R}, R)$ is interdefinable with $\Sa R$.
However $(\overline{\Sa R}, R)$ is in general not trace definable in $\Th(\Sa R)$.
Suppose that $\Sa R$ expands an ordered field.
It is easy to see that $\overline{\Sa R}$ also expands an ordered field.
Fix $t \in \overline{R} \setminus R$.
A weakly o-minimal field is real closed~\cite{MMS-weak}, so $\Sa R$ is real closed.
Hence $\overline{\Sa R}$ is a purely transendental field extension of $\Sa R$, so $t$ is transendental over $R$.
Then the map $R^k \to \overline{R}$, $(\alpha_0,\ldots,\alpha_{k - 1}) \mapsto t^{k} +  \alpha_{k - 1}t^{k - 1} + \cdots + \alpha_1 t + \alpha_0$ is injective for each $k \ge 2$.
Hence $(\overline{\Sa R}, R)$ has infinite dp-rank, and is therefore not trace definable in a dp-rank one theory such as $\Sa R$, see Proposition~\ref{thm:dp-rank}.

\medskip
It is natural to ask what happens when we drop the divisibility assumption.

\begin{conj}
\label{conj:2}
Suppose that $\Sa R$ is a non-valuational expansion of a dp-minimal ordered abelian group.
Then there is an o-minimal expansion $\Sq R$ of an ordered abelian group such that $\Sa R$ is trace equivalent to $(R;+)\sqcup\Sq R$.
If $\Sa R$ is a dp-minimal expansion of an archimedean ordered abelian group then we may take $\Sq R$ to be an expansion of $(\R;+,<)$.
\end{conj}

If $\Sa R$ is a dp-minimal non-valuational expansion of a discrete ordered abelian group then $\Sa R$ is interdefinable with $(R;+,<)$ and $(R;+,<)$ is a model of Presburger arithmetic, see \cite{SW-dp}.
Therefore this case of the conjecture follows by Proposition~\ref{prop:final Z}.
The case when $\Sa R=(R;+,<)$ follows by Theorem~\ref{thm:regular rank}.
If $A$ is a dense finite rank subgroup of $(\R_{>0};\times)$ then the structure induced on $A$ by $\rfield$ is dp-minimal, so Proposition~\ref{prop:mordell lang} is an instance of the conjecture.
Finally, this conjecture would give us another reason to want to know which torsion free abelian groups are trace definable in an o-minimal theory.

\section{A little more about fields}
\label{section:fields}
In Section~\ref{section:two examples} we gave some examples of weakly o-minimal structures above which trace define infinite fields.
In each case the field in question is real closed.
It is natural to conjecture that this should always be the case, e.g. any infinite field trace definable in a weakly o-minimal theory should be real or algebraically closed.
I don't even know if this holds for $\rcf$.
There are other similar conjectures that could be made.

\medskip
Many interesting $\nip$ structures have finite dp-rank, finiteness of dp-rank is preserved under trace definibility and there is now a reasonable classification of finite dp-rank fields~\cite{Johnson-dp-finite}.
It should be possible to classify finite dp-rank fields up to trace definibility.
This should yield information about what fields can be trace defined in theories of interest.
For example, any infinite field trace definable in $\rcf$ should be real or algebraically closed, and the right proof of this should show that any infinite field trace definable in an o-minimal structure is real or algebraically closed.
These questions are tied to the question as to whether or not $(\Z;+)$ is trace definable in an o-minimal structure.
A theory that cannot trace define $(\Z;+)$ cannot trace define $\Q_p$ as $(\Z;+)$ is the value group of $\Q_p$.
If one could show that $\rcf$ cannot trace definable $(\Z;+)$, then the proof should show that $\rcf$ can't define lots of abelian groups.
Unstable dp finite fields admit definable valuations (and hence admit many externally definable valuations), so one could apply this to the value group.

\medskip
In this section we take some small steps towards a classification of dp-finite fields up to trace definibility.
We first recall what we already know.
Of course $\rcf$ trace defines the theory of algebraically closed fields of characteristic zero, and by Proposition~\ref{prop:p adic fields} all finite extensions of $\Q_p$ are trace equivalent.
By Lemma~\ref{lem:fp} a dp finite field of characteristic $p>0$ trace defines the algebraic closure of the field with $p$ elements.

\subsection{Strong Erd\H{o}s-Hajnal Property}
\label{section:eh}
We show that $\C$, $\R$, and $\Q_p$ cannot trace define infinite positive characteristic fields.
This involves an interesting combinatorial property.

\medskip
A bipartite graph $(V,W;E)$ satisfies the \textbf{Strong Erd\H{o}s-Hajnal property} if there is a real number $\delta > 0$ such that for every $A \subseteq V, B \subseteq W$ there are $A^* \subseteq A$, $B^* \subseteq B$ such that $|A^*| \geq \delta |A|$, $|B^*| \geq \delta B$, and $A^* \times B^*$ is either contained in or disjoint from $E$.
Then $\Sa M$ has the \textbf{strong Erd\H{o}s-Hajnal property} if all definable bipartite graphs have the strong Erd\H{o}s-Hajnal property and $T$ has the strong Erd\H{o}s-Hajnal property if its models do.

\medskip\noindent
There are countable bipartite graphs that do not have the strong Erd\H{o}s-Hajnal property, so the generic countable bipartite graph does not have the strong Erd\H{o}s-Hajnal property.
Thus by Proposition~\ref{prop:trace-0} a theory with the strong Erd\H{o}s-Hajnal property is $\nip$.

\medskip
Proposition~\ref{prop:eh-trace} is clear from the definitions.

\begin{proposition}
\label{prop:eh-trace}
If $T$ has the strong Erd\H{o}s-Hajnal property than any theory trace definable in $T$ has the strong Erd\H{o}s-Hajnal property.
\end{proposition}

Fact~\ref{fact:distal} is due to Chernikov and Starchenko~\cite[Theorem~1.9]{CS}.

\begin{fact}
\label{fact:distal}
A distal structure has the strong Erd\H{o}s-Hajnal property.
\end{fact}

Fact~\ref{fact:eh} is also due to Chernikov and Starchenko~\cite[Section 6]{CS}.

\begin{fact}
\label{fact:eh}
Infinite positive characteristic fields don't have the strong Erd\H{o}s-Hajnal property.
\end{fact}



\noindent
Proposition~\ref{prop:trace-distal} follows from the previous three results.

\begin{proposition}
\label{prop:trace-distal}
Distal structures can't trace define infinite positive characteristic fields.
\end{proposition}

It follows that real closed fields, $p$-adically closed fields, and algebraically closed fields of characteristic zero cannot trace define infinite positive characteristic  fields (the first two are distal, the last admits a distal expansion).
Any dp-rank one expansion of a linear order is distal~\cite[Example 9.20]{Simon-Book} and so cannot trace define an infinite positive characteristic field.
Hence a weakly o-minimal theory cannot trace define an infinite positive characteristic field.

\medskip
Of course distality is not preserved under interpretations as it is not preserved under reducts.
We say that a structure is \textbf{pre-distal} if it has a distal expansion, or equivalently is interpretable in a distal structure.
Is pre-distality preserved under trace definibility?
Equivalently: does trace definibility in a distal structure imply interpretability in a distal structure?
(I don't see any reason why it should or should not.)

\subsection{Algebraically and real closed fields}
\label{section:acf}

\begin{theorem}
\label{thm:acf}
Suppose that $K$ is an algebraically closed field.
Then any infinite field trace definable in $K$ is algebraically closed.
If $K$ is in addition characteristic zero then any field trace definable in $K$ is characteristic zero.
\end{theorem}

\noindent
It easily follows that if $K$ is algebraically closed of characteristic zero, the transendence degree of $K/\Q$ is infinite, and $F$ is an infinite field trace definable in $K$, then there is an elementary embedding $F \to K$.
This is sharp by Proposition~\ref{prop:trace-basic}.3.

\begin{proof}
The first claim follows from Corollary~\ref{cor:superstable} and Macintyre's theorem~\cite{Macintyre-omegastable} that an infinite $\aleph_0$-stable field is algebraically closed.
Suppose $K$ is characteristic zero.
Fix a real closed subfield $R$ of $K$ such that $K = R(\sqrt{-1})$.
Then $(K;R)$ is distal.
Apply Proposition~\ref{prop:trace-distal}.
\end{proof}

\begin{proposition}
\label{prop:weak-o-min}
Suppose $\Sa M$ is weakly o-minimal and
$K$ is an infinite field trace definable in $\Sa M$ via an injection $K \to M$.
Then $K$ is real or algebraically closed of characteristic zero.
\end{proposition}

Proposition~\ref{prop:weak-o-min} follows easily from some known results.
We recall a notion of Flenner and Guingona~\cite{flenner-guingona}.
We say that $\Sa M$ is \textbf{convexly orderable} if there is a linear order $\triangleleft$ on $M$ such that if $(X_\alpha)_{\alpha \in M^n}$ is a definable family of subsets of $M$ then there is $n$ such that each $X_\alpha$ is a union of at most $n$ $\triangleleft$-convex sets for all $b \in M^{|y|}$.
A convexly orderable structure $\Sa M$ need not define a linear order on $M$, for example a strongly minimal structure is convexly orderable.
It is easy to see that convex orderability is preserved under elementary equivalence, so we say that $T$ is convexly orderable if some (equivalently: every) $T$-model is convexly orderable.
Structures with weakly o-minimal theory are clearly convexly orderable and $C$-minimal structures are convexly orderable~\cite{flenner-guingona}.

\medskip\noindent
Lemma~\ref{lem:convex-order} is clear from the definitions.

\begin{lemma}
\label{lem:convex-order}
If $\Sa M$ is convexly orderable and $\Sa O$ is trace definable in $\Sa M$ via an injection $O \to M$ then $\Sa O$ is convexly orderable.
\end{lemma}

\noindent
By Proposition~\ref{prop:trace-distal} a weakly o-minimal structure cannot trace definable an infinite positive characteristic field.
Proposition~\ref{prop:weak-o-min} follows from this observation, Lemma~\ref{lem:convex-order}, and Johnson's theorem~\cite{Johnson} that an infinite convexly orderable field is real or algebraically closed.
By \cite[Theorem~5.14]{flenner-guingona} $(\Z;+)$ is convexly orderable, so we cannot apply this kind of argument to $(\Z;+)$.
By \cite[Corollary~3.4]{flenner-guingona} a convexly orderable oag is divisible.
Hence any oag $(H;+,\prec)$ trace definable in a weakly o-minimal structure $\Sa M$ via an injection $H\to M$ is divisible.

\medskip
C-minimal structures are also convexly orderable, so  if $(F,v)\models\mathrm{ACVF}$ then any infinite field trace definable in $(F,v)$ via an injection $\to F$ is real or algebraically closed.
Of course we expect that any infinite field trace definable in  $\mathrm{ACVF}$ is algebraically closed.

\section{Expansions of $(\R;<,+)$}
\label{section:open core}
We make some comments concerning open cores of expansions of $(\R;+,<)$ trace definable in natural theories.
Let $\Sa R$ be an expansion of $(\R;+,<)$.
The \textbf{open core} $\Sa R^\circ$ of $\Sa R$ is the reduct of $\Sa R$ generated by all open (hence all closed) definable sets.
Note $\Sa R$ is interdefinable with $\Sa R^\circ$ iff $\Sa R$ is interdefinable with an expansion of $(\R;+,<)$ by open and closed subsets of $\R^m$.

\medskip
By Corollary~\ref{cor:rama-2} $\Sa R$ is interpretable in an o-minimal expansion of an ordered group if and only if $\Sa R$ is o-minimal.
When is $\Sa R$ trace definable in an o-minimal structure?
By Proposition~\ref{prop:dense pair} $(\R;+,<,\Q)$ is trace definable in $(\R;+,<)$, so $\Sa R$ need not be o-minimal.
By Theorem~\ref{thm:dp-rank} $\Sa R$ is strongly dependent if $\Sa R$ is trace definable in an o-minimal structure.
By \cite{big-nip} we know that if $\Sa R$ is strongly dependent then one of the following holds:
\begin{enumerate}
\item $\Sa R^\circ$ is o-minimal, or
\item there is a collection $\Cal B$ of bounded subsets of Euclidean space such that $\Sa R^\circ$ is interdefinable with $(\R;+,<,\Cal B,\az)$ for some positive $\alpha\in\R$.
\end{enumerate}
If (2) holds then $\Sa R^\circ$ is trace equivalent to $(\Z;+)\sqcup(\R;+,<,\Cal B)$ by Proposition~\ref{prop:expansion by Z}.
So we see that the following are equivalent:
\begin{enumerate}[leftmargin=*, label=(\alph*)]
\item An o-minimal structure cannot trace define $(\Z;+)$.
\item An expansion of $(\R;+,<)$ which is trace definable in an o-minimal theory necessarily has o-minimal open core.
\end{enumerate}
This gives another reason to want to know if an o-minimal structure can trace define $(\Z;+)$.

\medskip
When is $\Sa R$ trace definable in Presburger arithmetic?
By Proposition~\ref{prop:final Z} Presburger arithmetic trace defines $(\R;+,<,\Z)$.
By Proposition~\ref{prop:press} Presburger arithmetic cannot trace define an infinite field.
Applying Proposition~\ref{prop:o-min-rigid} we see that if (1) holds then $\Sa R^\circ$ is trace equivalent to an ordered vector space and if (2) holds then $\Sa R^\circ$ is trace equivalent to the disjoint union of $(\Z;+)$ and an ordered vector space.
So which ordered vector spaces are trace definable in Presburger arithmetic? 
I suspect that the presence of irrational scalars in $\rvec$ produce ``aperiodic" structure that obstructs trace definibility in Presburger arthimetic.

\section{Finitely homogeneous structures}
\label{section:conj}
We discuss trace definibility between finitely homogeneous structures.
We consider the generic countable $k$-hypergraph, $\dlo$, and the generic countable binary branching $C$-relation.

\medskip
We make some conjectures and then gather some evidence for these conjectures.
There is a long line of research on classification of various kinds of finitely homogeneous structures.
These classifications seem to become rather complex.
When we pass to trace equivalence there is a considerable reduction in complexity of the classifications.

\begin{wk}
There are $\aleph_0$ finitely homogeneous structures up to trace equivalence.
\end{wk}

Note that by Proposition~\ref{prop:omega cat} two finitely homogeneous structures are trace equivalent if and only if each trace defines the other.
It should be noted that Simon has conjectured that there are only countably many finitely homogeneous $\nip$ structures up to isointerdefinibility~\cite{Pierre2}.

\medskip
For each $k \ge 2$ the generic countable $k$-hypergraph is $(k-1)$-independent and $k$-dependent and a finitely $k$-ary structure is $k$-dependent \cite{cpt}.
So by Proposition~\ref{prop:k-dep} a finitely $k$-ary structure cannot trace define the generic countable $(k + 1)$-hypergraph.
In particular the generic countable $k$-hypergraph does not trace define the generic countable $\ell$-hypergraph when $k < \ell$.
(We could also apply Proposition~\ref{prop:airity}.)
So there are infinitely many finitely homogeneous structures up to trace equivalence.

\begin{sk}
For every $k$ there are only finitely many $k$-ary finitely homogeneous structures up to trace equivalence.
\end{sk}

The vague underlying conjecture is that trace equivalence classes of finitely homogeneous structures should correspond to natural model-theoretic properties.
We prove the binary case of the weak conjecture below, see Proposition~\ref{prop:binary}.
We also show that the trace equivalence classes of $\dlo$ and the generic countable $k$-hypergraph are large.

\medskip\noindent
Let $\Upomega_k$ be the cardinality set of trace equivalence classes of finitely homogeneous structures of airity exactly $k$.
We have $\Upomega_2 \ge 3$ as the trivial structure, $(\Q;<)$, and Erd\H{o}s-Rado graph give distinct classes.
For all I know we could have $\Upomega_2=3$.

\subsection{Trace definibility in the generic countable $k$-hypergraph}
We let $\hyp_k$, $\relk_k$ be the theory of the generic countable $k$-hypergraph, $k$-ary relation, respectively.

\begin{lemma}
\label{lem:airity}
The generic countable $k$-ary relation and the generic countable $k$-ary hypergraph both trace define any countable finitely $k$-ary structure.
\end{lemma}

\noindent
We remind the reader that the definition of ``finitely $k$-ary" is back in the conventions.

\begin{proof}
The second claim follows from the first by Proposition~\ref{prop:random} and the third follows by Fact~\ref{fact:airity}.
Suppose that $\Sa O$ is a countable finitely $k$-ary $L$-structure.
We first reduce to the case when every $R \in L$ has airity $k$.
Morleyizing reduces to the case when $L$ is a finite relational language, every $R \in L$ has airity $\le k$, and $\Sa O$ admits quantifier elimination.
Let $L^*$ be the language which contains each $k$-ary $R \in L$ and contains a $k$-ary relation $R^*$ for each $R \in L$ of airity $< k$.
We let $\Sa O^*$ be the $L^*$-structure with domain $O$ where each $k$-ary $R \in L$ is given the same interpretation as in $L$ and if $R \in L$ has airity $\ell < k$ then for any $\alpha_1,\ldots,\alpha_k \in O$ we declare 
$$\Sa O^* \models R^*(\alpha_1,\ldots,\alpha_k) \quad\Longleftrightarrow\quad \Sa O \models R(\alpha_1,\ldots,\alpha_\ell) \land (\alpha_{\ell} = \alpha_{\ell + 1} = \cdots = \alpha_k).$$
Note that $\Sa O^*$ admits quantifier elimination and is interdefinable with $\Sa O$.
After possibly replacing $\Sa O$ with $\Sa O^*$ we may suppose that every $R \in L$ is $k$-ary.

\medskip\noindent
Let $(V;R)$ be the generic countable $k$-ary relation.
We may suppose that $L = \{ R_1,\ldots,R_m\}$ where each $R_i$ is $k$-ary.
Note that each $(O;R_i)$ is isomorphic to a substructure of $(V;R)$ and apply Proposition~\ref{prop:fusion}.
\end{proof}

\begin{proposition}
\label{prop:k-dependent}
Fix $k \ge 2$.
The following are equivalent.
\begin{enumerate}
\item $T$ is $(k-1)$-independent,
\item $T$ trace defines $\hyp_k$,
\item $T$ trace defines $\relk_k$,
\item $T$ trace defines any finitely $k$-ary theory.
\end{enumerate}
\end{proposition}


\begin{proof}
Proposition~\ref{prop:random} shows that $(2)$ and $(3)$ are equivalent.
The random $k$-ary relation is $(k-1)$-independent, so $(3)$ implies $(1)$ by Proposition~\ref{prop:k-dep}.
It is clear that $(4)$ implies $(3)$.
By Lemma~\ref{lem:airity} $(3)$ implies $(4)$.
Proposition~\ref{prop:k-dep-0} shows that $(1)$ implies $(2)$.
\end{proof}

Corollary~\ref{cor:henson} follows from Proposition~\ref{prop:k-dependent}.

\begin{corollary}
\label{cor:henson}
Any finitely binary $\mathrm{IP}$ structure is trace equivalent to the Erd\H{o}s-Rado graph.
\end{corollary}

We can now prove the binary case of the Weak Conjecture.
Proposition~\ref{prop:binary} follows from Corollary~\ref{cor:henson} and the theorem of Onshuus and Simon~\cite{onshuus-simon} that there are only countably many binary finitely homogeneous $\nip$ structures up to isointerdefinibility.

\begin{proposition}
\label{prop:binary}
$\Upomega_2\le\aleph_0$.
\end{proposition}

Corollary~\ref{cor:henson} also reduces the binary case of the Strong Conjecture to the binary $\nip$ case.

\medskip
The Ramsey case of the conjectures is of particular interest owing to the connection to indiscernible collapse.

\begin{corollary}
\label{cor:ramsey fh}
Suppose that $\Sa O$ is finitely homogeneous and one of the following holds:
\begin{itemize}
\item $\age(\Sa O)$ has the Ramsey property, or 
\item $\age(\Sa O)$ is an unstable free amalgamation class.
\end{itemize}
Fix $k \ge 2$.
Then
\begin{enumerate}
\item $\Sa O$ is trace definable in the generic countable $k$-hypergraph if and only if $\Sa O$ is $k$-ary.
\item $\Sa O$ is trace equivalent to the generic countable $k$-hypergraph if and only if $\Sa O$ is $k$-ary and $
(k - 1)$-independent.
\end{enumerate}
\end{corollary}

The example discussed after Proposition~\ref{prop:airity} shows that Corollary~\ref{cor:ramsey fh} does not hold for arbitrary finitely homogeneous structures.

\begin{proof}
Suppose that $\age(\Sa O)$ has the Ramsey property.
Then (1) follows from Lemma~\ref{lem:airity}  and Proposition~\ref{prop:airity} and (2) follows from (1) and Proposition~\ref{prop:k-dependent}.
Suppose $\age(\Sa O)$ is an unstable free amalgamation class.
We now let $(\Sa O,\triangleleft)$ be as defined before Fact~\ref{fact:lo-exp}.
By Lemma~\ref{lem:free homo} and Proposition~\ref{prop:omega cat} the generic $k$-hypergraph trace defines $\Sa O$ if and only if it trace defines $(\Sa O,\triangleleft)$.
As $(\Sa O,\triangleleft)$ admits quantifier elimination and $k \ge 2$ we see that $(\Sa O,\triangleleft)$ is $k$-ary if and only if $\Sa O$ is $k$-ary and by Lemma~\ref{lem:free homo} $(\Sa O,\triangleleft)$ is $(k - 1)$-dependent iff $\Sa O$ is $(k - 1)$-dependent.
So we replace $\Sa O$ with $(\Sa O,\triangleleft)$.
Apply Fact~\ref{fact:lo-exp} and the Ramsey case. 
\end{proof} 

\begin{corollary}
\label{cor:fusions}
Fix $k \ge 2$.
Suppose that $L$ is relational, $L^*$ is an expansion of $L$ by relations of airity $\le k$, $\Sa M$ is a $(k- 1)$-independent $L$-structure, $\Sa O$ is an $L^*$-structure expanding $\Sa M$, and $\Sa O$ admits quantifier elimination.
Then $\Sa M$ and $\Sa O$ are trace equivalent.
\end{corollary}

Hence any expansion of an $\mathrm{IP}$ structure $\Sa M$ by binary relations which admits quantifier elimination is trace equivalent to $\Sa M$.

\begin{proof}
It suffices to show that $\Th(\Sa M)$ trace defines $\Sa O$.
Let $L^* \setminus L = \{R_1,\ldots,R_{n - 1}\}$.
As in the proof of Lemma~\ref{lem:airity} we  suppose that each $R_i$ is $k$-ary.
Let $(V;R)$ be the generic $k$-ary relation, so $\Th(\Sa M)$ trace defines $(V;R)$.
Apply Prop~\ref{prop:fusion} with $L_1 = L$ and $L_i = \{R_{i - 1}\}$ for $i \ge 2$, $\Sa O_1 = \Sa O$ and $\Sa O_i = (M;R_{i - 1})$ for $i \ge 2$, $\Sa P_1 = \Sa O_1$, and $\Sa P_i = (V;R)$ when $i \ge 2$.
\end{proof}

\subsection{Trace definibility in $(\Q;<)$}

By Proposition~\ref{prop:stable-0} $T$ is unstable if and only if $T$ trace defines $\dlo$, so the minimal unstable trace equivalence class is that of $\dlo$.
By Proposition~\ref{prop:airity} every Ramsey finitely homogeneous structure trace definable in $(\Q;<)$ is binary, so it is natural to ask if the converse holds.
More generally, it is natural to ask if every binary finitely homogeneous structure $\nip$ structure is trace definable in $(\Q;<)$.
By Corollary~\ref{cor:henson} this would go a long way towards showing that $\Upomega_2<\aleph_0$.
We show that some binary finitely homogeneous structures are trace equivalent to $(\Q;<)$.

\medskip
Suppose that $R$ is a $k$-ary relation on $M$.
We say that $R$ is \textit{supported} on $X\subseteq M$ if we have $\neg R(\alpha_1,\ldots,\alpha_k)$ when $\{\alpha_1,\ldots,\alpha_k\}\nsubseteq X$.
Given $X \subseteq M$ we let $R|_X$ be the relation on $M$ given by declaring $R|_X(\alpha_1,\ldots,\alpha_k)$ if and only if $\alpha_1,\ldots,\alpha_k \in X$ and $R(\alpha_1,\ldots,\alpha_k)$ and let $R\!\upharpoonright_X$ be the relation on $X$ given by restricting $R$ to $X$.

\begin{proposition}
\label{prop:in dlo}
Suppose that $L$ consists of finitely many unary relations and $L^*$ is an expansion of $L$ by two finite families $\Cal R$ and $\Cal E$ of binary relations.
Suppose that $\Sa M$ is an $L^*$-structure satisfying the following:
\begin{enumerate}[leftmargin=*]
\item Each $P \in \Cal P$ defines an infinite set.
\item For each $R \in \Cal R$ there is $P \in \Cal P$ such that $R$ is supported on $P$ and $R|_P$ is a linear order.
\item Every $E \in \Cal E$ is an equivalence relation on $M$.
\end{enumerate}
Finally, suppose  $\Sa M$ admits quantifier elimination.
Then $\Sa M$ is trace equivalent to $(\Q;<)$.
\end{proposition}

\begin{lemma}
\label{lem:eq relation}
Suppose that $E$ is an equivalence relation on $M$.
Let $E^*$ be the equivalence relation on $M^2$ given by $E^*((a,a'), (b,b')) \Longleftrightarrow a = b$.
Then $(M;E)$ embedds into $(M^2;E^*)$.
\end{lemma}

\begin{proof}
As $|M/E| \le |M|$ we fix a function $\upsigma \colon M \to M$ such that $E(a,b) \Longleftrightarrow \upsigma(a) = \upsigma(b)$.
Then the map $M \to M^2$, $a \mapsto (\upsigma(a),a)$ gives an embedding $(M;E) \to (M^2;E^*)$.
\end{proof}

We now prove Proposition~\ref{prop:in dlo}.

\begin{proof}
Note that $\Sa M$ is unstable, hence $\Th(\Sa M)$ trace defines $(\Q;<)$.
After passing to a countable elementary substructure we suppose that $\Sa M$ is countable.
We show that $(\Q;<)$ trace defines $\Sa M$.
We may suppose that $L = \{P_1,\ldots,P_n\}$.
For each $I \subseteq \{1,\ldots,n\}$ let $P_I = \left(\bigwedge_{i \in I} P_i\right) \land \left(\bigwedge_{i \notin I} \neg P_i\right)$.
After replacing $L$ with the set of $P_I$ that define an infinite set and replacing $\Cal R$ with the set of $R|_{P_I}$ such that $R\!\upharpoonright_{P_I}$ defines a linear order we suppose that the $P_i$ define pairwise disjoint sets.
Let $\Sa M_L$ be the $L$-reduct of $\Sa M$.
We consider $\Sa M$ as the fusion of the following structures: $(\Sa M_L,R)$ for each $R \in \Cal R$ and $(\Sa M_L,E)$ for each $E \in \Cal E$.
By Proposition~\ref{prop:fusion} it is enough to show that each $(\Sa M_L,R)$ and $(\Sa M_L,E)$ is a substructure of a structure trace definable in $(\Q;<)$.
Fix $R \in \Cal R$.
It is easy to see that $(\Sa M_L,R)$ is a substructure of a countable $L \cup \{R\}$-structure $\Sa P$ such that $R$ defines a dense linear order on some $P_i$.
It is easy to see that $\Sa P$ is definable in $(\Q;<)$.
Now fix $E \in \Cal E$.
Then we consider $(\Sa M_L,E)$ to be the fusion of $(M;E)$ with the structures $(M;P)$ for $P \in \Cal P$.
Each $(M;P)$ is interpretable in the trivial structure $M$ and by Lemma~\ref{lem:eq relation} $(M;E)$ is a substructure of a structure interpretable in the trivial structure $M$.
\end{proof}

We now give some special cases of Proposition~\ref{prop:in dlo}.

\medskip
Given $2\le \uplambda \le \upomega$ we let $\dloeq^\uplambda$ be the theory of $(L;<,E)$ where $(L;<) \models \dlo$ and $E$ is an equivalence relation on $L$ with $\uplambda$ classes each of which is dense and codense in $L$.
It is easy to see that each $\dloeq^\uplambda$ is $\aleph_0$-categorical and admits quantifier elimination.
Note that $\dloeq^2$ is just the theory of $(L;<,X)$ when $(L;<) \models \dlo$ and $X\subseteq L$ is dense and co-dense.

\begin{proposition}
\label{prop:dcd}
$\dlo$ is trace equivalent to, but does not interpret, $\dloeq^\uplambda$.
\end{proposition}

Hence $(\R;<)$ is trace equivalent to, but does not interpret, $(\R;<,\Q)$.

\begin{proof}
The first claim follows from Proposition~\ref{prop:in dlo}, the second claim from Corollary~\ref{cor:rama-1}.
\end{proof}


\noindent
A $k$-order is a set equipped with $k$ linear orders.
When $k \ge 2$ these are sometimes referred to as ``permutation structures".
When $X$ is a finite $k$-order $(X;\prec_0,\ldots,\prec_{k - 1})$ may be identified with the unique sequence $\upsigma_1,\ldots,\upsigma_{k - 1}$ of permutations of $X$ such that each $\upsigma_i$ gives an isomorphism $(X;\prec_0) \to (X;\prec_i)$.
Consider the case when $k = 2$, $X = \{1,\ldots,n\}$, and $\prec_0$ is $<$, and list the elements in the order given by $\prec_1$.
This is quite close to how one learns to think about permutations of $\{1,\ldots,n\}$ in a basic algebra course.

\medskip
Proposition~\ref{prop:k-order} follows from Proposition~\ref{prop:in dlo}.

\begin{proposition}
\label{prop:k-order}
All countable homogeneous $k$-orders are trace equivalent to $(\Q;<)$.
\end{proposition}

We let $\mathrm{DLO}_k$ be the theory of the generic countable $k$-order, so $\mathrm{DLO}_1 = \mathrm{DLO}$.
Suppose $k \ge 2$ and $(O;\prec_1,\ldots,\prec_k) \models \dlo_k$.
It is easy to see that any nonempty open $\prec_1$-interval is dense and codense in the $\prec_2$-topology.
By Corollary~\ref{cor:rama-1} $\dlo$ does not interpret $\dlo_k$.
Simon and Braunfeld have given a complex classification of homogeneous $k$-orders~\cite{Simon-Braunfeld}.

\medskip
We now prove a general result.
Theorem~\ref{thm:in dlo} is a corollary to Simons classification of primitive rank one $\aleph_0$-categorical unstable $\nip$ structures~\cite{Pierre}.
We assume some familiarity with this work.
The structure $\Sa M$ is \textit{not} \textbf{rank $\mathbf{1}$} if there is a definable family $\mathcal{X}$ of infinite subsets of $M$ and $n$ such that any intersection of $n$ distinct elements of $\mathcal{X}$ is empty and $\Sa M$ is \textit{not} \textbf{primitive} if there is a non-trivial $\emptyset$-definable equivalence relation on $M$.

\begin{theorem}
\label{thm:in dlo}
Suppose that $\Sa M$ is countable, $\aleph_0$-categorical, primitive, rank $1$, $\nip$, and unstable.
Then $\Sa M$ is trace equivalent to $(\Q;<)$.
\end{theorem}

A primitive rank $1$ finitely homogeneous unstable $\nip$ structure is trace equivalent to $(\Q;<)$.

\begin{proof}
Lemma~\ref{lem:stable-0} shows that $\Th(\Sa M)$ trace defines $(\Q;<)$.
(In fact by Simon~\cite{Pierre2} $\Sa M$ interprets $(\Q;<)$.)
So it is enough to prove the first claim.
We first recapitulate Simon's description from \cite[Section 6.6]{Pierre}.
We only describe the parts that we use and omit other details\footnote{If you look at Pierre's paper it's worth keeping in mind that, unlike him, we are free to add constants.}.
There is a set $W$, a finite collection $\Cal C$ of unary relations on $W$, another finite collection $\Cal P$ of unary relations on $W$, binary relations $(R_C : C \in \Cal C)$ on $W$, $k$, a $k$-to-one surjection $\uppi \colon  W \to M$, an enumeration $a_1,\ldots,a_k$ of $\uppi^{-1}(a)$ for each $a \in M$, two collections $\Cal P^* = (P_i : P \in \Cal P, i \in \{1,\ldots,k\})$ and $\Cal C^* = (C_i : C \in \Cal C, i \in \{1,\ldots,k\})$ of unary relations on $M$, and a collection of binary relations $\Cal R^* = (R^{ij}_C : C \in \Cal C, i,j \in \{1,\ldots,k\})$ on $M$ such that:
\begin{enumerate}
\item The $C \in \Cal C$ partition $W$.
\item Each $R_C$ defines a linear order on $C \in \Cal C$,
\item $C_i(a) \Longleftrightarrow C(a_i)$ for every $C \in \Cal C$, $a \in M$, and $i \in \{1,\ldots,k\}$.
\item $P_i(a) \Longleftrightarrow P(a_i)$ for every $P \in \Cal P$, $a \in M$, and $i \in \{1,\ldots,k\}$.
\item $R_C^{ij}(a,b) \Longleftrightarrow R_C(a_i,b_j)$ for every $C \in \Cal C$, $a,b \in M$, and $i,j \in \{1,\ldots,k\}$.
\item $\Sa W = (W;\Cal C,\Cal P, \Cal R)$ and $\Sa M^* = (M; \Cal C^*, \Cal P^*, \Cal R^*)$ both admit quantifier elimination.
\item $\Sa M^*$ is interdefinable with $\Sa M$.
\end{enumerate}

Proposition~\ref{prop:in dlo} shows that $\Sa W$ is trace definable in $(\Q; <)$.
We let $\uptau
\colon W \to \Q^n$ be an injection such that $(\Q;<)$ trace defines $\Sa W$ via $\uptau$.
Let $\Cal J = (J_C : C \in \Cal C)$ be a collection of $(\Q;<)$-definable subsets of $\Q^n$, $\Cal Q = \{Q_1,\ldots,Q_n\}$ be another collection of $(\Q;<)$-definable subset of $\Q^n$, and $\Cal S = (S_C : C \in \Cal C)$. be a collection of $(\Q;<)$-definable binary relations on $\Q^n$ such that $\uptau$  gives a isomorphism $\Sa W \to (\uptau(W); \Cal J, \Cal Q, \Cal S)$.

\medskip\noindent
We now let $\upsigma \colon M \to (\Q^n)^k$ be given by $\upsigma(a) = (\uptau(a_1),\ldots,\uptau(a_k))$.
We let $c = (c_1,\ldots,c_k)$ range over elements of $(\Q^n)^k$, likewise for $c'$.
For each $C_i \in \Cal C^*$ we let $J^*_{C_i}$ be the set of $c$ such that $c_i \in J_C$ and let $\Cal J^*$ be the collection of such $J^*_{C_i}$.
For each $P_i \in \Cal P$ we let $Q^*_i$ be the set of $c$ such that $c_i \in Q_i$ and let $\Cal Q^*$ be the collection of such $Q^*_i$.
For each $R^{ij}_C \in \Cal R$ let $S^{ij}_C$ be the collection of $(c,c')$ such that $S_C(c_i,c'_j)$ and let $\Cal S^*$ be the collection of such $S^{ij}_C$.
Let $P = \upsigma(M)$.
Note that $\upsigma$ gives an isomorphism $\Sa M^* \to (P; \Cal J^*,\Cal Q^*,\Cal S^*)$.
Quantifer elimination for $\Sa M^*$ and Corollary~\ref{cor:qe} together show that $\Sa M^*$ is trace definable in $(\Q;<)$ via $\upsigma$.
By (7) above $\Sa M$ is trace definable in $(\Q;<)$ via $\upsigma$.
\end{proof}




Fact~\ref{fact:lachlan-big} is a theorem of Lachlan~\cite{lachlan-order}.

\begin{fact}
\label{fact:lachlan-big}
If $\Sa M$ is finitely homogeneous and stable then $\Sa M$ is interpretable in $(\Q;<)$.
\end{fact}

Corollary~\ref{thm:in dlo 2} follows by Theorem~\ref{thm:in dlo} and Fact~\ref{fact:lachlan-big}.

\begin{corollary}
\label{thm:in dlo 2}
Suppose that $\Sa M$ is finitely homogeneous, primitive, rank $1$, and $\nip$.
Then $\Sa M$ is trace definable in $(\Q;<)$.
\end{corollary}








\medskip
We showed in Theorem~\ref{thm:trivial} that $\dlo$ does not trace define an infinite group.
Equivalently: there is an unstable theory that does not trace define an infinite group.
It seems reasonable to conjecture that a finitely homogeneous structure cannot trace define an infinite group.
Macpherson~\cite{macpherson-interpreting-groups} showed that a finitely homogeneous structure cannot interpret an infinite group, but his proof does not seem to go through in our setting.

\subsection{Trees}
\label{section:trees}
We have seen some examples of finitely homogeneous $\nip$ structures that are trace definable in $(\Q;<)$.
We now describe an interesting finitely homogeneous structure $\Sa C$ which is $\nip$ and not trace equivalent to $(\Q;<)$.
This structure is not rosy.

\medskip
We make some definitions following Bodirsky, Jonsson, and Pham~\cite[3.2]{binary-C}, but our terminology is a bit different.
A \textbf{tree} is an acyclic connected graph with a distinguished vertex called the \textbf{root}.
Let $\Sa T$ be a tree.
A \textbf{leaf} is a non-root vertex with degree $1$.
We say that $\Sa T$ is \textbf{binary} if the root has degree $0$ or $2$ and all other vertices that are not leaves have degree $3$.
Given vertices $u,u^* \in T$ we declare $u \prec u^*$ when $u$ lies on the path from the root to $u^*$, so $\prec$ is a partial order.
A maximal subset of $\Sa T$ which is linearly ordered under $\prec$ is a \textbf{branch}.
Let $\B$ be the set of branches through $\Sa T$.
If $\Sa T$ is finite then each branch contains a unique leaf so we canonically identify $\B$ with the set of leaves.
Given $b, b^* \in \B$ we let $b \land b^*$ be the $\prec$-maximal $u \in T$ such that $u$ lies on both $b$ and $b^*$.
We define the canonical ternary $C$-relation on $\B$ by declaring
$C(a,b,b^*)  \Longleftrightarrow  a \land b \prec b \land b^*$.
\begin{center}
\includegraphics[scale=.3]{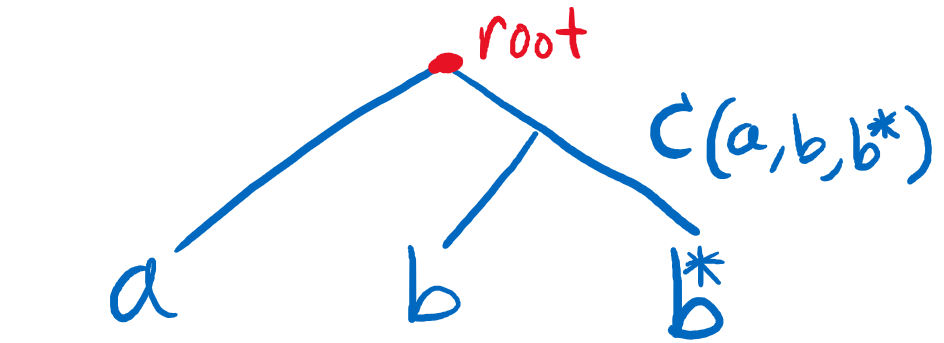}
\end{center}
We call $(\B;C)$ the \textbf{branch structure} of $\Sa T$.
This is an example of a $C$-set.
See \cite{adeleke-neumann,delon-no-density,C-minimal} for background and an axiomatic definition of a $C$-set.
If $\Sa T$ is binary then we say that $(\B;C)$ is a \textbf{binary branch structure}.
Note that if $\Sa T$ is binary and $\B$ is finite then $\Sa T$ is finite, so any finite binary branch structure is the branch structure of a finite binary tree.

\medskip
We describe a family of examples.
Suppose $1\le\uplambda \le \upomega$.
Let $\Upsigma$ be a (finite or infinite) alphabet, $\Upsigma^{<\uplambda}$ be the collection of finite words on $\Upsigma$ of length $< \uplambda$, and $\Upsigma^{\uplambda}$ be the collection of words of length $\uplambda$ on $\Upsigma$.
(If $\uplambda = \upomega$ then $\Upsigma^\uplambda$ is the just the usual collection of infinite words on $\Upsigma$.)
We consider $\Upsigma^{<\uplambda}$ to be a tree by taking the root to be the empty word and connecting $u$ to $v$ when either $v = u^\frown \sigma$ or $u = v^\frown \sigma$ for some $\sigma \in \Upsigma$.
(Here $u^\frown v$ is the usual concatenation.)
Then the set of branches through $\Upsigma^{<\uplambda}$ is canonically identified with $\Upsigma^\uplambda$.
Given $u,v \in \Upsigma^\uplambda$ we let $\mathrm{Ar}(u,v)$ be the maximal $n$ such that the length $n$ initial segments of $u$ and $v$ agree.
Then for $u,v,v^* \in \Upsigma^\uplambda$ we have $C(u,v,v^*) \Longleftrightarrow \mathrm{Ar}(u,v) < \mathrm{Ar}(v,v^*)$.

\medskip
The collection of finite binary branch structures is a \Fraisse class \cite[Proposition 7]{binary-C}.
We let $\Sa C$ be the \Fraisse limit of this class.
Then $\Sa C$ admits quantifier elimination and is therefore $C$-minimal, hence $\nip$.
(See \cite[A.1.4]{Simon-Book} for a definition of $C$-minimality and a proof that $C$-minimal structures are $\nip$.)
There is also a nice axiomatic definition of $\Sa C$, see \cite[3.3]{binary-C}.

\medskip
The structure $\Sa C$ is not so familiar, so we give some other definitions of non-trace definibility of $\Sa C$.
Proposition~\ref{prop:C} follows from Lemma~\ref{lem:substructure} and definition of $\Sa C$.

\begin{proposition}
\label{prop:C}
The following are equivalent:
\begin{enumerate}
\item $\Sa M$ trace defines $\Sa C$,
\item there is $m$ and $\Sa M$-definable $X \subseteq M^{m} \times M^{m} \times M^{m}$ such that for every finite binary branch structure $(\B;C)$ there is an injection $\uptau \colon \B \to M^m$ such that
\[
C(a,b,b^*) \quad \Longleftrightarrow \quad (\uptau(a),\uptau(b), \uptau(b^*)) \in X.
\]
\end{enumerate}
\end{proposition}

\medskip
Let $C_\bin$ be the canonical $C$-relation on $\{0,1\}^{\upomega}$ and let $\bin = (\{0,1\}^\upomega; C_\bin)$.

\begin{fact}
\label{fact:finite tree}
$\age(\Sa C) = \age(\bin)$, i.e. $\age(\bin)$ is the class of finite binary branch structures.
\end{fact}

We leave the easy verification of Fact~\ref{fact:finite tree} to the reader.
Lemma~\ref{lem:tree 0} follows from Lemma~\ref{lem:substructure} and Fact~\ref{fact:finite tree}.

\begin{lemma}
\label{lem:tree 0}
The following are equivalent:
\begin{enumerate}[leftmargin=*]
\item $T$ trace defines $\Sa C$.
\item there is $\Sa M \models T$, an injection $\uptau : \{0,1\}^{\upomega} \to M^m$, and $\Sa M$-definable $Y \subseteq M^{3m}$ so that
\[ C_\bin(b,b',b'') \quad \Longleftrightarrow \quad (\uptau(b),\uptau(b'),\uptau(b'')) \in Y \quad \text{for all} \quad b,b',b'' \in \{0,1\}^{\upomega}.\]
\end{enumerate}
\end{lemma}

\medskip
Our next goal is to show that $\dlo$ cannot trace define $\Sa C$.
This will go through indiscernible collapse and Proposition~\ref{prop:airity}.
As $\Sa C$ is not Ramsey we will need to pass to another structure.

\medskip
Suppose that $(\B;C)$ is a branch structure.
Let $C_{\alpha\beta} = \{ \beta^* \in \B : C(\alpha,\beta,\beta^*) \}$ for all $\alpha,\beta \in B$.
A \textbf{convex order} on $\B$ is a linear order such that each $C_{\alpha\beta}$ is convex and a \textbf{convexly ordered branch structure} is an expansion of a branch structure by a convex ordering.

\medskip
If $\Upsigma$ is equipped with a linear order then the resulting lexicographic order on $\Upsigma^\uplambda$ is convex.
Let $\Sa T$ be a finite tree and let $\Cal E$ be the collection of topological embeddings of $\Sa T$ into the upper half plane $\{(a,b) \in \R^2 : b \ge 0\}$ which take all leaves to the boundary.
Then any $f \in \Cal E$ induces an ordering $<_f$ on $\B$ in the natural way, and $<_f$ is convex.
Any convex order on $\B$ is of the form $<_f$ for a unique-up-to-boundary-preserving isotopy $f \in \Cal E$.
(I am sure this is known, but in any event it is easy to prove all of this via induction on finite trees.)
\begin{center}
\includegraphics[scale=.46]{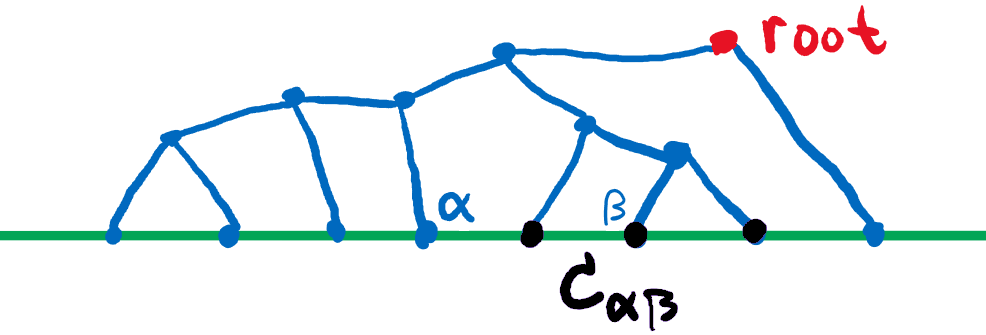}
\end{center}
The collection of finite convexly ordered binary branch structures is a \Fraisse class with the Ramsey property and there is an ordering $\triangleleft$ on $\Sa C$ such that $(\Sa C,\triangleleft)$ is the \Fraisse limit of the class of finite convexly ordered binary branch structures \cite[Prop~17, Thm~31]{binary-C}.
(The Ramsey property follows by Milliken's Ramsey theorem for trees~\cite{Milliken}.)

\medskip
Corollary~\ref{cor:C-tree} follows from Corollary~\ref{cor:add order}.

\begin{corollary}
\label{cor:C-tree}
$\Sa C$ and $(\Sa C, \triangleleft)$ are trace equivalent.
Hence $T$ trace defines $\Sa C$ if and only if some $T$ admits an uncollapsed indiscernible picture of $(\Sa C,\triangleleft)$.
\end{corollary}

Now we can show that $\dlo$ does not trace define $\Sa C$.

\begin{proposition}
\label{prop:airity tree}
A binary theory cannot trace define $\Sa C$.
\end{proposition}



\begin{proof}
Suppose that $T$ trace defines $\Sa C$.
By Corollary~\ref{cor:C-tree} $T$ trace defines $(\Sa C,\triangleleft)$.
By Proposition~\ref{prop:airity} it is enough to show that $(\Sa C,\triangleleft)$ is ternary.
By Fact~\ref{fact:airity} it is enough to produce $a_1,a_2,a_3,b_1,b_2,b_3$ from $(\Sa C,\triangleleft)$ such that $\tp(a_1a_2 a_3) \ne \tp(b_1b_2 b_3)$ and $\tp(a_{i}a_{j}) = \tp(b_{i}b_{j})$ for all distinct $i,j \in \{1,2,3\}$.
By quantifier elimination and the definition of $(\Sa C,\triangleleft)$ it is enough to produce convexly ordered binary branching trees $\Sa T, \Sa T^*$, branches $a_1,a_2,a_3$ in $\Sa T$, and branches $b_1,b_2,b_3$ in $\Sa T^*$ such that the induced ordered branch substructures on $\{a_i,a_j\}$ and $\{b_i,b_j\}$ are isomorphic for all distinct $i,j \in \{1,2,3\}$ and the induced ordered branch substructures on $\{a_1,a_2,a_3\}$ and $\{b_1,b_2,b_3\}$ are not isomorphic.
Consider the figure below.
\begin{center}
\includegraphics[scale=.33]{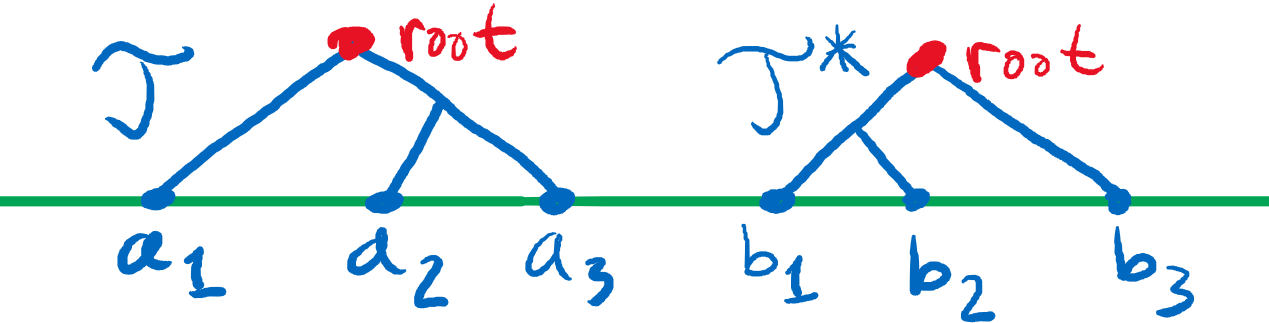}
\end{center}
\end{proof}

By \cite[Corollary~A.6.6]{Hodges} any colored linear order is binary, Proposition~\ref{prop:lo tree} follows.

\begin{proposition}
\label{prop:lo tree}
A colored linear order cannot trace define $\Sa C$.
\end{proposition}

We now give some examples of $\nip$ structures which trace define $\Sa C$.
A partial order $(O;\prec)$ is \textbf{semilinear} if it is upwards directed and $\{\alpha\in O:\beta\prec\alpha\}$ is a linear order for all $\beta\in O$.
A semilinear order $(O;\prec)$ is \textbf{branching} if for every $\alpha \in O$ there are incomparable $\alpha\prec\beta,\beta^*$.

\begin{proposition}
\label{prop:semi order}
The theory of any branching semilinear order trace defines $\Sa C$.
\end{proposition}

\begin{proof}
Suppose that $(O;\prec)$ is an $\aleph_1$-saturated branching semilinear order.
Applying induction construct a family $(\beta_u : u \in \{0,1\}^{<\upomega})$ such that for all $u \in \{0,1\}^{<\upomega}$ we have
\begin{enumerate}
\item $\beta_u \prec \beta_{u^\frown 0}, \beta_{u^\frown 1}$, and
\item $\beta_{u^\frown 0}$ and $\beta_{u^\frown 1}$ are incomparable.
\end{enumerate}
Semilinearity implies that if $u,v \in \{0,1\}^\upomega$ are not comparable in the tree order on $\{0,1\}^\upomega$ then $\beta_u,\beta_v$ are not comparable under $\prec$.
For each $u \in \{0,1\}^\upomega$ fix $\beta_u \in O$ such that $\beta_u \succ \beta_v$ for every initial segment $v \in \{0,1\}^{<\upomega}$ of $u$.
Let $Y$ be the set of $(a,b,b^*) \in O^3$ such that there $b,b^*$ are incomparable and there is $c \in O$ such that $b,b^* \succ c$ and $c$ is incomparable to $a$.
Then $Y$ is $(O;\prec)$-definable.
Let $\uptau \colon \{0,1\}^\upomega \to O$ be given by $\uptau(u) = \beta_u$.
It is easy to see that $\uptau$ and $Y$ satisfy the conditions of Lemma~\ref{lem:tree 0}.
\end{proof}

The Adeleke-Macpherson construction shows that any dense $C$-set interpret a branching semilinear order~\cite[12.4]{adeleke-neumann}.
Hence Proposition~\ref{prop:dense C relation} follows from Proposition~\ref{prop:semi order}.

\begin{proposition}
\label{prop:dense C relation}
The theory of any dense $C$-set trace defines $\Sa C$.
\end{proposition}

It is well-known that if $v$ is a non-trivial valuation on $K$ then we can define a dense $C$-relation on $K$ by declaring $C(\alpha,\beta,\beta')\Longleftrightarrow v(\alpha-\beta)<v(\beta-\beta')$.
Corollary~\ref{cor:valued field}.

\begin{corollary}
\label{cor:valued field}
The theory of any non-trivially valued field trace defines $\Sa C$
\end{corollary}

We show that a structure which defines a dense metric space trace defines $\Sa C$.
We work with a rather general class of metric spaces.
Let $X$ be a set and $\mathbb{L}$ be a linear order with a minimal element $0$.
We say that an $\mathbb{L}$-valued metric on $X$ is a function $d \colon X^2 \to \mathbb{L}$ such that:
\begin{enumerate}
\item for all $a,a' \in X$ we have $d(a,a') = d(a',a)$ and $d(a,a') = 0$ if and only if $a = a'$.
\item for every $r \in \mathbb{L}$ and open interval $r\in I \subseteq \mathbb{L}$ there is positive $s \in \mathbb{L}$ such that if $d(a,b) = r$ and $d(a,a'),d(b,b') < s$ then $d(a,b) \in I$.
\end{enumerate}
(3) is a weak form of the triangle inequality.
Note that by taking $r = 0$ in (3) we see that for any positive $r' \in \mathbb{L}$ there is positive $s \in \mathbb{L}$ such that $d(a,b),d(b,c) < s$ implies $d(a,c) < r'$.
Note that if $(H;+,\prec)$ is an ordered abelian group and $(X,d)$ is a metric space (in the usual sense) taking values in $H$ then $(X,d)$ is an $\{b\in H: b\ge 0\}$-valued metric space in our sense.

\medskip
Given $a \in X$ and $r \in \mathbb{L}$ we let $B(a,r)$ be the set of $a' \in X$ such that $d(a,a') < r$ and note that the collection of such sets forms a basis for a Hausdorff topology on $X$.
We say that an $\mathbb{L}$-valued metric space $(X,d)$ is $\Sa M$-definable if $\mathbb{L}$ is an $\Sa M$-definable linear order, $X$ is a definable set, and $d$ is a definable function.

\begin{proposition}
\label{prop:uniform 0}
Suppose that $(X,d)$ is an $\Sa M$-definable $\mathbb{L}$-valued metric space with no isolated points.
The $\Th(\Sa M)$ trace defines $\Sa C$.
\end{proposition}

\begin{proof}
We suppose that $\Sa M$ is $\aleph_1$-saturated.
Working inductively we produce  descending sequences $(r_i : i < \upomega)$ and $(d_i : i < \upomega)$ of elements of $\mathbb{L}$ and a family $(\alpha_u : u \in \{0,1\}^{<\upomega})$ of non-zero elements of $X$ such that for all $n$ and $u \in \{0,1\}^{< \upomega}$ of length $n$ we have
\begin{enumerate}
\item $d(x,y),d(y,z) < r_n$ implies $d(x,z) < d_n$.
(If the metric takes values in divisible ordered abelian group we set $r_n = d_n/2$.)
\item $B(\alpha_{u^\frown 0}, r_{n + 1})$ and $B(\alpha_{u^\frown 1}, r_{n + 1})$ are contained in  $B(\alpha_u, r_n)$, hence if $\beta_i \in B(\alpha_{u^\frown i},r_{n + 1})$ for $i \in \{0,1\}$ then $d(\beta_0,\beta_1) < d_n$.
\item if $\beta_i \in B(\alpha_{u^\frown i}, r_{n + 1})$ for $i \in \{0,1\}$ then $d(\beta_0,\beta_1) >  d_{n + 2}$.
\end{enumerate}
Let $B_u = B(a_u, r_{n})$ for $u \in \{0,1\}^{<\upomega}$ of length $n$.
Applying saturation, we let $\uptau \colon \{0,1\}^\upomega \to X$ be such that $\uptau(b) \in B_u$ when $u$ is an initial segment of $b$.
Let $Y$ be the set of $(a,b,b^*) \in X^3$ with $d(b,b^*) < d(a,b)$.
Then $\uptau$ and $Y$ satisfy the conditions of Lemma~\ref{lem:tree 0}.
\end{proof}

\begin{proposition}
\label{prop:oag tree}
The theory of any ordered abelian group trace defines $\Sa C$.
\end{proposition}

\begin{proof}
By Prop~\ref{prop:re-oag} it suffices to show that $\doag$ trace defines $\Sa C$.
Apply Prop~\ref{prop:uniform 0}.
\end{proof}

$\nip$ theories that don't trace define $\Sa C$ should be very close to stable theories or linear orders.
Of course trivial o-minimal structures are the o-minimal structures closest to linear orders.

\begin{proposition}
\label{prop:o min triv}
If $T$ is o-minimal then $T$ trace defines $\Sa C$ if and only if $T$ is non-trivial.
\end{proposition}

\begin{proof}
Suppose $T$ is trivial.
By a theorem of Mekler, Rubin, and Steinhorn~\cite{mekler-rubin-steinhorn} $T$ is binary, hence $T$ cannot trace define $\Sa C$ by Proposition~\ref{prop:airity tree}.
Suppose $\Sa M$ is non-trivial.
By Theorem~\ref{thm:trivial} $T$ trace defines $\doag$.
By Proposition~\ref{prop:oag tree} $\doag$ trace defines $\Sa C$.
\end{proof}

It is a famous conjecture that any unstable $\nip$ field admits a definable field order or a definable non-trivial valuation.
By Corollary~\ref{cor:valued field} and Prop~\ref{prop:uniform 0}, this conjecture implies that the theory of an unstable $\nip$ field trace defines $\Sa C$.
The dp finite case of the conjecture is a theorem of Johnson~\cite{Johnson-dp-finite}, so the theory of an unstable dp finite field trace defines $\Sa C$.

\subsection{Stable?}
We have shown that $3\le\Upomega_2\le\aleph_0$.
Note that $\Upomega_2=3$ iff the following holds:
\begin{enumerate}[leftmargin=*]
\item Any stable binary finitely homogeneous structure is trace equivalent to the trivial theory.
\item Any unstable $\nip$ binary finitely homogeneous structure is trace equivalent to $(\Q;<)$.
\item An $\mathrm{IP}$ binary finitely homogeneous structure is trace equivalent to the Erd\H{o}s-Rado graph.
\end{enumerate}

We proved (3), see Corollary~\ref{cor:henson}.
A binary finitely homogeneous structure is rosy of finite thorn rank by \cite[Lemma~7.1]{Pierre}.
So Theorem~\ref{thm:in dlo} gives the primitive rank $1$ version of (2).
We make some comments concerning (1).
Let $\mathbb{D}$ be the class of $\aleph_0$-stable, $\aleph_0$-categorical structures $\Sa M$ such that every interpretable strongly minimal set is trivial and let $\mathbb{D}_\mathrm{tot}$ be the class of totally categorical structures in $\mathbb{D}$.
Lachlan showed that every stable finitely homogeneous structure is in $\mathbb{D}$, he also showed that any $\Sa M \in \mathbb{D}$ expands to some $\Sa M^*\in\mathbb{D}_\mathrm{tot}$ \cite[Theorem~2.6]{lachlan-order}.
It is tempting to try to show that any structure in $\mathbb{D}_\mathrm{tot}$ is trace definable in a trivial structure.
I haven't been able to get this to work, but it might be easy for someone who actually knows something about totally categorical structures.

\medskip
A structure in a countable language interpretable in a trivial structure is $\aleph_0$-categorical, $\aleph_0$-stable, and trivial.
According to Hrushovski there are $\aleph_0$-categorical, $\aleph_0$-stable, trivial structures not interpretable in a trivial theory~\cite{szymon}.
It seems that any abstract characterization of theories interpretable in the trivial theory must be subtle.
Is there a nice abstract characterization of \textit{trace minimal} theories, i.e. theories trace definable in the trivial theory?

\newpage
\appendix

\section{Non-interpretation results}
\label{section:o-min}
We give above many examples of interesting structures $\Sa M$, $\Sa O$ such that $\Sa M$ trace defines but does not interpret $\Sa O$.
The results of this section are used to get the non-interpretations.

\medskip
We let $<_{\mathrm{Lex}}$ be the lexicographic order.

\subsection{Interpretations in o-minimal structures}

Fact~\ref{fact:rigid} is the rigidity result for interpretations between o-minimal expansions of fields.
It is due to Otero, Peterzil, and Pillay~\cite{OPP-groups-rings}.

\begin{fact}
\label{fact:rigid}
Let $\Sa R$ be an o-minimal expansion of an ordered field $\mathbf{R}$, $\mathbf{C} = \mathbf{R}[\sqrt{-1}]$, and $\mathbf{F}$ be an infinite field interpretable in $\Sa R$.
There is a definable isomorphism $\mathbf{F} \to \mathbf{R}$ or $\mathbf{F} \to \mathbf{C}$.
\end{fact}

Suppose $\Sa F$ is an expansion of an ordered field  interpretable in $\Sa R$.
Fact~\ref{fact:rigid} shows that $\Sa F$ is isointerdefinable with a structure intermediate between $\mathbf{R}$ and $\Sa R$.
In particular $\Sa F$ is o-minimal.
The analogue for ordered abelian groups fails.
For example $(\R;+,<)$ interprets the non-o-minimal structure $(\R^2;+,<_{\mathrm{Lex}}, \{0\} \times \R)$.
(Note that $\{0\}\times \R$ is a proper convex subgroup.)
We prove weaker results for expansions of ordered abelian groups.

\begin{fact}
\label{fact:discrete}
If $\Sa O$ is an oag interpretable in an o-minimal expansion of an oag then $\Sa O$ is dense.
Hence $(\Z;+,<)$ is not interpretable in an o-minimal expansion of an ordered abelian group.
\end{fact}

\begin{proof}
A discrete ordered abelian group does not eliminate $\exists^\infty$ and an o-minimal expansion of an ordered abelian group eliminates imaginaries and eliminates $\exists^\infty$.
\end{proof}

\begin{fact}
\label{fact:rama}
Suppose that $\Sa M$ is an o-minimal expansion of an ordered abelian group and $(L;\triangleleft)$ is a definable linear order.
Then there is a definable embedding $(L;\triangleleft) \to (M^k;\lex)$.
\end{fact}

Fact~\ref{fact:rama} is due to Ramakrishnan~\cite{ramakrishnan}.

\begin{corollary}
\label{cor:rama}
Suppose that $\Sa M$ is an o-minimal expansion of an ordered abelian group and $(L;\triangleleft)$ is a dense linear order interpretable in $\Sa M$.
Then there are non-empty open intervals $I \subseteq M$ and $J \subseteq L$ and a definable isomorphism $\upiota : (J;\triangleleft) \to (I;<)$.
If $\dim L = 1$ then we may suppose that $J$ is cofinal in $(L;\triangleleft)$.
\end{corollary}

Corollary~\ref{cor:rama} fails without the assumption that $(L;\triangleleft)$ is dense, consider $([0,1] \times \{0,1\}; \lex)$.
(This is known as the ``double arrow space" or ``split interval".)


\begin{proof}
We may suppose that $(L;\triangleleft)$ is definable in $\Sa M$ as $\Sa M$ eliminates imaginaries.
By Fact~\ref{fact:rama} we may suppose that $L \subseteq M^k$ and that $\triangleleft$ is the restriction of the lexicographic order to $L$.
We now prove the first claim.
We apply induction on $k$.
Suppose that $k = 1$.
As $L$ is infinite $L$ contains a nonempty open interval $I 
\subseteq M$, so we let $\upiota$ be the identity $I \to I$.
Suppose $k \ge 2$.
Let $\uppi \colon L \to M^{k - 1}$ be given by $\uppi(x_1,\ldots,x_k) = \uppi(x_1,\ldots,x_{k - 1})$.
Note that each $\uppi^{-1}(b)$ is a $\lex$-convex subset of $L$ and that the restriction of $\lex$ to any $\uppi^{-1}(\beta)$ agrees with the usual ordering on the $k$th coordinate.
Suppose that $\beta \in M^{k - 1}$ and $|\uppi^{-1}(\beta)| \ge 2$.
Then $\uppi^{-1}(\beta)$ is infinite as $(L;\triangleleft)$ is dense, so there is a nonempty open interval $I \subseteq M$ such that $\{\beta\} \times I \subseteq \uppi^{-1}(\beta)$.
In this case we let $\upiota \colon I \to \{\beta \} \times I$ be $\upiota(x) = (\beta,x)$.
We may suppose that $\uppi$ is injective.
Note that $\uppi$ is a monotone map $(L;\lex) \to (M^{k-1};\lex)$, so $\uppi$ induces an isomorphism $(L;\lex) \to (\uppi(L);\lex)$.
Apply induction on $k$.

\medskip
We now suppose that $L$ is one-dimensional.
We again apply induction on $k$, the case when $k = 1$ is easy.
Suppose $k \ge 2$ and let $\uppi$ be as above.
Let $X$ be the set of $\beta \in M^{k - 1}$ such that $|\uppi^{-1}(\beta)| \ge 2$.
By one-dimensionality $X$ is finite.
Suppose $X$ is not cofinal in $(\uppi(L); \lex)$.
Fix $\alpha \in L$ with $\uppi(\alpha) > X$.
Replace $L$ with $(\alpha,\infty)$ and apply induction on $k$.
Suppose that $X$ is cofinal in $(\uppi(L);\lex)$ and let $\beta$ be the maximal element of $X$.
Replacing $L$ with $\uppi^{-1}(\beta)$ reduces to the case when $k = 1$.
\end{proof}

\noindent
Corollary~\ref{cor:rama-1} follows easily from Corollary~\ref{cor:rama}.

\begin{corollary}
\label{cor:rama-1}
Suppose that $\Sa O$ is an expansion of a dense linear order $(O;<)$, $I$ is a nonempty open interval, and $X \subseteq I$ is definable and dense and co-dense in $I$.
Then $\Sa O$ is not interpretable in an o-minimal expansion of an ordered group.
\end{corollary}

\begin{corollary}
\label{cor:rama-1.5}
Suppose that $\Sa M$ is o-minimal and $A$ is a dense and co-dense subset of $M$.
Then the structure $\Sa A$ induced on $A$ by $\Sa M$ is not interpretable in an o-minimal expansion of an ordered abelian group.
\end{corollary}

\begin{proof}
Suppose that $I$ is a nonempty open interval in $A$.
By co-density there are $\alpha,\beta \in M \setminus A$ which lie in the convex hull of $I$ in $M$.
By density $I \cap A$ is a nonempty open convex subset of $A$ and $I \cap A$ is $\Sa A$-definable.
Apply Corollary~\ref{cor:rama} and the fact that every definable convex set in an o-minimal structure is an interval.
\end{proof}

We say that an expansion $\Sa M$ of a linear order is \textbf{locally o-minimal} if for every $p \in M$ there is an interval $I \subseteq M$ containing $p$ such that the structure induced on $I$ by $\Sa M$ is o-minimal\footnote{Some authors use ``locally o-minimal" to mean something else.}.

\begin{corollary}
\label{cor:rama1/2}
Suppose that $\Sa O$ is an expansion of an ordered abelian group and $\Sa O$ is interpretable in an o-minimal expansion of an ordered group.
Then $\Sa O$ is locally o-minimal.
\end{corollary}

\begin{proof}
By Fact~\ref{fact:discrete} $\Sa O$ is dense.
By Corollary~\ref{cor:rama} there is a non-empty open interval $I \subseteq O$ on which the induced structure is o-minimal.
Fix $\alpha \in I$.
Then every $\beta \in M$ lies in $(\beta - \alpha) + I$ and the induced structure on $(\beta - \alpha) + I$ is is o-minimal by translating. 
\end{proof}

An expansion of a linear order $(M;<)$ is \textbf{definably complete} if every nonempty definable bounded above subset of $M$ has a supremum in $M$.
Fact~\ref{fact:on R} is well-known and easy to see.

\begin{fact}
\label{fact:on R}
Suppose that $\Sa M$ expands an oag $(M;+,<)$.
\begin{enumerate}
\item If $\Sa M$ is locally o-minimal then every definable subset of $M$ is the union of a definable open set and a definable closed and discrete set.
\item $\Sa M$ is o-minimal if and only if $\Sa M$ is definably complete and every definable subset of $M$ is the union of a definable open set and a finite set.
\end{enumerate}
\end{fact}

\begin{corollary}
\label{cor:rama-2}
Suppose that $\Sa Q$ is an expansion of an archimedean ordered abelian group $(Q;+,<)$ and $\Sa Q$ is interpretable in an o-minimal expansion of an oag.
Then $\Sa Q$ is o-minimal.
\end{corollary}

\noindent
This does not extend to expansions of non-archimedean ordered abelian groups.
Consider for example $(\R^2;+,<_{\mathrm{Lex}}, \{0\} \times \R)$.
By \cite{LasStein} any o-minimal expansion of an archimedean oag is an elementary substructure of an o-minimal expansion of $(\R;+,<)$.
Hence if $\Sa Q$ is as above then $\Sa Q$ is elementarily equivalent to an o-minimal expansion of $(\R;+,<)$.

\begin{proof}
By Fact~\ref{fact:discrete} $(Q;+,<)$ is dense.
By Corollary~\ref{cor:rama1/2} $\Sa Q$ is locally o-minimal.
By Fact~\ref{fact:on R}.1 any definable $X\subseteq Q$ is a union of a definable open set and a definable closed discrete set.

\medskip
Note that $\Sa Q$ eliminates $\exists^\infty$ as o-minimal expansions of ordered abelian groups eliminate both imaginaries and $\exists^\infty$.
We first show any closed and discrete definable subset of $Q$ is finite.
Suppose that $X$ is a closed and discrete subset of $Q$.
By Hahn embedding we suppose that $(Q;+,<)$ is a substructure of $(\R;+,<)$.
If $X$ has an accumulation point in $\R$ then $[\alpha,\alpha^*]\cap X$ can be made finite and arbitrarily large, contradiction.
Then $X$ is discrete and closed in $\R$, hence $[-\alpha,\alpha] \cap X$ is finite for any $\alpha \in Q$.
It follows by elimination of $\exists^\infty$ that there is $n$ such that $|[-\alpha,\alpha] \cap X| \le n$ for all $\alpha \in Q$, hence $|X| \le n$.

\medskip
By Fact~\ref{fact:on R} it is enough to show that $\Sa Q$ is definably complete.
Suppose that $C \subseteq Q$ is a nonempty bounded above downwards closed subset of $Q$ whose supremum is not in $Q$.
Note that for any non-empty open interval $I \subseteq Q$ there is $\alpha \in Q$ such that $I \cap (\alpha+C)$ and $I \setminus (\alpha+ C)$ are both nonempty, hence $I \cap (\alpha+ C)$ is convex set that is not an interval.
Hence for any nonempty open interval $I$ contains a definable convex set which is not an interval.
Apply Corollary~\ref{cor:rama}.
\end{proof}

\begin{proposition}
\label{prop:cyclic order 1}
$(\R/\Z;+,C)$ does not interpret an ordered abelian group.
\end{proposition}

We identify $\R/\Z$ with $[0,1)$ by identifying $\alpha + \Z$ with $\alpha \in [0,1)$.
The quantifier elimination for $(\R;+,<)$ shows that $(\R/\Z;+,C)$ is interdefinable with the structure induced on $[0,1)$ by $(\R;+,<)$.
Hence Proposition~\ref{prop:cyclic order 1} follows from Proposition~\ref{prop:semilinear group} and the fact that the induced structure eliminates imaginaries.

\begin{proposition}
\label{prop:semilinear group}
Suppose $(G;+,\triangleleft)$ is an $(\R;+,<)$-definable ordered abelian group.
Then $G$ is unbounded (as a definable subset of Euclidean space).
\end{proposition}

We use two facts about $(\R;+,<)$.
Both follow easily from the fact that if $f \colon \R^m \to \R^n$ is definable then there is a partition of $\R^m$ into definable $X_1,\ldots,X_k$ so that $f$ is affine on each $X_i$.
Both are well known, at least to me.

\begin{fact}
\label{fact:doag}
Work in $(\R;+,<)$.
Any image of a definable bounded set under a definable function is bounded.
If $\Cal F$ is a definable family of germs of functions $(0,1) \to (0,1)$ at $1$ such that $\lim_{t \to 1} f(t) = 1$ for all $f \in \Cal F$ then $\Cal F$ has only finitely many elements.
\end{fact}

We now prove Proposition~\ref{prop:semilinear group}.
There is probably a more straightforward proof.

\begin{proof}
Suppose towards a contradiction that $G$ is bounded.
By \cite{edmundo-otero} a torsion free definable group cannot be definably compact and by \cite[Theorem~1.3]{one-dim-subgroup} a definable group which is not definably compact has a  one-dimensional definable subgroup.
So we suppose $\dim G = 1$.
By \cite{razenj} a one-dimensional definable torsion free group is divisible, hence $(G;+)$ is divisible.
By Corollary~\ref{cor:rama} there is a cofinal interval $J$ in $(G;\triangleleft)$, an open interval $I \subseteq \R$, and a definable isomorphism $\upiota \colon (J;\triangleleft) \to (I;<)$.
By Fact~\ref{fact:doag} $I$ is bounded.
After translating, rescaling by a rational, and shrinking $J$, we suppose that $I = (0,1)$.
For each $\alpha \in G$ let $f_\alpha$ be the function germ $(0,1) \to (0,1)$ at $1$ where we have $f_\alpha(\upiota(s)) = \upiota(t)$ when $\alpha + s = t$.
Observe that $\lim_{t \to \infty} f_\alpha(t) = 1$ for all $\alpha \in G$.
By Fact~\ref{fact:doag} $(f_\alpha : \alpha \in G)$ contains only finitely many germs.
This is a contradiction as each $\alpha \in G$ is determined by the germ of $f_\alpha$ at $1$.
\end{proof} 

We have given a number of examples of structures that are trace equivalent to $(\Z;+,<)$.
We gather here some non-interpretation results between these structures.

\medskip
Fact~\ref{fact:Zapryagaev and Pakhomov} is due to Zapryagaev and Pakhomov~\cite{zap-pak}.

\begin{fact}
\label{fact:Zapryagaev and Pakhomov}
Any linear order interpretable in $(\Z;+,<)$ is scattered of finite Hausdorff rank.
In particular $(\Z;+,<)$ does not interpret $\dlo$.
\end{fact}

Fact~\ref{fact:zp} is proven in \cite{big-nip}.

\begin{fact}
\label{fact:zp}
Suppose that $\Sa Z$  is an $\mathrm{NTP}_2$ expansion of $(\Z;<)$ and $\Sa G$ is an expansion of a group which defines a non-discrete Hausdorff group topology.
Then $\Sa Z$ does not interpret $\Sa G$.
\end{fact}

A group is non-trivial if it contains more than one element.

\begin{fact}
\label{fact:pres group}
$\Th(\Z;+,<)$ does not interpret a non-trivial divisible abelian group.
\end{fact}

Fact~\ref{fact:pres group} follows from the work of Onshuus-V\'{i}caria~\cite{Onshuus-pres} and L\'{o}pez~\cite{Lopez} on definable groups in Presburger arthimetic.

\begin{proof}
Let $(Z;+,<)$ be a model of Presburger arithmetic.
Then $(Z;+,<)$ eliminates imaginaries so it is enough to show that $(Z;+,<)$ does not define a non-trivial divisible abelian group.
A subset of $Z^n$ is bounded if it is contained in $[-a,a]^n$ for some $a \in Z$.
Note that if $X \subseteq Z^n$ is a bounded definable set then the induced structure on $X$ is pseudofinite.
An obvious transfer argument shows that a non-trivial pseudofinite group cannot be divisible, so a divisible $(Z;+,<)$-definable group with bounded domain is trivial.
Suppose that $G$ is a $(Z;+,<)$-definable abelian group.
By \cite{Lopez} there is a definable exact sequence $0 \to (Z^m;+) \to G \to H \to 0$ for some $m$ and bounded definable group $H$.
If $G$ is divisible then $H$ is divisible, hence $H$ is trivial, hence $G$ is isomorphic to $(Z^m;+)$, contradiction.
\end{proof}

\begin{proposition}
\label{prop:general cyclic}
Suppose that $(H;+,\prec)$ is a dense regular ordered abelian group such that $|H/pH| < \infty$ for all primes $p$.
Then $(H;+,\prec)$ eliminates imaginaries and eliminates $\exists^\infty$.
Hence $(H;+,\prec)$ does not interpret Presburger arithmetic.
\end{proposition}

By Fact~\ref{fact:finite rank} this applies to  regular ordered abelian groups  of finite rank.
A dense oag can interpret Presburger arthimetic, consider the lexicographic product $(\Q;+,<)\times(\Z;+,<)$.

\begin{proof}
The last claim follows from the previous.
We apply Fact~\ref{fact:qe for oag}.
Note first that there is no infinite discrete definable subset of $(H;+,\prec)$, pass to an $\aleph_1$-saturated elementary extension, and show that this implies elimination of $\exists^\infty$.
To obtain elimination of imaginaries it is enough to let $(X_\alpha : \alpha \in Y)$ be a definable family of non-empty subsets of $H$ and show that there is a definable function $f \colon Y \to H$ such that $f(\alpha) \in X_\alpha$  and $f(\alpha) = f(\beta)$ when $X_\alpha = X_\beta$.
We first treat the case when each $X_\alpha$ is a finite union of intervals.
In this case we select  the midpoint of the minimal convex component of $X_\alpha$.
We now treat the general case.
By Fact~\ref{fact:qe for oag} there is $m$ such that every intersection of every $X_\alpha$ with a coset of $mH$ is of the form $\gamma + mX$ where $X \subseteq H$ is a finite union of intervals.
Let $\gamma_1,\ldots,\gamma_k$ be representatives of the cosets of $mH$.
For each $i \in \{1,\ldots,k\}$ let $f_i \colon H \to H$ be given by $f_i(\beta) = \gamma_i  + m\beta$.
Then $f^{-1}_i(X_\alpha)$ is a finite union of intervals for each $i \in \{1,\ldots,k\}$ and $\alpha \in H^n$.
For each $\alpha\in Y$ we fix the minimal $j \in \{1,\ldots,k\}$ such that $X_\alpha \cap (\gamma_j + mH) \ne \emptyset$ and select an element $\eta$ from $f^{-1}_j(X_\alpha)$ as above.
Then the selected element from $X_\alpha$ is $f_j(\eta)$.
\end{proof}

\begin{corollary}
\label{cor:cyclic order}
Suppose that $(G;+,C)$ is an archimedean cyclically ordered abelian group and $|G/pG| < \infty$ for all primes $p$.
Then $(G;+,C)$ does not interpret Presburger arithmetic.
\end{corollary}

By Fact~\ref{fact:finite rank} this applies to finite rank archimedean cyclically ordered abelian groups.

\begin{proof}
Let $(H;u,+,\prec)$ be the universal cover of $(G;+,C)$ as defined in Section~\ref{section:cyclic order}.
Then $(H;+,\prec)$ interprets $(G;+,C)$ so it is enough to show that $(H;+,\prec)$ does not interpret $(\Z;+,<)$.
Note that  $(H;+,\prec)$ is archimedean as the universal cover of an archimedean cyclic order is an archimedean oag.
Hence $(H;+,\prec)$ is regular.
By Proposition~\ref{prop:general cyclic} it is enough to show that $|H/pH| < \infty$ for all primes $p$.
This follows by applying Fact~\ref{fact:finite rank} to the exact sequence $0\to(\Z;+) \to H \to G \to 0$.
\end{proof}

Fact~\ref{fact:group in a group} is due to Evans, Pillay, and Poizat~\cite{group-in-a-group}.

\begin{fact}
\label{fact:group in a group}
Suppose that $A$ is an abelian group and $G$ is a group interpretable in $A$.
Then $G$ has a finite index subgroup which is definably isomorphic to $B/B'$ where $B'\subseteq B$ are $A$-definable subgroups of $A^m$ for some $m$.
\end{fact}

\begin{corollary}
\label{cor:group in a group}
Suppose that $A$ is a finite rank free abelian group.
If $B$ is an infinite group interpretable in $A$ then $B$ has a finite rank subgroup $B^*$ such that $B^*$ is free abelian and $\rank(A)$ divides $\rank(B^*)$.
\end{corollary}

\begin{proof}
Let $\rank(A)=n$.
We first prove the claim.
\begin{Claim*}
Suppose that $B$ is a definable subgroup of $A^m$.
Then $n$ divides $\rank(B)$.
\end{Claim*}
\begin{claimproof}
If $B$ is trivial then $n$ divides $\rank(B)=0$, so we suppose $B$  is non-trivial.
We apply induction on $m$.
Suppose $m=1$.
As $A$ is torsion free $B$ is infinite.
By Fact~\ref{fact:abelian qe} $B$ contains $\beta+kA$ for some $\beta\in A$ and $k\ge 1$, so $B$ contains $kA$.
Then $n=\rank(A)\ge \rank(B) \ge  \rank(kA) = n$.
Hence $\rank(B) = n$.
Suppose that $m \ge 1$, let $\uppi\colon A^m\to A$ be the projection onto the first coordinate, and let $B'$ be the set of $b\in A^{m-1}$ such that $(0,b)\in B$.
Then $B'$ is a definable subgroup of $A^{m - 1}$ and $\uppi(B)$ is a definable subgroup of $A$.
By induction $n$ divides both $\rank(B')$ and $\rank(\uppi(B))$.
We have an exact sequence $0\to B'\to B\to\uppi(B)\to 0$, hence $\rank(B) = \rank(B')+\rank(\uppi(B))$, so $n$ divides $\rank(B)$.
\end{claimproof}
By Fact~\ref{fact:group in a group} we suppose $G$ is of the form $B/B'$ for $A$-definable subgroups $B'\subseteq B$ of $A^m$.
By the claim $n$ divides $\rank(B)$ and $\rank(B')$, so $n$ divides $\rank(B/B')=\rank(B)-\rank(B')$.
As $B/B'$ is a finitely generated abelian group $B/B'$ has a free abelian subgroup of the same rank.
\end{proof}

Corollary~\ref{cor:group in a group 2} is immediate from Corollary~\ref{cor:group in a group}.

\begin{corollary}
\label{cor:group in a group 2}
$(\Z^n;+)$ interprets $(\Z^m;+)$ if and only if $n$ divides $m$.
\end{corollary}

\begin{proposition}
\label{prop:two sorts}
Any ordered abelian group interpretable in $(\Z;+)\sqcup\Sa M$ is definably isomorphic to an ordered abelian group interpretable in $\Sa M$.
If $\Sa M$ is an o-minimal expansion of an ordered abelian group then any oag interpretable in $(\Z;+)\sqcup\Sa M$ is divisible.
\end{proposition}

\begin{proof}
The second claim follows from the first, elimination of imaginaries for o-minimal expansions of ordered abelian groups, and the fact that a torsion free abelian group definable in an o-minimal structure is divisible.
Divisibility of torsion free definable abelian groups follows from work of Strzebonski on Euler characteristic and definable groups~\cite[Lemma~2.5, Prop~4.4, Lemma~4.3]{strzebonski}.
We prove the first claim.
Suppose $\Sa M$ is an arbitrary structure and $(H;+,\prec)$ is an ordered abelian group interpretable in $(\Z;+)\sqcup\Sa M$.
By work of Berarducci and  Mamino~\cite{two-disjoint-sorts} there is a definable short exact sequence $0\to A\to (H;+)\to B\to 0$ for a $(\Z;+)$-definable group $A$, and $\Sa M$-interpretable group $B$.
The image of $A$ is ordered by $\prec$, so $A$ is finite by stability of $(\Z;+)$.
As $(H;+)$ is torsion-free the image of $A$ is trivial, hence $(H;+)\to B$ is an isomorphism.
The structure induced on $\Sa M$ by $(\Z;+)\sqcup\Sa M$ is interdefinable with $\Sa M$, so $\Sa M$ interprets $(H;+,\prec)$.
\end{proof}

\begin{lemma}
\label{lem:dislo}
Suppose that $(D;\prec)$ is an infinite discrete linear order and $D_1,\ldots,D_n$ are subsets of $D$ such that $D=D_1\cup\cdots\cup D_n$.
Then $\Th(D_i;\prec)$ defines an infinite discrete linear order for some $i\in\{1,\ldots,n\}$.
\end{lemma}

\begin{proof}
After possibly passing to an elementary extension we suppose that $(D;\prec,D_1,\ldots,D_n)$ is $\aleph_1$-saturated.
Fix $\beta\in D$.
Then either $(-\infty,\beta]$ or $[\beta,\infty)$ is infinite.
We identify $\beta$ with $0$ and identify $\N$ with the set of $\alpha\in D$ such that $\alpha\ge\beta$ and $[\beta,\alpha]$ is finite.
Fix $j\in\{1,\ldots,n\}$ such that $D_j\cap\N$ is infinite.
Then $D_j\cap[0,n]$ is a discrete linear order for all $n$.
By saturation there is $\alpha\in D_j,\alpha>\N$ such that $D_j\cap[0,\alpha]$ is an infinite discrete linear order.
Hence $(D_j;\prec)$ defines an infinite discrete linear order.
\end{proof}

As above $\mathbf{s}(x)=x+1$.

\begin{proposition}
\label{prop:two sorts again}
$(\Z;\mathbf{s})\sqcup(\R;<)$ does not interpret an infinite discrete linear order.
\end{proposition}

\begin{proof}
First note that $(\Z;\mathbf{s})\sqcup(\R;<)$ eliminates imaginaries.
(A disjoint union of two structures which eliminate imaginaries eliminates imaginaries.)
Hence it is enough to show that $(\Z;\mathbf{s})\sqcup(\R;<)$ does not define an infinite discrete linear order.
Suppose otherwise.
Let $(X;\prec)$ be a $(\Z;\mathbf{s})\sqcup(\R;<)$-definable infinite discrete linear order.
Then $X$ is a finite union of products of $(\Z;\mathbf{s})$-definable sets and $(\R;<)$-definable sets.
Let $Y_1,\ldots,Y_n$, $Y'_1,\ldots,Y'_n$ be $(\Z;\mathbf{s})$, $(\R;<)$-definable sets, respectively, with $X=(Y_1\times Y'_1)\cup\cdots\cup(Y_n\times Y'_n)$.
Fix $a \in Y'_i$ and let $\triangleleft$ be the linear order on $Y_i$ where $\beta\triangleleft\beta^* \Longleftrightarrow (a,\beta)\prec(a,\beta^*)$.
Then $(Y_i;\triangleleft)$ is a $(\Z;\mathbf{s})$-definable linear order, hence $Y_i$ is finite.
Hence each $Y_1,\ldots,Y_n$ is finite.
Thus $(\{\beta\}\times Y'_i : i \in \{1,\ldots,n\},\beta\in Y_i)$ is a finite cover of $X$.
By Lemma~\ref{lem:dislo} there is $\beta\in Y_i$ such that $\Th(\{\beta\}\times Y'_i;\prec)$ defines an infinite discrete linear order.
Let $\triangleleft'$ be the linear order on $Y'_i$ given by $a\triangleleft' a^*\Longleftrightarrow (\beta,a)\prec(\beta,a^*)$.
Then $(Y'_i;\triangleleft')$ is definable in $(\R;<)$, so $\dlo$ interprets an infinite discrete linear order.
This is a contradiction by elimination of $\exists^\infty$.
\end{proof}

\begin{proposition}
\label{prop:euler}
$(\Z;\mathbf{s})$ does not interpret $(\N;\mathbf{s})$.
\end{proposition}

\begin{proof}
Any easy back-and-forth argument shows that $(\R;\mathbf{s})\equiv(\Z;\mathbf{s})$.
Hence it is enough to show that $(\R;+,<)$ does not interpret $\Th(\N;\mathbf{s})$.
As $(\R;+,<)$ eliminates imaginaries it's enough to show that an o-minimal structure $\Sa M$ does not define a model of $\Th(\N;\mathbf{s})$.
An application of o-minimal Euler characteristic shows that $\Sa M$ cannot define an injection $f\colon X \to X$ with $|X\setminus f(X)|=1$, so $\Sa M$ cannot define a model of $\Th(\N;\mathbf{s})$, see~\cite[4.2.4]{lou-book}.
\end{proof}

See \cite{ez-fields} for the definition of an \'ez field.
The unpublished Fact~\ref{fact:ez} follows by combining \cite{ez-fields} and the methods of \cite{pillay-p-adic}.
It should also hold in positive characteristic, but this is more complicated as definable functions are not generically Nash in positive characteristic \'ez fields.

\begin{fact}
\label{fact:ez}
Suppose that $K$ is an \'ez field of characteristic zero.
Then any infinite field definable in $K$ is definably isomorphic to a finite extension of $K$.
\end{fact}

\subsection{The generic $k$-hypergraph}
Fix $k \ge 2$.
We show below that the generic countable $k$-ary hypergraph is trace equivalent to the generic countable $k$-ary relation.
It is easy to see that the generic countable $k$-ary relation interprets the generic countable $k$-hypergraph, see the proof of Proposition~\ref{prop:random}.
In this section we show that the generic countable $k$-hypergraph does not interpret the generic countable $k$-ary relation.
I asked on mathoverflow if the Erd\H{o}s-Rado graph interprets the generic countable binary relation.
This question was answered by Harry West~\cite{harry-west}.
My proof is a generalization of his argument that covers hypergraphs.

\begin{theorem}
\label{thm:west}
Fix $k \ge 2$.
The generic countable $k$-hypergraph does not interpret the generic countable $k$-ary relation.
In particular the Erd\H{os}-Rado graph does not interpret the generic countable binary relation.
\end{theorem}

\noindent
We first describe definable equivalence relations in the generic countable $k$-hypergraph.
This is a special case of \cite[Proposition 5.5.3]{Kruckman-thesis}.

\begin{fact}
\label{fact:wei}
Suppose $(V;E)$ is the generic countable $k$-hypergraph, $A$ is a finite subset of $V$, $X$ is an $A$-definable subset of $V^n$, $\approx$ is an $A$-definable equivalence relation on $X$, and $p$ is a complete $n$-type over $A$ concentrated on $X$.
Then there is $I_p \subseteq \{1,\ldots,n\}$ and a group $\Sigma$ of permutations of $I_p$ such that if $\alpha = (\alpha_1,\ldots,\alpha_n)$ and $\beta = (\beta_1,\ldots,\beta_n)$ in $V^n$ realize $p$ then 
\[
\alpha \approx \beta \quad \Longleftrightarrow \quad \bigvee_{\upsigma \in \Sigma} \bigwedge_{i \in I_p} \alpha_{\upsigma(i)} = \beta_i.
\]
\end{fact}
Given $\alpha = (\alpha_1,\ldots,\alpha_n) \in V^n$ and $I \subseteq \{1,\ldots,n\}$ we let $\alpha|I$ be the tuple $(\alpha_i : i \in I)$.
Given an $n$-type $p$ over $A$ and a realization $\alpha$ of $p$ we let $p|I$ be the type of $\alpha|I$ over $A$.

\begin{proof} \textit{(Of Theorem~\ref{thm:west})}
Let $(V;E)$ be the generic countable $k$-hypergraph.
Suppose that $A \subseteq V$ is finite, $X \subseteq V^n$ is $A$-definable, $\approx$ is an $A$-definable equivalence relation on $X$, and $R^*$ is an $A$-definable binary relation on $X/\!\approx$ such that $(X/\!\approx; R^*)$ is the generic countable binary relation.
Let $R \subseteq X^2$ be the pre-image of $R^*$ under the quotient map $X^2 \to (X/\!\approx)^2$.
Let $S_n(A,X)$ be the set of complete $n$-types in $(V;E)$ over $A$ concentrated on $X$.

\medskip\noindent
For each $p \in S_n(A,x)$ fix $I_p \subseteq \{1,\ldots,n\}$ as in Fact~\ref{fact:wei} and let $p^* = p|I_p$.
We may suppose that $I_p$ is minimal.
Let $p$ be a type in the variables $x_1,\ldots,x_n$.
Minimality of $I_p$ ensures that $p\models x_i \ne x_j$ for distinct $i,j \in I_p$ and $x_i \ne a$ for all $i \in I_p$ and $a \in A$.
(If $p$ satisfies $x_i = a$ then we can remove $i$ from $I_p$ and if $p$ satisfies $x_i = x_j$ then we can remove $j$ from $I_p$.)

\begin{Claim}
Fix a finite subset $B$ of $V$, elements $b_1,\ldots,b_{k - 1}$ of $X$ with coordinates in $B$, an index $i \in \{1,\ldots,n\}$, a realization $a \in X$ of $p \in S_n(A,X)$, and let $a^* = a|I_p$.
Then the truth value of $R(b_1,\ldots,b_{i - 1}, a, b_{i + 1},\ldots,b_{k - 1})$ is determined by $\tp(a^*|AB)$.
\end{Claim}

\begin{claimproof}
Suppose that $\tp(c^*|AB) = \tp(a^*|AB)$.
As $c^*$ and $a^*$ have the same type over $A$ there is a realization $c$ of $p$ such that $c^* = I_p$.
As $a|I_p = c|I_p$ the definition of $I_p$ above shows that $a \approx c$.
By $\approx$-invariance of $R$ we have
\[
R(b_1,\ldots,b_{i - 1}, a, b_{i + 1},\ldots,b_{k - 1}) \quad\Longleftrightarrow\quad R(b_1,\ldots,b_{i - 1}, c, b_{i + 1},\ldots,b_{k - 1}).\]

\end{claimproof}



\medskip\noindent
Let $d = \max \{ |I_p| : p \in S_n(A,X)\}$.
If $d = 0$ then each $\approx$-class is $A$-definable, hence there are only finitely many classes by $\aleph_0$-categoricity.
Hence we may suppose $d \ge 1$.
Fix $r \in S_n(A,X)$ with $|I_r| = d$.
We also fix $m \ge 2$, which we will eventually take to be sufficiently large.

\medskip\noindent
Fix a realization $(\alpha_{i} : i \in I_r)$ of $r^*$.
By the remarks above we have $\alpha_{i} \ne \alpha_{j}$ when $i \ne j$ and no $\alpha_{i}$ is in $A$.
By the extension axioms we fix an array $\beta = (\beta_{ij} : i \in I_r, j \in \{1,\ldots,m\})$ of distinct elements of $V \setminus A$ so that for any distinct $i_1,\ldots,i_k \in I_r$  we have
\[
E(\beta_{i_1 j_1},\ldots,\beta_{i_k j_k}) \quad \Longleftrightarrow \quad E(\alpha_{i_1},\ldots,\alpha_{i_k}) \quad \text{for any  } j_1,\ldots,j_k \in \{1,\ldots,m\}.
\]
Hence the induced hypergraph on $\{ \beta_{i,\upsigma(i)} : i \in I_r \}$ is isomorphic to the induced hypergraph on $\{ \alpha_i : i \in I_r \}$ for all $\upsigma \colon I_r \to \{1,\ldots,m\}$.
Quantifier elimination for the generic countable $k$-hypergraph shows that $\beta_\upsigma := (\beta_{i \upsigma(i)} : i \in I_p)$ realizes $r^*$ for any $\upsigma \colon I_r \to \{1,\ldots,m\}$.

\medskip\noindent
We prove three more claims, the conjunction of these claims is a contradiction.

\medskip\noindent
For each $q \in S_n(A,X)$ we let $S_q$ be the set of types $\tp(b|A\beta)$ where $b \in V^{|I_q|}$ realizes $q^*$.
We let $S$ be the union of the $S_q$.
\begin{Claim}
$|S| \le |S_n(A,X)| \left(dm + 2^{\binom{md}{k - 1}}\right)^d$.
\end{Claim}

\begin{claimproof}
It is enough to fix $q \in S_n(A,X)$ and show that $|S_q| \le \left(dm + 2^{\binom{md}{k - 1}}\right)^d$.
Let $|I_q|  = e$ and let $b = (b_1,\ldots,b_e)$ realize $q^*$.
Again $\{b_1,\ldots,b_e\}$ and $\beta$ are both disjoint from $A$, so the $A$-type of $(\beta,b)$ depends only on the induced hypergraph on $(\beta,b)$.
The induced hypergraphs on $\beta$ and $b$ are fixed, so we only need to consider the relations between $\beta$ and $b$.
It is enough to show that for each $b_i$ there are at most $dm + 2^{\binom{md}{k - 1}}$ induced hypergraphs on $(\beta,b_i)$.
There are two cases for each $b_i$.
Either $b_i$ is equal to some $\beta_{jj^*}$ or it is not.
There are $md$ possibilities in the first case.
In the second case we run through all subsets of $\beta$ with cardinality $k - 1$ and decide if $b$ is connected to each subset.
There are $2^{\binom{md}{k - 1}}$ possibilities in the second case.
\end{claimproof}

\begin{Claim}
$2^{km^{d(k-1)}}\le |S|$.
\end{Claim}

\begin{claimproof}
Let $B$ be the set of $\beta_\upsigma$ for $\upsigma \colon I_r \to \{1,\ldots,m \}$.
Note that $|B| = m^d$.
Given $b \in X$ we define $\upeta_b \colon B^{k - 1} \times \{1,\ldots,k\} \to \{0,1\}$ by declaring 
\[
\upeta_b(c_1,\ldots,c_{k - 1}, i) = 1 \quad \Longleftrightarrow \quad R(c_1,\ldots,c_{i - 1}, b, c_{i + 1},\ldots,c_{k - 1}).
\]
The extension axioms for the generic countable $k$-ary relation show that for any function $\upeta \colon B^{k - 1} \times \{1,\ldots,k\} \to \{0,1\}$ there is $b \in X$ such that $\upeta = \upeta_b$.
Hence $\{ \upeta_b : b \in X \}$ contains $2^{k(m^d)^{k-1}} = 2^{km^{d(k-1)}}$ distinct functions.
Finally, Claim $1$ shows that if $q = \tp(b|A)$ and $b^* = b|I_q$ then $\upeta_b$ is determined by $\tp(b^*|AB)$, equivalently by  $\tp(b^*|A\beta)$.
\end{claimproof}

\begin{Claim}
We have $2^{k(m^d)^{k - 1}} > |S_n(A,X)| \left(dm + 2^{\binom{md}{k - 1}}\right)^d$ when $m$ is large enough.
\end{Claim}
When $m$ is large enough we have
\[
\left(dm + 2^{\binom{md}{k - 1}}\right)^d \le \left(2 \cdot 2^{\binom{md}{k - 1}}\right)^d = \left( 2^{\binom{md}{k-1} + 1} \right)^d = 2^{d \binom{md}{k - 1} + d}
\]
and
\[
d\binom{md}{k - 1} + d \le d\left(\frac{(md)^{k - 1}}{(k - 1)!} \right) + d =  \frac{d^{k}}{(k - 1)!} m^{k - 1} + d.
\]
So it is enough to show that if $m$ is sufficiently large then
\[
km^{d(k - 1)} > \frac{d^{k}}{(k - 1)!} m^{k - 1} + d.
\]
If $d \ge 2$ then this holds as $d(k - 1) > k - 1$.
Suppose $d = 1$.
Then we want to show that
\[
km^{(k - 1)} > \frac{1}{(k - 1)!} m^{k - 1} + 1
\]
when $m$ is sufficiently large.
This holds as $k > 1/(k - 1)!$.
\end{proof}




\section{Expansions realizing definable types}
\label{section:artem}
All results and proofs in this section are due to Artem Chernikov.
Our conventions in this section are a bit different.
We let $\monster \models T$ and suppose that $\Sa M$ is a small submodel of $\monster$.
In this section, and this section only, we consider types of infinite tuples.
Give an ordinal $\uplambda$, tuple $(\alpha_i : i < \uplambda)$, and $j < \uplambda$ we let $\alpha_{\le j}$ be the tuple $(\alpha_i : i \le j)$, likewise $\alpha_{< j}$.

\medskip\noindent
We say that $T$ has \textbf{property (D)} if for any small subset $A$ of parameters from $\monster$ and nonempty $A$-definable $X \subseteq \monsterset$ there is a definable type over $A$ concentrated on $X$.
We say that $\Sa M$ has property (D) if $\Th(\Sa M)$ does.
It is easy to see that a theory with definable Skolem functions has property (D).
Fact~\ref{fact:artem} is proven in \cite{on-forking}.

\begin{fact}
\label{fact:artem}
If $\Sa M$ expands a linear order and has dp-rank one then $\Sa M$ has property (D).
\end{fact}

We only use the weakly o-minimal case of Fact~\ref{fact:artem} so we explain why it is true.
Suppose that $T$ is weakly o-minimal and $X \subseteq \monsterset$ is $A$-definable and nonempty.
Then $X$ decomposes as a union of its convex components, and each convex component is $A$-definable. 
Let $C$ be some convex component of $X$.
If $C$ has a maximal element $\beta$ then $\beta$ is in the definable closure of $A$, hence $p = \tp(\beta|A)$ is definable.
If $C$ does not have a maximum then we let $p$ be given by declaring $Y \in p$ if and only if $Y$ is cofinal in $C$ for all $A$-definable $Y \subseteq \monsterset$.
By weak o-minimality this determines a complete type over $A$ which is $A$-definable by definition.

\medskip\noindent
We prove Theorem~\ref{thm:artem}.

\begin{theorem}
\label{thm:artem}
Suppose that $\Sa M$ has property (D).
Then there is $\Sa M \prec \Sa N$ such that
\begin{enumerate}
\item Every definable type over $\Sa M$ in finitely many variables is realized in $\Sa N$, and
\item if $a \in N^n$ then $\tp(a|M)$ is definable.
\end{enumerate}
\end{theorem}

We require Lemma~\ref{lem:artem}, which is easy and left to the reader.

\begin{lemma}
\label{lem:artem}
Suppose that $C \subseteq B$ are small subsets of $\monsterset$ and $a,a^*$ are tuples from $\monster$.
If $\tp(a|B)$ is $C$-definable and $\tp(a^*|Ba)$ is $Ca$-definable then $\tp(aa^*|B)$ is $C$-definable.
\end{lemma}

\begin{lemma}
\label{lem:artem 2}
Suppose that $T$ has property (D), $a$ is a tuple from $\monster$, $\tp(a|M)$ is definable, and $X \subseteq \monsterset$ is $Ma$-definable.
Then there is $a^* \in X$ such that $\tp(aa^*|M)$ is $M$-definable.
\end{lemma}

\begin{proof}
There is definable type $p$ over $Ma$ concentrated on $X$.
Let $a^*$ be a realization of $p$ and apply Lemma~\ref{lem:artem}.
\end{proof}

We now prove Theorem~\ref{thm:artem}.

\begin{proof}
Let $(p_\uplambda : \uplambda < |T| + |M|)$ be an enumeration of all definable types in finitely many variables over $M$.
As $\Sa M$ is a model each $p_\uplambda$ extends to an $M$-definable global type $p^*_\uplambda$.
Let $\alpha_{0}^{\uplambda}$ be a realization of the restriction of $p^*_\uplambda$ to $M(\alpha_{0}^{j} : j < \uplambda)$ for all $\uplambda<|T|+|M|$ and let $\alpha_0$ be the tuple $(\alpha_0^{\uplambda} : \uplambda < |T| + |M|)$.
Lemma~\ref{lem:artem} and induction show that $\tp(\alpha_0|M)$ is definable.

\medskip\noindent
We now construct a sequence $(\alpha_i : i < \upomega)$ of tuples in $\monster$ by induction.
Suppose we have $\alpha_0,\ldots,\alpha_{i - 1}$.
Let $(X_\upeta : \upeta < |\alpha_{< i}| + |T|)$ be an enumeration of all nonempty $M\alpha_{< i}$-definable sets in $\monster$.
Applying Lemma~\ref{lem:artem 2} and induction we choose elements $\alpha_{i}^\upeta$ of $\monsterset$ such that $\alpha_i^\upeta$ is in $X_\upeta$ and $\tp(\alpha_{< i}\alpha_{i}^{\le \upeta}|M)$ is definable.

\medskip\noindent
Let $N = M \cup \{ \alpha^i : i < \upomega\}$.
By Tarski-Vaught $N$ is the domain of a submodel $\Sa N$ of $\monster$.
(1) holds by choice of $\alpha_0$ and (2) holds as $\tp(\alpha^{< \upomega}|M)$ is definable.
\end{proof}

\section{Powers in $p$-adic fields}
In this section we fix a prime $p$, a finite extension $\K$ of $\Q_p$, and $n$.
We prove Fact~\ref{fact:p adic field}.

\begin{fact}
\label{fact:p adic field}
There is $m$ so that if $\alpha \in \Q_p$~is~an~$m$th~power~in~$\K$~then~$\alpha$~is~an~$n$th~power~in~$\Q_p$.
\end{fact}

I originally worked out the unramified case of Fact~\ref{fact:p adic field} and then asked about the ramified case on mathoverflow.
Will Sawin answered almost immediately and gave a proof of Fact~\ref{fact:p adic field} which involved a clever computation and some facts on binomial coefficients.
It turned out that all the clever parts were completely unnecessary and could be replaced by a much more prosaic argument.
I am sure that Fact~\ref{fact:p adic field} is far from original in any event.

\medskip
We first prove a general fact about profinite groups which must be well known.

\begin{fact}
\label{fact:profinite}
Suppose that $\mathbb{A}$ is a profinite abelian group, written additively.
For any open neighbourhood $U$ of $0$ there is $m$ such that $m\mathbb{A}\subseteq U$.
\end{fact}

\begin{proof}
As $\mathbb{A}$ is profinite the collection of open subgroups of $\mathbb{A}$ forms a neighbourhood basis at $0$.
It is enough to suppose that $U$ is an open subgroup of $\mathbb{A}$ and show that $m\mathbb{A} \subseteq U$ for some $m$.
As $\mathbb{A}$ is compact $U$ is of index $k \in \N$.
Fix $\beta \in \mathbb{A}$.
Then $\beta,2\beta,\ldots, (k + 1)\beta$ cannot lie in distinct cosets of $U$, so there are $1 \le i < i^* \le k + 1$ such that $i\beta, i^*\beta$ lie in the same coset of $U$.
Then $(i^* - i)\beta \in U$, hence $k! \beta \in U$.
Let $m = k!$.
\end{proof}

We now prove Fact~\ref{fact:p adic field}.
We assume some familiarity with the $p$-adics.

\begin{proof}
We let $\valp \colon \K^\times \to \R$ be the valuation on $\K$, so the restriction of $\valp$ to $\Q^\times_p$ is the usual $p$-adic valuation.
Recall that $\valp(\Q^\times_p) = \Z$ and $\valp(\K^\times) = (1/r)\Z$ where $r \in \N, r \ge 1$ is the ramification index of $\K/\Q_p$. 
Let $\V$ be the valuation ring of $\K$.

\medskip
We may suppose that $\alpha \ne 0$.
If $\alpha \notin \Z_p$ then $1/\alpha \in \Z_p$, $1/\alpha$ is an $n$th power in $\Q_p$ iff $\alpha$ is an $n$th power in $\Q_p$, and for any $m$, $1/\alpha$ is an $m$th power in $\K$ iff $\alpha$ is an $m$th power in $\K$.
Thus we suppose that $\alpha \in \Z_p$.
We have $\alpha = p^k \beta$ for $k = \valp(\alpha)$ and $\beta \in \Z^\times_p$.
Note that $\alpha$ is an $n$th power in $\Q_p$ when $n|k$ and $\beta$ is an $n$th power in $\Q_p$.

\medskip
We first produce $m^*$ such that if $\beta$ is an $m^*$th power in $\K$ then $\beta$ is an $n$th power in $\Q_p$, i.e. we treat the case $k = 1$.
For each $m$ we let $Q_m = (\gamma^m : \gamma \in \Z^\times_p)$  and $K_m = (\gamma^m : \gamma \in \V^\times)$.
Recall that $Q_m$, $K_m$ is an open subgroup of $\Z^\times_p$, $\V^\times$, respectively.
As $\Z^\times_p,\V^\times$ is a profinite abelian group Fact~\ref{fact:profinite} shows that $(Q_m : m \in \N)$, $(K_m : m \in \N)$ is a neighbourhood basis at $1$ for $\Z^\times_p$,$\V^\times$, respectively.
The topology on $\Q_p$ agrees with that induced on $\Q_p$ by the topology on $\K$, so there is $m^*$ such that $K_{m^*} \cap \Q_p \subseteq Q_n$.
Hence if $\beta \in \Z^\times_p$ is an $m^*$th power in $\K$ then $\beta$ is an $n$th power in $\Q_p$.

\medskip
We now let $m = rnm^*$.
Suppose that $\alpha = \gamma^{m}$ for $\gamma \in \K$.
Then
\[
k = \valp(\alpha) = m\valp(\gamma) = rnm^*\valp(\gamma).
\]
As $\valp(\gamma) \in (1/r)\Z$ we see that $nm^*|k$, hence $m^*|k$ and $n|k$.
As $\alpha$ is an $m^*$th power in $\K$ and $\beta = \alpha/p^k = \alpha/(p^{k/m^*})^{m^*}$ we see that $\beta$ is an $m^*$th power in $\K$.
By choice of $m^*$, $\beta$ is an $n$th power in $\Q_p$.
Hence $\alpha=p^k\beta=(p^{k/n})^n\beta$ is an $n$th power in $\Q_p$.
\end{proof}

\bibliographystyle{abbrv}
\bibliography{NIP}

\begin{thebibliography}{100}

\bibitem{adeleke-neumann}
S.~A. Adeleke and P.~M. Neumann.
\newblock Relations related to betweenness: their structure and automorphisms.
\newblock {\em Mem. Amer. Math. Soc.}, 131(623):viii+125, 1998.

\bibitem{AldE}
E.~Alouf and C.~d'Elb\'{e}e.
\newblock A new dp-minimal expansion of the integers.
\newblock {\em J. Symb. Log.}, 84(2):632--663, 2019.

\bibitem{toomanyII}
M.~Aschenbrenner, A.~Dolich, D.~Haskell, D.~Macpherson, and S.~Starchenko.
\newblock Vapnik-{C}hervonenkis density in some theories without the
  independence property, {II}.
\newblock {\em Notre Dame J. Form. Log.}, 54(3-4):311--363, 2013.

\bibitem{toomanyI}
M.~Aschenbrenner, A.~Dolich, D.~Haskell, D.~Macpherson, and S.~Starchenko.
\newblock Vapnik-{C}hervonenkis density in some theories without the
  independence property, {I}.
\newblock {\em Trans. Amer. Math. Soc.}, 368(8):5889--5949, 2016.

\bibitem{trans}
M.~Aschenbrenner, L.~van~den Dries, and J.~van~der Hoeven.
\newblock {\em Asymptotic differential algebra and model theory of
  transseries}, volume 195 of {\em Annals of Mathematics Studies}.
\newblock Princeton University Press, Princeton, NJ, 2017.

\bibitem{stewart-baldwin}
S.~Baldwin.
\newblock A complete classification of the piecewise monotone functions on the
  interval.
\newblock {\em Transactions of the American Mathematical Society},
  319(1):155--178, 1990.

\bibitem{BHP-pairs}
E.~Bar-Yehuda, A.~Hasson, and Y.~Peterzil.
\newblock A theory of pairs for non-valuational structures.
\newblock {\em J. Symb. Log.}, 84(2):664--683, 2019.

\bibitem{zaran}
A.~Basit, A.~Chernikov, S.~Starchenko, T.~Tao, and C.-M. Tran.
\newblock Zarankiewicz's problem for semilinear hypergraphs,
  2020,arXiv:2009.02922.

\bibitem{poly-regular}
O.~Belegradek.
\newblock Poly-regular ordered abelian groups.
\newblock In {\em Logic and algebra}, volume 302 of {\em Contemp. Math.}, pages
  101--111. Amer. Math. Soc., Providence, RI, 2002.

\bibitem{two-disjoint-sorts}
A.~Berarducci and M.~Mamino.
\newblock Groups definable in two orthogonal sorts.
\newblock {\em Israel J. Math.}, 208(1):413--441, 2015.

\bibitem{BV-one-based}
A.~Berenstein and E.~Vassiliev.
\newblock Weakly one-based geometric theories.
\newblock {\em J. Symbolic Logic}, 77(2):392--422.

\bibitem{BV-independent}
A.~Berenstein and E.~Vassiliev.
\newblock Geometric structures with a dense independent subset.
\newblock {\em Selecta Math. (N.S.)}, 22(1):191--225, 2016.

\bibitem{binary-C}
M.~Bodirsky, P.~Jonsson, and T.~V. Pham.
\newblock The reducts of the homogeneous binary branching {$C$}-relation.
\newblock {\em J. Symb. Log.}, 81(4):1255--1297, 2016.

\bibitem{borovik-nesin}
A.~Borovik and A.~Nesin.
\newblock {\em Groups of finite {M}orley rank}, volume~26 of {\em Oxford Logic
  Guides}.
\newblock The Clarendon Press, Oxford University Press, New York, 1994.
\newblock Oxford Science Publications.

\bibitem{brlas}
S.~Braunfeld and M.~Laskowski.
\newblock Characterizations of monadic nip.
\newblock {\em arXiv:2104.12989}, 2021.

\bibitem{Simon-Braunfeld}
S.~Braunfeld and P.~Simon.
\newblock The classification of homogeneous finite-dimensional permutation
  structures, 2020, arxiv:1807.07110.

\bibitem{Chatzidakis}
Z.~Chatzidakis.
\newblock Simplicity and independence for pseudo-algebraically closed fields.
\newblock In {\em Models and computability ({L}eeds, 1997)}, volume 259 of {\em
  London Math. Soc. Lecture Note Ser.}, pages 41--61. Cambridge Univ. Press,
  Cambridge, 1999.

\bibitem{Cha-Pi}
Z.~Chatzidakis and A.~Pillay.
\newblock Generic structures and simple theories.
\newblock {\em Ann. Pure Appl. Logic}, 95(1-3):71--92, 1998.

\bibitem{cd}
G.~Cherlin and M.~A. Dickmann.
\newblock Real closed rings. {II}. {M}odel theory.
\newblock {\em Ann. Pure Appl. Logic}, 25(3):213--231, 1983.

\bibitem{CHL}
G.~Cherlin, L.~Harrington, and A.~H. Lachlan.
\newblock {$\aleph_0$}-categorical, {$\aleph_0$}-stable structures.
\newblock {\em Ann. Pure Appl. Logic}, 28(2):103--135, 1985.

\bibitem{cpt}
A.~Chernikov, D.~Palacin, and K.~Takeuchi.
\newblock On {$n$}-dependence.
\newblock {\em Notre Dame J. Form. Log.}, 60(2):195--214, 2019.

\bibitem{CS}
A.~{Chernikov} and S.~{Starchenko}.
\newblock {Regularity lemma for distal structures}.
\newblock {\em Journal of the European Mathematical Society}, July 2015.

\bibitem{coag}
R.~CLUCKERS and I.~HALUPCZOK.
\newblock {QUANTIFIER} {ELIMINATION} {IN} {ORDERED} {ABELIAN} {GROUPS}.
\newblock {\em Confluentes Mathematici}, 03(04):587--615, dec 2011.

\bibitem{delon-no-density}
F.~Delon.
\newblock {$C$}-minimal structures without density assumption.
\newblock In {\em Motivic integration and its interactions with model theory
  and non-{A}rchimedean geometry. {V}olume {II}}, volume 384 of {\em London
  Math. Soc. Lecture Note Ser.}, pages 51--86. Cambridge Univ. Press,
  Cambridge, 2011.

\bibitem{DG}
A.~Dolich and J.~Goodrick.
\newblock Strong theories of ordered {A}belian groups.
\newblock {\em Fund. Math.}, 236(3):269--296.

\bibitem{DMS1}
A.~Dolich, C.~Miller, and C.~Steinhorn.
\newblock Structures having o-minimal open core.
\newblock {\em Trans. Amer. Math. Soc.}, 362(3):1371--1411, 2010.

\bibitem{DMS-generic}
A.~Dolich, C.~Miller, and C.~Steinhorn.
\newblock Extensions of ordered theories by generic predicates.
\newblock {\em J. Symbolic Logic}, 78(2):369--387, 2013.

\bibitem{DMS-Indepedent}
A.~Dolich, C.~Miller, and C.~Steinhorn.
\newblock Expansions of o-minimal structures by dense independent sets.
\newblock {\em Ann. Pure Appl. Logic}, 167(8):684--706, 2016.

\bibitem{duret}
J.-L. Duret.
\newblock Les corps faiblement alg\'{e}briquement clos non s\'{e}parablement
  clos ont la propri\'{e}t\'{e} d'ind\'{e}pendence.
\newblock In {\em Model theory of algebra and arithmetic ({P}roc. {C}onf.,
  {K}arpacz, 1979)}, volume 834 of {\em Lecture Notes in Math.}, pages
  136--162. Springer, Berlin-New York, 1980.

\bibitem{edmundo-otero}
M.~J. Edmundo and M.~Otero.
\newblock Definably compact abelian groups.
\newblock {\em J. Math. Log.}, 4(2):163--180, 2004.

\bibitem{E-small}
P.~Eleftheriou.
\newblock Small sets in mann pairs.
\newblock {\em arXiv:1812.07970}, 2018.

\bibitem{Elef-small-sets}
P.~E. Eleftheriou.
\newblock Small sets in dense pairs.
\newblock {\em Israel J. Math.}, 233(1):1--27, 2019.

\bibitem{erdos-zaran}
P.~Erd\H{o}s.
\newblock On extremal problems of graphs and generalized graphs.
\newblock {\em Israel J. Math.}, 2:183--190, 1964.

\bibitem{group-in-a-group}
D.~Evans, A.~Pillay, and B.~Poizat.
\newblock Le groupe dans le groupe.
\newblock {\em Algebra i Logika}, 29(3):368--378, 382, 1990.

\bibitem{flenner-guingona}
J.~Flenner and V.~Guingona.
\newblock Convexly orderable groups and valued fields.
\newblock {\em J. Symb. Log.}, 79(1):154--170.

\bibitem{fuchs}
L.~Fuchs.
\newblock {\em Abelian groups}.
\newblock Springer Monographs in Mathematics. Springer, Cham, 2015.

\bibitem{GoHi-Pairs}
A.~B. Gorman, P.~Hieronymi, and E.~Kaplan.
\newblock Pairs of theories satisfying a mordell-lang condition, 2018.

\bibitem{gh-op}
V.~Guingona and C.~D. Hill.
\newblock On a common generalization of {S}helah's 2-rank, dp-rank, and
  o-minimal dimension.
\newblock {\em Ann. Pure Appl. Logic}, 166(4):502--525, 2015.

\bibitem{gcs}
V.~Guingona, C.~D. Hill, and L.~Scow.
\newblock Characterizing model-theoretic dividing lines via collapse of
  generalized indiscernibles.
\newblock {\em Ann. Pure Appl. Logic}, 168(5):1091--1111, 2017.

\bibitem{Guingona-Parnes}
V.~Guingona and M.~Parnes.
\newblock Ranks based on algebraically trivial fraisse classes, 2020,
  arXiv:2007.02922.

\bibitem{GH-Dependent}
A.~G{\"u}naydin and P.~Hieronymi.
\newblock Dependent pairs.
\newblock {\em J. Symbolic Logic}, 76(2):377--390, 2011.

\bibitem{expanded}
Y.~Gurevich.
\newblock Expanded theory of ordered abelian groups.
\newblock {\em Ann. Math. Logic}, 12(2):193--228, 1977.

\bibitem{halevi-hasson-peterzil}
Y.~Halevi, A.~Hasson, and Y.~Peterzil.
\newblock Fields interpretable in p-minimal fields.
\newblock {\em arXiv:2103.15198}, 2021.

\bibitem{halevi-palacin}
Y.~Halevi and D.~Palac\'{\i}n.
\newblock The dp-rank of abelian groups.
\newblock {\em J. Symb. Log.}, 84(3):957--986, 2019.

\bibitem{james-hanson}
J.~Hanson.
\newblock Can superstability of a countable theory be characterized in terms of
  not weakly trace interpreting; a particular structure?
\newblock MathOverflow.
\newblock URL:https://mathoverflow.net/q/394205 (version: 2021-05-31).

\bibitem{C-minimal}
D.~Haskell and D.~Macpherson.
\newblock Cell decompositions of c-minimal structures.
\newblock {\em Ann. Pure Appl. Logic}, 66(2):113--162, 1994.

\bibitem{hasson-peterzil}
A.~Hasson and Y.~Peterzil.
\newblock Interpretable fields in real closed valued fields and some
  expansions.
\newblock {\em arXiv:2102.00814}, 2021.

\bibitem{hempel-chernkov}
N.~Hempel and A.~Chernikov.
\newblock On n-dependent groups and fields ii, 2020, arxiv:1912.02385.

\bibitem{Hodges}
W.~Hodges.
\newblock {\em Model theory}, volume~42 of {\em Encyclopedia of mathematics and
  its applications}.
\newblock Cambridge University Press, 1993.

\bibitem{HP-invariant}
E.~Hrushovski and A.~Pillay.
\newblock On {NIP} and invariant measures.
\newblock {\em J. Eur. Math. Soc. (JEMS)}, 13(4):1005--1061, 2011.

\bibitem{metastable}
E.~Hrushovski and S.~Rideau-Kikuchi.
\newblock Valued fields, metastable groups.
\newblock {\em Selecta Mathematica}, 25(3), July 2019.

\bibitem{szymon}
S.~T. (https://mathoverflow.net/users/87983/szymon toru
\newblock omega-categorical, omega-stable structure with trivial geometry not
  definable in the pure set.
\newblock MathOverflow.
\newblock URL:https://mathoverflow.net/q/326895 (version: 2019-05-18).

\bibitem{jahnke-when}
F.~Jahnke.
\newblock When does nip transfer from fields to henselian expansions?
\newblock {\em arXiv:1607.02953:}, 2019.

\bibitem{Johnson}
W.~Johnson.
\newblock On dp-minimal fields.
\newblock {\em arXiv:1507.02745}, 2015.

\bibitem{Johnson-dp-finite}
W.~Johnson.
\newblock Dp-finite fields vi: the dp-finite shelah conjecture.
\newblock {\em arXiv:2005.13989}, 2020.

\bibitem{1stpaper}
W.~Johnson, M.~C. Tran, E.~Walsberg, and J.~Ye.
\newblock The \'etale-open topology and the stable fields conjecture.
\newblock {\em accepted in J. Eur. Math. Soc. (JEMS), arXiv:2009.02319}.

\bibitem{dp-rank-additive}
I.~Kaplan, A.~Onshuus, and A.~Usvyatsov.
\newblock Additivity of the dp-rank.
\newblock {\em Trans. Amer. Math. Soc.}, 365(11):5783--5804, 2013.

\bibitem{Kaplan2011}
I.~Kaplan, T.~Scanlon, and F.~O. Wagner.
\newblock Artin-schreier extensions in {NIP} and simple fields.
\newblock {\em Israel Journal of Mathematics}, 185(1):141--153, Sept. 2011.

\bibitem{Keren-thesis}
G.~Keren.
\newblock Definable compactness in weakly o-minimal structures.
\newblock Master's thesis, Ben Gurion University of the Negev, 2014.

\bibitem{Kruckman-thesis}
A.~Kruckman.
\newblock {\em Infinitary limits of finitary structures}.
\newblock PhD thesis, 2016.

\bibitem{KRExp}
A.~Kruckman and N.~Ramsey.
\newblock Generic expansion and {S}kolemization in {${\rm NSOP}_1$} theories.
\newblock {\em Ann. Pure Appl. Logic}, 169(8):755--774, 2018.

\bibitem{lachlan-order}
A.~H. Lachlan.
\newblock Structures coordinatized by indiscernible sets.
\newblock volume~34, pages 245--273. 1987.
\newblock Stability in model theory (Trento, 1984).

\bibitem{Laskowski-shelah-karp}
M.~Laskowski and S.~Shelah.
\newblock Karp complexity and classes with the independence property.
\newblock {\em Annals of Pure and Applied Logic}, 120(1-3):263--283, Apr. 2003.

\bibitem{LasStein}
M.~C. Laskowski and C.~Steinhorn.
\newblock On o-minimal expansions of {A}rchimedean ordered groups.
\newblock {\em J. Symbolic Logic}, 60(3):817--831, 1995.

\bibitem{Lopez}
J.~P.~A. L\'{o}pez.
\newblock Groups definable in presburger arithmetic.
\newblock {\em arXiv:1904.00321}, 2019.

\bibitem{loveys}
J.~Loveys and Y.~Peterzil.
\newblock Linear o-minimal structures.
\newblock {\em Israel J. Math.}, 81(1-2):1--30, 1993.

\bibitem{Macintyre-omegastable}
A.~Macintyre.
\newblock On {$\omega _{1}$}-categorical theories of fields.
\newblock {\em Fund. Math.}, 71(1):1--25. (errata insert), 1971.

\bibitem{macintyre-p-adic}
A.~MacIntyre.
\newblock On definable subsets of p-adic fields.
\newblock {\em The Journal of Symbolic Logic}, 41(3):605, Sept. 1976.

\bibitem{macintyre-generic}
A.~Macintyre.
\newblock Generic automorphisms of fields.
\newblock volume~88, pages 165--180. 1997.
\newblock Joint AILA-KGS Model Theory Meeting (Florence, 1995).

\bibitem{macpherson-interpreting-groups}
D.~Macpherson.
\newblock Interpreting groups in {$\omega$}-categorical structures.
\newblock {\em J. Symbolic Logic}, 56(4):1317--1324, 1991.

\bibitem{macpherson-survey}
D.~Macpherson.
\newblock A survey of homogeneous structures.
\newblock {\em Discrete Math.}, 311(15):1599--1634, 2011.

\bibitem{MMS-weak}
D.~Macpherson, D.~Marker, and C.~Steinhorn.
\newblock Weakly o-minimal structures and real closed fields.
\newblock {\em Trans. Amer. Math. Soc.}, 352(12):5435--5483, 2000.

\bibitem{Ma-adic}
N.~Mariaule.
\newblock Model theory of the field of $p$-adic numbers expanded by a
  multiplicative subgroup.
\newblock {\em arXiv:1803.10564}, 2018.

\bibitem{MPP}
D.~Marker, Y.~Peterzil, and A.~Pillay.
\newblock Additive reducts of real closed fields.
\newblock {\em J. Symbolic Logic}, 57(1):109--117, 1992.

\bibitem{mekler-rubin-steinhorn}
A.~Mekler, M.~Rubin, and C.~Steinhorn.
\newblock Dedekind completeness and the algebraic complexity of {$o$}-minimal
  structures.
\newblock {\em Canad. J. Math.}, 44(4):843--855, 1992.

\bibitem{Milliken}
K.~R. Milliken.
\newblock A {R}amsey theorem for trees.
\newblock {\em J. Combin. Theory Ser. A}, 26(3):215--237, 1979.

\bibitem{samaria-imaginary}
S.~Montenegro.
\newblock Imaginaries in bounded pseudo real closed fields.
\newblock {\em Ann. Pure Appl. Logic}, 168(10), 2017.

\bibitem{M-prc}
S.~Montenegro.
\newblock Pseudo real closed fields, pseudo {$p$}-adically closed fields and
  {${\rm NTP}_2$}.
\newblock {\em Ann. Pure Appl. Logic}, 168(1):191--232, 2017.

\bibitem{nese}
J.~Nesetril, P.~O. de~Mendez, M.~Pilipczuk, R.~Rabinovich, and S.~Siebertz.
\newblock Rankwidth meets stability.
\newblock {\em CoRR}, abs/2007.07857, 2020.

\bibitem{onshuus-simon}
A.~Onshuus and P.~Simon.
\newblock Dependent finitely homogneneous rosy structures, 2021,
  arxiv:2107.02727.

\bibitem{Onshuus-pres}
A.~Onshuus and M.~Vicar\'{\i}a.
\newblock Definable groups in models of {P}resburger arithmetic.
\newblock {\em Ann. Pure Appl. Logic}, 171(6):102795, 27, 2020.

\bibitem{OPP-groups-rings}
M.~Otero, Y.~Peterzil, and A.~Pillay.
\newblock On groups and rings definable in o-minimal expansions of real closed
  fields.
\newblock {\em Bull. London Math. Soc.}, 28(1):7--14, 1996.

\bibitem{palacin-sklinos}
D.~Palac\'{\i}n and R.~Sklinos.
\newblock On superstable expansions of free {A}belian groups.
\newblock {\em Notre Dame J. Form. Log.}, 59(2):157--169, 2018.

\bibitem{o-minimal-simple}
Y.~Peterzil, A.~Pillay, and S.~Starchenko.
\newblock Simple algebraic and semialgebraic groups over real closed fields.
\newblock {\em Trans. Amer. Math. Soc.}, 352(10):4421--4450, 2000.

\bibitem{PS-Tri}
Y.~Peterzil and S.~Starchenko.
\newblock A trichotomy theorem for o-minimal structures.
\newblock {\em Proc. London Math. Soc. (3)}, 77(3):481--523, 1998.

\bibitem{one-dim-subgroup}
Y.~Peterzil and C.~Steinhorn.
\newblock Definable compactness and definable subgroups of o-minimal groups.
\newblock {\em J. London Math. Soc. (2)}, 59(3):769--786, 1999.

\bibitem{pillay-p-adic}
A.~Pillay.
\newblock On fields definable in {${\bf Q}_p$}.
\newblock {\em Arch. Math. Logic}, 29(1):1--7, 1989.

\bibitem{pillay-book}
A.~Pillay.
\newblock {\em Geometric stability theory}, volume~32 of {\em Oxford Logic
  Guides}.
\newblock The Clarendon Press, Oxford University Press, New York, 1996.
\newblock Oxford Science Publications.

\bibitem{poizat-propos}
B.~Poizat.
\newblock \`a propos de groupes stables.
\newblock In {\em Logic colloquium '85 ({O}rsay, 1985)}, volume 122 of {\em
  Stud. Logic Found. Math.}, pages 245--265. North-Holland, Amsterdam, 1987.

\bibitem{Poizat}
B.~Poizat.
\newblock {\em A course in model theory}.
\newblock Universitext. Springer-Verlag, New York, 2000.
\newblock An introduction to contemporary mathematical logic, Translated from
  the French by Moses Klein and revised by the author.

\bibitem{ramakrishnan}
J.~Ramakrishnan.
\newblock Definable linear orders definably embed into lexicographic orders in
  o-minimal structures.
\newblock {\em Proc. Amer. Math. Soc.}, 141(5):1809--1819, 2013.

\bibitem{razenj}
V.~Razenj.
\newblock One-dimensional groups over an {$o$}-minimal structure.
\newblock {\em Ann. Pure Appl. Logic}, 53(3):269--277, 1991.

\bibitem{robinson-zakon}
A.~Robinson and E.~Zakon.
\newblock Elementary properties of ordered abelian groups.
\newblock {\em Trans. Amer. Math. Soc.}, 96:222--236, 1960.

\bibitem{rubin}
M.~Rubin.
\newblock Theories of linear order.
\newblock {\em Israel J. Math.}, 17:392--443, 1974.

\bibitem{schmerl}
J.~H. Schmerl.
\newblock Countable homogeneous partially ordered sets.
\newblock {\em Algebra Universalis}, 9(3):317--321, 1979.

\bibitem{Sch.habil}
P.~H. Schmitt.
\newblock {\em Model theory of ordered abelian groups}.
\newblock Habilitationsschrift, 1982.

\bibitem{scow}
L.~Scow.
\newblock Characterization of {NIP} theories by ordered graph-indiscernibles.
\newblock {\em Ann. Pure Appl. Logic}, 163(11):1624--1641, 2012.

\bibitem{Shelah-external}
S.~Shelah.
\newblock Dependent first order theories, continued.
\newblock {\em Israel J. Math.}, 173:1--60, 2009.

\bibitem{Simon-Book}
P.~Simon.
\newblock {\em A guide to NIP theories}, volume~44 of {\em Lecture Notes in
  Logic}.
\newblock Cambridge University Press.

\bibitem{Simon-dp}
P.~Simon.
\newblock On dp-minimal ordered structures.
\newblock {\em J. Symbolic Logic}, 76(2):448--460, 2011.

\bibitem{Pierre2}
P.~Simon.
\newblock Linear orders in {NIP} theories.
\newblock {\em arXiv:1807.07949}, 2018.

\bibitem{Pierre}
P.~Simon.
\newblock {NIP} omega-categorical structures: the rank 1 case.
\newblock {\em arXiv:1807.07102}, 2018.

\bibitem{on-forking}
P.~Simon and S.~Starchenko.
\newblock On forking and definability of types in some {DP}-minimal theories.
\newblock {\em J. Symb. Log.}, 79(4):1020--1024, 2014.

\bibitem{SW-dp}
P.~Simon and E.~Walsberg.
\newblock Dp and other minimalities.
\newblock {\em Preprint}, arXiv:1909.05399, 2019.

\bibitem{strzebonski}
A.~W. Strzebonski.
\newblock Euler characteristic in semialgebraic and other {${\rm o}$}-minimal
  groups.
\newblock {\em J. Pure Appl. Algebra}, 96(2):173--201, 1994.

\bibitem{Sw}
S.~\'Swierczkowski.
\newblock On cyclically ordered groups.
\newblock {\em Fund. Math.}, 47:161--166, 1959.

\bibitem{thomas-hyper}
S.~Thomas.
\newblock Reducts of random hypergraphs.
\newblock {\em Ann. Pure Appl. Logic}, 80(2):165--193, 1996.

\bibitem{sousa}
S.~Torrez\~{a}o~de Sousa and J.~K. Truss.
\newblock Countable homogeneous coloured partial orders.
\newblock {\em Dissertationes Math.}, 455:48, 2008.

\bibitem{TW-cyclic}
M.~C. Tran and E.~Walsberg.
\newblock A family of dp-minimal expansions of $(\mathbb{Z};+)$, 2017.

\bibitem{lou-book}
L.~van~den Dries.
\newblock {\em Tame topology and o-minimal structures}, volume 248 of {\em
  London Mathematical Society Lecture Note Series}.
\newblock Cambridge University Press, Cambridge, 1998.

\bibitem{vddG}
L.~van~den Dries and A.~G\"{u}nayd\i~n.
\newblock The fields of real and complex numbers with a small multiplicative
  group.
\newblock {\em Proc. London Math. Soc. (3)}, 93(1):43--81, 2006.

\bibitem{big-nip}
E.~Walsberg.
\newblock Externally definable quotients and nip expansions of the real ordered
  additive group, 2019, arXiv:1910.10572.

\bibitem{ez-fields}
E.~Walsberg and J.~Ye.
\newblock \'{E}z fields.
\newblock {\em arXiv preprint arXiv:2103.06919}, 2021.

\bibitem{Weispfenning}
V.~Weispfenning.
\newblock Elimination of quantifiers for certain ordered and lattice-ordered
  abelian groups.
\newblock {\em Bull. Soc. Math. Belg. S\'{e}r. B}, 33(1):131--155, 1981.

\bibitem{Wencel-1}
R.~Wencel.
\newblock Weakly o-minimal nonvaluational structures.
\newblock {\em Ann. Pure Appl. Logic}, 154(3):139--162, 2008.

\bibitem{harry-west}
H.~West.
\newblock Does the random graph interpret the random directed graph?
\newblock MathOverflow.
\newblock URL:https://mathoverflow.net/q/398026 (version: 2021-07-21).

\bibitem{Winkler}
P.~M. Winkler.
\newblock Model-completeness and {S}kolem expansions.
\newblock pages 408--463. Lecture Notes in Math., Vol. 498, 1975.

\bibitem{zakon}
E.~Zakon.
\newblock Model-completeness and elementary properties of torsion free abelian
  groups.
\newblock {\em Canadian J. Math.}, 26:829--840, 1974.

\bibitem{zap-pak}
A.~Zapryagaev and F.~Pakhomov.
\newblock Interpretations of {P}resburger arithmetic in itself.
\newblock In {\em Logical foundations of computer science}, volume 10703 of
  {\em Lecture Notes in Comput. Sci.}, pages 354--367. Springer, Cham, 2018.

\end{thebibliography}
\end{document}